\documentclass[a4paper,11pt]{amsart}

\usepackage{amssymb, amsfonts, latexsym, amsthm, amsmath, verbatim}
\pdfoutput=1
\usepackage{marginnote} 
\usepackage{array,longtable}
\usepackage{color}
\usepackage[all]{xy}
\usepackage{amscd}
\usepackage{mathrsfs} 
\usepackage[english]{babel}
\usepackage{tikz}
\usepackage{graphicx} 
\usepackage{standalone}
\usepackage{epstopdf}
\usepackage[paper=a4paper, text={155mm,218mm},centering]{geometry}
\definecolor{darkblue}{rgb}{0,0,0.4} 
\usepackage[colorlinks=true, citecolor=darkblue, filecolor=darkblue, linkcolor=darkblue,urlcolor=darkblue]{hyperref}
\usetikzlibrary{arrows,cd,decorations.pathmorphing,decorations.pathreplacing}
\usepackage{amscd}
\usepackage{enumitem}
\usepackage{transparent}
\usepackage{cancel}
\usepackage{tcolorbox}
\usepackage{makecell}
\tikzstyle{crossing}=[circle,fill=white,minimum height=6pt,inner sep=0pt, outer sep=0pt, style={transform shape=false}]

\usepackage{caption} 
\usepackage{thm-restate}
\usepackage{cleveref}

\numberwithin{equation}{section}

\theoremstyle{plain}
\newtheorem{thm}[equation]{Theorem}
\newtheorem{theorem}[equation]{Theorem}

\newtheorem*{thm*}{Theorem}
\newtheorem{corollary}[equation]{Corollary}
\newtheorem{lem}[equation]{Lemma}
\newtheorem{lemma}[equation]{Lemma}

\theoremstyle{definition}

\newtheorem{remark}[equation]{Remark}
\newtheorem{example}[equation]{Example}

\newtheorem{defn}[equation]{Definition}
\newtheorem{definition}[equation]{Definition}
\newtheorem{con}[equation]{Condition}

\numberwithin{figure}{section}

\usepackage{etoolbox}
\def\do#1{\csdef{c#1}{\mathcal{#1}}}
\docsvlist{A,B,C,D,E,F,G,H,I,J,K,L,M,N,O,P,Q,R,S,T,U,V,W,X,Y,Z}
\def\do#1{\csdef{#1}{\mathbb{#1}}}
\docsvlist{N,Z,Q,C,R,F,A,W,M}
\def\do#1{\csdef{#1#1}{\mathbf{#1}}}
\docsvlist{A,B,C,D,E,F,G,H,I,J,K,L,M,N,O,P,Q,R,S,T,U,V,W,X,Y,Z,a,b,c,d,e,f,h,i,j,k,m,n,o,p,q,r,s,t,u,v,w,x,y,z}
\def\do#1{\csdef{m#1}{\mathfrak{#1}}}
\docsvlist{a,b,e,j,m,E,I,J,K}
\def\do#1{\csdef{#1}{\operatorname{#1}}}
\docsvlist{atan,codim,Fix,Hom,icodim, Id, Int,lk, Mod,pcodim,Sym, supp}
\def\do#1{\csdef{sc#1}{\mathscr{#1}}}
\docsvlist{M,S}

\newcommand{\Sfree}{\mathcal{S}^-}
\newcommand{\Sfix}{\mathcal{S}^{\Z_2}}
\newcommand{\Sigmafree}{\Sigma^-}
\newcommand{\Sigmafix}{\Sigma^{\Z_2}}
\newcommand{\wt}[1]{\widetilde{#1}}

\newcommand{\ol}[1]{\overline{#1}}

\newcommand{\wtimes}{\boxtimes}

\usepackage{imakeidx}

\makeatletter
\newcommand*{\wackyenum}[1]{%
  \expandafter\@wackyenum\csname c@#1\endcsname%
}

\newcommand*{\@wackyenum}[1]{%
\ifcase#1\or{IR-1}\or{IR-2}\or{IR-3}\or{R-1}\or{R-2}\or{M-1}\or{M-2}\or{M-3}\else\@ctrerr\fi%
}
\AddEnumerateCounter{\wackyenum}{\@wackyenum}{53.13}
\makeatother

\newcounter{nparcount}
\newcounter{sparcount}
\def\npar#1#2{\stepcounter{sparcount}%
  \colorbox{green!10}{$?^{\arabic{sparcount}}$}\marginpar{\fcolorbox{green}{green!10}{\begin{minipage}{2cm}\fontsize{5}{6}\selectfont\color{green!20!black}${}^{\arabic{sparcount}}$#1\end{minipage}}\\ \fcolorbox{blue}{blue!10}{\begin{minipage}{2cm}\fontsize{5}{6}\selectfont\color{green!20!black}
${}^{\arabic{sparcount}}$#2\end{minipage}}}}
\def\ypar#1{\stepcounter{nparcount}%
  \colorbox{yellow!10}{$!^{\arabic{nparcount}}$}%
  \marginpar{\fcolorbox{yellow}{yellow!10}{\begin{minipage}{2cm}\fontsize{5}{6}\selectfont\color{green!20!black}
${}^{\arabic{nparcount}}$#1\end{minipage}}}}
\def\mpar#1{\stepcounter{nparcount}%
  \colorbox{yellow!5!red!5}{$!^{\arabic{nparcount}}$}%
  \marginpar{\fcolorbox{red}{orange!10}{\begin{minipage}{2cm}\fontsize{5}{6}\selectfont\color{green!20!black}
${}^{\arabic{nparcount}}$#1\end{minipage}}}}

\makeatletter

\def\@setref#1#2#3{%
  \ifx#1\relax
   \protect\G@refundefinedtrue
   \nfss@text{\colorbox{green!30}{[REF]}}%
   \@latex@warning{Reference '#3' on page \thepage \space
             undefined}%
  \else
   \expandafter#2#1\null
  \fi}
\def\@citex[#1]#2{\leavevmode
  \let\@citea\@empty
  \@cite{\@for\@citeb:=#2\do
    {\@citea\def\@citea{,\penalty\@m\ }%
     \edef\@citeb{\expandafter\@firstofone\@citeb\@empty}%
     \if@filesw\immediate\write\@auxout{\string\citation{\@citeb}}\fi
     \@ifundefined{b@\@citeb}{\colorbox{blue!30}{\reset@font REF}%
       \G@refundefinedtrue
       \@latex@warning
         {Citation `\@citeb' on page \thepage \space undefined}}%
       {\@cite@ofmt{\csname b@\@citeb\endcsname}}}}{#1}}
\makeatother

\def\mpar#1{\relax}
\def\npar#1#2{\relax}
\def\ypar#1{\relax}
\title{Reidemeister and Movie moves for involutive links}

\author{Maciej Borodzik}
\address{Institute of Mathematics\\ University of Warsaw
\\Warsaw, Poland}
\email{mcboro@mimuw.edu.pl}

\author{Irving Dai}
\address{Department of Mathematics\\The University of Texas at Austin\\ Austin, TX, USA}
\email{irving.dai@math.utexas.edu}

\author{Abhishek Mallick}
\address{Department of Mathematics\\Dartmouth College \\ Hanover, NH, USA}
\email{abhishek.mallick@dartmouth.edu}

\author{Matthew Stoffregen}
\address{Department of Mathematics\\Michigan State University\\ East Lansing, MI, USA}
\email{stoffre1@msu.edu}

\makeindex[title=Index of Notation, columns=2]
\begin{document}


\begin{abstract}
An involutive link is a link which is invariant under the standard rotation by 180 degrees in $S^3$. We establish an equivariant analogue of the work of Carter and Saito aimed at studying equivariant cobordisms between involutive links. This gives a set of $39$ equivariant movie moves that suffice to go between any two movie presentations of a pair of equivariantly isotopic cobordisms. Along the way, we give a singularity-theoretic proof of the equivariant Reidemeister theorem and study loops of equivariant Reidemeister moves. Our approach proceeds by analyzing codimension $2$ singularities of equivariant maps from $S^1$ to $\R^2$, as well as utilizing embedded equivariant Morse theory.
\end{abstract}
\maketitle
\tableofcontents

\section{Introduction}
\subsection{Motivation}
Equivariant knots and links are well-studied \cite{ChaKo, ND, Murasugi, Naik, Sakuma} and have appeared in a broad range of other topological contexts. Recent research in this area has utilized both classical techniques (see for example \cite{boyle2025equivariantunknottingnumbersstrongly,DiPrisa0,DiPrisa,Merz,MillerPowell,zampa2025boundequivariantunknottingnumber}) and gauge, symplectic and quantum invariants (see for example \cite{borodzik2025khovanov,dai2023equivariant, LSinvertible,lobb-watson,mallick2024knot,Sano,Watson}).  
Several of these papers have focused on questions involving equivariant surfaces and equivariant cobordisms. 
Results in this direction have often been aimed at obstructing equivariant concordance or bounding quantities such as the equivariant slice genus from below.

In prior work \cite{borodzik2025khovanov}, the authors constructed a $\Z_2$-equivariant version of Khovanov homology for strongly invertible knots and showed that given an equivariant cobordism (with a specified equivariant movie decomposition), there exists a cobordism map between the equivariant invariants associated to its ends. This was used to provide equivariant slice genus bounds and establish non-existence results regarding equivariant surfaces. However, for more refined questions, additional functoriality results must be established. For example, to \textit{distinguish} pairs of equivariant surfaces up to equivariant isotopy, we need to know that the homotopy type of our cobordism map is independent of all choices and preserved under equivariant isotopy. Such problems thus have both theoretical and practical interest.

A similar story has, of course, already played out in the non-equivariant setting. In \cite{khovanov}, Khovanov showed that given a presentation of a surface cobordism, one can construct a corresponding cobordism map on Khovanov homology. This was subsequently shown to be well-defined and invariant under isotopies of link cobordisms in $\mathbb{R}^3\times[0,1]$: first by 
Jacobsson \cite{jacobsson} (up to overall sign, also followed by \cite{khovanov-naturality,bar-natan-tangle}), with improved statements in \cite{blanchet,clark-morrison-walker} and generalizations to $\mathfrak{sl}(N)$
homologies in \cite{ETW}. Full functoriality in $S^3 \times I$ was established in \cite{Garfield} and played a foundational role in Hayden and Sundberg's work on exotic slice disks \cite{hayden-sundberg} and in defining skein
lasagna modules \cite{Garfield,Lasagna}; compare \cite{wedrich2025linkhomologytopologicalquantum}. 
Establishing functoriality in this sense has thus been essential for applications of Khovanov homology to low-dimensional topology.

In the non-equivariant case, cobordism maps on Khovanov homology are constructed by first giving a presentation for the cobordism as a \textit{movie} of Reidemeister moves and elementary cobordisms. We associate to each such piece a corresponding elementary cobordism map and define the overall cobordism map to be the composition of these. To establish well-definedness and/or isotopy invariance, one needs to have a comprehensive set of \textit{movie moves} such that any two movies for the same isotopy class of cobordism are related by a sequence of these moves. This is a foundational result of Carter and Saito \cite{CS}, based on \cite{Giller} (see also \cite{CRS,Roseman}). One then checks that if two movies are related by a movie move, then their associated maps on Khovanov homology are unchanged (in this case, up to overall sign).

The aim of the present paper is to carry out the topological side of this program in the $\Z_2$-equivariant setting. We concern ourselves with the general case of links which are invariant under the standard rotation by 180 degrees in $S^3$. Such links are called \textit{involutive links}.\footnote{The notion of an involutive link encompasses several other studies of equivariant links in the literature. Many other papers place restrictions on how the involution interacts with various link components. (For example, in terms of preserving/reversing orientation, the number of fixed points on each component, and so on.) Our only restriction is that the involution on $S^3$ is the standard rotation. In fact, it follows from the resolution of the Smith conjecture \cite{Waldhausen, MorganBass} that any orientation-preserving involution on $S^3$ with non-empty fixed-point set is conjugate to the standard rotation.} We first give a rigorous proof that any pair of diagrams for an involutive link are connected by a sequence of generalized equivariant Reidemeister moves, see Theorem~\ref{thm:intro1}. This is comparable to the classical Reidemeister theorem in the non-equivariant category. Moreover, we show that an equivariant cobordism between two involutive links can be decomposed into an equivariant movie, see Theorem~\ref{thm:intro2}. These claims were implicitly used in \cite{borodzik2025khovanov} to define equivariant Khovanov homology and construct equivariant cobordism maps. 

The bulk of the paper is then concerned with comparing pairs of equivariant cobordisms. We give a list of equivariant elementary movie moves such that if two equivariant movies represent the same equivariant isotopy class of cobordism, then they are related by a sequence of these moves, see Theorem~\ref{thm:intro3}. As a special case, we study loops of equivariant Reidemeister moves, see Theorem~\ref{thm:intro4}. These results represent the first step towards proving full functoriality of the cobordism maps for equivariant Khovanov homology, see \cite{borodzik2025khovanov}.

Like \cite{CS}, our results rely heavily on the machinery of singularity theory. However, the details differ significantly. The approach of Carter and Saito, is via studying singularities of surfaces. We present a somewhat longer (but perhaps more intuitive) approach. For movies that do not change the isotopy type of the underlying link (such as for a sequence of Reidemeister moves), we use the classification of codimension 2 singularities of equivariant maps from the circle to the plane. In this approach, we follow \cite{Perestroikas,David,FiedlerKurlin,Wall_Pgen}, porting their results
to the equivariant setting via \cite{Wall_jet}. For movies associated with a change of topology, we study equivariant embedded Morse theory,
alloying the results of \cite{Wassermann} with \cite{BP} and using also \cite{Akaho,BB} for studying families of equivariant Morse functions.

\subsection{Main results}
Our first result enumerates equivariant Reidemeister moves. View $S^3$ as $\R^3$ together with the point at $\infty$, and consider rotation about the $y$-axis. Our convention in this paper will be to fix the standard projection from $\mathbb{R}^3$ onto the $xy$-plane. Under this projection, an involutive link (which we always assume to be in generic position) gives a \textit{transvergent diagram}.  Equivariantly isotopic link diagrams
are connected by a sequence of equivariant Reidemeister moves; see \cite{lobb-watson,Sano}. We give a singularity theory based proof of
the following result.

\begin{restatable}[see \cite{lobb-watson,Sano}]{thm}{introone}\label{thm:intro1}
Let $L_0$ and $L_1$ be a pair of involutive links. An equivariant isotopy from $L_0$ to $L_1$ induces, after perturbation, a sequence of the following moves:
\begin{itemize}
  \item Equivariant planar isotopies;
  \item Off-axis Reidemeister moves (IR-1), (IR-2), (IR-3);
  \item On-axis Reidemeister moves (R-1), (R-2);
  \item On-axis mixed moves (M-1), (M-2), (M-3);
  \item The non-local I-move of Figure~\ref{fig:ISmove};
\end{itemize}
connecting their diagrams. See Figure~\ref{fig:AllReidemeister}. If the isotopy is in $\R^3$, the I-move is not needed.
\end{restatable}

The next result deals with cobordisms. Fix the standard product involution on $S^3 \times [0, 1]$ given by rotation about the $y$-axis in every $S^3$-slice. 
The following result is a simplified version of Theorem~\ref{thm:526}, proved in Section~\ref{sub:applications}.
\begin{thm}[see Theorem~\ref{thm:526}]\label{thm:intro2}
Let $L_0$ and $L_1$ be a pair of involutive links. Suppose $\Sigma \subset S^3\times[0,1]$ is an equivariant cobordism from $L_0$ to $L_1$ such that the $\Z_2$-action
  on $\Sigma$ has no isolated fixed points. Then $\Sigma$ is equivariantly isotopic rel boundary to an equivariant cobordism $\Sigma_0$, where $\Sigma_0$ is built by stacking elementary equivariant cobordisms of the forms:
  \begin{itemize}
    \item Equivariant isotopies and equivariant Reidemeister moves as in Theorem~\ref{thm:intro1};
    \item Births, saddles and deaths off-axis, see Figure~\ref{fig:outside_crits};
    \item Births, saddles and deaths on-axis, see Figure~\ref{fig:inside_crits}.
  \end{itemize}
  If $\Sigma$ has isolated fixed points, then additionally, singular births, saddles, or deaths on-axis (see Figures~\ref{fig:iso_fix3} and~\ref{fig:iso_fix4}) may occur.
\end{thm}
The assumption on the lack of isolated fixed points is satisfied, for instance, if the $\Z_2$-action inverts the orientation of every connected component
of the cobordism. (As an example, this is the case for a connected cobordism between a pair of strongly invertible knots.)

We refer to births, saddles and deaths mentioned in the statement of Theorem~\ref{thm:intro2} as \emph{equivariant Morse moves}.

Next, we discuss how two sequences of Reidemeister moves may be compared. In Section~\ref{sec:codim2}, we give $18$ loops of equivariant Reidemeister moves. An \textit{elementary loop replacement} consists of replacing a contiguous fragment of one of these loops by its complementary fragment, going in the opposite direction. These loops are summarized in Table~\ref{tab:local_loops}.

\begin{restatable}{thm}{introfour}\label{thm:intro4}
Let $L_0$ and $L_1$ be a pair of involutive links. Suppose we have two equivariant isotopies from $L_0$ to $L_1$ and let $M_0$ and $M_1$ be corresponding sequences of equivariant Reidemeister moves afforded by Theorem~\ref{thm:intro1}. If our equivariant isotopies can be connected by an equivariant isotopy-of-isotopies, then $M_0$ and $M_1$ can be transformed into each other by a sequence of the following:
  \begin{itemize}
    \item Changing the order of Reidemeister moves performed at different places;
    \item Inserting/deleting a pair of mutually inverse Reidemeister moves;
    \item One of 18 elementary loop replacement moves from Table~\ref{tab:local_loops}.
  \end{itemize}
\end{restatable}

Finally, we describe the elementary movie moves for equivariant cobordisms. 
\begin{thm}\label{thm:intro3}
  Let $L_0$ and $L_1$ be a pair of involutive links. Suppose $\Phi_s\colon \Sigma\to S^3\times[0,1]$, $s\in[0,1]$ is a family of equivariant cobordisms from $L_0$ to $L_1$ such that $\Phi_s(\Sigma)$ has no isolated fixed points of $\tau$ for each $s$. Let $M_0$ and $M_1$ be movies for $\Phi_0$ and $\Phi_1$ afforded by Theorem~\ref{thm:intro2}. Then $M_0$ and $M_1$ can be transformed into each other by elementary movie moves associated with the following:
  \begin{itemize}
    \item Commuting two moves occurring at different places, e.g. an equivariant Reidemeister move and an equivariant Morse move as in Subsection~\ref{sub:col_deg};
    \item A loop of Reidemeister moves of Theorem~\ref{thm:intro4};
    \item A loop of Reidemeister moves around the point at infinity, see Subsection~\ref{sub:iloop};
    \item A loop of equivariant Morse moves around the point at infinity as in Subsection~\ref{sub:ml_morse_inf};
    \item A loop of equivariant Morse moves around the point on the axis as in Subsection~\ref{sub:morse_online};
    \item A loop related to a singular equivariant Morse move as in Subsection~\ref{sub:twist_morse};
    \item A loop related to an equivariant Morse move over another point in the diagram as in Subsection~\ref{sub:morse_on_diagram};
    \item A loop of equivariant Morse moves related to a failure of the Morse condition, see Subsection~\ref{sub:morse_loop}. 
  \end{itemize}
\end{thm}
Theorem~\ref{thm:intro3} is proved as Theorem~\ref{thm:main_loop} in Section~\ref{sec:carter_saito}, where technical details
regarding Theorem~\ref{thm:intro3} are also explained. In particular,
we describe
in Subsection~\ref{sub:list} the relation between loops of moves of Theorem~\ref{thm:intro3} and elementary movie moves.

\begin{remark}
  Studying movie moves for cobordisms with isolated fixed points of the $\tau$-action requires significantly more involved
  tools, such as an analysis of higher codimension singularities of equivariant maps from $\R^2$ to $\R^3$. These are beyond the scope of this paper.
\end{remark}

\subsection{Main methods}
We now discuss some of the central methods used in the present work.

\subsubsection{Stratification induced by singularities}

A large portion of this paper -- especially our results concerning Reidemeister moves -- is based on understanding the space $\cF$ of smooth equivariant maps from a disjoint union of circles $\cS$ to $\R^2$. The idea will be to stratify $\cF$ and investigate the behavior of families of maps as strata are crossed. As a starting point, if we require
that the image is a link diagram (no cusps, and no tangencies, no triple points), then we obtain the space $\cF^0 \subset \cF$ of \emph{regular} maps, which
is open-dense. Thus, a generic map can be assumed to be regular.

However, a one-parameter family of maps will not usually stay inside $\cF^0$, but will instead cross its complement
\[
\wt{\cF}^1 = \cF \setminus \cF^0.
\]
We thus find a subset $\cF^1 \subset \wt{\cF}^1$ whose complement in $\wt{\cF}^1$ has codimension $2$. Any one-parameter family of maps can then be perturbed so that all intersections with $\smash{\wt{\cF}^1}$ occur at points of $\cF^1$. The space $\cF^1$ will be constructed as the union of strata $\cF^1_i$, each of which describes the failure of one of the regularity conditions, but which limits the failure to be as mild as possible. For instance, we might consider the space of maps with triple points, but these triple points are required to be ordinary, and no quadruple points are allowed. 

Importantly, this will allow us to write down a model singularity from each $\cF^1_i$, called a \textit{normal form}. In fact, in each case we will be able to write down a model one-dimensional deformation of such a singularity, called a \textit{versal deformation}. These versal deformations will allow us to understand how our one-parameter family crosses $\cF^1$. The qualitative change in behavior as each $\cF^1_i$ is crossed gives rise to the equivariant Reidemeister moves.

A similar story holds for two-parameter families. While a one-parameter family will generically stay inside $\cF^0 \cup \cF^1$, a two-parameter family of maps into $\cF$ will usually cross
\[
\wt{\cF}^2 = \cF \setminus (\cF^0 \cup \cF^1).
\]
We thus find a subset $\cF^2 \subset \wt{\cF}^2$ whose complement in $\wt{\cF}^2$ has codimension $3$. Any one-parameter family of maps can then be perturbed so that all intersections with $\smash{\wt{\cF}^2}$ occur at points of $\cF^2$. As before, $\cF^2$ will be constructed as the union of strata $\cF^2_i$, each describing a particular singularity which is (roughly speaking) the next level of severity as compared to the singularities in $\cF^1$. 

Once again, we write down a normal form and versal deformation for each $\cF^2_i$. In this case, the versal deformation will be two-dimensional, as it describes how a generic two-dimensional family intersects the codimension $2$ subset $\cF^2$. We think of the parameter space of the versal deformation as a disk, where the center corresponds to a singularity in some $\cF^2_i$. Generally, our deformation will continue to exhibit milder singularities along arcs radiating outwards from the center of the disk, corresponding to singularities of the various $\cF^1_i$. (The set of parameters along which a deformation is singular is called the \textit{discriminant locus}; see Section~\ref{sub:versal}.) Going around the boundary of the disk gives a loop of equivariant Reidemeister moves. These are precisely the loops appearing in the statement of Theorem~\ref{thm:intro4}.

\subsubsection{Equivariant Morse theory}
The second tool that we use -- especially in the context of analyzing equivariant cobordisms -- is embedded equivariant Morse theory. We briefly discuss some of the more subtle points which arise in this context. 

For the moment, let us consider the classical case of non-equivariant links and cobordisms. Let $\Sigma\subset \R^3\times[0,1]$ be a cobordism and consider the height function $f\colon\Sigma\to[0,1]$ given by projection onto the second factor of $\R^3\times[0,1]$. If $\Sigma$ is in general position,
then $f$ is Morse. As it is well-known, a critical point of $f$ induces a change of the topology of the link. It is usually accepted that an index zero critical point of $f$ corresponds to adding a round circle to the diagram of the link, while an index one critical point results in performing a standard saddle move. Finally, an index two critical point of $f$ destroys a circle. 

However, it turns out that there is a subtlety in this discussion which arises especially in the study of families of cobordisms. 
To explain this, let $z\in\Sigma$ be a critical point of $f$. The embedded Morse Lemma (see e.g.\ \cite[Lemma 2.17]{BP}) provides us with local coordinates $(u, v)$ on $\Sigma$ near $z$, such that $f$ is quadratic in $(u, v)$ and $z$ is identified with $(0,0)$. Let $(x, y)$ be coordinates on $\R^2$.
The family of diagrams of links as $t$ crosses the singular value is studied using the map $\pi\colon (u, v)\mapsto (x, y)$ obtained by composing the
embedding $\Sigma\hookrightarrow \R^3\times[0,1]$ 
with the projection $\R^3\times[0,1]\to\R^3\to\R^2$. 
If this
has non-vanishing derivative at $(0,0)$ (that is, if it is a local diffeomorphism) and if the inverse image consists of the singular point only, then the discussion in the previous paragraph holds. That is, the change of a link diagram after crossing an index zero critical point is adding a round circle, and so on, see Definition~\ref{con:generic}. 

If we study a one-parameter family of embeddings of $\Sigma$ into $\R^3\times[0,1]$, however, then degenerate situations will usually occur at certain parameter values. That is, for certain embeddings in this family, we may have that $\det D\pi(0,0)=0$, or that $\pi^{-1}(0,0)$ consists of more than just one point. These situations lead to movie moves involving Morse handles. The equivariant setting is even more challenging. While the degeneracy of $D\pi$ can be handled in a similar manner as in the non-equivariant setting, the condition that $\pi$ be one-to-one on its image near $z$ is significantly more involved in the presence of the
$\Z_2$-action.

\subsection{Structure of the article}
For the first few sections, we set up the general language we use for our equivariant application. In Section~\ref{sec:warmup}, we explain the classical theory of Reidemeister moves in the language of singularity theory. We introduce
jet spaces, state the Thom transversality theorem, and prove the usual Reidemeister theorem. We define regularity for diagrams and paths of diagrams. This section serves as an introduction to the techniques of the paper. In Section~\ref{sec:mutt}, we start generalizing these notions to the equivariant setting. We recall an equivariant version of the jet transversality theorem
due to Wall \cite{Wall_jet} and define regularity for involutive links. In Section~\ref{sec:def_versal}, we make a small detour: before we proceed to a classification
of codimension 1 and 2 singularities of equivariant maps, we need to recall several notions from singularity theory. We define RL equivalence and provide a definition of versality and infinitesimal versality.

We begin our work in earnest in Section~\ref{sec:patterns}, where we provide normal forms and versal deformations for several classes of equivariant singularities. We introduce the notion of a strikethrough, which is when a relatively non-complicated singularity is met by an extra line passing through it. Many codimension 2 singularities arise as strikethroughs of codimension 1 singularities, so a general way of handling these cases makes the description more concise.
We also discuss equivariant cusps and tangencies. As we will see, these must be subdivided into several cases depending on their behavior with respect to the axis of symmetry.

We use our work in Section~\ref{sec:patterns} to determine the codimension $1$ and codimension $2$ equivariant singularities and prove Theorems~\ref{thm:intro1} and \ref{thm:intro4}. Theorem~\ref{thm:intro1} is established in Section~\ref{sec:one_para}, where we study one-parameter families of equivariant diagrams.
We first prove an equivariant Reidemeister theorem for diagrams in $\R^3$ and then discuss the question of diagrams in $S^3$, where the extra I-move
is needed. We introduce another non-local move, called an S-move, consisting of moving a double point over a point at infinity. Unlike the I-move, the S-move can be deduced from other Reidemeister moves. However, it allows for a more concise description of loops of Reidemeister moves
in later sections.

In Section~\ref{sec:one_para}, we also discuss 
the possibility of eliminating one of the two R-1 moves in the case of a strongly invertible knot, which is useful
for discussing pointed link diagrams. Theorem~\ref{thm:intro4} is established in Section~\ref{sec:codim2}, where we study codimension~2 singularities and their bifurcations. We classify all codimension~2 singularities and draw bifurcation diagrams for each.

In Section~\ref{sec:4}, we study equivariant cobordisms. We develop the theory of equivariant embedded Morse functions. We prove density as
well as a variant of the embedded equivariant Morse lemma (Theorem~\ref{thm:eem}). We use this to prove Theorem~\ref{thm:intro2}. Finally, in Section~\ref{sec:carter_saito}, we study movie moves for equivariant cobordism and prove Theorem~\ref{thm:intro3}.

As a last remark: in Section~\ref{sec:def_versal} we state the versality theorem (Theorem~\ref{thm:versality_theorem}) and note that the equivariant variant of the theorem has an analogous proof to the classical one. This relies on an equivariant version of Malgrange Preparation Theorem. The generalization is well-known to experts, but for completeness we include a quick argument in the appendix.

\subsection*{Acknowledgments} The authors would like to thank Kristen Hendricks, Dmitry Kerner, Slava Krushkal, Robert Lipshitz, Mark Powell, Liam Watson, and Ian Zemke for helpful conversations. Part of this work was completed when the authors were at the IMPAN Knots, Homology, and Physics Simons Semester and joint with Warsaw University IDUB programme. We also thank the Simons Center for Geometry and Physics for hosting some of the authors during the conference Gauge Theory and Floer Homology in Low Dimensional Topology. ID was supported by NSF grant DMS--2303823. MB was supported by the NCN grant OPUS 2024/53/B/ST1/03470. MS was partially supported by NSF grant DMS-2203828.

\section{Classical Reidemeister theory}\label{sec:warmup}

In this section, we briefly review jet spaces and state several related transversality theorems. We then provide a warm-up for the rest of the paper by using jet spaces to prove the usual (non-equivariant) Reidemeister theorem. This will serve as a foundation for the introduction of equivariant jet spaces and involutive links in the sequel. 

\subsection{Jet spaces}\label{sec:jet}

We begin by defining jet spaces and jet extensions; see \cite[Chapter 2]{GoluGuille} and \cite[Section 29]{Arnold_additional}. We start with the Euclidean case. Let $V$ and $W$ be two finite-dimensional vector spaces. Let $\cE(V)$ be the ring of $\R$-valued germs at $0 \in V$ and let $\mm(V)$ be its maximal ideal; this consists of germs vanishing at $0$. Write $\cE(V, W)$ for the $\cE(V)$-module of germs of smooth functions from $V$ to $W$ taking $0 \in V$ to $0 \in W$. Then the classical space of $k$-jets taking $0$ to $0$ is defined to be
\[
\mJ^k(V, W) = \cE(V, W)/\mm(V)^{k+1}\cE(V, W).
\]
Two germs in $\mJ^k(V, W)$ are the same if all of their partial derivatives up to order $k$ coincide. 

Now let $X$ and $Y$ be smooth manifolds and fix $x \in X$ and $y \in Y$. Choose Euclidean neighborhoods $V$ and $W$ of $x$ and $y$ in which $x$ and $y$ are taken to the origin. The space of $k$-jets taking $x$ to $y$ is defined to be
\[
\mJ^k(x, y) = \mJ^k(V, W).
\]
We denote a typical element of $\mJ^k(x, y)$ by $\mj^k f(x)$. 

\begin{defn}
The \textit{$k$-jet space} $\mJ^k(X, Y)$ is the union of $\mJ^k(x, y)$ over all $(x, y) \in X \times Y$. This can canonically be given the structure of a smooth bundle over $X \times Y$.
\end{defn}

\begin{defn}
Let $f \colon X \rightarrow Y$. The \emph{$k$-jet extension} of $f$ is the map 
\[
\mj^kf\colon X\to \mJ^k(X,Y)
\]
assigning to each $x\in X$ the $k$-jet $\mj^kf(x)$ obtained by restricting $f$ to a germ from $x$ to $f(x)$. 
\end{defn}

Roughly speaking, $\mj^k f(x)$ records the partial derivatives of $f$ at $x$ up to order $k$. Clearly, $\mJ^0(X,Y)=X\times Y$, while $\mJ^1(X,Y)$ is the bundle over $X\times Y$ with fiber at $(x,y)$ given by $\Hom(T_xX,T_yY)$.
Note that we have two maps
\[
\ma\colon \mJ^k(X,Y)\to X \quad \text{and} \quad \mb\colon \mJ^k(X,Y)\to Y.
\]
The first is given by sending a $k$-jet $\mj^k f(x)$ to $x$ and is referred to as the \emph{source map}. The second is given by sending a $k$-jet $\mj^k f(x)$ to $f(x)$ and is referred to as the \emph{target map}.

\begin{example}
  Suppose $X = Y=\R$. Then $\mJ^1(X,Y) = \R^3$, and the $1$-jet extension of a map $f\colon\R\to\R$ is the map $\mj^1f(x)=(x,f(x),f'(x))$.
\end{example}

Jets often arise in connection with degeneracy conditions on maps. 
Our general strategy will be to express various degeneracy conditions in terms of hitting some subspace of an appropriately defined jet space. If these subspaces have sufficiently large codimension, one can argue that such degeneracies are avoided by a generic map. 

\begin{thm}[Thom Transversality Theorem, see \expandafter{\cite[Theorem II.4.9]{GoluGuille}}]\label{thm:ttt}
  Let $W \subset \mJ^k(X,Y)$ be a submanifold. Then the set of maps $f\colon X\to Y$ such that $\mj^kf\colon X\to \mJ^k(X,Y)$
  is transverse to $W$ is residual in $C^\infty(X, Y)$. 
\end{thm}

In order to restrict the behavior of a function at multiple points simultaneously, we introduce the notion of a multijet space. Let $s \in \mathbb{N}$ and $X^{(s)}$ be the set of ordered $s$-tuples
of pairwise distinct points in $X$:
\[X^{(s)}=\{(x_1,\dots,x_s)\in X^s\colon x_i\neq x_j\textrm{ unless } i=j\}.\]
For each $k$, let $\ma^s$ and $\mb^s$ be the $s$-fold Cartesian powers 
\[
\ma^s\colon \mJ^k(X,Y)^s\to X^s \quad \text{and} \quad \mb^s\colon \mJ^k(X,Y)^s\to Y^s
\]
of the source and target maps. 
\begin{definition}
The \emph{$s$-fold multijet space} $\mJ^{k,s}(X,Y)$ is the preimage of $X^{(s)}$ under $\ma^s$. 
\end{definition}

The $s$-fold multijet space consists of $s$-tuples of jets that live over $s$ distinct points in $X$. Note that $\mJ^k(X,Y)$ is $\mJ^{k,1}(X,Y)$. We likewise have the notion of a multijet extension.

\begin{defn}
Let $f \colon X \rightarrow Y$. The \emph{$s$-fold multijet extension} of $f$ is the map 
\[
\mj^{k,s}f\colon X^{(s)}\to \mJ^{k,s}(X,Y)
\]
that assigns to each $(x_1,\dots,x_s) \in X^{(s)}$ the multijet
\[\mj^{k,s}f(x_1,\dots,x_s)=(\mj^kf(x_1),\dots,\mj^kf(x_s)).\]
\end{defn}

The following result generalizes the Thom transversality theorem:

\begin{thm}[Multijet Transversality Theorem, see \expandafter{\cite[Theorem II.4.13]{GoluGuille}}]\label{thm:asdfg}
  Let $W \subset \mJ^{k,s}(X,Y)$ be a submanifold. Then the set of maps $f\colon X\to Y$
  such that $\mj^kf\colon X^{(s)}\to \mJ^{k,s}(X,Y)$ is transverse to $W$ is residual in $C^\infty(X, Y)$. 
\end{thm}

In the above situation, we will sometimes say $f$ is transverse to $W$ to mean that the appropriate jet extension $\mj^{k,s}f$ is transverse to $W$.
It is a standard procedure to translate this into a ready-to-use parameter counting argument. The idea is as follows. Suppose we have an $r$-dimensional family of maps from $X$ to $Y$. For any map $f$ in this family, the associated multijet extension $\mj^{k, s} f \colon X^{(s)} \rightarrow \mJ^{k, s}(X, Y)$ has domain of dimension $sn$, where $n$ is the dimension of $X$. As $f$ varies, these multijet extensions thus sweep out a locus of dimension $sn + r$ in $\mJ^{k, s}(X, Y)$. Generically, this will not intersect a submanifold of $\mJ^{k, s}(X, Y)$ with codimension greater than $sn + r$. We make this precise below:

\begin{thm}[Parameter Counting]\label{thm:pc_argument}
Let $R$ be an $r$-dimensional manifold, which we think of as a parameter space, and $\dim X=n$. Let $W\subset \mJ^{k,s}(X,Y)$ be a submanifold of
  codimension $w$. If $sn+r<w$, then there is a residual subset of $C^\infty(R\times X,Y)$ such
  that for any $F$ in this subset and any $u\in R$, the jet extension 
  $\mj^{k,s}f_u$ of $f_u =F(u,\cdot)\colon X\to Y$ misses $W$. 
\end{thm}
\begin{proof}
The result is classical, see e.g. \cite[Section 3]{FiedlerKurlin}, but we provide a quick proof which will generalize to the equivariant setting. The only subtlety in the above discussion is that (strictly speaking) we should phrase our argument in terms of the multijet transversality theorem for maps from $R \times X$ to $Y$. Let $\Delta'$ be the preimage of the diagonal in $R^s$ under
  \[
  \mJ^{k,s}(R \times X,Y)\to (R\times X)^s \to R^s,
  \]
  where the first map is the source map and the second is projection. Let $\pi_X\colon R \times X \to X$ be projection onto $X$. This induces a projection $\Pi_X\colon \mJ^{k,s}(R \times X,Y)\to \mJ^{k,s}(X,Y)$ which is a locally trivial fibration. Set $W' = \Pi_X^{-1}(W)\cap \Delta'$. By the multijet transversality theorem, the set of maps $F \colon R \times X \to Y$ which are transverse to $W'$ is residual. 

Now, $\codim W$ is $w$, and hence $\codim\Pi_X^{-1}(W)$ is also $w$, since $\Pi_X$ is a fibration. Likewise, $\codim \Delta'$ 
  is $r(s-1)$. Thus $\codim W'$ is $r(s-1)+w$. On the other hand, the dimension of $R \times X$ is $r+n$, and hence the dimension of $(R \times X)^s$ is $s(r + n)$. Thus, if 
  \[
s(r + n) < r(s-1)+w, 
  \]
  the mutlijet transversality theorem provides a residual subset of functions $F$ in $C^\infty(R \times X,Y)$ such that the multijet extension $\mj^{k,s}F$ misses $W'$. This inequality is of course identical to the one in the hypothesis of the theorem. It remains to show that $F$ missing $W'$ implies that $\mj^{k,s}f_u$ misses $W$ for all $u\in R$.

Suppose $\mj^{k,s}f_u$ hits $W$. This means that there exist points $x_1,\dots,x_s\in X$ such that
  \[(\mj^kf_u(x_1),\mj^kf_u(x_2),\dots,\mj^kf_u(x_s))\in W.\]
Then
  \[(\mj^kF(u,x_1),\mj^kF(u,x_2),\dots,\mj^kF(u,x_s))\in \Pi_X^{-1}(W)\subset \mJ^{k,s}(R \times X,Y).\]
  However, the point $(\mj^kF(u,x_1),\mj^kF(u,x_2),\dots,\mj^kF(u,x_s))$ also belongs to $\Delta'$. Hence,
  if $\mj^{k,s}f_u$ hits $W$, then $\mj^{k,s}F$ hits $W'$.
\end{proof}

Our usage of the parameter counting argument is so common that it will be useful to introduce some auxiliary notation:

\begin{defn}
Let $D$ be a submanifold of $\mJ^{k, s}(X, Y)$. The \textit{(jet space) codimension of $D$} (sometimes called the \textit{number of defining conditions}) is the usual codimension $\codim(D)$ of $D$ in $\mJ^{k, s}(X, Y)$. We define the \textit{(parametric) codimension of $D$} to be 
\[
\pcodim(D) = \codim(D) - s(\dim X).
\]
Note that $s(\dim X)$ is the dimension of the domain of the relevant multijet extension. In this language, the parameter counting argument states that a generic $r$-parameter family of maps in $C^\infty(X, Y)$ will miss $D$ so long as $r < \pcodim(D)$. 
\end{defn}

\begin{example}
Consider the case of an $s$-tuple self-intersection singularity. The presence of an $s$-tuple self-intersection of $f$ means that $\mj^{0, s}f$ hits $D\subset \mJ^{0,s}(X,Y)$, where $D$ is the preimage of the diagonal in $Y^s$ under the target map. This has $\codim D=(s-1)(\dim Y)$, so 
  \[
  \pcodim(D) = (s-1)(\dim Y) - s(\dim X).
  \]
For example, suppose $f$ is a link diagram (as we discuss presently), so that $X$ is a circle or disjoint union of circles and $Y$ is $\R^2$. Then $\pcodim(D) = s - 2$, so an $s$-tuple self-intersection does not occur in a generic $r$-parameter family of diagrams unless $r$ is at least $s - 2$.
\end{example}

\subsection{Regular diagrams}\label{sub:reidemeister}
We now illustrate the theory by using jet spaces to establish standard facts regarding link diagrams and Reidemeister moves. We begin with the former. 

\begin{defn}
Fix a disjoint union of circles $\cS$. A \textit{link} is an embedding 
\[
\wt{\phi} \colon \cS \to \R^3. 
\]
Let $\pi\colon\R^3\to\R^2$ be the fixed projection from $\R^3$ onto the $xy$-plane. We refer to 
\[
\phi\colon \cS\to\R^2 \quad \text{given by} \quad \phi=\pi\circ\wt{\phi}
\]
as the corresponding \emph{link diagram}. In Subsection~\ref{sub:ES3}, we will pass to links embedded in $S^3$ by viewing $S^3$ as the one-point compactification of $\R^3$.
\end{defn}

Our convention will be that maps or objects in $\R^3$ or $S^3$ are decorated with a tilde, while maps or objects in $\R^2$ are not. In general, a map $\phi \colon \cS \rightarrow \R^2$ might have arbitrarily complicated singularities. Our goal will to show that we can always perturb $\phi$ to have singularities which are particularly simple. The argument and notation here will be generalized to the equivariant setting in subsequent sections.

\begin{defn}\label{def:F0}
For a fixed $\cS$, let $\cF$ be the space of all smooth maps from $\cS$ to $\R^2$. 
\end{defn}

The following describes a generic element $\phi \in \cF$.

\begin{defn}\label{def:generic_R2}
Let $\cF^0$ be the subspace of $\cF$ consisting of maps whose only singularities are ordinary double points. Thus,
  $\phi$ belongs to $\cF^0$ if it has the following three properties:
  \begin{enumerate}[label=(Reg-\arabic*)]
    \item $\phi'$ is nonzero at every point of $\cS$;\label{item:R1}
    \item for any $w\in\R^2$, if $\phi^{-1}(w)$ consists of two points $t_1, t_2$, then $\phi'(t_1)$, $\phi'(t_2)$ are linearly independent;\label{item:R2}
    \item for any $w\in\R^2$, the preimage $\phi^{-1}(w)$ consists of at most two points.\label{item:R3}
  \end{enumerate}
Conditions~\ref{item:R1}, \ref{item:R2}, and~\ref{item:R3} can be phrased as: $\phi$ has no cusps, 
$\phi$ has no tangencies,
and $\phi$ has no triple points. We refer to such a $\phi$ as being \textit{regular}.
\end{defn}

This means there are three (non-exclusive) ways that $\phi$ can fail to be regular. Hence the complement of $\cF^0$ decomposes into three strata:

\begin{defn}\label{def:stratum}
For each $i$, let $\wt{\cF}^1_i$ be the subset of $\cF$ on which (\textrm{Reg}-$i$) is violated. Thus:
  \begin{enumerate}
    \item[($\wt{\cF}^1_1$)] \ref{item:R1} is violated: $\phi$ has a cusp; that is, there is at least one point where $\phi'(t) = 0$;
    \item[($\wt{\cF}^1_2$)] \ref{item:R2} is violated: $\phi$ has a tangency; that is, there is at least one pair of points $t_1$, $t_2$ such that $\phi(t_1) = \phi(t_2)$ and $\phi'(t_1)$, $\phi'(t_2)$ are linearly dependent;
    \item[($\wt{\cF}^1_3$)] \ref{item:R3} is violated: $\phi$ has a triple point; that is, there is at least one $w \in \R^2$ for which $\phi^{-1}(w)$ consists of three or more points.
  \end{enumerate}
Denote
\[
\wt{\cF}^1 = \wt{\cF}^1_1  \cup \wt{\cF}^1_2 \cup \wt{\cF}^1_3 = \cF \setminus \cF^0.
\]
\end{defn}

By examining these three strata, we obtain the following standard result:
\begin{thm}[see \cite{Whitney_closed}]\label{thm:cf0}
  The subset $\cF^0$ of regular maps is open-dense in $\cF$. 
\end{thm}
\begin{proof}

It suffices to show that for each $i$, the set of maps missing each $\wt{\cF}^1_i$ is open-dense. By definition, $\cF^0$ will then be the intersection of three such sets. We proceed via multijet transversality. 

A map $\phi$ lies in $\wt{\cF}^1_1$ if $\mj^1\phi$ intersects the subspace of $\mJ^1(\cS,\R^2)$ defined by $\phi'(t)=0$. This is a codimension~2 subset of jet space (the vanishing of both components of the derivative) while the source is of dimension~1; Theorem~\ref{thm:asdfg} thus shows that the set of maps missing this subspace is residual. Because having $\phi'(t) \neq 0$ is clearly an open condition, in this case we have that the complement of $\smash{\wt{\cF}^1_1}$ is open-dense.

A map $\phi$ lies in $\wt{\cF}^1_2$ if $\mj^{1,2}\phi$ intersects the subspace of $\mJ^{1,2}(\cS,\R^2)$ defined by the conditions $\phi(t_1)=\phi(t_2)$ and $\det(\phi'(t_1),\phi'(t_2))=0$. This is a codimension~3 subset of jet space (the equality $\phi(t_1)=\phi(t_2)$ gives two coordinate constraints, and the linear dependence gives a third) while the source is of dimension~2; Theorem~\ref{thm:asdfg} thus shows that the set of maps missing this subspace is residual. Once again we may directly check that the complement of $\smash{\wt{\cF}^1_2}$ is open and thus that it is open-dense.

Finally, a map $\phi$ lies in $\wt{\cF}^1_3$ if $\mj^{0,3}\phi$ intersects the subspace of $\mJ^{0,3}(\cS,\R^2)$ defined by $\phi(t_1)=\phi(t_2)=\phi(t_3)$. This is a codimension~4 subset of jet space (there are four independent coordinate equalities) while the source is of dimension~3. Once again, the complement of $\smash{\wt{\cF}^1_3}$ is open-dense. 
  \end{proof}
  
  We quickly obtain the following corollary: 

\begin{corollary}\label{cor:perturb}
 Any link $\wt{\phi}$ can be perturbed so that $\phi$ is in $\cF^0$.
\end{corollary}
\begin{proof}
  Let $\phi_n \in \cF^0$ be a sequence converging to $\phi$ in the $C^\infty$-topology and
  $\wt{\phi}_n=\wt{\phi}+\phi_n-\phi$. Here, we view the codomain of $\phi_n-\phi$ as $\mathbb{R}^3$ by setting the $z$-coordinate equal to zero. We have $\pi \circ \wt{\phi}_n=\phi_n$
  and $\wt{\phi}_n\to\wt{\phi}$ in $C^\infty$. In particular, the convergence is also in the $C^1$-norm.
  It follows from stability of embeddings that $\wt{\phi}_n$ is an 
  embedding isotopic to $\wt{\phi}$. 
\end{proof}

\subsection{Reidemeister moves}\label{sec:noneqreidemeister}
We complete our discussion of classical knot theory by sketching a proof of the Reidemeister theorem using jet spaces. For this, we will need to understand one-parameter families of maps in $\cF$. While a generic element of $\cF$ lies in $\cF^0$, a one-parameter family of maps in $\cF$ will usually cross the strata of Definition~\ref{def:stratum}. We thus begin by identifying a subset $\cF^1_i$ of each stratum consisting of the mildest possible singularities within that stratum. Our treatment is brief, since we give a detailed discussion of the equivariant case in subsequent sections.

\begin{defn}\label{def:generic_R3}
For each $i$, define $\cF^1_i \subset \wt{\cF}^1_i$ as follows.

  \begin{enumerate}
    \item[$(\cF^1_1)$] $\phi$ has a cusp of the lowest possible order; i.e., a $(2, 3)$-cusp (see Definition~\ref{def:cusppq}). This means that at the cusp point, $\phi''(t)$ and $\phi'''(t)$ are linearly independent.
    \item[$(\cF^1_2)$] $\phi$ has a tangency of the lowest possible order; i.e., a $1$-fold tangency (see Definition~\ref{def:tangency-fold}). This means that at the point of tangency, we can find coordinates around $t_1$, $t_2$, and $\phi(t_1) = \phi(t_2)$ such that the two branches of $\phi$ are given by
    \[
    \phi(t) = (t, 0) \quad \text{and} \quad \phi(s) = (s, h(s)),
    \]
    for some $h$ with $h(0) = h'(0) = 0$ and $h''(0) \neq 0$.
    \item[$(\cF^1_3)$] $\phi$ has an ordinary triple point. This means that $\phi(t_1) = \phi(t_2) = \phi(t_3)$ for some $t_1$, $t_2$, and $t_3$, but $\{\phi'(t_1), \phi'(t_2), \phi'(t_3)\}$ are pairwise linearly independent.
 \end{enumerate}
 In each case, we furthermore require $\phi$ has no coincidences (see Section~\ref{sub:coin}) or strikethroughs (see Section~\ref{sub:strike}). The former means that each $\phi$ has exactly one singularity; for example, we do not allow the simultaneous appearance of two cusps or the appearance of a cusp and a tangency. The latter means that we do not allow the participation of any other branch of $\phi$ in one of the singularities defined above; for example, we do not allow a triple point to become a quadruple point. Denote
\[
\cF^1 = \cF^1_1 \cup \cF^1_2 \cup \cF^1_3.
\]
\end{defn}

We claim that a generic one-parameter family of maps may cross $\smash{\wt{\cF}^1_i}$, but only does so in a controlled manner, by staying within $\cF^1_i$. For instance, when $i = 1$, Definition~\ref{def:generic_R3} says that a one-parameter family may develop a cusp, but this cusp will be of the lowest possible order. Or, when $i = 3$, Definition~\ref{def:generic_R3} says that such a family may develop a triple point, but this triple point will occur between three mutually transverse branches. 

\begin{defn}\label{def:generic_path}
A path $\phi_s$ for $s \in [0, 1]$ in $\cF$ is called \textit{regular} if:
  \begin{enumerate}[label=(PR-\arabic*)]
    \item $\phi_0$ and $\phi_1$ are regular maps; \label{item:eboundary_regular}
    \item $\phi_s$ belongs to $\cF^0\cup\cF^1$ for all $s$; \label{item:eno_big_deal}
    \item there are finitely many values $s_1,\dots,s_m$ such that $\phi_{s_i}\in\cF^1$, while for other
      values $\phi_s\in\cF^0$;\label{item:efinitely_many_events}
    \item regarded as a path in $\cF$,  $\phi_s$ crosses $\cF^1$ transversely.\label{item:etransverse}
  \end{enumerate}
\end{defn}

To show that a generic path stays within $\cF^0 \cup \cF^1$, we consider its complement:

\begin{defn}
For each $i$, let $\wt{\cF}^2_i = \wt{\cF}^1_i \setminus \cF^1_i$. Denote
\[
\wt{\cF}^2 = \wt{\cF}_1^2 \cup \wt{\cF}_2^2 \cup \wt{\cF}_3^2 = \cF \setminus (\cF^0 \cup \cF^1).
\]
\end{defn}

As in the proof of Theorem~\ref{thm:cf0}, we proceed by showing:

\begin{thm}[see \cite{Perestroikas}]\label{thm:perestrojka}
  Let $\phi_s$ be any path in $\cF$ with $\phi_0,\phi_1\in\cF^0$. Then there exists an arbitrarily close path $\phi_s'$,
  equal to $\phi_s$ at the ends, which is regular.
\end{thm}
\begin{proof}[Sketch of proof]
  We prove \ref{item:eno_big_deal} by showing that each $\wt{\cF}_i^2$ is the union of strata with $\pcodim$ at least $2$ and is thus missed by a generic path. Our discussion is brief, as we describe the codimension $2$ strata in greater detail in Section~\ref{sec:codim2}. 

There are two ways for $\phi$ to be in $\smash{\wt{\cF}_i^2}$. The first (and less interesting) way is for $\phi$ to develop a coincidence or strikethough. A coincidence means $\phi$ violates two or more of the conditions $(\textrm{Reg}$-$i)$ at the same time, or the same condition twice. For instance, $\phi$ could have two cusps, which would correspond to the jet extension hitting a codimension~4 subspace of $\mJ^{1,2}(\cS, \R^2)$ (two equations for each cusp). As the domain in this case has dimension~2, such a subspace has $\pcodim = 2$. Other coincidences of singularities are dealt with similarly. In the case of a strikethrough, for each new branch we add at least two equations (one equation for each coordinate of the branch) and increase the dimension of the domain by one. Hence all coincidences or strikethroughs have $\pcodim$ at least $2$.

The other way to be in $\smash{\wt{\cF}_i^2}$ is to fail one of the specific conditions laid out in Definition~\ref{def:generic_R3}. 

\begin{itemize}
\item For $\smash{\wt{\cF}_1^2}$, $\phi$ must have a higher-order cusp. That is, at the cusp point we have that $\phi''(t)$ and $\phi'''(t)$ are linearly dependent.
\item For $\smash{\wt{\cF}_2^2}$, $\phi$ must have a degenerate or higher-order tangency. That is, at the point of tangency, we either have that one of the branches is actually a cusp, or the two branches are tangent to higher order.
\item For $\smash{\wt{\cF}_3^2}$, $\phi$ must have a degenerate triple point. That is, at the triple point, we either have that one (or more) of the branches is actually a cusp, or two (or more) branches are tangent.
\end{itemize}
In Section~\ref{sec:one_para}, we verify that the above subsets are unions of strata with $\pcodim$ at least $2$. For example, $\phi$ has a higher-order cusp if and only if its jet extension hits the subspace of $\mJ^3(\cS,\R^2)$ cut out by three equations: $\phi'(t) = 0$, together with the linear dependence of $\phi''(t)$ and $\phi'''(t)$. Since the domain has dimension $1$, this subspace has $\pcodim = 2$. It follows that none of these singularities occur in a generic 1-parameter family of maps. This proves that a residual set of maps from $[0,1]$ to $\cF$
stays in $\cF^0\cup\cF^1$, establishing item~\ref{item:eno_big_deal}.

The other conditions are straightforward. Item~\ref{item:eboundary_regular} is contained in the assumptions of the theorem. We thus deal with \ref{item:efinitely_many_events} and~\ref{item:etransverse}. By the transversality theorem, a residual set of maps
from can be made transverse
to the strata defining $\cF^1$. For such a map $\phi\colon[0,1]\to\cF$, the number of parameter values $s$ for which $\phi(s)\in\cF^1$
is finite, and the intersection is transverse.
\end{proof}

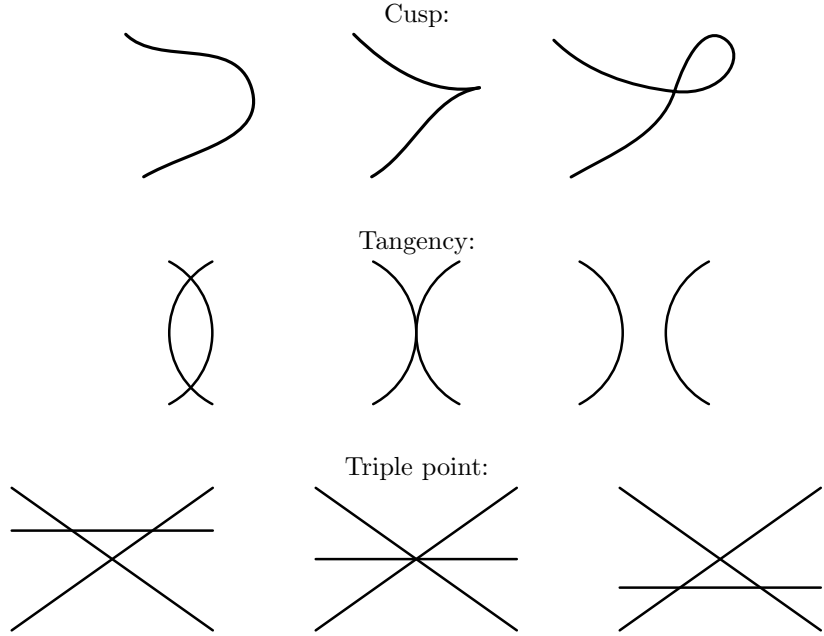
\begin{figure}
  \begin{tikzpicture}
    \draw (-3,3) node[rectangle, minimum width=2cm] {\includegraphics[height=2cm]{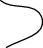}};
    \draw (0,3) node[rectangle, minimum width=2cm] {\includegraphics[height=2cm]{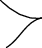}};
    \draw (3,3) node[rectangle, minimum width=2cm] {\includegraphics[height=2cm]{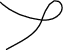}};
    \draw (-3,0) node[rectangle, minimum width=2cm] {\includegraphics[height=2cm]{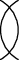}};
    \draw (0,0) node[rectangle, minimum width=2cm] {\includegraphics[height=2cm]{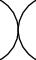}};
    \draw (3,0) node[rectangle, minimum width=2cm] {\includegraphics[height=2cm]{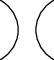}};
    \draw (-4,-3) node[rectangle, minimum width=1.5cm] {\includegraphics[height=2cm]{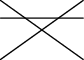}};
    \draw (0,-3) node[rectangle, minimum width=1.5cm] {\includegraphics[height=2cm]{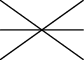}};
    \draw (4,-3) node[rectangle, minimum width=1.5cm] {\includegraphics[height=2cm]{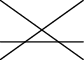}};
    \draw (0,4.2) node [scale=0.9] {Cusp:};
    \draw (0,1.2) node [scale=0.9] {Tangency:};
    \draw (0,-1.8) node [scale=0.9] {Triple point:};
  \end{tikzpicture}
  \caption{Perestroikas of critical points.}\label{fig:perestrojka}
\end{figure}

Theorem~\ref{thm:perestrojka} implies the classical Reidemeister theorem. The essential idea is as follows: for each $i$, the singularities described by $\cF^1_i$ are sufficiently constrained so that we have a local model for any $\phi$ lying in $\cF^1_i$. Such a model is called a \textit{normal form} for the corresponding singularity; see Section~\ref{sub:cA}. In fact, we moreover have a local model for any one-parameter family of maps transversely crossing $\cF^1_i$. This uses the notion of a \textit{versal deformation}; see Section~\ref{sub:versal}. Versal deformations allow us to construct qualitative before-and-after pictures for the maps in such a family.

For instance, up to a choice of coordinates, a $(2,3)$-cusp is locally given by $\phi(t) = (t^2, t^3)$. A model one-parameter family of maps passing through this cusp is given by $\phi_\varepsilon(t) = (t^2, t^3 + \varepsilon t)$. This family is universal in an appropriate sense discussed in Section~\ref{sub:versal}, and hence describes the behavior of a general one-parameter family transversely crossing $\cF^1_1$. The before-and-after pictures of this family for $\varepsilon < 0$ and $\varepsilon > 0$ represent the usual R1 Reidemeister move.

A formal summary of this argument is given below: 

\begin{corollary}\label{cor:Reidemeister}
  Any two diagrams of the same knot can be connected by a sequence of isotopies and finitely many Reidemeister moves.
\end{corollary}
\begin{proof}[Sketch of proof]
  A path $\wt{\phi}_s\colon \cS\to\R^3$, $s\in[0,1]$ induces a path $\phi_s\colon \cS\to\R^2$
  via $\phi_{s}=\pi\circ\wt{\phi}_s$. We perturb $\phi_{s}$ to a regular path $\phi_{s,n}$
  in such a way that $\phi_{s,n}$ agrees with $\phi_s$ for $s=0,1$.
  The perturbation
  is small in $C^\infty$ topology, in particular, we may assume that $||\phi_{s,n}-\phi_s||_{C^1}\le \frac1n$ for all $s$.
  By \ref{item:efinitely_many_events}, $\phi_{s,n}$ intersects
  $\cF^1$ at finitely many points. We study how $\phi_{s,n}$ changes when $s$ crosses such a parameter.

  Suppose for $s_0\in[0,1]$, $\phi_{s_0,n}\in\cF^1_i$, for $i=1,\dots,3$, and suppose $\varepsilon>0$ is such that
  for $s\in[s_0-\varepsilon,s_0+\varepsilon]\setminus\{s_0\}$, $\phi_{s,n}\notin\cF^1$. By \ref{item:etransverse},
  the maps $\phi_{s_0-\varepsilon,n}$ and $\phi_{s_0+\varepsilon,n}$ belong to different connected components of $\cF^0$. It
  follows (see \cite{Perestroikas}) that the change $\phi_{s_0-\varepsilon,n}(\cS)$ and $\phi_{s_0+\varepsilon,n}(\cS)$
  is given by a \emph{perestroika}; see Figure~\ref{fig:perestrojka}.

  As $\phi_{s,n}$ converges to $\phi_s$, the sequence $\wt{\phi}_{s,n}:=\wt{\phi}_s+\phi_{s,n} - \phi_s$ converges
  to $\wt{\phi}_{s,n}$, where it is enough to use convergence in $C^1$-norm. By stability of embeddings, for sufficiently large $n$, $\wt{\phi}_{s,n}$ is isotopic to $\wt{\phi}$.
  Also $\pi(\wt{\phi}_{s,n})=\phi_{s,n}$. A perestroika at the level of $\phi_{s,n}$ induces a Reidemeister move at the level
  of $\wt{\phi}_{s,n}$.
  As displayed in Figure~\ref{fig:perestrojka}, a Reidemeister move R1 occurs at each cusp, a Reidemeister move R2 occurs at each tangency,
  and a Reidemeister move R3 happens at each triple point.
\end{proof}

\section{Equivariant jet spaces}\label{sec:mutt}
We now begin our discussion of involutive links. As in Section~\ref{sec:warmup}, we start by introducing equivariant jet spaces. We then establish the genericity of regular diagrams in the equivariant setting; this closely follows the argument of Section~\ref{sec:warmup}. The proof of the equivariant Reidemeister theorem, however, will require significantly more investment, so we delay it to future sections. Nevertheless, our discussion here will be foundational to the rest of the paper.

\subsection{General setup}\label{sub:ejet}
We begin with a general discussion of equivariant jet spaces following \cite{Wall_jet}. Although we only deal with the group $G = \Z_2$, we give Wall's construction for $G$ a compact Lie group in order to illustrate the subtleties that arise in the definition.

We start with the Euclidean case. Let $V$ and $W$ be two finite-dimensional vector spaces acted on by $G$; we assume this action to be orthogonal. Let $\cE^G(V) \subset \cE(V)$ be the ring of $G$-invariant germs at $0 \in V$ and $\mm^G(V)=\mm(V)\cap\cE^G(V)$ be its maximal ideal. Let $\cE^G(V, W) \subset \cE(V, W)$ be the $\cE^G(V)$-module of $G$-equivariant germs taking $0 \in V$ to $0 \in W$. The space of $G$-equivariant $k$-jets taking $0$ to $0$ is then defined to be 
\[
\mJ_G^k(V, W) = \cE^G(V, W)/\mm^G(V)^{k+1}\cE^G(V, W).
\]

Now suppose $G$ acts smoothly on two manifolds $X$ and $Y$. We first note that if $f$ is a $G$-equivariant map, then the behavior of $f$ in a neighborhood of $x$ determines the behavior of $f$ near the entire orbit of $x$. Thus, one might expect the equivariant jet space to be a bundle over $(X/G) \times Y$. Indeed, if the action of $G$ is free, then smooth $G$-equivariant maps from $X$ to $Y$ are in bijection with smooth maps from $X/G$ to $Y$, and the study of equivariant jets can be reduced to the study of non-equivariant jets on the quotient. However, in general the equivariant jet space at $x$ should depend on the isotropy group of $x$. This is because the isotropy group preserves a neighborhood of $x$ and imposes a nontrivial restriction on germs at $x$ when requiring $G$-equivariance.

Let $x \in X$ thus have isotropy group $H \subset G$. Fix a slice $S$ of the $G$-action at $x$; this is a locally $H$-invariant submanifold which is transverse to the $G$-orbit of $x$ at $x$. Fix a Euclidean neighborhood $V$ of $x$ in $S$ which is preserved by $H$ and on which $H$ acts orthogonally. Note that a $G$-equivariant map from $X$ to $Y$ necessarily takes $x$ into the $H$-fixed point set $Y^H$ of $Y$. Thus, let $y \in Y^H$ and let $W$ be a Euclidean neighborhood of $y$ in $Y$ on which $H$ acts orthogonally. The space of equivariant $k$-jets taking $x$ to $y$ is then defined to be
\begin{equation}\label{eq:eqxy}
\mJ_G^k(x, y) = \mJ_H^k(V, W).
\end{equation}
Note that the right-hand side consists of $H$-equivariant, rather than $G$-equivariant, germs. 

Now consider a point $g \cdot x$ in the orbit of $x$. This has isotropy group $g H g^{-1} \subset G$. Clearly, $g \cdot S$ is a slice through $g \cdot x$, and $g \cdot V$ and $g \cdot W$ are $gHg^{-1}$-invariant neighborhoods of $g \cdot x$ and $g \cdot y$, respectively. We thus have an obvious isomorphism
\begin{equation}\label{eq:eqiso}
\mJ_G^k(x, y) \cong \mJ_G^k(g \cdot x, g \cdot y)
\end{equation}
obtained by sending a germ $f$ from $x$ to $y$ to the germ $gfg^{-1}$ from $g \cdot x$ to $g \cdot y$. 

\begin{defn}
Let $[x] = G \cdot x$ be a $G$-orbit in $X$. An \textit{equivariant $k$-jet at $[x]$} (or \textit{along $[x]$}) consists of a $k$-jet at each point of $[x]$, such that each jet is equivariant in the sense of \eqref{eq:eqxy} and the different $k$-jets are taken to each other by the isomorphism \eqref{eq:eqiso}. 
\end{defn}

Roughly speaking, the reader should think that there are two steps to ensuring that a $k$-jet is $G$-equivariant. The first is that the germ itself should be equivariant under the local action of the isotropy group at that point. The second is that a germ at $x$ gives a class of germs defined along the orbit $[x]$, each of which is equivariant in the previous sense. Note that because of the first condition, equivariant germs at points with non-isomorphic isotropy groups are qualitatively different and should not be regarded as elements of the same jet space. To this end, we introduce the following partition of $X$:

\begin{defn}
The \textit{orbit type} of $x \in X$ is determined by the isotropy group $H$ of $x$ and the representation $\rho$ of $H$ on the tangent space $T_xS$. Two pairs $(H_1, \rho_1)$ and $(H_2, \rho_2)$ are equivalent if there exists some $g \in G$ such that $gH_1g^{-1} =H_2$ and the representations $h \mapsto \rho(h)$ and $h \mapsto \rho_2(g h g^{-1})$ of $H_1$ are equivalent; that is, $\rho_1(h) = \rho_2(ghg^{-1})$ for all $h \in H_1$.
\end{defn}

Let $\lambda$ now be a fixed orbit type and $X_\lambda$ be the subset of points in $X$ with orbit type $\lambda$. Clearly, $X_\lambda$ is composed of the union of $G$-orbits. Write $(X/G)_\lambda$ for the image of $X_\lambda$ in $X/G$. This is of course the same as the set of orbits in $X_\lambda$.

\begin{defn}
Let $\lambda$ be a fixed orbit type. The \textit{equivariant $k$-jet space} $\mJ_G^{k, \lambda}(X, Y)$ is the collection of all equivariant $k$-jets over all $[x] \in (X/G)_\lambda$. This can canonically be given the structure of a smooth manifold; see \cite{Wall_jet}.
\end{defn}

\begin{defn}
Let $\lambda$ be a fixed orbit type and $f \colon X \rightarrow Y$ be equivariant. The \emph{$k$-jet extension} of $f$ is the map 
\[
\mj^{k, \lambda}_{G} f \colon (X/G)_\lambda \to \mJ_G^{k, \lambda}(X,Y)
\]
that assigns to each $[x]\in (X/G)_\lambda$ the equivariant $k$-jet $\mj^{k, \lambda}_{G}f([x])$ along $[x]$ obtained by restricting $f$ to a germ at each point of $[x]$.
\end{defn}

The important takeaway is that there are many different equivariant jet spaces, each corresponding to a fixed orbit type and involving the points in $X$ of that orbit type. We will sometimes assemble all of these jet spaces into a single jet space by taking their disjoint union; we denote this by $\mJ^k_G(X, Y)$. However, as discussed in \cite{Wall_jet} there does not seem to be a satisfactory way to topologize this union to make it into a single bundle over $X$ or $X/G$. Nevertheless, if $f \colon X \rightarrow Y$ is a $G$-equivariant map, we obtain a $k$-jet extension
\[
\mj^k_Gf \colon X \to \mJ_G^k(X,Y)
\]
which assigns to each $x \in X$ the equivariant $k$-jet $\mj^{k, \lambda}_Gf(x) = \mj^{k, \lambda}_{G}f([x]) \in \mJ_G^{k, \lambda}(X,Y)$, where $\lambda$ is the orbit type of $x$ and $[x]$ is the equivalence class of $x$ in $(X/G)_\lambda$.

We now turn to equivariant multijets. For each orbit type $\lambda$, we have a source map
\[
\ma^\lambda \colon \mJ_G^{k, \lambda}(X, Y) \rightarrow (X/G)_\lambda.
\]
More generally, let $\Lambda = (\lambda_1, \lambda_2, \cdots, \lambda_s)$ be a tuple of orbit types, not necessarily all distinct. We then have the multijet source map
\[
\ma^{\Lambda} \colon \mJ^{k, \lambda_1}_G(X, Y) \times \cdots \times \mJ^{k, \lambda_s}_G(X, Y) \rightarrow (X/G)_{\lambda_1} \times \cdots \times (X/G)_{\lambda_s}.
\]
We introduce the notation
\[
X^{(\Lambda)} \subset (X/G)_{\lambda_1} \times \cdots \times (X/G)_{\lambda_s}
\]
to be the subsets of \textit{distinct} tuples, in analogy with $X^{(s)}$ from Section~\ref{sec:jet}. This condition is, of course, only meaningful if two of the $\lambda_i$ coincide.

\begin{definition}
Fix a tuple $\Lambda = (\lambda_1, \lambda_2, \cdots, \lambda_s)$ of orbit types. The \emph{$\Lambda$-fold multijet space} is the preimage
\[
\mJ_G^{k,\Lambda}(X,Y) = (\ma^{\Lambda})^{-1}(X^{(\Lambda)}).
\]
\end{definition}

\begin{defn}
Fix a tuple $\Lambda = (\lambda_1, \lambda_2, \cdots, \lambda_s)$ of orbit types and let $f \colon X \rightarrow Y$ be equivariant. The \emph{$\Lambda$-fold multijet extension} of $f$ is the map 
\[
\mj_G^{k,\Lambda}f\colon X^{(\Lambda)}\to \mJ_G^{k,\Lambda}(X,Y)
\]
that assigns to each $([x_1],\dots,[x_s]) \in X^{(\Lambda)}$ the multijet
\[\mj_G^{k,\Lambda}f([x_1],\dots,[x_s])=(\mj^{k, \lambda_1}_{G}f([x_1]),\dots,\mj^{k, \lambda_s}_{G}f([x_s])).\]
\end{defn}

Cleary, an equivariant multijet is analogous to the notion of a non-equivariant multijet, except that we speak of $s$ distinct orbits rather than $s$ distinct points, and we must specify the types of these orbits. We are now ready to state the multijet transversality theorem in the equivariant setting:

\begin{theorem}[see \expandafter{\cite[Theorem 2.1]{Wall_jet}}]\label{thm:ejet}
  Fix a tuple $\Lambda = (\lambda_1, \lambda_2, \cdots, \lambda_s)$ of orbit types.
  Let $\smash{W \subset \mJ_G^{k,\Lambda}(X,Y)}$ be a submanifold. Then the set of equivariant maps $f\colon X\to Y$
  such that $\smash{\mj_G^{k,\Lambda}f}$ is transverse to $W$ is residual. 
\end{theorem}

\subsection{The case $G = \Z_2$}
We now apply our discussion of equivariant jets to the case of involutive links. Unless otherwise specified, henceforth we use $\cS$ to denote a disjoint union of circles, but now equipped with a smooth involution $\sigma$ with isolated fixed points. (That is, we do not allow the degenerate case where $\sigma$ acts as the identity on an entire component.) Write $\Sfix$ for the fixed-point set of $\sigma$ and $\Sfree$ for the subset of $\cS$ on which $\sigma$ acts freely. 

\begin{defn} Let $\wt{\tau}(x, y, z) = (-x, y, -z)$ be the involution on $\R^3$ given by rotation about the $y$-axis. We denote the fixed-point axis of $\wt{\tau}$ by $\wt{\cL} \subset \R^3$. Likewise, let $\tau(x, y) = (-x, y)$ be the corresponding involution on $\R^2$. We denote the fixed-point axis of $\tau$ by $\cL = \pi(\wt{\cL}) \subset \R^2$. 
\end{defn}

\begin{defn}
Let $\cS$ be a disjoint union of circles equipped with an involution $\sigma$. A link $\wt{\phi} \colon \cS \rightarrow \R^3$ is \textit{involutive} if 
\[
\wt{\phi} \circ \sigma = \wt{\tau} \circ \wt{\phi}
\]
and the image of $\phi$ intersects $\wt{\cL}$ at a finite number of points.\footnote{It is easily checked that this means each component of $\phi(\cS)$ must intersect $\wt{\cL}$ in either zero or two points.} The resulting link diagram $\phi = \pi \circ \wt{\phi}$ is equivariant in the following sense:
\[
\phi \circ \sigma = \tau \circ \phi.
\]
Note that the axis of symmetry is visible in the plane of the page. We call such a diagram a \textit{transvergent diagram}, in contrast to an \textit{intravergent diagram}, for which the axis of symmetry intersects the plane of projection at a single point. 
\end{defn}

We now describe the equivariant jet space $\mJ^k_{\Z_2}(\cS, \R^2)$, which will be fundamental to the rest of the paper. In this context, there are two orbit types, corresponding to points with trivial isotropy group and points that are fixed by $G$. We denote these orbit types by $\lambda_0$ and $\lambda_1$, respectively. By definition, the equivariant jet space corresponding to $\lambda_0$ is simply the non-equivariant jet space over $\Sfree/\Z_2$, which is diffeomorphic to a collection of open intervals:
\[
\mJ^{k, \lambda_0}_{\Z_2}(\cS, \R^2) = \mJ^k(\Sfree/\Z_2, \R^2).
\]
Elements of $\mJ^{k, \lambda_0}_{\Z_2}(\cS, \R^2)$ are pairs of germs of the form $(f, \tau f \sigma)$, where $f$ is defined on a small subinterval of $\Sfree$.

To understand the equivariant jet space corresponding to $\lambda_1$, we consider $\Z_2$-equivariant germs from points in $\Sfix$ to points in $(\R^2)^{\Z_2} = \cL$. Thus, let $x \in \Sfix$. Choose a $\Z_2$-invariant neighborhood $V$ of $x$ with parameter $t$ such that $x$ is taken to $0$ and $\sigma$ acts by negation. The set of real $\Z_2$-invariant germs at $x$ then consists of power series in $t^2$. The corresponding maximal ideal is given by $(t^2)$, and hence 
\[
  \mm^{\Z_2}(V)^{k+1}=(t^{2k+2}).
\]
Now let $y$ be a point in $\cL$ and consider $\cE^{\Z_2}(V, W)$, where $W = \R^2$ is equipped with the involution $\tau$. We treat this as a global equivariant coordinate chart and relax the condition that $y$ be identified with the origin. In local coordinates, such a germ is given by $(x(t), y(t))$, where $x(t)$ and $y(t)$ are power series in $t$. Imposing $\Z_2$-equivariance means that $x(-t) = -x(t)$ and $y(-t) = y(t)$; this occurs when $x(t)$ consists of odd-exponent terms and $y(t)$ consists of even-exponent terms. Quotienting out by $\mm^{\Z_2}(V)^{k+1}$, we thus have that 
\[
\mJ^{k, \lambda_1}_{\Z_2}(x, y) = \{(\alpha_1t+ \alpha_3 t^3 + \cdots + \alpha_{2k+1} t^{2k+1}, \beta_0+\beta_2t^2+ \cdots +\beta_{2k}t^{2k})\}.
\]
Note that this means for points in $\Sfix$, the $\Z_2$-equivariant $k$-jet space records derivatives of germs up to order $2k + 1$ rather than $k$, although equivariance causes half of these to be zero. This convention is discussed in \cite{Wall_jet}. 

We now consider equivariant multijets. As discussed in Section~\ref{sub:ejet}, we obtain an equivariant multijet space for every choice of tuple $\Lambda = (\lambda_0, \cdots, \lambda_0, \lambda_1, \cdots, \lambda_1) = (m \lambda_0, n \lambda_1)$. An element of this multijet space lives over $m$ distinct equivalence classes in $\Sfree/\Z_2$ and $n$ distinct points in $\Sfix$. 

\begin{defn}
We write 
\[
\mJ^{k, (m, n)}_{\Z_2}(\cS, \R^2) = \mJ^{k, (m \lambda_0, n \lambda_1)}_{\Z_2}(\cS, \R^2).
\]
When our multijet space only involves orbits of type $\lambda_0$, we frequently use a shorter notation
\[
\mJ^{k, m}(\cS, \R^2):=
\mJ^{k, (m, 0)}_{\Z_2}(\cS, \R^2).
\]
\end{defn}

\begin{example}
Consider the subset of equivariant multijet space that corresponds to a triple point off of the axis of symmetry. Phrased in terms of $\Z_2$-equivariant maps from $\cS^2$ to $\R^2$, this is given by the condition
\[
\exists t_1,t_2,t_3\in \Sfree \text{ pairwise distinct, such that } \phi(t_1)=\phi(t_2)=\phi(t_3) \notin \cL.
\]
Note that the above occurs if and only if $\sigma(t_1)$, $\sigma(t_2)$, and $\sigma(t_3)$ form a triple point on the other side of the axis of symmetry. However, as these conditions are not independent, we still think of the singularity as involving three (rather than six) points in the domain. It is straightforward to rephrase the above condition in terms of jets over three distinct classes in $\Sfree/\Z_2$. This takes place in
\[
\mJ^{0, (3, 0)}_{\Z_2}(\cS, \R^2) = \mJ^{0, 3}(\cS, \R^2).
\]
\end{example}

\begin{example}
Consider the subset of equivariant multijet space that corresponds to having a double point on the axis of symmetry. Phrased in terms of $\Z_2$-equivariant maps from $\cS^2$ to $\R^2$, this is given by the condition
\[
\exists t_1\in \Sfree, t_2\in \Sfix \text{ such that } \phi(t_1)=\phi(t_2).
\]
Note that the above occurs if and only if $\phi(\sigma(t_2))$ also coincides with $\phi(t_1) = \phi(t_2)$. Once again, however, we think of the singularity as involving two (rather than three) points in the domain. It is straightforward to rephrase the above condition in terms of a jet over a class in $\Sfree/\Z_2$ and a jet over a $\Z_2$-fixed point. This takes place in
\[
\mJ^{0, (1, 1)}_{\Z_2}(\cS, \R^2).
\]
\end{example}

We now have the analogue of the parameter counting theorem, which we state for the case of equivariant maps from $\cS$ to $\R^2$. This follows immediately from the same argument as before, after observing that the dimension of the domain of the relevant jet extension is equal to the number of factors of orbit type $\lambda_0$.

\begin{lemma}[Equivariant Parameter Counting]\label{lem:codimension}
  Let $R$ be an $r$-dimensional manifold, which we think of as a parameter space specifying a family of equivariant link diagrams. Let $W\subset \mJ^{k,(m, n)}_{\Z_2}(\cS,\R^2)$ be a submanifold of
  codimension $w$. If $r + m < w$, then there is a residual subset of equivariant maps from $R\times \cS$ to $\R^2$ such
  that for any $F$ in that subset and any $u\in R$, the equivariant jet extension 
  $\smash{\mj^{k,(m, n)}_{\Z_2}f_u}$ of $f_u =F(u,\cdot)\colon X\to Y$ misses $W$.
\end{lemma}
\begin{proof}[Sketch of proof]
  The proof is analogous to the proof of Theorem~\ref{thm:pc_argument}. The difference is that
  in Theorem~\ref{thm:pc_argument}, we consider maps from $R\times X$ to $Y$. Here, we consider maps
  from $R\times(\Sfree)^m\times(\Sfix)^n$. The choice of the source is dictated by the parameters $m,n$ depending on
  the orbit type $W$ belongs to. The dimension of the source is $r+m$. The rest of the proof uses the same argument.
  \end{proof}

Once again, we can phrase this in terms of the codimension of a set of equivariant maps:

\begin{definition}\label{def:codimension}
Let $D$ be a submanifold of $\smash{\mJ_{\Z_2}^{k,(m, n)}(\cS,\R^2)}$.
The \textit{(jet space) codimension of $D$} (sometimes called the \textit{number of defining conditions}) is the usual codimension $\codim(D)$ of $D$. We define the \textit{(parametric) codimension of $D$} to be
\[
\pcodim(D) = \codim(D) - m.
\]
Note that $m$ is the dimension of the domain of the relevant multijet extension. In this language, the parameter counting argument states that a generic $r$-parameter family of equivariant link diagrams will miss $D$ so long as $r < \pcodim(D)$. 

\end{definition}
\begin{remark}
  The parametric codimension in the complex setting was studied e.g.\ in \cite{BZ}. In the real case, many calculations can be repeated.
\end{remark}

\subsection{Equivariant regular diagrams}\label{sub:2sym}

We close this section by establishing the genericity of equivariant regular diagrams. For this, we generalize the argument of Section~\ref{sec:warmup}. As henceforth all of our objects will be equivariant, we re-use the same notation $\cF$ (and $\cF^0$, and so on) as in the non-equivariant caase.

\begin{definition}\label{def:half-knot}
For a fixed $\cS$, let $\cF$ be the space of all equivariant smooth maps from $\cS$ to $\R^2$ that take distinct points in $\Sfix$ to distinct points in $\cL$.\footnote{It is clear that an actual link diagram will have this property, but here we are simply studying the space of equivariant maps $\phi$ from $\cS$ to $\R^2$. The reason that this condition is separated from Definition~\ref{def:regular-half-knot} below is that it does not lead to an equivariant Reidemeister move; rather, it should be thought of as a necessary condition for $\phi$ to lift to an actual involutive link.}
\end{definition}

The following is the analogue of Definition~\ref{def:generic_R2}:

\begin{defn}\label{def:regular-half-knot}
  Let $\cF^0$ be the subspace of $\cF$ consisting of equivariant maps satisfying the following conditions:
  \begin{enumerate}[label=(ER-\arabic*)]
    \item $\phi'$ is nonzero at every point of $\Sfree$; \label{item:nocusp}
    \item for any $w\in\R^2\setminus \cL$, if $\phi^{-1}(w)$ consists of two points $t_1,t_2$, then
      $\phi'(t_1)$, $\phi'(t_2)$ are linearly independent;\label{item:notang}
    \item for any $w\in\R^2\setminus \cL$, the preimage $\phi^{-1}(w)$ consists of at most two points; \label{item:notriple}
    \item if $t\in \Sfree$ is mapped to $\cL$, then $\phi'(t)$ is not tangent to $\cL$;\label{item:no_local_tang}
    \item if $t\in \Sfree$ is mapped to $\cL$, then $\phi'(t)$ is not perpendicular to $\cL$;\label{item:perp}
    \item if $t_1,t_2\in \Sfree$ and $\phi(t_1),\phi(t_2)\in \cL$, then $\phi(t_1)\neq \phi(t_2)$ unless $t_1,t_2$
      are in the same orbit of $\sigma$;\label{item:nocross}
    \item $\phi'$ is nonzero at every point of $\Sfix$;\label{item:no_end_tang}
    \item $\phi(\Sfree)$ and $\phi(\Sfix)$ are disjoint.\label{item:no_pass_end}
  \end{enumerate}
  We refer to such a $\phi \in \cF^0$ as being \textit{regular}.
\end{defn}
\begin{remark}
Several of these items are quite similar: for example, \ref{item:nocusp} and \ref{item:no_end_tang} can obviously be combined into the condition that $\phi'$ is nonzero at every point of $\cS$. We separate these because their violation leads to different Reidemeister moves in the equivariant setting.
\end{remark}

The following theorem relates $\cF^0$ above with $\cF^0$ from Definition~\ref{def:generic_R2}, and justifies the usage of the term regularity for both: 

\begin{thm}\label{thm:generic_half_knot}
Let $\phi \colon \cS \rightarrow \R^2$ be equivariant. Then $\phi$ is in $\cF^0$ (as in Definition~\ref{def:regular-half-knot}) if and only if it is in $\cF^0$ (as in Definition~\ref{def:generic_R2}) as a non-equivariant map.
  \end{thm}
\begin{proof}
Assume $\phi$ satisfies Definition~\ref{def:regular-half-knot} and consider Definition~\ref{def:generic_R2}. Clearly, \ref{item:R1} is implied by \ref{item:nocusp} and \ref{item:no_end_tang}. Next, consider \ref{item:R2}, which means that there are no tangencies at double points. If $w \in \R^2 \setminus \cL$, this is guaranteed by
  \ref{item:notang}. If $w \in \cL$, then the corresponding points $t_1, t_2$ in $\cS$ are either both in $\Sfix$, or one is in $\Sfix$ and one is in $\Sfree$, or both are in $\Sfree$. The first is disallowed by Definition~\ref{def:half-knot} and the second cannot happen by \ref{item:no_pass_end}. In the third case, we must have $t_1 = \sigma(t_2)$ by \ref{item:nocross}; it is then easily seen that $\phi'(t_1)$ and $\phi'(t_2)$ are linearly independent by the equivariance of $\phi$ combined with \ref{item:no_local_tang} and \ref{item:perp}. Finally, consider \ref{item:R3}, which means that there are no triple points. If $w \in \R^2 \setminus \cL$, this is guaranteed by \ref{item:notriple}. If $w \in \cL$, we consider the possibilities for $t_1, t_2, t_3$. If all three points lie in $\Sfree$, we obtain a contradiction with \ref{item:nocross} since each orbit has size two. Otherwise, exactly one point is in $\Sfix$ by Definition~\ref{def:half-knot}. But this contradicts \ref{item:no_pass_end}.
  
Conversely, assume $\phi$ satisfies Definition~\ref{def:generic_R2} and consider Definition~\ref{def:regular-half-knot}. Clearly, \ref{item:nocusp} and \ref{item:no_end_tang} are implied by \ref{item:R1}. Next, \ref{item:notang} is obviously implied by \ref{item:R2}. Considering \ref{item:no_local_tang} and \ref{item:perp}, we see that these are implied by \ref{item:R2} as well. Indeed, if $t \in \Sfree$ is mapped to $\cL$, then $\sigma(t)$ is also mapped to $\cL$; hence if \ref{item:no_local_tang} or \ref{item:perp} were violated, then the equivariance of $\phi$ would create a double point tangency. Likewise, \ref{item:R2} implies Definition~\ref{def:half-knot}, since if $t_1, t_2 \in \Sfix$ were both mapped to the same point of $\cL$, it is easily checked that $\phi'(t_1)$ and $\phi'(t_2)$ would both be perpendicular to $\cL$ and thus create a double point tangency. Finally, \ref{item:R3} obviously implies \ref{item:notriple}. Considering \ref{item:nocross} and \ref{item:no_pass_end}, we see that these are implied by \ref{item:R3} as well. Indeed, if \ref{item:nocross} was violated, we would obtain a quadruple point with preimage consisting of a pair of two-element orbits of $\sigma$; if \ref{item:no_pass_end} were violated, we would obtain a triple point with preimage consisting of a two-element orbit of $\sigma$ and a point in $\Sfix$.
\end{proof}

We close this section by establishing the analogue of Theorem~\ref{thm:cf0}. As in Section~\ref{sec:warmup}, this will be established by investigating the codimension of the strata corresponding to violations of \ref{item:nocusp} through \ref{item:no_pass_end}.

\begin{defn}\label{def:equivariant-stratum}
  For each $i$, let $\wt{\cF}^1_i$ be the subset of $\cF$ on which (\textrm{ER}-$i$) is violated. These are described explicitly in Table~\ref{tab:cpl}. Denote 
\[
\wt{\cF}^1 = \wt{\cF}^1_1  \cup \cdots \cup \wt{\cF}^1_8 = \cF \setminus \cF^0.
\]
\end{defn}

\begin{table}
  \begin{tabular}{|c|c|p{2.8cm}|p{5.4cm}|c|c|}\hline
    Violates & Stratum & Name & Condition & Jet& Codim\\ \hline
    \ref{item:nocusp} & $\wt{\cF}^1_1$ & off-axis cusp & $\exists t \in \Sfree, \phi(t) \notin \cL, \phi'(t)=0$ & $\mJ^1$ & 2\\ \hline
    \ref{item:notang} & $\wt{\cF}^1_2$ & off-axis tangency & $\exists t_1\neq t_2\in \Sfree, \phi(t_1)=\phi(t_2) \notin \cL$ and 
    $\phi'(t_1),\phi'(t_2)$
    are linearly dependent &$\mJ^{1,2}$ & 3 \\ \hline
    \ref{item:notriple} & $\wt{\cF}^1_3$ & off-axis triple~point \hfill &
    $\exists t_1,t_2,t_3\in \Sfree$ pairwise distinct, such that $\phi(t_1)=\phi(t_2)=\phi(t_3) \notin \cL$ & $\mJ^{0,3}$ & 4\\ \hline
    \ref{item:no_local_tang} & $\wt{\cF}^1_4$ & on-axis line tangency &  $\exists t\in \Sfree$, $\phi(t)\in \cL$, $\phi'(t)$
    is tangent to $\cL$ & $\mJ^{1}$ & 2\\ \hline
    \ref{item:perp} & $\wt{\cF}^1_5$ & on-axis perpendicular tangency & $\exists t\in \Sfree$, $\phi(t)\in \cL$, $\phi'(t)$
    is perpendicular to $\cL$ & $\mJ^{1}$ & 2\\\hline
    \ref{item:nocross} & $\wt{\cF}^1_6$ & on-axis double point & $\exists t_1\neq t_2 \in \Sfree$, $t_1\neq \sigma(t_2)$ such that $\phi(t_1)=\phi(t_2)\in \cL$
		       & $\mJ^{0,2}$ & 3\\ \hline
    \ref{item:no_end_tang} & $\wt{\cF}^1_7$ & fixed-point cusp & $\exists t\in \Sfix$ such that $\phi'(t)=0$ & $\mJ_{\Z_2}^{1, (0, 1)}$ & 1\\ \hline
    \ref{item:no_pass_end} & $\wt{\cF}^1_{8}$ & fixed double point & $\exists t_1\in \Sfree$, $t_2\in \Sfix$ such that $\phi(t_1)=\phi(t_2)$ & $\mJ_{\Z_2}^{0, (1, 1)}$ & 2
    \\\hline
  \end{tabular}
  \\
  \caption{The strata $\smash{\wt{\cF}^1_i}$. We refer to $\smash{\wt{\cF}^1_6}$ and $\smash{\wt{\cF}^1_8}$ as double points, even though taking into account all symmetric branches makes them into quadruple points and triple points, respectively.}\label{tab:cpl}
\end{table}
\begin{thm}\label{thm:open_dense}
  The subset $\cF^0$ of regular maps is open-dense in $\cF$.
\end{thm}
\begin{proof}
It is enough to show that a generic map misses $\wt{\cF}^1_i$ for $1 \leq i \leq 8$.
This is done through a case-by-case analysis. In the last two columns of Table~\ref{tab:cpl}, we give the relevant multijet spaces
and the (jet space) codimension of the appropriate submanifold. Note that for the first six entries, the conditions do not involve points of $\Sfix$, so by the discussion in Section~\ref{sub:ejet}, we may phrase our analysis in terms of non-equivariant jet spaces.

We thus only comment on $\wt{\cF}^1_7$ and $\wt{\cF}^1_{8}$. For $\wt{\cF}^1_7$, the relevant condition is codimension~1 since in the equivariant jet space, one component of the derivative already automatically vanishes at points of $\Sfix$. The source is of dimension~0, so the stratum is missed by a generic map. For $\wt{\cF}^1_{8}$, we relevant condition is codimension~2. The source is of dimension~1, since $t_1$ varies in a space of dimension~1 while $t_2$ varies in a space of dimension~0, so the stratum is missed by a generic map.
\end{proof}
Theorem~\ref{thm:perestrojka} and Corollary~\ref{cor:Reidemeister} can be promoted to the equivariant setting. We will do this in
Subsection~\ref{sub:reg_path}, when we use the machinery of versality and normal forms to provide more rigorous proofs than those
in the non-equivariant setting.

\section{Deformations of singularities}\label{sec:def_versal}

In this section, we introduce the language needed to study deformations of singularities. Our main goal will be to introduce the notion of a versal deformation. Roughly speaking, this is a deformation of a singularity which induces any other deformation, up to reparameterization. We use versal deformations to classify families of singularities in the next section. 

\subsection{RL equivalence}\label{sub:cA}
We begin with a general discussion of equivalence of germs, starting with the non-equivariant case. For $\{t_1, \dots, t_n\} \subset \cS$, let $\cO(t_1,\dots,t_n)$ be the space of multigerms of smooth functions from $\{t_1, \dots, t_n\}$ to $\R^2$. Clearly,
\[
\cO(t_1,\dots,t_n) = \cO({t_1})\oplus\dots\oplus \cO({t_n}),
\]
where each $\cO({t_i})$ is the germ of smooth functions into $\R^2$ at $t_i$. We refer to the restriction of $\phi \in \cO(t_1, \dots, t_n)$ to a neighborhood of $t_i$ as the \textit{branch} of $\phi$ through $t_i$.
\begin{defn}[see \cite{AVG}]\label{def:RLnoneq}
Let $f \in \cO(t_1, \cdots, t_n)$ and $g \in \cO(t_1', \cdots, t_n')$. We say that $f$ and $g$ are \textit{right-left (RL) equivalent} if there exist open neighborhoods $U$ and $U'$ in $\cS$ and $V$ and $V'$ in $\R^2$, together with germs of diffeomorphisms $h_R$ and $h_L$, making the following diagram commute:
\[\begin{tikzcd}
	U & V \\
	U' & V'
	\arrow["{f}", from=1-1, to=1-2]
	\arrow["{h_R}"', from=1-1, to=2-1]
	\arrow["{h_L}", from=1-2, to=2-2]
	\arrow["{g}", from=2-1, to=2-2]
\end{tikzcd}\]
Here, $U$ and $U'$ are open neighborhoods of $\{t_1, \ldots, t_n\}$ and $\{t_1', \ldots, t_n'\}$, respectively, while $V$ and $V'$ are open neighborhoods of $\{f(t_1),\dots,f(t_n)\}$ and $\{g(t'_1),\dots,g(t'_n)\}$, respectively. The map $h_R \colon U \rightarrow U'$ is the germ of a diffeomorphism sending $t_i$ to $t'_i$ for all $i=1,\dots,n$, while the map $h_L \colon V \rightarrow V'$ is the germ of a diffeomorphism sending $f(t_i)$ to $g(t'_i)$ for $i=1,\dots,n$. It is clear that RL equivalence is an equivalence relation. 
\end{defn}

Now consider the equivariant case. For a subset of points $\{t_1, \dots, t_n\}$ which is setwise preserved by $\sigma$, we have the space of equivariant multigerms, which we denote by
\[
\cO(t_1,\dots,t_n)^{\Z_2} = (\cO({t_1})\oplus\dots\oplus \cO({t_n}))^{\Z_2} 
\]
We usually group $t_1,\dots,t_n$ in such a way so that $\sigma t_i=t_{i+1}$ for $i=1,3,\dots,2\ell-1$ and $\sigma t_i=t_i$ for $i=2\ell+1,\dots,n$. The equivariant multigerm space then decomposes as
\[
\cO(t_{1}, t_{2})^{\Z_2} \oplus\cdots \oplus \cO(t_{{2\ell-1}}, t_{{2\ell}})^{\Z_2} \oplus \cO(t_{{2\ell+1}})^{\Z_2}\oplus\dots\oplus \cO(t_{n})^{\Z_2}.
\]
Since the involution $\sigma$ on $\cS$ interchanges neighborhoods of $t_i$ and $t_{i+1}$, it is clear that 
\[
\cO(t_{2i-1}, t_{2i})^{\Z_2} \cong \cO(t_{{2i-1}})
\]
by taking $f \in \cO({t_{2i-1}})$ to $(f, \tau f \sigma)$. We refer to the restriction of $\phi$ to a neighborhood of a two-element orbit $\{t_i, \sigma t_i\}$ as a \textit{symmetric pair of branches}, while the restriction of $\phi$ to a neighborhood of a one-element orbit $t_i = \sigma t_i$ is called a \textit{fixed-point branch}. We usually abuse terminology and still refer to a symmetric pair of branches as a single branch. When the context is clear, we will correspondingly write $f$ in place of $(f, \tau f \sigma)$. 

\begin{defn}\label{def:RLeq}
Let $f \in \cO(t_1, \cdots, t_n)^{\Z_2}$ and $g \in \cO(t_1', \cdots, t_n')^{\Z_2}$. We say that $f$ and $g$ are \textit{right-left (RL) equivalent} if there exist data as in Definition~\ref{def:RLnoneq} which are equivariant. That is, the neighborhoods $U$ and $U'$ are required to be preserved by $\sigma$, the neighborhoods $V$ and $V'$ are required to be preserved by $\tau$, and $h_R$ and $h_L$ are required to be equivariant germs of diffeomorphisms.
\end{defn}

Let $D$ be a subset of some non-equivariant or equivariant multijet space. Up to RL equivalence, we would like to explicitly describe the space of germs appearing in $D$. By this, we mean the following: recall that a multijet is an equivalence class of germs over some $\{t_1, \dots, t_n\}$. Our aim is to give a model family of germs such that any germ in any equivalence class from $D$ is RL equivalent to some germ in our family.  

\begin{defn}
Fix $\{t_1, \dots, t_n\} \subset \cS$ and let $\Lambda$ be a smooth manifold, which we think of as a parameter space. A \textit{family} of germs at $\{t_1, \dots, t_n\}$ is a smooth map $\Phi \colon \Lambda \rightarrow \cO(t_1, \dots, t_n)$. The definition extends to the equivariant case by replacing $\cO(t_1, \dots, t_n)$ with $\smash{\cO(t_1, \dots, t_n)^{\Z_2}}$. We denote $\phi_{\lambda} = \Phi(\lambda)$. 
\end{defn}

\begin{defn}\label{def:normalform}
Suppose $D$ is a subset of some non-equivariant or equivariant multijet space. A \textit{normal form} for $D$ is a family of germs $(\Phi, \Lambda)$ such that any germ in any equivalence class from $D$ is RL equivalent to $\Phi(\lambda)$ for a finite (non-empty) set of $\lambda \in \Lambda$. 
\end{defn}

The point of Definition~\ref{def:normalform} is the following. Let $\phi$ be a map from $\cS$ to $\R^2$ whose jet extension hits $D$. Each such intersection corresponds to a multijet over a particular set of points $\{t_1, \cdots, t_n\}$. The assertion of Definition~\ref{def:normalform} is that up to local reparameterization of the domain and codomain, the behavior of $\phi$ at $\{t_1, \cdots, t_n\}$ is determined modulo some set of parameters from the parameter space $\Lambda$. An especially salient case will be when there are no parameters; e.g., $\Lambda$ consists of a single point. In this case, we refer to $D$ as being \textit{simple}. 

\begin{defn}
Let $D$ be a subset of some non-equivariant or equivariant multijet space. We say that $D$ is \emph{simple} if it has a normal form with $\Lambda$ a single point. Equivalently, this means that whenever $\phi_1 \in \cO(t_1, \cdots, t_n)$ and $\phi_2 \in \cO(t_1', \cdots, t_n')$ are two multigerms whose jet extensions lie in $D$, we have that $\phi_1$ and $\phi_2$ are RL equivalent.
\end{defn}

The reader should think of a simple subset of jet space as consisting of (the jet extensions of) precisely one RL orbit. The archetypal example is the following. 

\begin{example}
Consider the subset $D$ of $\mJ^{3}(\cS, \R^2)$ defined by the condition that there exists (exactly) one point $t$ such that $\phi'(t) = 0$, and that at this point the vectors $\phi''(t)$ and $\phi'''(t)$ are linearly independent. This is the defining jet space for the stratum $\cF^1_1$ discussed in Section~\ref{sec:noneqreidemeister}: if the jet extension of $\phi$ hits $D$, then this means $\phi$ has a $(2, 3)$-cusp. Now consider the germ 
\[
\phi(t) = (t^2, t^3).
\]
As we show in Section~\ref{sub:cusp}, this is a normal form for $D$. That is, up to RL equivalence, every $(2, 3)$-cusp locally looks like $\phi$. Since the normal form consists of only one point, $D$ is simple.
\end{example}

However, not every subspace of jet space is simple. 

\begin{example}
  Consider the subset $D$ of $\mJ^{0, 4}(\cS, \R^2)$ that detects an ordinary quadruple point, i.e. the $X_9$ singularity of \cite{AVG}.
It can be shown, see \cite{David,FiedlerKurlin} that
a normal form for $D$ is given by the family of multigerms 
\[
\Phi(\lambda)(s_1, s_2, s_3, s_4) = ((s_1, 0), (0, s_2), (s_3, s_3), (s_4, \lambda s_4)).
\]
This multigerm looks like four lines intersecting at the origin: the claim is that the slopes of three of the lines can be set, but different choices for the slope of the fourth line will in general lead to RL-inequivalent multigerms. In particular, $D$ is not simple.
\end{example}

\subsection{Versal deformations}\label{sub:versal}

We now use our formalism to discuss deformations of germs. By a deformation of a germ $\phi$, we of course mean a family of germs $(\Phi, \Lambda)$ such that $\Phi(\lambda_0) = \phi$ for some $\lambda_0 \in \Lambda$. Usually $\Lambda$ will be a neighborhood of the origin in some Euclidean space and $\lambda_0$ will be the origin.

\begin{defn}
Let $(\Phi_1, \Lambda_1)$ and $(\Phi_2, \Lambda_2)$ be two deformations of $\phi = \Phi_1(\lambda_1) = \Phi_2(\lambda_2)$. We say that $\Phi_1$ is \emph{induced} from $\Phi_2$ if there exists open neighborhoods $U_1,U_2$ of $\lambda_1$ and $\lambda_2$ in $\Lambda_1$, respectively $\Lambda_2$, together with a smooth map $h\colon U_1\to U_2$ with $h(\lambda_1) = \lambda_2$ such that $\Phi_1(\lambda)$ and $\Phi_2(h(\lambda))$ are RL equivalent for all $\lambda\in U_1$. The reparametermizations realizing the RL equivalence
are required to smoothly depend on $\lambda$, see \cite[Section 8]{AVG}. A deformation is \emph{versal} if any other deformation can be induced from
it.
\end{defn}

A key tool to understanding versal deformations is the notion of infinitesimal versality, see \cite[Section 8.2]{AVG}. For this, fix a germ $\phi \in \cO(t_1, \dots, t_n)$. Write $[\phi]$ for space of germs which are RL equivalent to $\phi$. We think of $[\phi]$ by first forming the orbit of $\phi$ in $\cO(t_1, \dots, t_n)$ under germs of diffeomorphisms of the domain that fix $\{t_1, \dots, t_n\}$ and arbitrary germs of diffeomorphisms of the codomain. In a small neighborhood $U$ of $\{t_1, \dots, t_n\}$, it is clear that we may apply small translations of the $t_i$ and speak of the space of all germs at $n$ distinct points in $U$. This describes the local structure of $[\phi]$, which we think of as a smooth manifold. 

\begin{definition}\label{def:infinitesimallyversal}
Let $(\Phi, \Lambda)$ be a deformation of $\phi = \Phi(\lambda_0)$. We say that $\Phi$ is \emph{infinitesimally versal}
if it is transverse to $[\phi]$ at $\lambda_0$.
\end{definition}

The importance of infinitesimal versality is that it can be verified algebraically. Indeed, let $\phi \in \cO(t_1, \dots, t_n) = \cO(t_1) \oplus \cdots \oplus \cO(t_n)$. Choose local coordinates $s_i$ around each $t_i$ such that $t_i$ is taken to zero, and let $\phi_i$ be restriction of $\phi$ near $t_i$. Then $\phi$ is given by an $n$-tuple of pairs of power series
\[
\phi_i(s_i) = (x_i(s_i), y_i(s_i)).
\]
The tangent space to $[\phi]$ at $\phi$ is given by the space of germs in $\cO(t_1, \dots, t_n)$ of the form
\begin{equation}\label{eq:RL_versal}
  \left(h_1(s_1)\frac{d\phi_1}{ds_1},\dots,h_n(s_n)\frac{d\phi_n}{ds_n}\right)+(k(\phi_1),\dots,k(\phi_m)),
\end{equation}
where
\begin{enumerate}
\item Each $h_i$ is a germ at $t_i$ of a map into $\R$; and,
\item The map $k$ is the germ at $\{\phi(t_1), \ldots, \phi(t_n)\}$ of a map into $\R^2$.
\end{enumerate}

To see this, consider reparameterizations of the domain. Let the reparameterization near $t_i$ be given by $s_i \mapsto s_i + \epsilon h_i(s_i)$. (Strictly speaking, we have two kinds of reparameterizations: those that fix $t_i$ correspond to the case $h_i(t_i) = 0$, while the case $h_i(t_i) \neq 0$ is achieved by translating $t_i$.) To first order in $\epsilon$, such a reparameterization clearly alters $\phi$ by $\epsilon$ times the first term in \eqref{eq:RL_versal}. Likewise, consider a reparameterization of the codomain of the form $(x, y) \mapsto (x, y) + \epsilon k(x, y)$. Here, $k(x, y)$ is the germ at $\{\phi(t_1), \ldots, \phi(t_n)\}$ of a map into $\R^2$. Note that we treat $k$ as a single germ -- as opposed to a tuple of distinct germs -- since the codomain points $\phi(t_i)$ may not be distinct. This gives the second term in \eqref{eq:RL_versal}. 

We thus obtain the following; see \cite[Section 8.2, Theorem]{AVG}.
\begin{thm}\label{thm:d_deform}
Let $(\Phi, \Lambda)$ be a deformation of $\phi = \Phi(\lambda_0) \in \cO(t_1, \ldots, t_n)$. Then $\Phi$ is infinitesimally versal if and only if any $\nu\in\cO(t_1, \ldots, t_n)$ can be written as
  \begin{equation}\label{eq:d_deform}\nu=
  \left(h_1(s_1)\frac{d\phi}{ds_1},\dots,h_n(s_n)\frac{d\phi}{ds_n}\right)+(k(\phi_1),\dots,k(\phi_n))+\sum_{j=1}^m c_j\frac{\partial \Phi}{\partial\lambda_j}(\lambda_0),\end{equation}
  where $h_i,k$ are any germs as above, $\lambda_1,\dots,\lambda_m$ are local coordinates on $\Lambda$ near $\lambda_0$,
  and $c_j$ are real numbers.
\end{thm}

The above notions generalize in a straightforward manner to the equivariant setting by requiring all germs to be equivariant. We now give several examples to clarify our discussion. 

\begin{example}\label{ex:t25}
Consider a $(2, 5)$-cusp with germ $\phi(s)=(s^2,s^5)$. The tangent space to $[\phi]$ at $\phi$ consists of germs $\nu = (\nu^x, \nu^y)$ of the form
\[
(\nu^x(s), \nu^y(s)) = h(s) (2s, 5s^4) + (k_1(s^2, s^5), k_2(s^2, s^5))
\]
for germs $h(s)$ and $k(x, y) = (k_1(x, y), k_2(x, y))$. Consider which power series can be obtained for $\nu^x(s)$. By an appropriate choice of $k_1$, we can set any coefficient of $\nu^x(s)$ whose exponent lies in the semigroup generated by $2$ and $5$; this leaves only the coefficients of $s$ and $s^3$. These remaining coefficients can be set by an appropriate choice of $h(s)$. Likewise, we can simultaneously set any coefficient of $\nu^y(s)$ other than those of $s$ and $s^3$ by an appropriate choice of $k_2$. These last two coefficients cannot be altered by our choice of $h$ or $k$. Hence the tangent space to $[\phi]$ consists of germs $\nu = (\nu^x, \nu^y)$ such that the linear and cubic coefficients of $\nu^y$ are zero.

An example of an infinitesimally versal deformation is thus given by 
\[
\Phi(\lambda_1, \lambda_2)(s) = (s^2,s^5+\lambda_1s^3+\lambda_2s).
\]
\end{example}

\begin{example}\label{ex:R3_tangent}
  Consider a triple point with multigerm $\phi_1(s_1)=(s_1,0)$, $\phi_2(s_2)=(0,s_2)$, $\phi_3(s_3)=(s_3,s_3)$.
  The tangent space to $[\phi]$ at $\phi$ consists of multigerms $(\nu_1 = (\nu_1^x, \nu_1^y), \nu_2 = (\nu_2^x, \nu_2^y), \nu_3 = (\nu_3^x, \nu_3^y))$ of the form
  \begin{align*}
  &(\nu_1^x(s_1), \nu_1^y(s_1)) = h_1(s_1)(1, 0) + (k_1(s_1, 0), k_2(s_1, 0)) \\
  &(\nu_2^x(s_2), \nu_2^y(s_2)) = h_2(s_2)(0, 1) + (k_1(0, s_2), k_2(0, s_2)) \\
  &(\nu_3^x(s_3), \nu_3^y(s_3)) = h_3(s_3)(1, 1) + (k_1(s_3, s_3), k_2(s_3, s_3))
  \end{align*}
  for germs $h_i(s_i)$ and $k(x, y) = (k_1(x, y), k_2(x, y))$. An examination of the above shows that $\nu_1^y(0) = k_2(0, 0)$ and $\nu_2^x(0) = k_1(0, 0)$, and that $\nu_3^x(0) = h_3(0) + k_1(0, 0)$ and $\nu_3^y(0) = h_3(0) + k_2(0, 0)$. Hence we have the equality of constant terms
\[
\nu_3^x(0) - \nu_3^y(0) = \nu_2^x(0) - \nu_1^y(0).
\]
A rather tedious analysis shows that every other coefficient of $\nu_i^x$ and $\nu_i^y$ can be set via appropriate choices of $h_i$ and $k$. Hence the tangent space $[\phi]$ is given by multigerms satisfying the single constraint above.

An example of an infinitesimally versal deformation is thus given by 
 \[\Phi(\lambda)(s_1,s_2,s_3)=((s_1,0),(0,s_2),(s_3,s_3+\lambda)).
 \]

\end{example}
\begin{example}\label{ex:defining}
We now give an equivariant example. Consider an on-axis $(3, 4)$-cusp with germ $\phi(s) = (s^3, s^4)$. The tangent space to $[\phi]$ at $\phi$ consists of germs $\nu = (\nu^x, \nu^y)$ of the form
\[
(\nu^x(s), \nu^y(s)) = h(s) (3s^2, 4s^3) + (k_1(s^3, s^4), k_2(s^3, s^4))
\]
where now $h(s)$ and $k(x, y) = (k_1(x, y), k_2(x, y))$ are equivariant germs. That is: 
\begin{enumerate}
\item $h(-s) = - h(s)$, which means $h$ is a power series with only odd exponents; and,
\item $k_1(-x, y) = -k_1(x, y)$ and $k_2(-x, y) = k_2(x, y)$, which means every term in $k_1$ has an odd exponent of $x$ and every term of $k_2$ has an even exponent of $x$.
\end{enumerate}
Consider which power series can be obtained for $\nu^x(s)$. By an appropriate choice of $k_1$, we can set any coefficient of $\nu^x(s)$ whose exponent is the sum of an odd multiple of $3$ and a multiple of $4$. Note that we are computing the tangent space inside the set of equivariant germs, so $\nu^x(s)$ must satisfy $\nu^x(-s) = - \nu^x(s)$. Hence we only consider $\nu^x(s)$ consisting of power series with odd exponents. Among such power series, our choice of $k_1$ thus allows us to set any coefficient other than that of $s$ and $s^5$. The coefficient of $s^5$ can then be set by an appropriate choice of $h(s)$. Likewise, we can set any coefficient of $\nu^y(s)$ whose exponent is the sum of an even multiple of $3$ and a multiple of $4$. Since $\nu^y(s)$ must satisfy $\nu^y(-s) = \nu^y(s)$, we only consider $\nu^y(s)$ consisting of power series with even exponents; thus the only remaining coefficient is that of $s^2$. Hence the tangent space to $[\phi]$ consists of germs $\nu = (\nu^x, \nu^y)$ such the linear term of $\nu^x$ and the quadratic term of $\nu^y$ vanish.

An example of an infinitesimally versal deformation is thus given by 
 \[
 \Phi(\lambda_1, \lambda_2)(s) = (s^3+\lambda_1s ,s^4+\lambda_2s^2).
 \]
\end{example}

The following result is due to Martinet \cite{Martinet}, see also \cite[Section 8.3]{AVG}. 
\begin{theorem}[Versality Theorem]\label{thm:versality_theorem}
  An infinitesimally versal deformation is versal.
\end{theorem}
\begin{proof}
The proof in \cite{Martinet} is given for the non-equivariant case, but applies in the equivariant setting
with minor changes. 
The situation is reduced to the following problem. Suppose $\Phi\colon(\Lambda,\lambda_0)\to\cO$ 
is infinitesimally versal, and $\Phi'$ is a deformation from $\Lambda\times I$ (where $I$ is an interval), extending
$\Phi$. If we can show that $\Phi'$ is induced from $\Phi$, then a quick argument gives Theorem~\ref{thm:versality_theorem}.

The latter statement is the reduction lemma \cite[Lemma I.2.2]{Martinet}, \cite[Section 8.3, Lemma]{AVG} and constitutes
the nontrivial part of the proof. The reduction is first done formally, at the level of power series. Finally, the Malgrange preparation
theorem is applied to perform the reduction in the smooth category.

In this line, only the last step does not immediately carry over for the equivariant setting. In the last step,
we have to use the Malgrange theorem in the presence of a $\Z_2$-action. We explain this briefly in the appendix;
see Section~\ref{sec:app}. The basic tool is
the Whitney theorem \cite{Whitney}, stating that an even smooth function $f\colon\R\to\R$ can be written
as $f(x)=g(x^2)$ where $g$ is smooth. That is, the local ring of equivariant smooth functions is isomorphic
to the local ring of germs of smooth function, hence the equivariant version of Malgrange preparation theorem is deduced
from the classical, non-equivariant version.
\end{proof}

We use deformations to study singularities of maps in the following manner. Let $\phi$ be a multigerm that lies in one of the singular strata discussed in Sections~\ref{sec:warmup} or \ref{sec:mutt}. Suppose that $\Phi$ is a deformation of $\phi$. For generic $\lambda \in \Lambda$, we will usually have that $\phi_\lambda$ is regular in the sense of Definition~\ref{def:generic_R2}. However, depending on the dimension of $\Lambda$ and the nature of the singular stratum containing $\phi$, it may be that $\phi_\lambda$ continues to be singular on some subset of $\Lambda$ of positive codimension. We call this subset the \textit{discriminant}. 

\begin{definition}\label{def:discriminant}
Let $(\Phi, \Lambda)$ be a deformation of $\phi$. The discriminant of $(\Phi, \Lambda)$ is the subset $\Delta \subset \Lambda$ on which $\phi_\lambda$ is not regular. That is, for any $\lambda \in \Lambda$, the germ $\phi_\lambda$ is regular in the sense of Definition~\ref{def:generic_R2} if and only if $\lambda\notin\Delta$.
\end{definition}

The discriminant set is usually stratified, with strata of codimension~$1$ corresponding to singularities of codimension~1. Note that $\phi_\lambda$ for $\lambda \in \Lambda \setminus \Delta$ is regular, so within each connected component of $\Lambda \setminus\Delta$, the qualitative behavior
of $\phi_\lambda$ is the same.

\begin{example}\label{ex:codim1}
Suppose $\phi$ lies in a codimension~1 singular stratum and $(\Phi, \Lambda)$ is a $1$-parameter deformation. By shrinking $\Lambda$, we can usually assume that the discriminant consists of the single point $\lambda_0$ corresponding to $\phi$. Then $\Lambda \setminus \Delta$ has two connected components. In this case we think of the deformation as a path of germs $\phi_\lambda$ such that $\phi_\lambda$ crosses the singular stratum at exactly one point and has different behavior for $\lambda$ in the two components of $\Lambda \setminus \Delta$.

\end{example}

\begin{example}\label{ex:codim2}
Suppose $\phi$ lies in a codimension~2 singular stratum and $(\Phi, \Lambda)$ is a $2$-parameter deformation. We think of $\Lambda$ as a disk centered at $\lambda_0$. The discriminant locus will usually consist of $\lambda_0$ together with strata of dimension~$1$ that radiate outwards from $\lambda_0$ and divide $\Lambda$ into sectors. In most cases, for sufficiently small $r>0$, the circle $C_r$ with center $\lambda_0$ and radius $r$ (in some Riemannian metric), is transverse to each stratum of $\Delta$. That is, $C_r$ crosses $\Delta$ at finitely many points, corresponding to codimension~1 singularities. Going around $C_r$ leads to a loop of Reidemeister moves. We will return to this in Section~\ref{sec:codim2}.
\end{example}

We end this section by relating transversality in spaces of germs to transversality in jet space. Suppose $D$ is a subset of some non-equivariant or equivariant multijet space. Let $\Phi \colon \Lambda \rightarrow \cO(t_1,\dots,t_n)^{\Z_2}$ be a family of maps over a smooth parameter space $\Lambda$. Then we may consider the jet extension of $\Phi$; this is the map from $\Lambda$ into jet space obtained by taking the jet extension of $\Phi(\lambda)$ at each $\lambda$. 

\begin{thm}\label{thm:inf_vers}
	If the jet extension of $\Phi$ is transverse to a simple subset $D$ of jet space, then $\Phi$ is infinitesimally versal, hence versal.
\end{thm}
\begin{proof}[Sketch of proof]
  We give a sketch in the non-equivariant case, when $D\subset\mJ^k(\cS,\R^2)$. The general case is analogous, but involves more notation.
  Suppose $\dim\Lambda=r$.
  Suppose $\Phi\colon\Lambda\to\cO(t_1)$ is such that $\mj^k\Phi\colon\Lambda\to\mJ^k(\cS,\R^2)$ is transverse to $D$ at $\lambda_0$
  and set $\phi=\Phi(\lambda_0)$. Simplicity of $D$ means that for any 
  $\nu\in T_{\mj^k\phi}D$, there exists a 1-parameter family of RL-equivalence
  whose direction at the origin is $\nu$. Formally, in local coordinates, this means that for any $\nu\in\cO(t_1,\dots,t_n)$,
  there exists $h,k$, such that
  \[\nu(t)=h(t)\frac{d\phi}{dt}+k(\phi)\bmod t^{k+1},\]
  where we restrict modulo $t^{k+1}$, because we deal with transversality to the jet space.

  Next, transversality of $\Phi$ to $D$ implies that any direction transverse to $D$ can be realized by a derivative of $\Phi$
  along a parameter $\lambda$. Together with simplicity, this implies that for any $\nu\in\cO$, the equation
  \[\nu(t)=h(t)\frac{d\phi}{dt}+k(\phi)+\sum c_i\frac{\partial\Phi}{\partial \lambda_i}\]
  with unknowns $h,k,c_1,\dots,c_r$, has solution modulo $t^{k+1}$.
  This means that \eqref{eq:d_deform} is satisfied up to order $k+1$.
   We can also consider $D$ as a subset of a higher-order jet space by pulling back $D$
  via the projection map 
  \[
    \mJ^{k'}(\cS, \R^2) \rightarrow \mJ^{k}(\cS, \R^2)
  \]
  for any $k' > k$. Then $\mj^{k'}\Phi$ is still transverse to $D$. 

  This means that \eqref{eq:d_deform} is satisfied for all polynomials up to degree $t^{k'+1}$. By limiting procedure,
  we can find a solution  at the level of formal power series. By the Malgrange preparation
  theorem, it is satisfied at the level of smooth functions. Hence, the deformation $\Phi$ is infinitesimally versal, so it is versal.
\end{proof}
\begin{remark}
  Theorem~\ref{thm:inf_vers} is extensively used in Section~\ref{sec:patterns} to find versal deformations of simple singularities.
  In fact, we usually find a set of equations defining $D$. Then, a deformation that is transverse to the defining equation,
  is versal.
\end{remark}


\section{Singularities of equivariant maps}\label{sec:patterns}

We now determine normal forms and versal deformations of several classes of equivariant singularities. Our ultimate goal will be to list all singularities of $\pcodim \leq 2$ and find normal forms and versal deformations for each. We first describe coincidences and strikethrough singularities; these are general constructions by which more complicated singularities can be formed from simpler ones. We then move on to cusps and tangencies.

\subsection{Coincidences}\label{sub:coin}

We start by defining a coincidence of two singularities. This occurs when a map acquires two singularities independently. Formally, let $D_1$ and $D_2$ be subsets of $\mJ^{k_1, s_1}(\cS, \R)$ and $\mJ^{k_2, s_2}(\cS, \R)$, respectively. For $i = 1, 2$, let
\[
\pi_i \colon \mJ^{\max(k_1, k_2), s_1 + s_2}(\cS, \R) \rightarrow \mJ^{k_i, s_i}(\cS, \R) 
\]
be the projection map. This sends a multigerm over $s_1 + s_2$ points to its restriction over its first $s_1$ points (for $i = 1$) or its last $s_2$ points (for $i = 2$), and also throws out the information of the derivatives of order higher than $k_1$ (when $i = 1$) or $k_2$ (when $i = 2$). We also define
\[
(\R^2)^{s_1 \neq s_2} = \{(y_1,\dots,y_{s_1 + s_2})\in (\R^2)^{s_1 + s_2} \colon y_i\neq y_{j} \textrm{ if } i\le s_1\textrm{ and } j>s_1\}.
\]
to be the set consisting of $(s_1 + s_2)$-tuples of points in $\R^2$ such that none of the first $s_1$ points in $\R^2$ coincide with any of the last $s_2$ points in $\R^2$.  

\begin{defn}\label{def:coincidence}
Define the \textit{coincidence of $D_1$ and $D_2$} to be
\[
D_1 \wtimes D_2 = (\pi_1^{-1}(D_1) \cap \pi_2^{-1}(D_2)) \cap \mb^{-1}((\R^2)^{s_1 \neq s_2}).
\]
This is a subset of $\mJ^{\max(k_1, k_2), s_1 + s_2}(\cS, \R)$. Here, $\mb$ is the target map
\[
\mb \colon \mJ^{\max(k_1, k_2), s_1 + s_2}(\cS, \R) \rightarrow (\R^2)^{s_1 + s_2}.
\] 
\end{defn}

Explicitly, $D_1 \wtimes D_2$ consists of equivalence classes of multigerms $\phi$ such that: 
\begin{enumerate}
\item The ($k_1$th jet extension of the) restriction of $\phi$ to the first $s_1$ points lies in $D_1$;
\item The ($k_2$th jet extension of the) restriction of $\phi$ to the last $s_2$ points lies in $D_2$; and,
\item The image points under $\phi$ of its first $s_1$ domain points are disjoint from the image points of its last $s_2$ points.
\end{enumerate}
The space $D_1 \wtimes D_2$ is a relatively open subset of $\pi_1^{-1}(D_1) \cap \pi_2^{-1}(D_2)$. The difference corresponds to disallowing pairs of singularities occurring at the same image point. For example, if $D_1$ and $D_2$ describe triple points, then $D_1 \wtimes D_2$ describes a map with two triple points at two different places. On the other hand, in general $\pi_1^{-1}(D_1) \cap \pi_2^{-1}(D_2)$ allows singularities where the two triple points overlap; i.e., a 6-tuple point.

\begin{defn}
Definition~\ref{def:coincidence} clearly extends to the case of equivariant jet spaces, where if
\[
D_1 \subset \mJ_{\Z_2}^{k_1, (m_1, n_1)}(\cS, \R^2) \quad \text{and} \quad D_2 \subset \mJ_{\Z_2}^{k_2, (m_2, n_2)}(\cS, \R^2)
\]
then
\[
D_1 \wtimes D_2 \subset \mJ_{\Z_2}^{\max(k_1, k_2), (m_1 + m_2, n_1 + n_2)}(\cS, \R^2).
\]
The disjointness condition for the two halves of the domain is interpreted as follows: a germ $\phi$ in $D_1 \wtimes D_2$ has the property that no image points of $\phi$ associated to the first half of the domain (or the reflections of these points) coincide with any image points of $\phi$ associated to the second half of the domain (or the reflections of these points).
\end{defn}

The following computation is clear:

\begin{lemma}\label{lem:coincidence}
  We have $\pcodim(D_1 \wtimes D_2) = \pcodim(D_1) + \pcodim(D_2)$.
\end{lemma}
\begin{proof}
Clearly, 
\[
\codim(D_1 \wtimes D_2) = \codim(\pi_1^{-1}(D_1)) + \codim(\pi_2^{-1}(D_2)) = \codim(D_1) + \codim(D_2). 
\]
Moreover, the dimension of the domain for multigerms in $D_1 \wtimes D_2$ is the sum of the dimensions of the domains of multigerms in $D_1$ and $D_2$, respectively. This holds in both the non-equivariant case, where the dimension is $s_1 + s_2$, and in the equivariant case, where the dimension is $m_1 + m_2$. The claim follows.
\end{proof}

Given two multigerms $\phi_1$ and $\phi_2$ over disjoint sets of points and with disjoint images, we denote the obvious multigerm over the union of the domain points by $(\phi_1, \phi_2)$.

\begin{lemma}\label{lem:ver_coi}
Let $\Phi_1 \colon \Lambda_1 \rightarrow \cO_1$ and $\Phi_2 \colon \Lambda_2 \rightarrow \cO_2$ be versal deformations for $\phi_1$ and $\phi_2$, respectively. Then the obvious deformation
\[
\Phi_1 \times \Phi_2 \colon \Lambda_1 \times \Lambda_2 \rightarrow \cO_1 \oplus \cO_2
\]
is a versal deformation for $(\phi_1, \phi_2)$
\end{lemma}
\begin{proof}
We leave the verification to the reader.
\end{proof}

\subsection{Strikethroughs}\label{sub:strike}
We now turn to the case of strikethroughs. In a strikethrough singularity, we add an extra domain point to our multigerm $\phi$, and we require that the image of this point be equal to an already-existing image point (for concreteness, the image of the first domain point). Compare
the discussion in \cite[Section 2.2]{Stevens2}. For example, a strikethrough of a double point is a triple point, a strikethough of a triple point is a quadruple point, and so on. Formally, let $D \subset \mJ^{k, s}(\cS, \R^2)$. Let
\[
\pi \colon \mJ^{k, s + 1}(\cS, \R^2) \rightarrow \mJ^{k, s}(\cS, \R^2)
\]
be the projection and let
\[
\Delta = \left\{(u_1,\dots,u_{s+1})\in(\R^2)^{s+1}\colon u_1=u_{s+1}\right\}
\]
be the set of $(s+1)$-tuples of points in $\R^2$ where the first and last points coincide.

\begin{definition}\label{def:strikethrough}
Define the \textit{strikethrough} of $D$ to be
\[
\cancel{D} = \pi^{-1}(D) \cap \mb^{-1}(\Delta).
\]
This is a subset of $\mJ^{k, s + 1}(\cS, \R^2)$. Here, $\mb$ is the target map
\[
\mb \colon \mJ^{k, s + 1}(\cS, \R^2) \rightarrow (\R^2)^{s+1}.
\]
We say that a strikethrough is \textit{ordinary} if $\phi'(t_{s+1})$ is not parallel to any of the generalized tangent lines at the other $\phi(t_i)$. (Recall that the generalized tangent line to $\phi$ at $t_i$ is $\phi^{(m)}(t_i)$ for the least $m \geq 1$ for which this is nonzero.) In particular, note that $\phi'(t_{s+1})$ is required to be nonzero.
\end{definition}

Explicitly, $\cancel{D}$ consists of multigerms $\phi$ such that the restriction of $\phi$ to its first $s$ points lies in $D$, while the image of the $(s+1)$st point coincides with the image of the first point. (Usually, we think of $\phi$ as a multigerm with only one image point.) An ordinary strikethrough further has the property that the germ over $t_{s+1}$ has nonvanishing derivative and is not tangent to any of the other branches of the multigerm.

\begin{defn}
In the equivariant case, we will deal exclusively with the case that the new domain point $t_{s+1}$ lies in $\Sfree$, and so constitutes an orbit $\{t_{s+1}, t_{s+2} = \sigma t_{s+1}\}$. Thus, if
\[
D \subset \mJ^{k, (m, n)}_{\Z_2}(\cS, \R^2)
\]
then
\[
\cancel{D} \subset \mJ^{k, (m +1, n)}_{\Z_2}(\cS, \R^2).
\]
The requirement that $\phi(t_{s+1}) = \phi(t_1)$ is interpreted in the appropriate sense of orbits; that is, we require $\{\phi(t_{s+1}), \phi(t_{s+2}) = \tau \phi(t_{s+1})\}$ to intersect $\{\phi(t_1), \tau \phi(t_1)\}$ nontrivially, where the the latter may consist of one or two points. We say that such a strikethrough is ordinary if $\phi'(t_{s+1})$ is nonzero and is not parallel to any of the generalized tangent lines at the other $\phi(t_i)$, or their reflections. If $\phi(t_{s+1})$ lies on $\cL$, we further require that $\phi'(t_{s+1})$ is neither parallel nor perpendicular to $\cL$.
\end{defn}

The following computation is clear:

\begin{lemma}\label{lem:strikethrough}
  We have $\pcodim(\cancel{D}) = \pcodim(D) + 1$. Moreover, a non-ordinary strikethrough has codimension at least $\pcodim(D) + 2$.
\end{lemma}
\begin{proof}
  We have $\codim(\pi^{-1}(D)) = \codim(D)$, while $\cancel{D}$ adds two more conditions ($u_1=u_{s+1}$). However, we also increase the dimension
  of the domain by one, in both the non-equivariant and equivariant cases. This gives the first claim. If we moreover require that the strikethrough is non-ordinary, then clearly the codimension further increases by at least one.
\end{proof}

We now aim to find versal deformations of strikethrough singularities. In general, this is difficult (compare \cite{Stevens2}), so we focus on a few specific cases that will be sufficient for our purposes. We first work in the non-equivariant setting.

\begin{definition}\label{def:uhs-noneq}
An \emph{unbalanced homogenous singularity} is a multigerm which is RL equivalent to $\phi = (\phi_1, \phi_2, \phi_3)$, where:
  \begin{itemize}
    \item $\phi_1(s_1)=(s_1,0)$ parameterizes $\{u = 0\}$;
    \item $\phi_2(s_2)=(0,s_2)$ parameterizes $\{v = 0\}$;
    \item $\phi_3(s_3)=(s_3^q,s_3^p)$ parameterizes $\{u^p - v^q = 0\}$. We require $\gcd(p, q) = 1$ and $p,q\ge 1$, but $p\neq q$. 
  \end{itemize}
  We also allow our multigerm to be RL equivalent to a multigerm consisting of some subset of these branches; e.g., $\phi = (\phi_1, \phi_3)$. 
\end{definition}

We now have the following result. 

\begin{thm}\label{thm:unb}
Let $\phi$ be an unbalanced homogeneneous singularity and $(\Phi, \Lambda)$ be a versal deformation of $\phi$. Let $\cancel{\phi}$ be an ordinary strikethrough of $\phi$. Construct a deformation $(\cancel{\Phi}, \cancel{\Lambda})$ of $\cancel{\phi}$ by setting $\cancel{\Lambda} = \Lambda \times I$, and defining $\cancel{\Phi}$ to be $\Phi$ on $\Lambda$ and on $I$ moving the strikethrough branch in any direction transverse to its tangent. Then $(\cancel{\Phi}, \cancel{\Lambda})$ is a versal deformation of $\cancel{\phi}$.
\end{thm}
\begin{proof}
Let $\phi_1(s_1)=(s_1,0)$, $\phi_2(s_2)=(0,s_2)$, and $\phi_3(s_3)=(s_3^q,s_3^p)$. The case where the original singularity has
fewer components is an obvious simplification of this case. Let $\phi_4$ be the new branch of the strikethrough. For simplicity, we assume that $\phi_4(s_4)=(s_4,s_4)$, but the proof is valid for $\phi_4$ whose derivative at zero is neither horizontal nor vertical. 

To verify that $\cancel{\Phi}$ is a versal deformation, we must first understand the tangent space to the RL orbit $[\cancel{\phi}]$ of $\cancel{\phi}$. Recall that the tangent space to $[\cancel{\phi}]$ is a subspace of the space of all multigerms $(\nu_1, \nu_2, \nu_3, \nu_4)$, where each $\nu_i$ may be thought of as a pair of power series. We claim that for any pair of power series $\nu_4$ with $\nu_4^x(0) = \nu_4^y(0)$, we can find an RL equivalence which fixes the branches $\phi_1$, $\phi_2$, and $\phi_3$, but moves $\phi_4$ in the direction of $\nu_4$.

  To begin with, notice that for any $\lambda\in\R\setminus\{0\}$, the change $s_1\mapsto \lambda^q s_1$, $s_1\mapsto\lambda^p s_2$,
  $s_3\mapsto\lambda s_3$, followed by $u\mapsto \lambda^{-q}u$, $v\mapsto\lambda^{-p}v$, preserves both the horizontal and vertical lines,
  as well as the weighted homogenous singularity parameterized by $\phi_3$. The infinitesimal generator of this action is obtained by taking
  $\lambda=1+\varepsilon$ and considering the part of the action linear in $\varepsilon$: in this way we obtain $h_1,h_2,h_3,k$, which can be used to calculate the corresponding tangent vector to the orbit of the action. In this case, $h_1(s_1)=q s_1$, $h_2(s_2)=ps_2$, $h_3(s_3)=1$, $k(u,v)=(-qu,-pv)$.

  A more general action is given by allowing the parameter $\lambda$ to vary. For any smooth function $\theta(u, v)$ and for $\varepsilon$ sufficiently small, consider the local diffeomorphism sending
\[
s_1 \mapsto (1+\varepsilon\theta(s_1,0))^q s_1, \quad s_2 \mapsto (1+\varepsilon\theta(0,s_2))^ps_2, \quad s_3\mapsto(1+\varepsilon\theta(s_3^q,s_3^p))s_3
\]
in the domain and
\[
  u \mapsto u(1+\varepsilon\theta(u,v))^{-q}, \quad v \mapsto v(1+\varepsilon\theta(u,v))^{-p}
\]
in the codomain. It is straightforward to check that this fixes the germs $\phi_1,\phi_2,\phi_3$. The functions $h_1,h_2,h_3,k$ corresponding to this 1-parameter family of diffeomorphisms are obtained by taking
the derivative with respect to $\varepsilon$ and setting $\varepsilon$ to $0$:
\[
h_1(s_1)=q s_1\theta(s_1,0),\quad h_2(s_2)=ps_2\theta(0,s_2),\quad h_3(s_3)=s_3\theta(s_3^q,s_3^p)
\]
and
\[k(u,v)=(-qu\theta(u,v),-pv\theta(u,v)).\]
As explained in \eqref{eq:RL_versal}, the corresponding tangent vector to the orbit is given by
\[
\nu_j=h_j(s_j)\frac{d\phi_j}{ds_j}+k(\phi_j)
\]
as a function of $s_j$. It is straightforward to verify that we indeed have
\begin{align*}
  h_1(s_1)\frac{d\phi_1}{ds_1}+k(\phi_1)&=(qs_1\theta(s_1,0),0)+(-qs_1\theta(s_1,0),0)=(0,0)\\
  h_2(s_2)\frac{d\phi_2}{ds_2}+k(\phi_2)&=(0,ps_2\theta(0,s_2))+(0,-ps_2\theta(0,s_2))=(0,0)\\
  h_3(s_3)\frac{d\phi_3}{ds_3}+k(\phi_3)&=s_3\theta(s_3^q,s_3^q)(qs_3^{q-1},ps_3^{p-1})+(-qs_3^q\theta(s_3^q,s_3^p),-ps_3^p\theta(s_3^q,s_3^p))=(0,0).
\end{align*}
Now, for any function $\kappa(s_4)$, let
\[
h_4(s_4)=\kappa(s_4)+ps_4 \theta(s_4, s_4).
\]
The tangent vector corresponding to this $h_4$ is given by
\begin{align*}
\nu_4&=h_4(\phi)(1,1)+(-qs_4 \theta(s_4, s_4),-ps_4\theta(s_4, s_4))\\
&=(\kappa(s_4)+(p-q)s_4 \theta(s_4, s_4),\kappa(s_4)).
\end{align*}
This shows that any $\nu_4$ with $\nu_4^x(0) = \nu_4^y(0)$ can be achieved with an appropriate choice of $\kappa$ and $\theta$. Indeed, simply choose $\kappa(s_4) = \nu_4^y(s_4)$ and choose $\theta(u, v)$ such that 
\[
\theta(s_4, s_4) = (\nu_4^x(s_4) - \nu_4^y(s_4))/(s_4(p - q)),
\]
where we have used the fact that $\nu_4^x(0) = \nu_4^y(0)$ to see that we may divide by $s_4$. By our discussion above, for any choice of $\theta(u, v)$ we have a corresponding $h_1$, $h_2$, $h_3$ so that $\nu_1 = \nu_2 = \nu_3 = 0$. This gives the claim asserted at the beginning of the proof.

The theorem follows quickly. To say that a deformation is infinitesimally versal is to say that we can achieve any vector in the tangent space of all multigerms by summing together the tangent space to the deformation with the tangent space to the RL orbit. Let
\[
\cancel\Phi(\lambda_1,\lambda_2,\lambda_3,\lambda_4)=(\Phi(\lambda_1,\lambda_2,\lambda_3), \phi_4 + \lambda_4 v),
\]
where $v$ is any vector transverse to $(1, 1)$; e.g. $v = (1, 0)$. Then we can achieve any tangent direction for $\phi_4$ by varying $\lambda_4$ and summing this together with the tangent space to $[\cancel{\phi}]$. We have already supposed we can achieve any tangent direction for $\phi_1$, $\phi_2$, and $\phi_3$ using $\Phi$. This completes the proof.
\end{proof}

In the equivariant setting, there are three cases which we will need for future reference. First, we could have an equivariant germ with a pair of symmetric copies of Definition~\ref{def:uhs-noneq}: that is, an equivariant germ with six branches, with three of them meeting on one side of $\cL$ at a singularity RL equivalent to Definition~\ref{def:uhs-noneq}, and three of them meeting on the other side. (As in Definition~\ref{def:uhs-noneq}, we also allow ourselves to select some subset of these three branches.) The strikethrough of this is formed as in Definition~\ref{def:uhs-noneq} and consists of a pair of symmetric strikethroughs at points on either side of the axis. It is clear that this does not differ substantially from Theorem~\ref{thm:unb}. 

Second, we will be interested in the case where we have a branch of $\phi$ which is based at $t \in \Sfree$ and is RL equivalent to $\phi_3$ from Definition~\ref{def:uhs-noneq}, but now specifically $\phi(t) \in \cL$. The symmetric branch is then also RL equivalent to $\phi_3$, and the two branches meet at a point on $\cL$. An ordinary strikethrough consists of a pair of symmetric curves passing through the point of intersection, such that none of these four curves are tangent to each other. 

Finally, we will be interested in the case where $\phi$ has a single branch based at $t \in \Sfix$ and is RL equivalent to an equivariant copy of $\phi_3$ from Definition~\ref{def:uhs-noneq}. An ordinary strikethrough at this point consists of a pair of symmetric curves passing through $\phi(t) \in \cL$.

\begin{definition}\label{def:euh}
An \emph{equivariant unbalanced homogenous singularity} is an equivariant multigerm which is RL equivalent to one of the following:
  \begin{itemize}
    \item A germ $\phi \in \cO(t_1, \sigma t_1, t_2, \sigma t_2, t_3, \sigma t_3)^{\Z_2}$, where $\phi(t_3) \notin \cL$ and the germ over $t_1, t_2, t_3$ is RL equivalent to the triple of germs from Definition~\ref{def:uhs-noneq} (or some subset). We call this an \textit{off-axis homogenous singularity}.
    \item A germ $\phi \in \cO(t, \sigma t)^{\Z_2}$, where $\phi(t) \in \cL$ and $\phi$ is (equivariantly) RL equivalent to the symmetric pair $(s_1^q,s_1^p)$ and $(-s_2^q, s_2^p)$, where $s_2 = \sigma s_1$. We call this an \textit{on-axis homogenous singularity}. 
    \item A germ $\phi \in \cO(t = \sigma t)^{\Z_2}$ which is (equivariantly) RL equivalent to the equivariant germ $\phi(s) = (s^q, s^p)$ with $q$ odd and $p$ even. We call this a \textit{fixed-point homogenous singularity}.
\end{itemize}
In all cases, we require $\gcd(p, q) = 1$ and $p,q\ge 1$, but $p\neq q$.
\end{definition}
\begin{thm}\label{thm:euh_deform}
Let $\phi$ be an on-axis or fixed-point equivariant unbalanced homogeneneous singularity and $(\Phi, \Lambda)$  be a versal deformation of $\phi$. Let $\cancel{\phi}$ be an ordinary strikethrough of $\phi$. Construct a deformation $(\cancel{\Phi}, \cancel{\Lambda})$ of $\cancel{\phi}$ by setting $\cancel{\Lambda} = \Lambda \times I$, and defining $\cancel{\Phi}$ to be $\Phi$ on $\Lambda$ and on $I$ moving the strikethrough branch along $\cL$. Then $(\cancel{\Phi}, \cancel{\Lambda})$ is a versal deformation of $\cancel{\phi}$.
\end{thm}
\begin{proof}
We give a proof of the on-axis case, leaving the fixed-point case to the reader. Let $\phi_1(s_1)=(s_1^{q},s_1^p)$ and $\phi_2(s_2)=(-s_2^q, s_2^p)$. An equivariant diffeomorphism preserving $\phi_1$
  also preserves $\phi_2$. For any power series $\theta(x,y)$ in $x^2$ and $y$, consider the infinitesimal change of coordinates in the domain given by 
  \[
  h_1(s_1) = s_1\theta(s_1^q, s_1^p), \quad h_2(s_2) = s_2\theta(s_2^q, s_2^p)
  \]
  and
  \[
  k(u, v) = (-qu\theta(u, v), -p v\theta(u, v))
  \]
  in the codomain. These are as in the proof of Theorem~\ref{thm:unb}, but note that $\theta$ is invariant under the action of $\tau$ in the codomain. As before, these fix $\phi_1$ and $\phi_2$.
  
 Let $\phi_3$ and $\phi_4$ be the symmetric pair of strikethrough branches. For simplicity, assume that $\phi_3(s_3)=(s_3,s_3)$ and $\phi_4(s_4)=(-s_4,s_4)$; the general case is similar. As before, for any $\kappa(s_3)$, consider
\[
h_3(s_3)=\kappa(s_3)+ps_3 \theta(s_3, s_3).
\]
The tangent vector corresponding to this $h_3$ is given by
\begin{align*}
\nu_3&=h_3(\phi)(1,1)+(-qs_3 \theta(s_3, s_3),-ps_3\theta(s_3, s_3))\\
&=(\kappa(s_3)+(p-q)s_3 \theta(s_3, s_3),\kappa(s_3)).
\end{align*}
Just as in Theorem~\ref{thm:unb}, we attempt to achieve different tangent vectors $\nu_3$ by choosing $\kappa(s_3) = \nu_3^y(s_3)$ and $\theta(u, v)$ such that 
\[
\theta(s_3, s_3) = (\nu_3^x(s_3) - \nu_3^y(s_3))/(s_3(p - q)).
\]
The only additional requirement is that $\theta$ is a power series in $x^2$ and $y$, but it is clear that this does not present any difficulty. The proof thus proceeds as in the proof of Theorem~\ref{thm:unb}.
\end{proof}

\subsection{Cusps}\label{sub:cusp}
We now discuss cusps. Our goal will be to enumerate the cusps of small codimension and write down a normal form and versal deformation for each. We begin with the non-equivariant case.

\begin{defn}\label{def:noneqcusp}
A \textit{cusp} is a germ $\phi \in \cO(t_0)$ with $\phi'(t_0) = 0$. We say that a cusp has \textit{multiplicity $m > 1$} if $\smash{\phi^{(i)}(t_0) = 0}$ for all $i < m$ and $\smash{\phi^{(m)}(t_0) \neq 0}$. 
\end{defn}

We say that a cusp $\phi$ is in \textit{standard form} if we have chosen local coordinates centered at $t_0$ and $\phi(t_0)$ such that
\[
\phi(t) = (x(t), y(t)),
\]
where $x(0) = y(0) = x'(0) = y'(0) = 0$ and the leading degree of $x(t)$ does not divide the leading degree of $y(t)$ (or vice-versa). It is straightforward to check that every cusp can be put into standard form: start with any coordinate system and select whichever of $x(t)$ and $y(t)$ has lower leading degree. Without loss of generality, suppose $x(t)$ has leading degree $p$ and $y(t)$ has leading degree $k \geq p$. If $p$ divides $k$, then perform the coordinate change in the codomain which subtracts a multiple of $x^{k/p}$ from $y$ to cancel its leading term. Repeating this procedure, we either end up with $y(t) = 0$ or we put $\phi$ in standard form. The former case corresponds to degenerate cusps such as $(t^3, 0)$ or $(t^2, t^2)$; these situations are of infinite codimension in jet space (requiring the vanishing of infinitely many coefficients) and are generally disallowed.

Although the standard form for $\phi$ is not unique, the leading degrees of a cusp in standard form are well-defined. To see this, suppose $\phi$ has leading degrees $p$ and $k$ in some fixed coordinate system. Consider an RL equivalence that preserves the fact that $\phi$ is in standard form. An RL equivalence in the domain (which replaces $t$ by some power series in $t$ with linear leading term) does not change leading degrees. An RL equivalence in the codomain replaces $x$ and $y$ with linear functions of $x$ and $y$, plus higher-order power series in $x$ and $y$. In general, this can only change the leading degrees to $p$ and an element of $\{np \colon 1 \leq n \leq \lfloor q/p \rfloor\} \cup \{k\}$. Thus, under the supposition that the new expression for $\phi$ is still in standard form, the pair of degrees $p$ and $k$ is well-defined.

\begin{definition}\label{def:cusppq}
Let $p, q \geq 2$ with $\gcd(p, q) = 1$. We say that $\phi$ is \textit{$(p, q)$-cusp} if, when put into standard form, it has leading degrees $p$ and $q$. 
\end{definition}

Note that a $(p, q)$-cusp has multiplicity $\min(p, q)$. Not all cusps are $(p,q)$-cusps; for example, $(t^4,t^6+t^7)$. We single out the case when $p$ and $q$ are coprime since the theory for non-coprime cusps is more involved, and we only use the case when
the multiplicity is at most $3$,
because cases with cusps with multiplicity greater than $3$ have higher codimension.

\begin{example}\label{ex:23cusp}
We express the condition of being a $(2, 3)$-cusp in terms of jet space. Consider a germ $\phi$ of the form
\[
x(t) = a_0 + a_1 t + a_2 t^2 + a_3 t^3 + \cdots \quad \text{and} \quad y(t) = b_0 + b_1 t + b_2 t^2 + b_3 t^3 + \cdots.
\]
Without loss of generality we may assume $a_0 = b_0 = 0$. Clearly, $\phi$ is a cusp if and only if $a_1 = b_1 = 0$, and it has multiplicity $2$ if and only if at least one of $a_2$ and $b_2$ is nonzero. Without loss of generality suppose $a_2 \neq 0$. Consider the linear change of coordinates in the codomain replacing $y$ with $y - b_2/a_2 x$; the leading term of $y(t)$ is then $(b_3 - (b_2/a_2) a_3) t^3$. Hence $\phi$ is a $(2, 3)$-cusp if and only if this term is nonzero; that is, $a_2 b_3 - b_2 a_3 \neq 0$. More invariantly, this is equivalent to the condition that $\phi''(t_0)$ and $\phi'''(t_0)$ are linearly independent.
\end{example}

In the equivariant setting, we further distinguish cusps depending on their behavior with respect to the axis of symmetry. We have the following three cases.

\begin{defn}\label{def:eqcusp} Let $\phi \in \cO(t_0, \sigma t_0)^{\Z_2}$. We say $\phi$ is an \textit{off-axis cusp} if it satisfies Definition~\ref{def:noneqcusp} at $t_0$ (and thus also $\sigma t_0$) and $\phi(t_0) \notin \cL$.
\end{defn}

\begin{defn}
Let $\phi \in \cO(t_0, \sigma t_0)^{\Z_2}$. We say $\phi$ is an \textit{on-axis cusp} if it satisfies Definition~\ref{def:noneqcusp} at $t_0$ (and thus also $\sigma t_0$) and $\phi(t_0) \in \cL$. We consider three further subcases:
\begin{itemize}
\item An \textit{on-axis tangency cusp} occurs if the generalized tangent line at $t_0$ is parallel to $\cL$.
\item An \textit{on-axis perpendicular cusp} occurs if the generalized tangent line at $t_0$ is perpendicular to $\cL$.
\item An \textit{on-axis oblique cusp} occurs if the neither of the above happen.
\end{itemize}
\end{defn}

\begin{defn}
Let $\phi \in \cO(t_0 = \sigma t_0)^{\Z_2}$. We say $\phi$ is a \textit{fixed-point cusp} if it satisfies Definition~\ref{def:noneqcusp} at $t_0$. In this case, the generalized tangent line is automatically either parallel or perpendicular to $\cL$.
\end{defn}

An equivariant cusp has multplicity $m$ if it has multiplicity $m$ in the same sense as Definition~\ref{def:noneqcusp}. An equivariant $(p, q)$-cusp is an equivariant cusp which is a $(p, q)$-cusp in the sense of Definition~\ref{def:cusppq}. Note that the RL equivalence putting the cusp in standard form is \textit{not} required to be equivariant.

Our goal is to identify equivariant cusps of small codimension. Our starting point is the following calculation.

\begin{lemma}\label{lem:ext_nu}
 A non-equivariant $(p,q)$-cusp has parametric codimension
\[
\pcodim = p+q-\left\lfloor\frac{q}{p}\right\rfloor-3.
\] 
\end{lemma}
\begin{proof}
  The calculation is done in \cite{BZ} in the complex setting, and holds in the real case as well. The relevant conditions in jet space are given by vanishing of so-called \emph{Puiseux coefficients}. To get an idea for this codimension, we first approximate the condition of being a $(p, q)$-cusp by requiring the first $p - 1$ coefficients of $x(t)$ to vanish and the first $q - 1$ coefficients of $y(t)$ to vanish, leading to $p + q - 2$ conditions. However, as discussed after Definition~\ref{def:noneqcusp}, requiring the coefficients in $y(t)$ of degrees lying in $\{np \colon 1 \leq n \leq \lfloor q/p \rfloor\}$ is not actually an essential condition; so in fact there are $p + q - \lfloor q/p \rfloor - 2$ conditions. Since the domain is dimension $1$, this gives the calculation of $\pcodim$.
\end{proof}
\begin{remark}
  The proof of Lemma~\ref{lem:ext_nu} lets us find the conditions in the jet space specifying a $(p,q)$-cusp. This is helpful for finding
  a versal deformation: it essentially consists of modifying $x$ up to order $p$ and changing the 
  exponents of $y$ less than $q$ that are not divisible by $p$.
\end{remark}

The calculation of $\pcodim$ an off-axis cusp is the same. For an on-axis cusp, the calculation is similar, but we increase the codimension by one since we now have the condition that the cusp lies on $\cL$; we increase it by a further one if the cusp is required to be a tangency cusp or perpendicular cusp. The calculation for a fixed-point cusp is different (and will be done only in small cases), because in the fixed-point case only germs $(x(t), y(t))$ with odd powers of $t$ and even powers of $t$ (for $x(t)$ and $y(t)$, respectively) are allowed.

\begin{lemma}\label{lem:cusplist}
The only equivariant cusps with $\pcodim \leq 2$ are the following:
\begin{itemize}
\item an on-axis $(2, 3)$-cusp has $\pcodim = 1$, while an on-axis $(2, 5)$-cusp has $\pcodim = 2$;
\item an off-axis oblique $(2, 3)$-cusp has $\pcodim = 2$; and,
\item a fixed-point $(2, 3)$-cusp has $\pcodim = 1$, while a fixed-point $(2, 5)$-cusp or fixed-point $(3, 4)$-cusp has $\pcodim = 2$.
\end{itemize}
\end{lemma}
\begin{proof}
The enumeration of on-axis and off-axis cusps with $\pcodim \leq 2$ follows immediately from Lemma~\ref{lem:ext_nu} and the subsequent discussion. To analyze the case of a fixed-point cusp, we first claim that a fixed-point cusp with multiplicity $4$ or more has codimension at least $3$. To see this, consider the expansion
\[
x(t) = \sum_{i \text{ odd}} \alpha_i t^i \quad \text{and} \quad y(t) = \sum_{i > 0 \text{ even}} \beta_i t^i.
\]
Then $(x(t),y(t))$ has multiplicity at least $4$ if and only if $\alpha_1 = \alpha_3 = \beta_2 = 0$. Hence we may limit our search to fixed-point cusps of multiplicities $2$ and $3$. A fixed-point $(2, 3)$-cusp is defined by the condition $\alpha_1 = 0$, a fixed-point $(2, 5)$ cusp is defined by the conditions $\alpha_1 = \alpha_3 = 0$, and a fixed-point $(3, 4)$-cusp is defined by the conditions $\alpha_1 = \beta_2 = 0$. By direct inspection, we see that
  these are the only cases with $\pcodim \leq 2$.
\end{proof}

Having listed the cusps of small codimension, we now seek to establish a normal form and versal deformation for each. We begin with the on-axis $(2,3)$- and $(2,5)$-cusps, which are handled by the following slightly more general lemma.

\begin{lemma}\label{lem:normal_cusp}
An off-axis $(2, 2k+1)$-cusp is equivariantly RL equivalent to $(t^2, t^{2k+1})$. A versal deformation of such a cusp is provided by 
  \[
  \Phi(\lambda_1,\dots,\lambda_k)(t) = (t^2, t^{2k+1} + \lambda_1 t^{2k-1} + \cdots + \lambda _{k-1} t^3 + \lambda_k t).
  \]
 \end{lemma}

\begin{proof}
Since away from the axis, non-equivariant and equivariant RL equivalence coincide, the proof is the same as in the non-equivariant case. We thus repeat the proof of \cite[Lemma 4.1]{David}. Following Definition~\ref{def:cusppq}, we may choose coordinates so that our $(2, 2k+1)$-cusp is of the form
\[
x(t) = \sum_{i = 2}^\infty a_i t^i \quad \text{and} \quad y(t) = \sum_{i = 2k+1}^\infty b_i t^i
\]
and re-scale so that $a_2 = b_{2k+1} = 1$. Then $x(t) = t^2 w(t)$ where $w(0) = 1$. The function $\sqrt{w}$ is smooth near zero and hence we may define a new coordinate in the domain by $t \sqrt{w}$. After this change-of-coordinates we have that $x(t) = t^2$. Note that $y(t)$ still has leading term $t^{2k+1}$. Separate $y(t)$ into its even and odd parts by writing
\[
y(t) = p(t^2) + tq(t^2)
\]
for power series $p$ and $q$. Perform the coordinate change in the codomain which replaces $y$ with $y - p(x)$. After this coordinate change, $y(t) = tq(t^2)$. Write $q(t^2) = t^{2k} r(t^2)$ for some invertible power series $r$. Perform the further coordinate change in the codomain which replaces $y$ with $y/r(x)$. After this coordinate change, we finally have $y(t) = t^{2k+1}$, as desired. 

A versal deformation of such a cusp is described in \cite{David}, see also \cite{AVG,Wall_Pgen}. The construction of the versal deformation may be carried out using Theorem~\ref{thm:inf_vers}.
\end{proof}

We now turn to the case of the on-axis oblique $(2, 3)$-cusp. The proof of Lemma~\ref{lem:normal_cusp} does not immediately carry over to this situation, since the changes of variables in the codomain are not necessarily equivariant. Since the analysis of oblique cusps is more complicated, we will content ourselves with the specific case of the $(2, 3)$-cusp.

\begin{lemma}\label{lem:normal_bound}
An on-axis oblique $(2, 3)$-cusp is equivariantly RL equivalent to $(t^2, t^3 + t^2)$. A versal deformation of such a cusp is provided by 
\[
 \Phi(\lambda_1,\lambda_2)(t) = (t^2+\lambda_1, t^3 + t^2 + \lambda_2 t + \lambda_1 t).
 \]
\end{lemma}
\begin{proof}
Let $\phi$ be an oblique $(2, 3)$-cusp, which we write as $\phi(t) = (x(t), y(t))$ in the usual rectilinear coordinates $x$ and $y$ for the codomain. (We do \textit{not} allow ourselves to immediately choose coordinates as in Definition~\ref{def:cusppq}, as this coordinate change will not in general be equivariant.) Then $x'(0) = y'(0) = 0$ and at least one of $x''(0)$ and $y''(0)$ are nonzero, as $\phi$ has multiplicity two. Moreover, since the tangent line to $\phi$ is neither parallel nor perpendicular to $\cL$, we in fact have that both $x''(0)$ and $y''(0)$ must be nonzero. We can thus linearly re-scale the $x$ and $y$ variables in such a way that the tangent line is parallel to $(1,1)$; this is obviously equivariant. Write
\[
x(t) = t^2 + \cdots \quad \text{and} \quad y(t) = t^2 + \cdots.
\]
Reparameterize $t$ as in the proof of Lemma~\ref{lem:normal_cusp} in such a way so that $x(t)=t^2$. (Note that since $t_0 \in \Sfree$, coordinate changes in the domain have no restriction.) This changes $y(t)$ to $y(t) = t^2 + at^3 + \cdots$. We can further linearly re-scale $t,x,y$ in such a way so that $a=1$; this is again obviously equivariant. We thus have
\[
x(t) = t^2 \quad \text{and} \quad y(t) = t^2 + t^3 + \cdots.
\]
We aim to remove the higher order terms in $y(t)$ by an equivariant change of variables. 
 
We claim that for any power series $y(t)=t^2+t^3+\cdots$, there exists a power series $A$ in two variables such that $y(t) =A(t^4,t^2+t^3)$. Assuming the claim, we conclude the first part of the proof: the map $(x,y)\mapsto (x,A(x^2,y))$ is an equivariant diffeomorphism taking a curve parameterized by $(t^2,t^2+t^3)$ to the curve parameterized by $(x(t),y(t))$. The inverse of this map is the required change of variables.

To prove the claim, write $y(t)=t^2+t^3+c_4t^4+\cdots$. Set 
\[
p_0(u,v)=u, \quad p_1(u,v)=\frac12(v^2-u), \quad p_2(u,v)=uv, \quad \text{and} \quad p_3(u,v)=\frac12(v^3-vu)
\]
and define $q_i(t) = p_i(t^4, t^2 + t^3)$, so that
\[
q_0=t^4, \quad q_1=t^5+\frac12t^6, \quad q_2=t^6+t^7, \quad \text{and} \quad q_3=t^7+\frac32t^8+\frac12t^9.
\]
For any four real numbers $a_1,a_2,a_3,a_4$, define
  \begin{align*}
    P(a_1,a_2,a_3,a_4)&=a_1p_1+(a_2-\frac12a_1)p_2+(a_3-a_2+\frac12a_1)p_3+(a_4-\frac32a_3+\frac32a_2-\frac34a_1)p_0^2\\
    Q(a_1,a_2,a_3,a_4)&=a_1q_1+(a_2-\frac12a_1)q_2+(a_3-a_2+\frac12a_1)q_3+(a_4-\frac32a_3+\frac32a_2-\frac34a_1)q_0^2.
\end{align*}
Since $q_i(t) = p_i(t^4, t^2 + t^3)$, it is clear that 
\begin{equation}\label{eq:QP}
Q(a_1, a_2, a_3, a_4)(t) = P(a_1, a_2, a_3, a_4)(t^4, t^2 + t^3). 
\end{equation}
It is routine to check
\begin{equation}\label{eq:Q_values}
  Q(a_1,a_2,a_3,a_4)(t)=a_1t^5+a_2t^6+a_3t^7+a_4t^8+\left(\frac12a_3-\frac12a_2+\frac14a_1\right)t^9.
\end{equation}
Using \eqref{eq:Q_values}, we can easily show that 
\[
y(t)=t^2+t^3+c_4q_0+\sum_{j=0}^\infty Q(c_{1j},c_{2j},c_{3j},c_{4j})(t)t^{4j},
\]
for some sequences $c_{1j},c_{2j},c_{3j},c_{4j}$. Define
\[A(u,v)=v+c_4u^2+\sum_{j=0}^\infty P(c_{1j},c_{2j},c_{3j},c_{4j})(u^2,v)u^{2j}.\]
Then \eqref{eq:QP} shows that $A(t^2, t^2 + t^3) = y(t)$. 

This proves the claim in the category of formal power series. The Malgrange preparation theorem allows us to show that if $y(t)$ is smooth, then $y(t)=A(t^4,t^2+t^3)$ for smooth $A$. While it is not needed, we might at this point indicate that the proof works in the analytic category as well, i.e. if $y$ is a convergent power series, then so is $A$. Indeed, convergence of the power series defining $A$ can be proved by observing that the coefficients of the power series defining $A$ are bounded by the partial sums $|c_4|+\dots+|c_n|$ of the coefficients for $y$. A standard argument as in \cite[Section 5.4]{Softy}, allows us to state that if the power series for $y(t)$ has convergence radius $R>0$, then the power series $\sum_{j} c_{1j} t^{4j+1}+c_{2j}t^{4j+2}+c_{3j}t^{4j+3}+c_{4j}t^{4j+4}$ has convergence radius at least $\min(1,R)$. Thus the power series for $A$ is convergent in a neighborhood of $(0,0)$. We leave the details to the reader. 

Finally, a versal deformation for the germ $(t^2, t^3 + t^2)$ can be constructed by perturbing the defining equation according to Theorem~\ref{thm:inf_vers}.
\end{proof}

We are now left with the fixed-point cusps listed in Lemma~\ref{lem:cusplist}. We start with the $(2, 3)$- and $(2, 5)$-cusps. It will be no more difficult to consider fixed-point $(2, 2k+1)$-cusps in general.

\begin{lemma}\label{lem:fixedcusp}
A fixed-point $(2, 2k+1)$-cusp is equivariantly RL equivalent to $(t^{2k +1}, t^2)$. A versal deformation of such a cusp is given by
 \[
 \Phi(\lambda_1,\dots,\lambda_k)(t) = (t^{2k+1} + \lambda_1 t^{2k-1} + \cdots + \lambda _{k-1} t^3 + \lambda_k t,t^2).
 \]
\end{lemma}
\begin{proof}
It suffices to check that the analogue of the change of variables used in Lemma~\ref{lem:normal_cusp} is equivariant. We may assume our germ is of the form
\[
x(t) = \sum_{i \text{ odd}} \alpha_i t^i \quad \text{and} \quad y(t) = \sum_{i > 0 \text{ even}} \beta_i t^i.
\]
Then $\phi''(0)=(0,\beta_2)$ and since $\phi$ is multiplicity two, we have $\beta_2\neq 0$. Rescaling $y$ linearly, we may assume that $\beta_2=1$. Following Lemma~\ref{lem:normal_cusp}, write $y(t)=t^2w(t)$, but now note that $w$ is a power series in $t^2$. Thus $\sqrt{w}$
  is a power series in $t^2$ and the local reparameterization $t\mapsto t\sqrt{w}$ is equivariant. This turns $y$ into $t^2$ and preserves the fact that $x$ has
  only odd powers of $t$ in its expansion. Write $x(t)=t^{2k+1}r(t)$ for $r(t)$ an invertible power series in $t^2$. The coordinate change in the codomain replacing $x$ with $x/r(y)$ is clearly equivariant, since $r(y)$ is fixed by the action of $\tau$, and turns $x(t)$ into $t^{2k+1}$.
  
Again, a versal deformation may be carried out as in Examples~\ref{ex:t25} and \ref{ex:defining}.
\end{proof}
We are left with the case of the fixed-point $(3,4)$-cusp.

\begin{lemma}\label{lem:dot}
A fixed-point $(3, 4)$-cusp is equivariantly RL equivalent to $(t^3, t^4)$. A versal deformation of such a cusp is given by
 \[
 \Phi(\lambda_1, \lambda_2)(t) = (t^3 + \lambda_1 t, t^4 + \lambda_2 t^2).
 \]
\end{lemma}
\begin{proof}
After normalizing, we may assume
\[
x(t) = t^3 + \cdots \quad \text{and} \quad y(t) = t^4 + \cdots
\]
where $x(t)$ is a power series in odd powers of $t$ and $y(t)$ is a power series in even powers of $t$. Write $x(t)=t^3w(t)$ for $w$ a power series in $t^2$. Then $\sqrt[3]w$ is a power series in $t^2$; replacing $t$ by $t\sqrt[3]w$ is an equivariant coordinate change in the domain and gives $x(t)=t^3$. Note that $y(t)$ still has leading term $t^4$. Write
 \[
 y(t)=t^4p(t^4)+t^6q(t^4)
 \]
for power series $p$ and $q$ with $p(0) = 1$. Consider the map sending $(x,y)$ to $(x,yp(y) + x^2 q(y))$. The germ of this is an equivariant change of coordinates in the codomain which sends the curve $(t^3,t^4)$ to the curve $(t^3,t^4p(t^4)+t^6q(t^4))$. The inverse of this map thus gives the desired RL equivalence.

A versal deformation of this germ was given in Example~\ref{ex:defining}.
\end{proof}

\subsection{Tangencies}\label{sub:tang}
We now turn to tangencies, beginning with the non-equivariant case. 

\begin{defn}\label{def:noneqtan}
A germ $\phi \in \cO(t_1, t_2)$ is a \textit{tangency} if $\phi(t_1) = \phi(t_2)$ and the derivatives $\phi'(t_1)$ and $\phi'(t_2)$ are nonzero and parallel.
\end{defn}

We say that a tangency $\phi$ is in \textit{standard form} if we have chosen coordinates centered at $t_1$, $t_2$, and $\phi(t_1) = \phi(t_2)$ so that the branches through $t_1$ and $t_2$ are given by 
\[
\phi_1(t) = (t, 0) \quad \text{and} \quad \phi_2(s) = (s, h(s))
\]
for some power series $h$ with $h(0) = h'(0) = 0$. It is straightforward to check that every cusp can be put into standard form: first, rotate the codomain so that $\phi_1'(0)$ and $\phi_2'(0)$ are horizontal. By reparameterizing the domains of the two branches, we may then assume $\phi_1(t) = (t, p(t))$ and $\phi_2(s) = (s, q(s))$ for some power series $p$ and $q$ with $p(0) = q(0) = p'(0) = q'(0) = 0$. Now perform a reparameterization of the codomain sending $y$ to $y - p(x)$. This makes $\phi_1(t) = (t, 0)$ and transforms $\phi_2(s) = (s, h(s))$. We usually disregard the case that $h = 0$, which corresponds to the situation where the two branches exactly coincide; this is of infinite codimension in jet space.

Although the standard form for $\phi$ is not unique, the leading degree of $h$ is well-defined. To see this, suppose we have a coordinate system in which $\phi_1(t) = (t, 0)$ and $\phi_2(s) = (s, h(s))$. Consider the possible RL equivalences that preserve the form of $\phi_1$. To preserve the fact that the second component of $\phi_1$ is zero, such an RL equivalence must reparameterize the codomain by replacing $y$ with a power series in $y$ only, rather than a power series in $x$ and $y$. This preserves the leading degree of $h$. An RL equivalence of the domain around $t_2$ also preserves the leading degree of $h$; hence this leading degree is well-defined.

\begin{defn}\label{def:tangency-fold}
  Let $k \geq 1$. We say $\phi \in \cO(t_1, t_2)$ is a \textit{$k$-fold tangency}, or an \textit{order $k$ tangency}, if, when put into standard form, it is given by $\phi_1(t) = (t, 0)$ and $\phi_2(s) = (s, h(s))$ for $h$ of leading degree $k + 1$.
\end{defn}

As in the case of cusps, in the equivariant setting we classify tangencies differently depending on their behavior with respect to the axis of symmetry. The different cases of equivariant tangencies are displayed in Figure~\ref{fig:tangencies}. 

\begin{figure}
  \begin{tikzpicture}
    \draw (-2,0) node[rectangle, minimum width=2.5cm] {\includegraphics[height=2.5cm]{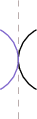}}; 
    \draw (2,0) node[rectangle,minimum width=2.5cm] {\includegraphics[height=2.5cm]{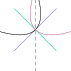}}; 
    \draw (-2,-3) node[rectangle,  minimum width=2.5cm]  {\includegraphics[height=2.5cm]{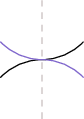}}; 
    \draw (2,-3) node[rectangle, minimum width=2.5cm] {\includegraphics[height=2.5cm]{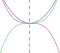}}; 
  \end{tikzpicture}
  \caption{Top: a central tangency and an oblique tangency. Bottom: a perpendicular tangency and a fixed-point tangency.}\label{fig:tangencies}
\end{figure}

The first case is when the point of tangency occurs off of the axis of symmetry.

\begin{definition}\label{def:interior}
Let $\phi \in \cO(t_1, \sigma t_1, t_2, \sigma t_2)^{\Z_2}$. We say $\phi$ is an \textit{off-axis tangency} if it satisfies Definition~\ref{def:noneqtan} at either $\{t_1, t_2\}$ or $\{t_1, \sigma t_2\}$ and $\phi(t_1) \notin \cL$.
\end{definition}

The second case occurs when the point of tangency lies on the axis of symmetry, but all domain points are still in $\Sfree$.

\begin{definition}
We say that a tangency is \textit{on-axis} if the point of tangency lies on $\cL$. There are three subcases, depending on whether the tangency is between a branch and its symmetric copy, or between two unrelated branches. In the former case, the derivatives of the branches are necessarily either vertical or horizontal.
\begin{itemize}
\item Let $\phi \in \cO(t_1, \sigma t_1)^{\Z_2}$. We say $\phi$ is an \textit{on-axis line tangency} if it satisfies Definition~\ref{def:noneqtan} at $\{t_1, \sigma t_1\}$ and $\phi'(t_1)$ is vertical.
\item Let $\phi \in \cO(t_1, \sigma t_1)^{\Z_2}$. We say $\phi$ is an \textit{on-axis perpendicular tangency} if it satisfies Definition~\ref{def:noneqtan} at $\{t_1, \sigma t_1\}$ and $\phi'(t_1)$ is horizontal.
\item Let $\phi \in \cO(t_1, \sigma t_1, t_2, \sigma t_2)^{\Z_2}$. We say $\phi$ is an \textit{on-axis oblique tangency} if it satisfies Definition~\ref{def:noneqtan} at either $\{t_1, t_2\}$ or $\{t_1, \sigma t_2\}$ and $\phi(t_1) \in \cL$. We also require that $\phi'(t_0)$ is neither horizontal nor vertical.
\end{itemize}
\end{definition}

Finally, we have the most subtle case, in which a branch of $\phi$ through a point $t_1 \in \Sfree$ is tangent to a branch of $\phi$ through a point $t_2 \in \Sfix$. Note that this means the branch through $\sigma t_1$ is also tangent to the branch through $t_2$, and the branches through $t_1$ and $\sigma t_1$ are necessarily tangent to teach other.

\begin{definition}\label{def:Tbound}
  Let $\phi \in \cO(t_1, \sigma t_1, t_2 = \sigma t_2)^{\Z_2}$. We say $\phi$ has a \textit{fixed-point tangency} if the branches of $\phi$ through $t_1$ and $t_2$ are tangent and the branches through $t_1$ and $\sigma t_1$ are tangent. In this case the tangent vectors are necessarily horizontal.
\end{definition}

Defining the order of a tangency in the equivariant case is slightly subtle. In the case of an off-axis or on-axis oblique tangency, an order-$k$ tangency is simply an equivariant tangency of order $k$ in the non-equivariant sense. Note that the RL equivalence in Definition~\ref{def:tangency-fold} is \textit{not} required to be equivariant. In the case of an on-axis line or perpendicular tangency, the same definition holds, but with a few simplifications and caveats. 

For a line tangency, we claim the order of tangency can be calculated as the order of tangency between $\phi$ and $\cL$. To see this, note that $\phi'(t_1)$ is already vertical, so after performing a pair of symmetric linearizations, we may assume $\phi$ is given by a symmetric pair of branches
\[
\phi_1(t) = (h(t), t) \quad \text{and} \quad \phi_2(s) = (-h(s), s).
\]
To calculate the order of tangency of these two branches with each other, we put $\phi$ into standard form. Consider the change of variables sending $x$ to $x - h(y)$. This gives
\[
\phi_1(t) = (0, t) \quad \text{and} \quad \phi_2(s) = (-2h(s), s)
\] 
and shows the order of tangency is the leading degree of $-2h(s)$ minus one. This is obviously equal to the leading degree of $h(s)$ minus one, which is the order of tangency between $\phi_1$ and $\cL$, as desired.

We claim that the order of a perpendicular tangency is always even. To see this, note that $\phi'(t_1)$ is already horizontal, so after performing a pair of symmetric linearizations, we may assume $\phi$ is given by a symmetric pair of branches
\[
\phi_1(t) = (t, h(t)) \quad \text{and} \quad \phi_2(s) = (-s, h(s)).
\]
Reparameterize the branch $\phi_2$ by precomposing with the map sending $s$ to $-s$; this makes $\phi_2(s) = (s, h(-s))$. To calculate the order of tangency, we perform the change of variables in the codomain sending $y$ to $y - h(x)$. This turns our germ into
\[
\phi_1(t) = (t, 0) \quad \text{and} \quad \phi_2(s) = (s, h(-s) - h(s)).
\]
The order of tangency is thus the leading degree of $h(-s) - h(s)$ minus one. As the leading degree of $h(-s) - h(s)$ is odd, the order of the perpendicular tangency is even.

Finally, let us consider the case of a fixed-point tangency. This will actually be described by \textit{two} tangency orders: the tangency between the fixed-point branch and either of the two free branches, and the tangency between the two free branches. We say that a fixed-point tangency is order $(k, 2\ell)$ if the branches through $t_1$ and $t_2 = \sigma t_2$ are tangent to order $k$, and the branches through $t_1$ and $\sigma t_1$ are tangent to order $2 \ell$. Note that the branches through $t_1$ and $\sigma t_1$ constitute a perpendicular tangency of order $2\ell$, as in the previous paragraph.

We now calculate the codimensions of different equivariant tangencies. Our starting point is the following calculation.

\begin{lem}\label{lem:noneqtancalc}
A non-equivariant $k$-fold tangency has $\pcodim = k$.
\end{lem}
\begin{proof}
This is shown in \cite{BZ}, where it is also implicitly explained how a $k$-fold tangency can be defined by a stratum
in jet space. Roughly speaking, the reader may think that a $k$-fold tangency is defined by putting $\phi$ into the form
\[
\phi_1(t) = (t, 0) \quad \text{and} \quad \phi_2(s) = (s, a_1 s + a_2 s^2 + a_3 s^3 + \cdots)
\]
and then requiring the vanishing of the $k$ coefficients $a_1 = \cdots = a_{k} = 0$. In addition, there are two extra conditions from requiring $\phi(t_0) = \phi(t_1)$, while the domain is two-dimensional. This gives the desired calculation.
\end{proof}

This extends to the equivariant setting as follows. 

\begin{lem}\label{lem:complicatedtan}
We have the following:
\begin{itemize}
\item The space of off-axis $k$-fold tangencies has $\pcodim = k$.
\item The space of on-axis $k$-fold line tangencies has $\pcodim = k$.
\item The space of on-axis $2k$-fold perpendicular tangencies has $\pcodim = k$.
\item The space of on-axis $k$-fold oblique tangencies has $\pcodim = k + 1$.
\item The space of fixed-point $(k, 2\ell)$-fold tangencies has 
\[
\pcodim = k +\max(\ell - \lfloor (k + 1)/2 \rfloor, 0) + 1.
\]
\end{itemize}
\end{lem}
\begin{proof}
The off-axis case is essentially the same as in Lemma~\ref{lem:noneqtancalc}. For a line tangency, we require $k$ coefficients to vanish and impose one additional condition (that $\phi(t_1) \in \cL$), while the domain is one-dimensional. For an oblique tangency, we require $k$ coefficients to vanish and impose three conditions (that $\phi(t_1) = \phi(t_2) \in \cL$), while the domain is two-dimensional. The case of a perpendicular tangency is only slightly more complicated: as in the discussion above, we put our tangency in the standard form 
\[
\phi_1(t) = (t, 0) \quad \text{and} \quad \phi_2(s) = (s, h(-s) - h(s)).
\]
The even-degree coefficients of $h(-s) - h(s)$ automatically vanish. Thus, the condition that $h(-s) - h(s)$ have leading degree $2k + 1$ is equivalent to the vanishing of the $k$ odd-degree coefficients $a_1 = a_3 = \cdots = a_{2k - 1} = 0$ of $h$. There is one additional condition (that $\phi(t_1) \in \cL$), while the domain is one-dimensional.

For the last case, parameterize the branches through $t_1$, $\sigma t_1$, and $t_2$ by 
\[
\phi_1(t) = (t, h(t)), \quad \phi_2(s) = (s, h(-s)), \quad \phi_3(r) = (r, g(r)).
\]
Here, $g$ is an even power series (with vanishing constant term) since $\phi_3$ is equivariant. The condition that $\phi_1$ and $\phi_3$ are tangent to order $k$ means that the coefficients $c_1$ through $c_k$ of $g(r) - h(r)$ vanish; that is, the odd-indexed coefficients of $h(r)$ between $1$ and $k$ must vanish, and the even-indexed coefficients between $1$ and $k$ must coincide with $g(r)$. This constitutes $k$ conditions. As in the previous paragraph, the condition that $\phi_1$ and $\phi_2$ are tangent to order $2 \ell$ means that the $\ell$ odd-indexed coefficients of $h(s)$ between $1$ and $2\ell$ must vanish. Since we have already imposed that the odd-indexed coefficients between $1$ and $k$ vanish, this imposes $\max(\ell - \lfloor (k + 1)/2 \rfloor, 0)$ new conditions. Finally, there are two additional conditions (that $\phi(t_1) = \phi(t_2)$), while the domain is one-dimensional. This gives the desired calculation.
\end{proof}

For future reference, we single out a specific case of the most complicated family:

\begin{corollary}
The only fixed-point tangency with $\pcodim \leq 2$ is the fixed-point $(1, 2)$-fold tangency, which has $\pcodim = 2$.
\end{corollary}
\begin{proof}
Keeping in mind that $k$ and $\ell$ are strictly positive, this follows immediately from Lemma~\ref{lem:complicatedtan}.
\end{proof}

We now derive normal forms and versal deformations for equivariant tangencies. We suppress writing the symmetric branches when it is clear from context.

\begin{lemma}\label{lem:normal_tang}
An off-axis $k$-fold tangency is equivariantly RL equivalent to the (equivariant) multigerm $\phi = (\phi_1, \phi_2)$ with $\phi_1(t) = (t, 0)$ and $\phi_2(t) = (s, s^{k+1})$. A versal deformation of this is given by keeping the first branch fixed and deforming the second branch to
\[
\Phi(\lambda_1, \cdots, \lambda_k)(s) = (s, s^{k+1} + \lambda_{1} s^{k-1} + \lambda_{2} s^{k-2} + \cdots + \lambda_{k}).
\]
\end{lemma}
\begin{proof}
Since off-axis tangencies are essentially the same as non-equivariant tangencies, it suffices to establish the claim in the non-equivariant setting. Let $\phi_1(t) = (t, 0)$ and $\phi_2(s) = (s, h(s))$ for $h$ of leading degree $k + 1$. Write $h(s) = s^{k+1} w(s)$ for $w$ an invertible power series. The change of coordinates sending $(x, y)$ to $(x, y/w(x))$ gives $\phi$ the desired form.

The obvious versal deformation varies all $k + 1$ coefficients of $\phi_2$ with degrees less than $k + 1$ (modifying the defining equation
of the singularity as in Theorem~\ref{thm:inf_vers}). However, the reparameterization of the domain replacing $s$ with $s + c$ allows us to set one of these coefficients to any given constant. This results in the versal deformation in the statement of the lemma. 
\end{proof}

Similar arguments hold in the case of a line tangency.

\begin{lemma}\label{lem:normal_int_tang}
A $k$-fold line tangency is equvariantly RL equivalent to the (equivariant) germ $\phi(t) = (t^{k+1}, t)$. A versal deformation of this is given by 
\[
\Phi(\lambda_1, \cdots, \lambda_k)(t) = (t^{k+1} + \lambda_{1} t^{k-1} + \lambda_{2} t^{k-2} + \cdots + \lambda_{k}, t).
\]

\end{lemma}

\begin{proof}
The proof is as in Lemma~\ref{lem:normal_tang}, with the $x$- and $y$-coordinates interchanged.
\end{proof}

For perpendicular tangencies we have the following. Here we write out the symmetric branch to emphasize the tangency between the two branches.

\begin{lemma}\label{lem:normal_perp}
A $2k$-fold perpendicular tangency is equivariantly RL equivalent to the (equivariant) multigerm $\phi = (\phi_1, \phi_2)$ with $\phi_1(t) = (t,t^{2k+1})$ and $\phi_2(s) = (-s, s^{2k+1})$, where $\phi_2$ is parameterized by $s = \sigma t$. A versal deformation of this is given by 
\[
\Phi(\lambda_1, \cdots, \lambda_k)(t) = (t, t^{2k+1} + \lambda_{1} t^{2k-1} + \lambda_{2} t^{2k-3} + \cdots + \lambda_{k} t)
\]
while simultaneously deforming the symmetric branch via
\[
\Phi(\lambda_1, \cdots, \lambda_k)(s) = (-s, s^{2k+1} + \lambda_{1} s^{2k-1} + \lambda_{2} s^{2k-3} + \cdots + \lambda_{k} s).
\]
\end{lemma}

\begin{proof}
As discussed above, we may assume that after a symmetric pair of linearizations we have $\phi_1(t) = (t, h(t))$ and $\phi_2(s) = (-s, h(s))$ for some power series $h$. Write $h(t) = p(t^2) + t q(t^2)$ and perform the equivariant coordinate change in the codomain replacing $y$ with $y - p(x^2)$. This changes our germ into $\phi_1(t) = (t, tq(t^2))$ and $\phi_2(s) = (-s, s q(s^2))$. Write $q(t^2) = t^{2k+1} r(t^2)$ for some invertible power series $r$. Then the equivariant coordinate change replacing $y$ with $y/r(x^2)$ gives the desired form for $\phi$.

A versal deformation is obtained by finding a family transverse to the defining set as in Theorem~\ref{thm:inf_vers}.
\end{proof}

Next is the normal form for an oblique tangency. 

\begin{lemma}\label{lem:oblique_form}
  A $k$-fold oblique tangency is equivariantly RL equivalent to the singularity with one local branch parameterized by $x(t)=t$, $y(t)=t$,
  and the other branch parameterized by $x(s)=s$, $y(s)=s+s^{k+1}$. The symmetric branches are parameterized symmetrically 
  (i.e. $t\mapsto (-t,t)$, and $s\mapsto (-s,s\pm s^{k+1})$.

  A deformation with parameters $\lambda_0,\dots,\lambda_{k}$ keeping the first branch fixed and modifying the second to
  \[x(s)=s, y(s)=\lambda_0+s\lambda_1+\dots+\lambda_ks^k\pm s^{k+1}\]
  is versal.
\end{lemma}
\begin{proof}
  Consider the first branch. We shift the coordinate system vertically so that the branch passes through $(0,0)$. As the branch is
  neither tangent nor perpendicular to the central line, we know that $x'(0)\neq 0$. In particular, $t\mapsto x(t)$ has local inverse.
  By changing reparametrizing $t$, we can
  assume that $x(t)=t$. Next, reparameterizing $y$, we assume that $y(t)=t$. We remark that reparametrizing $x$ is more difficult, because we are only allowed to reparameterize in a $\tau$-equivariant way.

  Having specified the first branch, we pass to the other one. It is convenient to choose $u=\frac12(x+y)$, $v=\frac12(x-y)$.
  Note that then $\tau u=v$, $\tau v=u$.
  The first branch is given by $u(t)=t$, $v(t)=0$. The second branch is given by $u(s),v(s)$ with $u(0)=v(0)$. We can
  reparameterize $s$ in such a way that $u(s)=s$.
  The fact that the two branches are tangent up to order $k$ implies that $v(s)=v_ss^{k+1}+\dots$ with $v_s\neq 0$. We aim to cancel the
  higher order terms. We will look for change of variables $v\mapsto P(v)$, $u\mapsto P(u)$, $t\mapsto P^{-1}(t)$, $s\mapsto P^{-1}(s)$.
  This change commutes with $\tau$, and fixes the form of the first branch. We have the following result, whose proof we defer to the end
  of the proof of Lemma~\ref{lem:oblique_form}.
  \begin{lemma}\label{lem:calculus1}
    Suppose $k>1$. Let $H(t)=H_{k}t^k+\dots$ be a germ of a real-valued smooth function in one variable. Then, there exist smooth functions $Q,R$ defined
    near $0\in\R$ such that $R\circ Q(t)=t$
    and $R(Q(t)^k)=\epsilon H(t)$. Here $\epsilon$ is the sign of $H_k$.
  \end{lemma}
  Given the lemma, we choose $P=Q$ with $H(s)=y(s)$. Then, the new $y$ is equal to $\pm s^k$.

  This finishes the proof of the first part of Lemma~\ref{lem:oblique_form}, in particular we show that the oblique tangency is simple. Then,
  the deformation described in the statement of lemma is transverse to the defining equation, so by Theorem~\ref{thm:inf_vers} it is versal.
  The latter observation concludes the proof of Lemma~\ref{lem:oblique_form}, up to the proof of Lemma~\ref{lem:calculus1}, which we now give.
\end{proof}
\begin{proof}[Proof of Lemma~\ref{lem:calculus1}]
  Suppose $h_k>0$, so that we can assume that $\epsilon=1$. It is enough to show that there exists a local diffeomorphism $Q$ taking $0$ to $0$
  such that $Q(t)^k=Q(H(t))$. As a local diffeomorphism, $Q$ is divisible by $t$ so $Q(t)=tq(t)$ for some smooth function $q$ non-vanishing at $0$. Also, by subsequent application of Hadamard lemma, write $H(t)=t^kh(t)$ with $h(0)\neq 0$. Equation $Q(t)^k=Q(H(t))$ translates into
  \[ q(t)^k=h(t)q(H(t)).\]
  We want to solve it for $q$. As both sides are non-zero for $t=0$, with $c(t)=\log h(t)$ and $r(t)=\log q(t)$, we have
  \begin{equation}\label{eq:on_rt}
    r(t)=\frac1k c(t)+\frac1k r(H(t)).
  \end{equation}
  We apply an iterative method. That is, \eqref{eq:on_rt} yields.
  \begin{equation}\label{eq:on_rt2}
    r(t)=\frac1k c(t)+\frac{1}{k^2}c(H(t))+\frac{1}{k^2} r(H(H(t)).
  \end{equation}
  With $H^{(n)}(t)$ denoting the $n$-fold composition of $H$, define now
  \begin{equation}\label{eq:on_rt3}
    r_\infty(t)=\sum_{j=0}^\infty \frac{1}{k^{j+1}} c(H^{(n)}(t)).
  \end{equation}
  Taking the convergence of the series in \eqref{eq:on_rt3} for granted, we see that $r_\infty$ satisfies \eqref{eq:on_rt}. So
  with $q(t)=\exp r_\infty(t)$, we obtain the statement. It remains to show that for some $\varepsilon>0$,
  the series \eqref{eq:on_rt3} converges on $[-\varepsilon,\varepsilon]$ uniformly with all derivatives.

  The Taylor expansion of $H$ starts with the term $H_{k+1}t^{k+1}$, and $k>1$. Hence, there exists $\varepsilon>0$ such that
  for $t\in[-\varepsilon,\varepsilon]$, $|H(t)|<C|t|^{k+1}$ with $C=2|H_{k+1}|$. We may assume that $C|t|^{k+1}\le \varepsilon$
  if $|t|\le\varepsilon$, hence $H$ takes $[-\varepsilon,\varepsilon]$ to itself. Consequently:
  \[\sup_{t\in[-\varepsilon,\varepsilon]} |c(H(t))|\le c_0:=\sup_{t\in[-\varepsilon,\varepsilon]} |c(t)|.\]
  This implies that $|c(H^{(n)})(t)|\le c_0$, so by Weierstrass criterion, the series \eqref{eq:on_rt3} converges uniformly.

  Convergence of all derivatives follows from a bound on the derivatives of $H^{(n)}$. Iterating $H(t)<C|t|^{k+1}$ implies
  that $H^{(2)}(t)<C^{k+2}|t|^{(k+1)^2}$ and consequently, $|H^{(n)}|$ vanishes at $t$ up to an order expanding superexponentially.
  From this, a straightforward but tedious estimate involving Faa di Bruno's formula shows that for any $\ell\ge 1$, the $\ell$-th derivative
  of $c(H^{(n)})$ is small enough to ensure uniform convergence of the $\ell$-th derivative of \eqref{eq:on_rt3}. 
\end{proof}

Finally, we have the most complicated case, which is for a fixed-point tangency. Note that the first case $(k=2,\ell=1)$
leads to a rather complicated singularity, consisting of 3 branches with tangency orders $(2,3,3)$. Even that case is not a simple
singularity. Finding a normal form of the fixed-point
tangency is beyond the scope of this article. Instead, we content ourselves with a description of a deformation

\begin{lemma}\label{lem:fixed_point_tan}
  Suppose $k\le 2\ell+1$. A $(k, 2\ell)$-fold fixed-point tangency is equivariantly RL equivalent to the (equivariant) multigerm 
  $x(s_1)=s_1, y(s_1)=s_1^{2\ell+1}$, $x(s_2)=s_2$, $y(s_2)=h(s_2)$ (here $\sigma s_2=s_2$, but $\sigma s_1\neq s_1$),
  where $h(s_2)=h_{k+1}s_2^{k+1}+\dots$.
  The deformation 
  \begin{align*}
    y(s_1)&\mapsto \lambda_1s_1+\lambda_2s_1^3+\dots+\lambda_{\ell-1}s_1^{2\ell-1}+s_1^{2\ell+1}\\
    y(s_2)&\mapsto \lambda_{\ell}s_2^2+\dots+\lambda_{\ell+k/2-2}s_2^{k-1}+h(s_2)
  \end{align*}
  depending on parameters $(\lambda_1,\dots,\lambda_{\ell+k/2-2})$ is transverse to the defining equation.
\end{lemma}

\begin{proof}
  Apply Lemma~\ref{lem:normal_perp} to find a coordinate system such that $x(s_1)=s_1$, $y(s_1)=s_1^{2\ell+1}$. Reparameterize the coordinate $s_2$
  in such a way that $x(s_2)=s_2$. Then, as the branches are tangent up to order $k$, and $k\le 2\ell$, we know that
  \[y(s_2)=h(s_2),\]
  with $h(s_2)$ starting with degree $k+1$.

  This coordinate change can be done in a family, so that we can find a deformation transverse to the defining set,
  even if the singularity is not simple.  To see this,
  suppose $\Psi\colon \Xi\to \cO(s_1,\sigma s_1,s_2)$ is a family of maps parametrizing a fixed-point singularity of order $(k,2\ell)$.
  For $\xi\in\Xi$, we denote the local parameterization by $(x_\xi(s_1),y_\xi(s_1))$ and $(x_\xi(s_2),y_\xi(s_2))$.
  Fix $\xi_0\in\Xi$ and suppose we have found a coordinate system in which $x_{\xi_0},y_{\xi_0}$ have the form as above. 
  
  The first branch singularity (the perpendicular singularity) is simple. Therefore, we can find coordinate systems (both in the domain and in the codomain), possibly depending on $\xi$, such that $x_{\xi}(s_1)=s_1$ and $y_{\xi}(s_1)=s_1^{2\ell+1}$. On reparametrizing the coordinate $s_2$
  (by a local diffeomorphism of $s_2$-variable, possibly dependent on $\xi$), we arrive at $x_{\xi}(s_2)=s_2$. Then, the tangency restriction implies that  $y_{\xi}(s_2)=h_{\xi}(s_2)$, where $h_{\xi}(s_2)$ has Taylor expansion starting at $k+1$.

  That is, in the local coordinates, with $x(s_1)=s_1$, $x(s_2)=s_2$, with $y(s_1)=b_1s_1+b_2s_2+\dots$ and $y(s_2)=\beta_1s_2^2+\beta_2s_2^4+\dots$, the defining equations are $b_1=b_3=\dots=b_{\ell-1}=0$, $\beta_1=\dots=\beta_{k/2-2}=0$. 
\end{proof}


\section{One-parameter deformations of equivariant link diagrams}\label{sec:one_para}

We now apply the formalism of the last few sections to give a rigorous proof of the equivariant Reidemeister theorem. This will arise from understanding the behavior of one-parameter families of maps in $\cF$. In general, a one-parameter family will not stay within the subset of regular maps, but will instead cross the codimension $1$ subset
\[
\wt{\cF}^1 = \cF \setminus \cF^0.
\]
We thus begin by identifying a particular subset $\cF^1 \subset \wt{\cF}^1$ whose complement in $\wt{\cF}^1$ has codimension $2$. This will consist of the mildest possible singularities within $\smash{\wt{\cF}^1}$. Any one-parameter family of maps can then be perturbed so that all intersections with $\smash{\wt{\cF}^1}$ occur at points of $\cF^1$. 

Importantly, each of these intersections corresponds to a codimension $1$ singularity which is sufficiently constrained so that we can write down a normal form and versal deformation. This gives a local model for how our one-parameter family crosses $\cF^1$. As discussed in Example~\ref{ex:codim1}, the before-and-after pictures for these versal deformations describe the qualitative change in $\phi$ as $\cF^1$ is crossed. As in Section~\ref{sec:noneqreidemeister}, these give rise to the different equivariant Reidemeister moves.

To give an exhaustive list of codimension $1$ singularities, recall that 
\[
\wt{\cF}^1 = \wt{\cF}^1_1 \cup \cdots \cup \wt{\cF}^1_8,
\]
where each of the strata $\wt{\cF}^1_i$ describes the violation of a particular regularity condition. For the sake of organization, we note that these may be separated into three types:
\begin{itemize}
  \item Off-axis singularities; that is, singularities involving domain points in $\Sfree$ and image points off the axis. These are described by $\wt{\cF}^1_1,\wt{\cF}^1_2$, and $\wt{\cF}^1_3$.
  \item On-axis singularities; that is, singularities involving domain points in $\Sfree$ and image points on the axis. These are described by $\wt{\cF}^1_4,\wt{\cF}^1_5$, and $\wt{\cF}^1_6$. 
  \item Fixed-point singularities; that is, singularities involving domain points in $\Sfix$. These are described by $\wt{\cF}^1_7$ and $\wt{\cF}^1_{8}$.
\end{itemize}
The reader should refer back to Table~\ref{tab:cpl} for a description of these singularities. We then construct $\cF^1$ as the union
\[
\cF^1 = \cF^1_1 \cup \cdots \cup \cF^1_8,
\]
where each $\wt{\cF}^1_i \setminus \cF^1_i$ has codimension $2$. As a rule, we define $\cF^1_i$ to be the complement in $\wt{\cF}^1_i$ of the following kinds of singularities:
\begin{itemize}
  \item coincidences, as in Definition~\ref{def:coincidence}. This means that if the definition of $\wt{\cF}_i$ specifies a singularity, then no other singularity occurs at the same time.
  \item strikethroughs, as in Definition~\ref{def:strikethrough}. This means that if the definition of $\wt{\cF}_i$ involves points $t_1,\dots,t_k$,
    then no other branch of $\phi$ passes through the images of these points.
  \item higher-order cusps. This means that if the definition of $\wt{\cF}_i$ involves some differentials vanishing, then they vanish only to the required order and not a higher order.
  \item higher-order tangencies. This means that if the definition of $\wt{\cF}_i$ involves two branches that intersect, then they intersect
    with the prescribed tangency and not to a tangency of higher order.
\end{itemize}
In Lemmas~\ref{lem:coincidence} and~\ref{lem:strikethrough}, we already proved that a coincidence or a strikethrough of a codimension~1 singularity
has codimension at least $2$. 

\subsection{Off-axis singularities}\label{sub:off_the_axis}
Off-axis singularities involve points on $\Sfree$ that are mapped to $\R^2\setminus\cL$. 
Each of these thus actually constitutes a symmetric orbit of singularities and leads to a symmetric orbit of Reidemeister moves. These result in symmetrized versions of the usual three Reidemeister moves.

\subsubsection{Off-axis cusp}
Define $\cF^1_1$ to be the space of off-axis $(2,3)$-cusps with no coincidences or strikethroughs. Then $\smash{\wt{\cF}^1_1\setminus\cF^1_1}$ consists
of strata describing a coincidence, a strikethrough, or a higher-order cusp. By Lemma~\ref{lem:cusplist}, an off-axis cusp is a codimension~1 if and only if it is a $(2,3)$-cusp, so this is a union of strata of codimension $2$. By Lemma~\ref{lem:normal_cusp}, a one-parameter versal deformation of the off-axis $(2, 3)$-cusp is given by $(t^2, t^3 + \lambda t)$.
For $\lambda>0$, this describes a smooth curve, while for $\lambda<0$ there is a self-intersection at $t=\pm\sqrt{-\lambda}$. The bifurcation
diagram is given in Figure~\ref{fig:outer_bifurcation}.

\begin{figure}
  \begin{tikzpicture}
    \draw (-4,3) node[rectangle, minimum width=3cm] {\includegraphics[height=2cm]{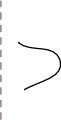}};
    \draw (0,3) node[rectangle, minimum width=3cm] {\includegraphics[height=2cm]{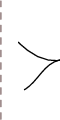}};
    \draw (4,3) node[rectangle, minimum width=3cm] {\includegraphics[height=2cm]{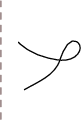}};

    \draw (-4,0) node[rectangle, minimum width=3cm] {\includegraphics[height=2cm]{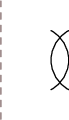}};
    \draw (0,0) node[rectangle, minimum width=3cm] {\includegraphics[height=2cm]{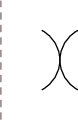}};
    \draw (4,0) node[rectangle, minimum width=3cm] {\includegraphics[height=2cm]{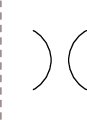}};

    \draw (-4,-3) node[rectangle, minimum width=3cm] {\includegraphics[height=2cm]{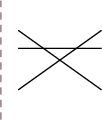}};
    \draw (0,-3) node[rectangle, minimum width=3cm] {\includegraphics[height=2cm]{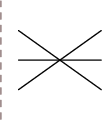}};
    \draw (4,-3) node[rectangle, minimum width=3cm] {\includegraphics[height=2cm]{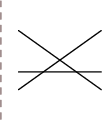}};
  \end{tikzpicture}
  \caption{Top: bifurcation of an outer cusp. The middle picture shows the cusp singularity. Middle: bifurcation diagrams
  of an outer tangency. Bottom: bifurcation diagrams of an outer triple point. }\label{fig:outer_bifurcation}
\end{figure}

\subsubsection{Off-axis tangency}
Define $\cF^1_2$ to be the space of $1$-fold tangencies with no coincidences or strikethroughs. Then $\smash{\wt{\cF}^1_2\setminus\cF^1_2}$ consists
of strata describing a coincidence, a strikethrough, or a higher-order tangency. By Lemma~\ref{lem:complicatedtan}, an outer tangency is codimension~1 if and only if it is a $1$-fold tangency, so this is a union of strata of
codimension $2$. By Lemma~\ref{lem:normal_tang}, our tangency has normal form given by the two branches $(t, 0)$ and $(s, s^2)$. A one-parameter versal deformation is given by perturbing the second branch to $(s, s^2 + \lambda)$. For $\lambda<0$ the branches intersect twice, while for $\lambda>0$, they
do not intersect. The bifurcation
diagram is given in Figure~\ref{fig:outer_bifurcation}.

\subsubsection{Off-axis triple points}
Define $\cF^1_3$ to be the space of ordinary triple points. Note that this comes from an ordinary strikethrough of the ordinary double point and hence has codimension $1$. Any extra tangency between branches leads to a higher codimension phenomenon, as well as any strikethroughs or coincidences. Thus the complement $\smash{\wt{\cF}^1_3\setminus\cF^1_3}$ consists of strata of codimension at least $2$. A versal deformation for the ordinary triple point was given in Example~\ref{ex:R3_tangent}. The bifurcation diagram is given in Figure~\ref{fig:outer_bifurcation}.

\subsection{On-axis singularities.}
On-axis singularities involve domain points on $\Sfree$ but image points lying on the axis of symmetry. In an on-axis singularity we have two symmetric branches of our map that intersect at a point on $\cL$. 

\subsubsection{On-axis line tangency}\label{ssub:central}
Define $\cF^1_4$ to be the space of $1$-fold line tangencies with no coincidences or strikethroughs. Then $\smash{\wt{\cF}^1_4\setminus\cF^1_4}$ consists
of strata describing a coincidence, a strikethrough, or a higher-order line tangency. By Lemma~\ref{lem:complicatedtan}, a line tangency is codimension~$1$ if and only if it is order $1$, so this is a union of strata of codimension $2$. By Lemma~\ref{lem:normal_int_tang}, a $1$-fold line tangency has normal form given by $(t^2, t)$. A one-parameter versal deformation is given by $(t^2 + \lambda, t)$. The bifurcation diagram is given in Figure~\ref{fig:central_bifurcat}.

\begin{figure}
  \begin{tikzpicture}
    \draw (-3,-3) node[rectangle, minimum width=3cm] {\includegraphics[height=2cm]{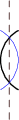}};
    \draw (0,-3) node[rectangle, minimum width=3cm] {\includegraphics[height=2cm]{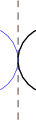}};
    \draw (3,-3) node[rectangle, minimum width=3cm] {\includegraphics[height=2cm]{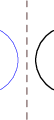}};
    \draw (-3,-6) node[rectangle, minimum width=3cm] {\includegraphics[height=2cm]{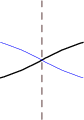}};
    \draw (0,-6) node[rectangle, minimum width=3cm]  {\includegraphics[height=2cm]{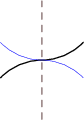}};
    \draw (3,-6) node[rectangle, minimum width=3cm]  {\includegraphics[height=2cm]{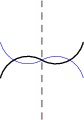}};
    \draw (-3,-9) node[rectangle, minimum width=3cm] {\includegraphics[height=2cm]{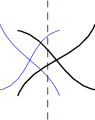}};
    \draw (0,-9) node[rectangle, minimum width=3cm] {\includegraphics[height=2cm]{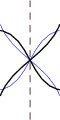}};
    \draw (3,-9) node[rectangle, minimum width=3cm] {\includegraphics[height=2cm]{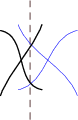}};
  \end{tikzpicture}
  \caption{Top: Bifurcation of central tangency. Middle: Bifurcation of perpendicularity tangency.
  Bottom: Bifurcaton of central double point.}\label{fig:central_bifurcat}.
\end{figure}

\subsubsection{On-axis perpendicular tangency}\label{ssub:perp_tang}
Define $\cF^1_5$ to be the space of $4$-fold perpendicular tangencies with no coincidences or strikethroughs. Then $\smash{\wt{\cF}^1_5\setminus\cF^1_5}$ consists
of strata describing a coincidence, a strikethrough, or a higher-order perpendicular tangency. By Lemma~\ref{lem:complicatedtan}, a perpendicular tangency is codimension~$1$ if and only if it is order $4$ as a tangency, so this is a union of strata of codimension $2$. By Lemma~\ref{lem:normal_perp}, a $4$-fold perpendicular tangency has normal form given by the two symmetric branches $(t, t^3)$ and $(-s, s^3)$. A versal deformation is given by simultaneously perturbing $(t, t^3 + \lambda t)$ and $(-s, s^3 + \lambda s)$. The bifurcation diagram is given in Figure~\ref{fig:central_bifurcat}.

\subsubsection{On-axis double point}\label{ssub:dcentral}
Define $\cF^1_6$ to be the space of on-axis double points with no coincidences or strikethroughs, satisfying the additional condition that there are no tangencies between the four branches and that none of them are parallel or perpendicular to $\cL$. More precisely, recall that $\smash{\wt{\cF}^1_6}$ consists of germs such that there exist $t_1,t_2\in \Sfree$ not in the same orbit of $\sigma$ with $\phi(t_1)=\phi(t_2)\in\cL$. To define $\cF^1_6$, we additionally require that any two vectors among $\{\phi'(t_1), \phi'(t_2), \tau \phi'(t_1), \tau \phi'(t_2)\}$ are linearly independent and that none of them are horizontal or vertical. Negating this condition adds an extra equation, so it is clear that $\smash{\wt{\cF}^1_6\setminus\cF^1_6}$ consists of strata of codimension~2. A versal deformation of $\cF^1_6$ is given by shifting the branches through $t_1$ and $t_2$ to the right according to $\phi(t)\mapsto
\phi(t)+(\lambda, 0)$, and shifting their symmetric branches to the left. The bifurcation diagram is somewhat complex,
because it involves four branches. See Figure~\ref{fig:central_bifurcat}.

\subsection{Fixed-point singularities} Like on-axis singularities, fixed-point singularities involve image points on the axis of symmetry, but (unlike on-axis singularities) involve domain points from $\Sfix$. 

\subsubsection{Fixed-point cusp}\label{ssub:bcusp}
Define $\cF^1_7$ to be the space of fixed-point $(2, 3)$-cusps with no coincidences or strikethroughs. Then $\smash{\wt{\cF}^1_7\setminus\cF^1_7}$ consists
of strata describing a coincidence, a strikethrough, or a higher-order fixed-point cusp. By Lemma~\ref{lem:cusplist}, a fixed-point cusp is codimension~$1$ if and only if it is a $(2, 3)$-cusp, so this is a union of strata of codimension $2$. By Lemma~\ref{lem:fixedcusp}, a versal deformation of the $(2, 3)$-fixed point cusp is given by $(t^3 + \lambda t, t^2)$. The bifurcation diagram is given in Figure~\ref{fig:bcusp}. 
\begin{figure}
  \begin{tikzpicture}
    \draw (-3,-6) node[rectangle, minimum width=3cm] {\includegraphics[height=2cm]{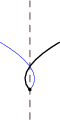}};
    \draw (0,-6) node[rectangle, minimum width=3cm] {\includegraphics[height=2cm]{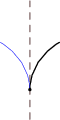}};
    \draw (3,-6) node[rectangle, minimum width=3cm] {\includegraphics[height=2cm]{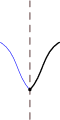}};
    \draw (-3,-9) node[rectangle, minimum width=3cm] {\includegraphics[height=2cm]{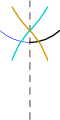}};
    \draw (0,-9) node[rectangle, minimum width=3cm] {\includegraphics[height=2cm]{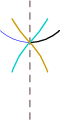}};
    \draw (3,-9) node[rectangle, minimum width=3cm] {\includegraphics[height=2cm]{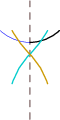}};
  \end{tikzpicture}
  \caption{Top: Bifurcation of a fixed point $(2,3)$ cusp. Bottom: Bifurcation of a fixed double point.}\label{fig:bcusp}
\end{figure}
\subsubsection{Fixed double point}\label{ssub:bdoublep}
Define $\cF^1_8$ to be the space of fixed double points with no coincidences or strikethroughs, satisfying the additional condition that none of the branches to the double point are tangent to each other. More precisely, recall that $\smash{\wt{\cF}^1_{8}}$ consists of germs such that for some $t_1 \in \Sfree$ and $t_2 \in \Sfix$, we have $\phi(t_1) = \phi(t_2)$. To define $\cF^1_8$, we require that none of $\{\phi'(t_1), \tau \phi'(t_1), \phi'(t_2)\}$ are parallel to each other (and that all of them are nonzero). Negating this condition  adds an extra equation, so it is clear that $\smash{\wt{\cF}^1_8\setminus\cF^1_8}$ consists of strata of codimension~2. A versal deformation of $\cF^1_8$ is given by shifting the branch through $t_1$ (and thus the symmetric branch through $\sigma t_1$) up according to $\phi(t) + (0, \lambda)$. The bifurcation diagram is given in Figure~\ref{fig:bcusp}. 

\subsection{Regular paths and equivariant Reidemeister moves in $\R^3$}\label{sub:reg_path}
Having described the strata $\cF^1_1,\dots,\cF^1_{8}$, we now define regularity of paths. Following Section~\ref{sec:noneqreidemeister}, let
\[
\cF^1 = \cF^1_1 \cup \cdots \cup \cF^1_{8}.
\]

\begin{defn}\label{def:regular_path_half}
  A path $\phi_s, s\in[0,1]$ in $\cF$ is called \emph{regular} if:
  \begin{enumerate}[label=(S-\arabic*)]
    \item $\phi_0,\phi_1$ are regular maps; \label{item:boundary_regular}
    \item $\phi_s$ belongs to $\cF^0\cup\cF^1$ for all $s$; \label{item:no_big_deal}
    \item there are finitely many values $s_1,\dots,s_m$ such that $\phi_{s_i}\in\cF^1$; for other
      values $\phi_s\in\cF^0$;\label{item:finitely_many_events}
    \item regarded as a path in $C^\infty(\cS,\R^2)$, $\phi_s$ crosses $\cF^1$ transversely.\label{item:fitrans}
  \end{enumerate}
\end{defn}

As in Section~\ref{sec:noneqreidemeister}, we set $\wt{\cF}^2_i = \wt{\cF}^1_i\setminus\cF^1_i$ for each $i$ and denote
\[
\wt{\cF}^2 = \wt{\cF}^2_1 \cup \cdots \cup \wt{\cF}^2_8 = \cF \setminus (\cF^0 \cup \cF^1).
 \]

\begin{thm}\label{thm:complem}
  Let $\phi_s$, $s\in[0,1]$, be any path in $\cF$ with $\phi_0,\phi_1\in\cF^0$. 
  Then there exists an arbitrarily close path $\phi'_s$,
  equal to $\phi_s$ at the ends, which is regular.
\end{thm}
\begin{proof}
  By construction, $\wt{\cF}^2$ is the union of codimension $2$ strata. Using the Equivariant Transversality Theorem~\ref{thm:ejet} applied as in Lemma~\ref{lem:codimension}, there is a path $\phi'_s$, equal to $\phi_s$ on the boundary (so that \ref{item:boundary_regular} holds),
  as close to $\phi_s$ as we please, missing the codimension $2$ strata (so \ref{item:no_big_deal} holds) and intersecting $\cF^1$ transversely. The
  latter amounts to saying that \ref{item:fitrans} holds; this implies \ref{item:finitely_many_events}.
\end{proof}

Now suppose $\wt{\phi}_s$ is a path of equivariant maps from $\cS$ to $\R^3$,
such that $\phi_s = \pi\circ\wt{\phi}_s$ is regular. Then $\phi_s$ crosses $\cF^1$ at finitely many parameters $s$. Lifting these crossings to $\smash{\wt{\phi}_s}$ gives the following behaviors:

\begin{defn}\label{def:equivar-reide}
  The following list of moves \ref{IR-1}--\ref{M-3} constitutes the set of \emph{equivariant Reidemeister moves}.  

\begin{enumerate}[label=(\wackyenum*)]
  \item Crossing $\cF^1_1$ (off-axis cusp) 
    leads to the (IR-1) move.\label{IR-1}
  \item Crossing $\cF^1_2$ (off-axis tangency) leads to the (IR-2) move.\label{IR-2}
  \item Crossing $\cF^1_3$ (off-axis triple point) leads to the (IR-3) move.\label{IR-3}
  \item Crossing $\cF^1_7$ (fixed-point cusp) leads to the (R-1) move.\label{R-1}
  \item Crossing $\cF^1_4$ (on-axis line tangency) leads to the (R-2) move.\label{R-2}
  \item Crossing $\cF^1_{8}$ (fixed double point) leads to the  (M-1) move.\label{M-1}
  \item Crossing $\cF^1_5$ (on-axis perpendicular tangency) leads to the (M-2) move.\label{M-2}
  \item Crossing $\cF^1_6$ (on-axis double point) leads to the (M-3) move.\label{M-3}
\end{enumerate}

See Figure~\ref{fig:AllReidemeister}. We remark that these figures should be considered as patterns, not as an exhaustive list of examples, in the sense that the signs of crossings may be changed (as long as the change is consistent). For future reference, we also consider the moves where the tunnels and bridges are not specified. A move among \ref{IR-1}--\ref{M-3} considered as a move on a diagram in $\R^2$, where the crossings are not specified, is called a \emph{planar equivariant Reidemeister move}.
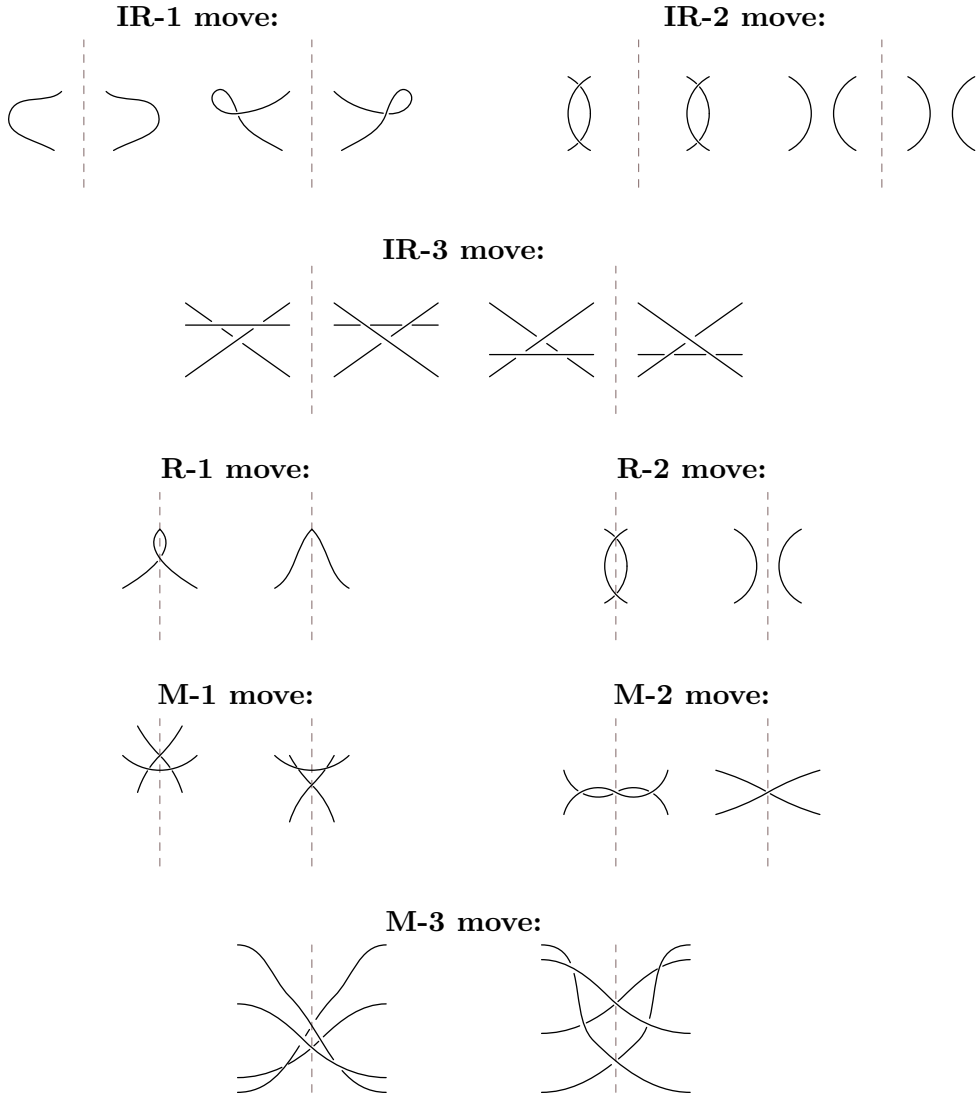
\begin{figure}
  \begin{tikzpicture}
    \draw (-3.5,4.3) node{\textbf{IR-1 move:}};
    \draw (3.7,4.3) node{\textbf{IR-2 move:}};
    \draw (0,1.2) node{\textbf{IR-3 move:}};
    \draw (-3,-1.7) node{\textbf{R-1 move:}};
    \draw (3,-1.7) node{\textbf{R-2 move:}};
    \draw (-3,-4.7) node{\textbf{M-1 move:}};
    \draw (3,-4.7) node{\textbf{M-2 move:}};
    \draw (0,-7.7) node{\textbf{M-3 move:}};
    \draw (-5,3) node[rectangle, minimum width=1.5cm] {\includegraphics[height=2cm]{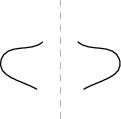}};
    \draw (-2,3) node[rectangle, minimum width=1.5cm] {\includegraphics[height=2cm]{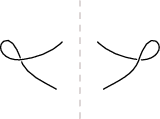}};
    \draw (2.3,3) node[rectangle, minimum width=1.5cm] {\includegraphics[height=2cm]{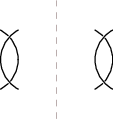}};
    \draw (5.5,3) node[rectangle, minimum width=1.5cm] {\includegraphics[height=2cm]{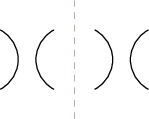}};
    \draw (-2,0) node[rectangle, minimum width=3cm] {\includegraphics[height=2cm]{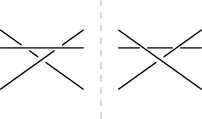}};
    \draw (2,0) node[rectangle, minimum width=3cm] {\includegraphics[height=2cm]{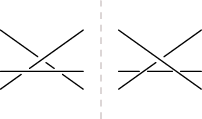}};

    \draw (-4,-3) node[rectangle, minimum width=3cm] {\includegraphics[height=2cm]{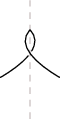}};
    \draw (-2,-3) node[rectangle, minimum width=3cm] {\includegraphics[height=2cm]{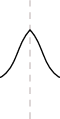}};
    \draw (2,-3) node[rectangle, minimum width=3cm] {\includegraphics[height=2cm]{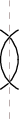}};
    \draw (4,-3) node[rectangle, minimum width=3cm] {\includegraphics[height=2cm]{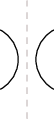}};

    \draw (-4,-6) node[rectangle, minimum width=2.5cm] {\includegraphics[height=2cm]{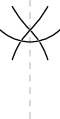}};
    \draw (-2,-6) node[rectangle, minimum width=2.5cm] {\includegraphics[height=2cm]{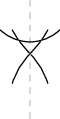}};
    \draw (2,-6) node[rectangle, minimum width=2.5cm] {\includegraphics[height=2cm]{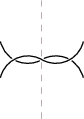}};
    \draw (4,-6) node[rectangle, minimum width=2.5cm] {\includegraphics[height=2cm]{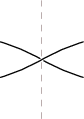}};
    \draw (-2,-9) node[rectangle, minimum width=3cm] {\includegraphics[height=2cm]{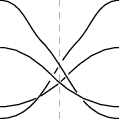}};
    \draw (2,-9) node[rectangle, minimum width=3cm] {\includegraphics[height=2cm]{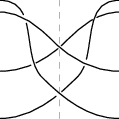}};
  \end{tikzpicture}
  \caption{The equivariant Reidemeister moves.}\label{fig:AllReidemeister}
\end{figure}
\end{defn}

This quickly gives the equivariant Reidemeister theorem.

\begin{thm}\label{thm:equiv_reid}
  Let $L_1$ and $L_2$ be involutive links whose projections to $\mathbb{R}^2$ give transvergent diagrams $D_1$ and $D_2$.  If $L_1$ and $L_2$ are isotopic as involutive links, then $D_1$ and $D_2$ can be connected by a sequence of equivariant Reidemeister moves as in Definition~\ref{def:equivar-reide}.
\end{thm}
\begin{proof}
Let $\wt{\phi}_s$ be an equivariant isotopy from $L_1$ to $L_2$ and set $\phi_s=\pi \circ \wt{\phi}_s$. Theorem~\ref{thm:complem} allows us to perturb $\phi_s$ to a path $\phi'_s$ that is regular. We can lift this perturbation to an isotopy of equivariant links by setting 
\[
\wt{\phi'}_s=\wt{\phi}_s+\phi'_s-\phi_s
\]
where $\phi'_s - \phi_s$ is considered as a map into $\R^3$. This is an equivariant isotopy, provided that the perturbation is sufficiently small. Each time $\phi'_s$ crosses a stratum $\cF^1_i$ yields a corresponding Reidemeister move, as described in Definition~\ref{def:equivar-reide}.
\end{proof}

\subsection{Undoing the \ref{R-1} move}\label{sub:undoing}
Two transvergent diagrams of the same strongly invertible knot in $\R^3$ are related by a sequence of equivariant Reidemeister moves.
These moves include (R-1)-moves performed at the two fixed points of the knot. We
assume that these points are $(0,1)$ and $(0,2)$. Although it will not play a large role in the present paper, it is sometimes helpful to know that we can choose one of these two points and perform
equivariant Reidemeister moves only at the second.

\begin{thm}\label{thm:undoing}
  Suppose $D$ and $D'$ are two diagrams of the same strongly invertible knot $K$ such that $D$ and $D'$ differ by a single \ref{R-1} move
  performed at the point $(0,1)$. Then, there is a sequence of Reidemeister moves connecting $D$ and $D'$ which may contain the \ref{R-1} move
  at $(0,2)$, but does not involve any \ref{R-1} move at $(0,1)$.
\end{thm}
\begin{proof}
The rough idea is that any time the \ref{R-1} move would be performed at $(0, 1)$, we simply rotate the rest of the diagram instead, keeping the diagram around $(0, 1)$ fixed. We include the details here for the reader who wishes to see the argument written out explicitly. Let $\smash{\wt{\phi}} \colon S^1 \rightarrow \R^3$ be a parameterization of $K$. Then $\smash{\phi = \pi \circ \wt{\phi}}$ is a regular diagram parameterizing $D$ and satisfies the conditions listed in Theorem~\ref{thm:generic_half_knot}. In particular, $\phi$ satisfies item~\ref{item:no_end_tang}, which states that $\phi'$ is nonzero at every point in $\Sfix$.

Choose cylindrical coordinates $(r,z,\theta)$ in $\R^3$. 
Fix any $r_1,r_2 > 0$ with $r_1 < r_2$. Let $C_1$ and $C_2$ be the concentric cylinders centered at $(0, 0, 1)$ given by
  $\{r\le r_i\}$ and $|z_i-1|\le r_i$ for $i=1,2$. We assume $r_1$ and $r_2$ are sufficiently small so that they intersect $K$ in small arc. Let $\kappa$ be any smooth function which is zero inside $C_1$ and $\pi$ outside of $C_2$. For $s\in [0,1]$, let
  \[\wt{\phi}_s\colon S^1\to\R^3 \quad \text{by} \quad \wt{\phi}_s(u)=R_{s\kappa(\wt{\phi}(u))} \circ \wt{\phi}(u),\]
where $R_w$ is the rotation of $\R^3$ about the $z$-axis by angle $w$. Clearly, $\wt{\phi}_s$ consists of taking a knot and rotating all of it, except for the part contained in $C_1$. The endpoint of this rotation gives the same diagram as if we had performed an \ref{R-1} move at $(0, 1)$.

Perturb the projection $\phi_s$ of this equivariant isotopy to be regular. Note, that $\phi_s$ already satisfies~\ref{item:no_end_tang} at $(0,1)$ for all $s$; since \ref{item:no_end_tang} is an open condition, our perturbation has this property also. Now, the \ref{R-1} move at $(0,1)$ arises precisely from
  violations of~\ref{item:no_end_tang} at $(0,1)$. We have thus constructed a regular path consists of equivariant Reidemeister moves without the \ref{R-1} move at $(0,1)$, as desired.
\end{proof}

It is not difficult to show that this procedure does, in fact, create an \ref{R-1} move at $(0, 2)$. In fact, it is easily seen that the parity of the number of \ref{R-1} moves needed to connect to equivariant diagrams is invariant.

\subsection{Equivariant Reidemeister moves in $S^3$}\label{sub:ES3}
In the non-equivariant setting, if $L_0$ and $L_1$ are two links in $\R^3$ that are isotopic as links in $S^3$, then they are isotopic as links in $\R^3$. This is evident from the usual dimension counting argument. Indeed, since $\cS$ is one-dimensional and hitting $\infty$ is a codimension 3 condition, a generic one-parameter family of maps from $\cS$ into $S^3$ can be assumed to avoid $\infty$. Unfortunately, this does not hold in the equivariant setting, leading to a difference between equivariant Reidemeister moves in $\R^3$ and $S^3$.

To investigate this, we must understand the possible singularities of a diagram at $\infty$. Let $\iota\colon S^3\to S^3$ be an inversion of the sphere that commutes with the symmetry action and moves $\infty$ to $0$. Write $\pi^{\iota}\colon\R^3\to\R^2$ for the composition $\pi\circ\iota$ defined on $\R^3=S^3\setminus\{0\}$. We then set 
\[
\phi^\iota = \pi^{\iota}\circ\wt{\phi} \colon\cS\to\R^2
\]
for any map $\wt{\phi} \colon \cS \rightarrow S^3$. This gives a diagram of $\wt{\phi}$ around $\infty$. We abuse of notation slightly since $\phi^\iota$ is not defined on $\smash{\wt{\phi}^{-1}(0)}$, but this will
not pose any problems. 

\begin{defn}\label{defn:reg_at_inft} Singularities of a map $\wt{\phi} \colon \cS \rightarrow S^3$ at $\infty$ are defined in terms of the behavior of $\phi^\iota$ near the origin.
  \begin{itemize}
    \item An equivariant map $\wt{\phi}$ is regular at infinity if $\phi^{\iota} \colon \cS \rightarrow \R^2$ misses the origin.
      
    \item A family of equivariant maps $\wt{\phi}_s$ is regular at infinity if the map $\phi^\iota_s \colon \cS \times I \rightarrow \R^2$ is transverse to the origin. In particular, $\phi^\iota_s$ hits the origin 
  only for finitely many parameters $s$. At such parameters, we require that the diagram $\phi^\iota_s$ is regular at the origin in the sense of Definition~\ref{def:regular-half-knot}. For example, $\phi^\iota_s$ can have at most a double point at the origin, and so on.

  \end{itemize}
\end{defn}

\begin{thm}\label{thm:reg_at_inf}
  Regular maps and paths are residual in the space of equivariant maps from $\cS$ to $S^3$. 
\end{thm}
\begin{proof}
The space of all equivariant maps is the disjoint union $\cI^0\cup\cI^1\cup\cI^2\cup\cI^3$, where:
  \begin{itemize}
    \item $\cI^0$ consists of maps such that $\phi^\iota$ misses the origin.
    \item $\cI^1$ consists of maps such that $\phi^{\iota}$ hits the origin and the diagram is regular at the origin in the sense of Definition~\ref{def:regular-half-knot}.
    \item $\cI^2$ consists of all remaining maps.
  \end{itemize}
We claim that $\cI^1 \cup \cI^2 \cup \cI^3$ has $\pcodim = 1$; and that $\cI^2$ has $\pcodim = 2$.

Consider the first claim. The union $\cI^1 \cup \cI^2$ consists of maps where $\phi^\iota$ hits the origin. This can happen two ways: either $\phi^\iota(\cS^-)$ hits the origin, which requires two conditions and has domain of dimension $1$; or $\phi^{\iota}(\cS^{\Z_2})$ hits the origin, which requires one condition and has domain of dimension $0$. Hence in both cases the situation is of codimension 1. Imposing the failure of regularity increases this codimension by $1$, showing the second claim; imposing the failure of a condition from $\cF^1$ increases the codimension further by $1$. This gives the third claim. The theorem then follows from the Equivariant Parameter Counting Lemma~\ref{lem:codimension}.
\end{proof}

A path of equivariant maps $\wt{\phi}_s$ from $\cS$ to $S^3$ may thus be assumed to lie in $\cI^0 \cup \cI^1$ and cross $\cI^1$ transversely at a finite number of parameters. As before, we attempt to understand the possible before-and-after pictures associated to these crossings. By construction, at each such parameter $s_i$, the diagram $\phi^\iota_{s_i}$ hits the origin at some $t_i \in \cS$. There are two possibilities:

\begin{itemize}
  \item If $t_i \in \Sfix$, then a fixed-point branch of $\phi_{s_i}$ hits $\infty$. As $s$ varies, this fixed point sweeps over $\infty$ as indicated in Figure~\ref{fig:reg_at_inf}.
  \item If $t_i \in \Sfree$, then there are two branches of $\phi_{s_i}$ that hit $\infty$, corresponding to a double point over $\infty$. As $s$ varies, the pair of arcs forming this double point sweep over $\infty$ as indicated in Figure~\ref{fig:reg_at_inf}.
\end{itemize}
This gives rise to the following moves at $\infty$: 

\begin{defn}
  We define the \textit{I-move} to be the move which takes a strand of the knot diagram and passes it through $\infty$ in an equivariant manner. See Figure~\ref{fig:ISmove}.
\end{defn}

\begin{defn}
  We define the \textit{S-move} to be the move which takes a double point in the diagram and passes it through $\infty$ in an equivariant manner. See Figure~\ref{fig:ISmove}.
\end{defn}

\begin{figure}
  \begin{tikzpicture}
    \draw (-4,0) node {\includegraphics[height=2.5cm]{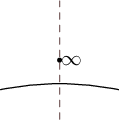}};
    \draw (0,0) node {\includegraphics[height=2.5cm]{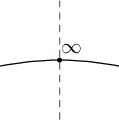}};
    \draw (4,0) node {\includegraphics[height=2.5cm]{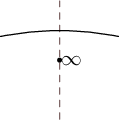}};
    \draw (-4,-3) node {\includegraphics[height=2.5cm]{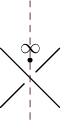}};
    \draw (0,-3)  node {\includegraphics[height=2.5cm]{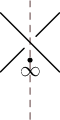}};
    \draw (4,-3)  node {\includegraphics[height=2.5cm]{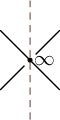}};
  \end{tikzpicture}
  \caption{Crossing the point at infinity. Top: the fixed branch crosses the infinity point. Bottom: a pair of branches crossing
  the infinity point.}\label{fig:reg_at_inf}
\end{figure}
\begin{figure}
  \begin{tikzpicture}
    \draw (-5,0) node {\includegraphics[width=2.5cm]{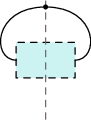}};
    \draw (-2,0) node {\includegraphics[width=2.5cm]{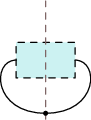}};
  \node at (2,0) {\includegraphics[width=2.5cm]{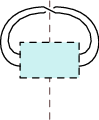}};
\node at (5,0) {\includegraphics[width=2.5cm]{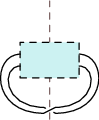}};
  \end{tikzpicture}
  \caption{Left: The I-move. The shaded part is the remaining diagram, which is not changed by the move. Right: the S-move.}\label{fig:ISmove}
\end{figure}

It is not hard to check that the I-move cannot be replicated by the composition of any other equivariant Reidemeister moves, since the I-move interchanges the order of the fixed points on the axis of symmetry. In contrast, any two diagrams related by the S-move are related by a different sequence of equivariant Reidemeister moves. To see this, consider Figure~\ref{fig:ISmove}. Take the top strand of the double point and move it over the rest of the diagram, while simultaneously moving the bottom strand of the double point under the rest of the diagram. This gives an equivariant isotopy in $\R^3$ that replicates the effect of the S-move. Perturbing this isotopy so that it is regular gives the desired sequence of equivariant Reidemeister moves. Nevertheless, the S-move will be useful when discussing loops of Reidemeister moves in Section~\ref{sec:carter_saito}.

We thus finally obtain a proof of Theorem~\ref{thm:intro4}:

\introfour*

\begin{proof}
This follows immediately from Theorem~\ref{thm:equiv_reid}, together with the above discussion of the I-move and S-move.
\end{proof}

Towards further applications in Section~\ref{sub:iloop}, we analyze codimension~2 singularities at infinity. We consider two-parameter families of equivariant embeddings
$\wt{\phi}_{s,s'}\colon\cS\to S^3$. The map $\phi_{s,s'}$ is the projection, and $\phi^{\iota}_{s,s'}$ is the projection
at infinity. The fact that $\phi_{s,s'}$ and $\phi^{\iota}_{s,s'}$ are defined only on a subset of $\cS$ (e.g. $\phi_{s,s'}$ is not
defined at $t\in\cS$ for which $\wt{\phi}_{s,s'}(t)\in S^3\setminus\R^3$), will not pose significant problems. 

These will essentially turn out to be the codimension~1 on-axis singularities of $\phi^\iota$ taking place at the point $(0,0)\in\cL$.
\begin{defn}\label{def:two_regular}
 For a two-parameter family $\phi_{s,s'}$, $(s,s')\in[0,1]^2$ its infinity locus
  \[\Sigma=\{(s,s')\in[0,1]^2\colon \phi^\iota_{s,s'}\textrm{ hits }(0,0)\in\cL\}\]
  is decomposed as a union of $\Sigmafix$ and $\Sigmafree$ corresponding to whether ${\phi^\iota}^{-1}_{s,s'}(0,0)$ belongs to $\Sfix$
  or $\Sfree$ (the spaces $\Sigmafix$ and $\Sigmafree$ may intersect).

  We say that the family
 is \textit{regular at infinity} if the infinity locus
  has the following properties.
  \begin{enumerate}[label=($\Sigma$-\arabic*)]
    \item $\Sigmafix$ is a smooth closed $1$-manifold and $\Sigmafree$ is an immersed $1$-manifold with boundary with possibly ordinary
      double points as self-intersections;\label{item:Sigma_smooth}
    \item the boundary of $\Sigmafree$ belongs to $\Sigmafix$;\label{item:Sigma_bound}
    \item the intersection of $\Sigmafree$ and $\Sigmafix$ is transverse; \label{item:Sigma_transverse}
    \item for $(s,s')\in\Sigma$ the diagram of $\phi^\iota_{s,s'}$ at $(0,0)$ is either regular, or has a codimension~1 singularity.\label{item:Sigma_codim}
    \item The function $(s,s')\mapsto s'$ restricts to a Morse function on $\Sigmafix$ and on $\Sigmafree$.\label{item:Sigma_Morse}.
  \end{enumerate}
\end{defn}
\begin{lemma}\label{lem:generic_at_infty}
  Any two-parameter family $\phi_{s,s'}$ can be $C^\infty$-perturbed to a family that is regular at infinity.
\end{lemma}
\begin{proof}
  The condition that $\phi^{\iota}_{s,s'}$ has
  a singularity at $(0,0)$ of codimension 2 or more is of codimension at least $3$ in $\cF$: in fact, to the standard dimension count behind
  the decomposition of $\wt{\cF}^1\setminus\cF^1$, we add one for an extra condition that the singularity occurs at $(0,0)$. This shows
  \ref{item:Sigma_transverse}.

  Next, consider the space $\wt{\Sigma}\in\cS\times[0,1]^2$ given by $(t,s,s')$ such that $\phi^\iota_{s,s'}(t)=0$. 
  This space can be regarded as the preimage $\Phi^{-1}(0,0)$ under the map $\Phi\colon \cS\times[0,1]^2\to\R^2$,
  $\Phi(t,s,s')=\phi^\iota_{s,s'}(t)$. Take $\cS'=\Sfree/\Z_2$, and let $\Phi'$ be the restriction of $\Phi$ to $\cS'$ (it is well-defined,
  because $\Phi$ is $\Z_2$-equivariant with respect to the $t$ variable).
  Also, let $\Phi^{\Z_2}$ be the restriction of $\Phi$ to $\Sfix$. On perturbing $\Phi$, we may assume that $(0,0)$ is a regular
  value of both $\Phi^{\Z_2}$ and $\Phi'$. Write $\wt{\Sigma}^{\Z_2}$ and $\wt{\Sigma}'$ as preimages. These are clearly smooth 1-dimensional manifolds, and $\wt{\Sigma}^{\Z_2}$ has no boundary. Let $\Sigmafix$ be the projection of $\wt{\Sigma}^{\Z_2}$ to $[0,1]^2$ under the projection map $\Sfix\times[0,1]^2\to[0,1]^2$, likewise
  we let $\Sigma'$ be the projection of $\wt{\Sigma}'$ to $[0,1]^2$. By equivariance of $\Phi$, $\Sigma'$ agrees with $\Sigmafree$.

  For a dense subset of functions $\Phi$, the projections $\Sigmafix$ and $\Sigmafree$ are regular in the sense of Definition~\ref{def:generic_R2}. That is, $\Sigmafix$ and $\Sigmafree$ have at most double points as singularities. Moreover, if the projection is regular,
  the only singularities of the intersection of $\Sigmafix$ and $\Sigmafree$ are transverse double points proving \ref{item:Sigma_transverse}.

  We claim that $\Sigmafix$ has no double points. Suppose it has at $(s_0,s'_0)$. As $\wt{\Sigma}^{\Z_2}$ is smooth, a double point corresponds
  to the pair $t_1,t_2\in\Sfix$ such that $t_1\neq t_2$ and $\phi^\iota_{s_0,s'_0}(t_i)=(0,0)$. But then, $\wt{\phi}_{s_0,s'_0}$ hits the infinity point at two values, $t_1$ and $t_2$. This leads to a contradiction, proving \ref{item:Sigma_smooth}.

  Next, suppose that $(s_n,s_n')\in\Sigmafree$ converges to $(s_0,s_0')$ that is not in $\Sigmafree$. By construction, there exists
  a sequence $t_n\in\Sfree$ such that $\Phi(t_n,s_n,s_n')=(0,0)$. On passing to a subsequence, we assume that $t_n$ converges to $t_0$.
  Clearly $\Phi(t_0,s_0,s_0')=(0,0)$.
  If $t_0\in\Sfree$, then $(s_0,s_0')\in\Sigmafree$, contradiction. Hence, $t_0\in\Sfix$, so $(s_0,s_0')\in\Sigmafix$, proving \ref{item:Sigma_bound}.

  Finally, we prove~\ref{item:Sigma_Morse}. We show more: on perturbing $\Phi$, we may assume that the functions $\cS'\times[0,1]^2\to[0,1]$,
  and $\Sfix\times[0,1]^2\to[0,1]$ given by
  $(t,s,s')\mapsto s'$ restrict to Morse functions on $\wt{\Sigma}'$ and $\wt{\Sigma}^{\Z_2}$, respectively. This proves \ref{item:Sigma_Morse},
  and also gives insight on what is the Morse condition on the set of double points of $\Sigmafree$, namely,
  the function restricts to a Morse function
  on each branch through a double points. Also, we may assume that there are no local minima or maxima at these double points.
\end{proof}

Suppose now $\wt{\phi}_{s,s'}$ is regular at infinity. Fix $s'\in[0,1]$. As $s$ goes from $0$ to $1$, the path
$\gamma_s\colon[0,1]\to[0,1]^2$, $\gamma_s=(s,s')$ can intersect $\Sigmafree\cup\Sigmafix$. Each intersection correspond either to an $S$-move
or an $I$-move, or to a Reidemeister move at infinity. There is a partial translation of the intersection type and the type of bifurcation
that occurs.

\begin{itemize}
  \item If $\gamma_s$ intersects $\Sigmafree$ transversely at a point disjoint from the closure of $\Sigmafix$, 
    then this corresponds to either to an S-move along the path $s\mapsto \phi_{s,s'}$;
  \item If $\gamma_s$ intersects $\Sigmafix$ transversely at a point disjoint from $\Sigmafree$, 
    then this corresponds either to an I-move along the path $s\mapsto\phi_{s,s'}$, or to an M-2 move (for finitely many values of $s'$);
  \item If $\gamma_s$ intersects a smooth point of $\Sigmafree$ non-transversally, then this can be either a cancellation of a pair of S-moves, or a tangency at infinity (R-2 move);
  \item If $\gamma_s$ intersects a smooth point of $\Sigmafix$ non-transversally, then this is a cancellation of a pair of I-moves;
  \item If $\gamma_s$ intersects a double point of $\Sigmafree$, then $\phi^{\iota}_{s,s'}$ has two pairs of branches intersection $(0,0)$.
    This leads to the M-3 move at infinity;
  \item If $\gamma_s$ intersects a point in $\Sigmafix$ intersected with a point in the relative interior of $\Sigmafree$, then
    this is a move involving one fixed branch and a pair of branches interchanged by $\tau$, but the double points do not disappear. 
    This is the M-1 move;
  \item If $\gamma_s$ intersects a point of the relative boundary of $\Sigmafree$ (this point is necessarily on $\Sigmafix$, then
    this corresponds to a Reidemeister move where a double point disappears. This is the R-1 move.
\end{itemize}


\section{Two-parameter deformations of equivariant link diagrams}\label{sec:codim2}
We now lay the groundwork for the proof of Theorem~\ref{thm:intro4} by studying two-parameter deformations of equivariant link diagrams. We have seen in Section~\ref{sec:one_para} that equivariant Reidemeister moves arise from studying one-parameter families of maps which cross $\cF^1$. The movie moves of Theorem~\ref{thm:intro4} will analogously arise from understanding two-parameter families of maps. Explicitly, a two-parameter family of maps will not stay within $\cF^0 \cup \cF^1$, but will intersect the codimension $2$ subset
\[
\wt{\cF}^2 = \cF \setminus (\cF^0 \cup \cF^1).
\]
We thus begin by identifying a subset $\smash{\cF^2 \subset \wt{\cF}^2}$ whose complement in $\wt{\cF}^2$ has codimension $3$. This will consist of the mildest possible singularities within $\smash{\wt{\cF}^2}$. Any two-parameter family of maps can then be perturbed such that all intersections with $\smash{\wt{\cF}^2}$ occur at points of $\cF^2$. 

Once again, each of these intersections corresponds to a codimension $2$ singularity which is sufficiently constrained so that we can write down a normal form and versal deformation. This provides a local model for how our two-parameter family crosses $\cF^2$. As described in Example~\ref{ex:codim2}, such a versal deformation is parameterized over the disk, with the center of the disk giving the codimension $2$ singularity. The discriminant locus of the deformation will consist of arcs radiating out from the center of the disk, along which our family has codimension $1$ singularities. Traveling around the boundary of the disk gives a cyclic sequence of Reidemeister moves, with each move coming from crossing one of the arcs in the discriminant locus. These are precisely the sequences from Theorem~\ref{thm:intro4}.

The general strategy will be to show that any codimension $2$ singularity lies in one of $18$ individual codimension $2$ strata $\cF^2_i$. (There is also an additional stratum $\cF^2_0$ corresponding to coincidence of codimension $1$ singularities.) We then simply define 
\[
\cF^2 = \cF^2_0 \cup \cF^2_1 \cup \cdots \cup \cF^2_{18}.
\]
These are enumerated with the help of Section~\ref{sec:patterns}, which we also use to provide normal forms and versal deformations. The bulk of this section will thus consist of understanding the two-parameter bifurcation diagram in each case.

\subsection{Towards a classification of codimension 2 singularities}
To begin with, we categorize potential codimension~2 singularities into several groups:
\begin{itemize}
  \item Off-axis singularities, see Section~\ref{sub:off2};
  \item On-axis or fixed-point cusps, see Section~\ref{sub:cusp2};
  \item On-axis or fixed-point strikethroughs, see Section~\ref{sub:strike2};
  \item On-axis or fixed-point tangencies, see Section~\ref{sub:tang2};
  \item Coincidences of two codimension 1 singularities.
\end{itemize}
Coincidences of singularities are easily understood and we treat them directly in the proof of Theorem~\ref{thm:two_loops} below. We treat off-axis cusps, strikethroughs, and tangencies together in Section~\ref{sub:off2}, since this case is the same as in the non-equivariant setting, and is thus straightforward.

The following theorem asserts that all codimension $2$ singularities are of one of the forms discussed in Section~\ref{sec:patterns}, except for one special case. This occurs when a fixed-point branch passes through an on-axis line tangency. We have not covered this possibility in Section~\ref{sub:strike}, as we prefer not to give the general theory of strikethroughs by fixed-point branches. We subsume this case in Section~\ref{sub:tang2}. 

\begin{thm}\label{thm:complete}
  Each singularity with $\pcodim = 2$ that is not a coincidence is either:
  \newcounter{Senum}
  \begin{enumerate}[label=(S-\arabic*)]
    \item an off-axis singularity;\label{item:Soff}
    \item an on-axis or fixed-point cusp; \label{item:Scusp}
    \item an on-axis or fixed-point (ordinary) strikethrough; \label{item:Sstrike} 
    \item an on-axis or fixed-point tangency; \label{item:Stang}
      \setcounter{Senum}{\value{enumi}}
     \end{enumerate}
     or the special case:
     \begin{enumerate}[label=(S-\arabic*)]
       \setcounter{enumi}{\value{Senum}}
        \item a fixed-point branch of $\phi$ intersecting an on-axis line tangency of order $1$. \label{item:Sltang_bound}
  \end{enumerate}
\end{thm}
\begin{proof}
  If the singularity occurs at two distinct points in the codomain, then it is a coincidence. Thus, suppose we have a codimension $2$ singularity occurring at $w\in\R^2$. If $w\notin\cL$, then the singularity is off-axis and we are in the situation~\ref{item:Soff}. Thus we assume $w\in\cL$. Let the corresponding domain points be $t_1, \ldots, t_n \in \cS$ and their orbits. We proceed by analyzing the possible values of $n$, together with what possibilities there are for the $t_i$ (i.e., whether they lie in $\Sfree$ or $\Sfix$). As usual, we speak of $\phi$ as having $n$ branches, even though if $t_i \neq \sigma t_i$ the branch through $t_i$ generates a symmetric copy.
    
Firstly, note that at most one of the $t_i$ lies in $\Sfix$, since different points of $\Sfix$ are mapped to different points of $\cL$. Let $f$ be the number of $t_i$ in $\Sfix$, so that either $f = 0$ or $f = 1$. The coincidence $\phi(t_1) = \phi(t_2) = \cdots = \phi(t_n)$ constitutes $2n - 2$ equations. We require that these lie on $\cL$; this gives one extra equation if $f = 0$ and is automatic if $f = 1$. The dimension of the domain is given by $n - f$. Hence the parametric codimension, even before imposing extra conditions like tangency or the singularity of branches, is given by $(2n - 2) + (1 - f) - (n - f) = n - 1$. Thus $n$ is at most $3$.

If $n = 3$, then to have $\pcodim = 2$ there must be no additional conditions. Since at least one branch is through a point of $\Sfree$, our singularity is a strikethrough of an on-axis or fixed-point singularity with two branches and $\pcodim = 1$. This means our overall singularity is a strikethrough of an on-axis double point or a fixed double point. In either case, we are in the setting of \ref{item:Sstrike}.

If $n = 2$, then to have $\pcodim = 2$ we must have exactly one additional condition. One possibility is that this condition is between the two branches; i.e., the two branches are tangent. This leads to an on-axis or fixed-point tangency, and we are in case \ref{item:Stang}. The other possibility is that the additional condition induces a singularity within a single branch (or between a branch and its symmetric copy). If the nonsingular branch is through a point of $\Sfree$, then our singularity is a strikethrough of an on-axis or fixed-point singularity with one branch and $\pcodim = 1$. This means our overall singularity is a strikethrough of an on-axis line tangency, on-axis perpendicular tangency, or fixed-point cusp. In each case we are in the setting of \ref{item:Sstrike}. It is also possible that the nonsingular branch is a fixed-point branch. In this case, our singularity must consist of a fixed-point branch intersecting an on-axis tangency. If this is a line tangency, then we have the special case \ref{item:Sltang_bound}; if it is a perpendicular tangency, then in fact the branches form a fixed-point tangency and we are in the case of \ref{item:Stang}.

Finally, if $n = 1$, then to have $\pcodim = 2$ we need exactly two additional conditions. If the branch is through a point of $\Sfree$, then we have a higher-order on-axis cusp or higher-order on-axis tangency; these are cases \ref{item:Scusp} and \ref{item:Stang}. If the branch is through a point of $\Sfix$, then we have a higher-order fixed-point cusp; this is case \ref{item:Scusp}.
\end{proof}

\subsection{Off-axis singularities}\label{sub:off2}

The enumeration of codimension $2$ off-axis singularities is essentially the same as in the non-equivariant case. It follows from our discussion in Section~\ref{sec:patterns} (see also \cite{David,Wall_Pgen}) that these consist of:  

\begin{itemize}
  \item off-axis $(2,5)$-cusp, called the $A_4$ singularity;
  \item off-axis $2$-fold tangency, called the $A_5$ singularity;
  \item ordinary strikethrough of an off-axis $(2,3)$-cusp;
  \item ordinary strikethrough of an off-axis $1$-fold tangency;
  \item ordinary strikethrough of an off-axis ordinary triple point (resulting in an off-axis quadruple point), called the  $X_9$ singularity.
\end{itemize}
We display these in Figure~\ref{fig:off-the-axis}. The bifurcation diagrams are described in \cite{FiedlerKurlin}, but for completeness we recall these here. Throughout, we alter various conventions throughout Section~\ref{sec:patterns} to more closely match \cite{FiedlerKurlin}.

\begin{figure}
  \begin{tikzpicture}
    \draw (-6,0) node  {\includegraphics[width=2.2cm]{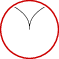}};
    \draw (-3.5,0) node  {\includegraphics[width=2.2cm]{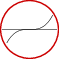}};
    \draw (-1,0) node  {\includegraphics[width=2.2cm]{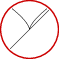}};
    \draw (1.5,0) node  {\includegraphics[width=2.2cm]{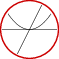}};
    \draw (4.0,0) node  {\includegraphics[width=2.2cm]{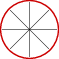}};
  \end{tikzpicture}
  \caption{Off-axis singularities. From left-to-right: a $(2,5)$-cusp, $2$-fold tangency, strikethrough of a $(2,3)$-cusp, strikethrough of a $1$-fold tangency, and quadruple point.}\label{fig:off-the-axis}
\end{figure}

\subsubsection{$(2,5)$-cusp}\label{ssub:A4}
By Lemma~\ref{lem:normal_cusp}, a $(2,5)$-cusp has normal form $(x = t^5, y = t^2)$ and versal deformation $(x = t^5+\lambda_1 t^3+\lambda_2 t, y = t^2)$. We explain the changes in the topology of the curves in this family by first finding the discriminant set. This procedure will be used for the remaining 17 cases of codimension 2 singularities.

We check for which values of $\lambda_1$ and $\lambda_2$ this map has a codimension $1$ singularity. In this case, the only possibilities are an off-axis cusp, off-axis tangency, or off-axis triple point. For cusps: a $(2, 3)$-cusp occurs along the locus $\{\lambda_2 = 0\}$.\footnote{If in addition $\lambda_1 = 0$, this further degenerates to the $(2, 5)$-cusp. Throughout this section, we abuse language slightly and implicitly assume $\lambda_1$ and $\lambda_2$ are not simultaneously zero.} For tangencies: we start by understanding when our deformation has a self-intersection. Suppose $x(t) = x(t')$ and $y(t) = y(t')$. The second equality forces $t = - t'$; the first equality then gives $x(t)=0$. A $1$-fold tangency will occur for values of $\lambda_1$ and $\lambda_2$ at which $x(t)$ has a double real root, which is equivalent to $t^4+\lambda_1t^2+\lambda_2$ having a double real root. This happens along the locus $\{\lambda_1^2 = 4\lambda_2, \lambda_1 \leq 0\}$. For triple points: we see that there are none, since $y(t)=y(t')=y(t'')$ implies that two out the three parameters $t,t',t''$ are equal.

The diagram is drawn in Figure~\ref{fig:A4}, see also \cite[Figure 8v]{FiedlerKurlin}.\footnote{There is a slight formal mistake in the figure in \cite{FiedlerKurlin} due to the omission of the condition $\lambda_1 \leq 0$ accompanying the relation $\lambda_1^2 = 4\lambda_2$.} The parameters $\lambda_1$, $\lambda_2$ vary over the disk bounded by the blue circle, which we take to be $\lambda_1^2 + \lambda_2^2 = \varepsilon$ for some small $\varepsilon$. We indicate the discriminant set (i.e., the set of parameters for which a singularity occurs) in red. This consists of $\{\lambda_2=0\}$ and $\{\lambda_1^2 = 4 \lambda_2,\lambda_1\leq0\}$. Note that the discriminant set intersects the boundary of the disk in a finite number of points and divides it into a collection of arcs. At each such point, we draw a singular diagram (circled in red), while for each arc, we draw a regular diagram (circled in blue) for a representative point within that arc.

As we travel around the central circle, we go from regular diagram to regular diagram by passing through the indicated singular diagrams. Two regular diagrams on either side of a singular diagram are related by the equivariant Reidemeister move associated to that singular diagram. Crossing the locus $\{\lambda_2 = 0\}$ corresponds to an \ref{IR-1} move, while crossing the locus $\{\lambda_1^2 = 4 \lambda_2,\lambda_1\leq0\}$ corresponds to an \ref{IR-2} move.

\begin{figure}
  \includegraphics[width=10cm]{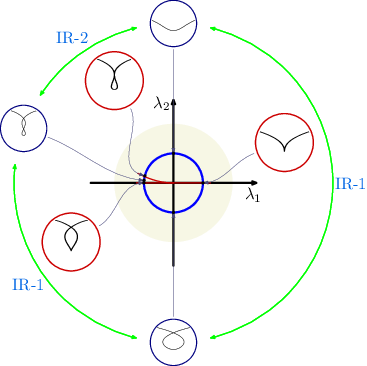}
  \caption{Off-axis $(2,5)$-cusp.}\label{fig:A4}
\end{figure}

\subsubsection{$2$-fold tangency}\label{ssub:A5}
By Lemma~\ref{lem:normal_tang}, a $2$-fold tangency has normal form with two branches $y = 0$ and $(x = t, y = t^3)$ and versal deformation which perturbs the second branch to $(x = t, y = t^3+\lambda_1 t+\lambda_2)$. 

There are no cusps or triple points regardless of $\lambda_1,\lambda_2$. 
A $1$-fold tangency will occur if $y(t)$ has a double root, which occurs along the locus $\{4\lambda_1^3=-27\lambda_2^2\}$ and leads to an \ref{IR-2} move.
The diagram is drawn in Figure~\ref{fig:A5}, see also \cite[Figure 8iv]{FiedlerKurlin}. 
\begin{figure}
  \includegraphics[width=10cm]{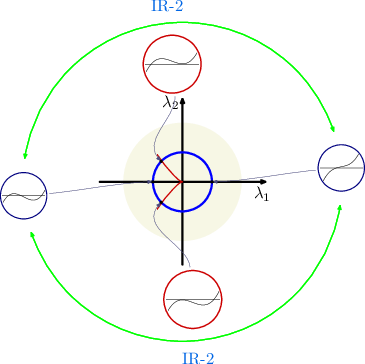}
  \caption{Off-axis $2$-fold tangency.}\label{fig:A5}
\end{figure}

\subsubsection{Strikethrough of a $(2,3)$-cusp}\label{ssub:sA2}
This has normal form with two branches $y = ax$ and $(x=t^3,y=t^2)$ for some $a\neq 0$. We can rescale our coordinates so that $a=1$. By Theorem~\ref{thm:euh_deform}, a versal deformation is given by $y=x+\lambda_1$ and $(x=t^3 - \lambda_2 t,y= t^2)$.

For $\{\lambda_2=0\}$, the cusp persists, leading to an \ref{IR-1} move. For $\lambda_2<0$, the second branch is injective, while for $\lambda_2 \geq 0$, there is a double point for $t=\pm\sqrt{\lambda_2}$ at $(0, \lambda_2)$. Along $\{\lambda_1 = \lambda_2, \lambda_2 \geq 0\}$, the strikethrough line passes through this double point, leading to a triple point and hence an \ref{IR-3} move. Finally, as computed in \cite[Lemma 4.4]{FiedlerKurlin}, along $\{3\lambda_1 + \lambda_2^2=0\}$ we have a tangency of the strikethrough line to the cusp branch, leading to an \ref{IR-2} move. The diagram is given in Figure~\ref{fig:sA2}, see also \cite[Figure 8iii]{FiedlerKurlin}. 
\begin{figure}
  \includegraphics[width=10cm]{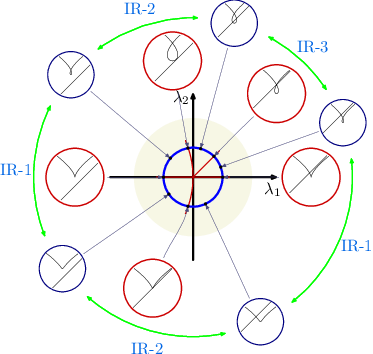}
  \caption{Strikethrough of an off-axis $(2,3)$-cusp.}\label{fig:sA2}
\end{figure}

\subsubsection{Strikethrough of a $1$-fold tangency}\label{ssub:sA3}
This has normal form with three branches $y = 0$, $(x = t, y = t^2)$, and $y = ax$ for some $a \neq 0$. We can rescale our coordinates so that $a=1$. By Theorem~\ref{thm:euh_deform}, a versal deformation is given by sending the second and third branches to $(x = t, y = t^2 + \lambda_1)$ and $y = x + \lambda_2$.

There are no cusps. Along $\{\lambda_1 = 0\}$ the tangency between the first and second branches persists, leading to an \ref{IR-2} move. The double points between the first and second branches occur at $(\pm\sqrt{-\lambda_1},0)$; a triple point occurs when the strikethrough line passes through one of these. This occurs along the locus $\{\lambda_2 = \pm \sqrt{-\lambda_1}, \lambda_1 \leq 0\}$, leading to an \ref{IR-3} move. The diagram is given in Figure~\ref{fig:sA3}, see also \cite[Figure 8ii]{FiedlerKurlin}.

\begin{figure}
  \includegraphics[width=10cm]{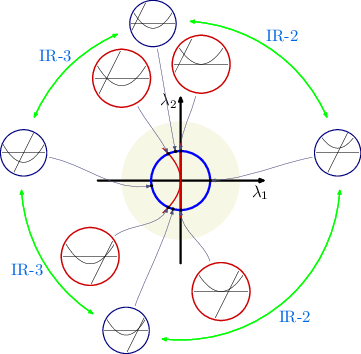}
  \caption{Strikethrough of an off-axis $1$-fold tangency.}\label{fig:sA3}
\end{figure}

\subsubsection{Strikethrough of a triple point}\label{ssub:X9}
This case is not quite covered by Theorem~\ref{thm:euh_deform} since the ordinary triple point is not an unbalanced homogenous singularity in the sense of Definition~\ref{def:euh}. We have a normal form given by $x = 0$, $y = 0$, $y = x$, and $y = ax$, where $a \neq 0$ or $1$; compare \cite{David,FiedlerKurlin}. However, this time the parameter $a$ cannot be removed by a local coordinate change. It is not difficult to check that a versal deformation of the singularity is given by fixing the first two branches and deforming the last two to $y = x + \lambda_1$ and $y = ax + \lambda_2$. However, this set of parameters is really 3-dimensional, because of the failure of simplicity: strictly speaking, the parameters are $(a,\lambda_1,\lambda_2)$. 

We can bypass this difficulty by freezing the parameter $a$. That is, if we have any two-dimensional deformation of the four-tuple point, then by versality it is induced
from the above deformation. Thus, there exist functions $\lambda_1(\zeta_1,\zeta_2)$, $\lambda_2(\zeta_1,\zeta_2)$, $a(\zeta_1,\zeta_2)$
such that the deformation is given by $y = x + \lambda_1(\zeta_1,\zeta_2)$ and $y = a(\zeta_1,\zeta_2) x +\lambda_2(\zeta_1,\zeta_2)$, where
  $\zeta_1,\zeta_2$ are the deformation parameters. We assume without loss of generality that the four-tuple point occurs when $\zeta_1 = \zeta_2=0$ and nowhere else;
  that is $\lambda_1,\lambda_2$ vanish only at $(0,0)$. We can then modify the deformation in such a way that $a$ is constant near $(0,0)$.

  Although in principle the behavior of the deformation may depend on the constant value of $a$, in this case it is not difficult to see that different values of $a$ result in essentially the same picture. We display the case of $a = -1$. There are no cusps or tangencies, because all the branches are straight lines.
It is straightforward to check that triple points occur along the loci $\{\lambda_1 = 0\}$, $\{\lambda_2 = 0\}$, $\{\lambda_1 = \lambda_2\}$, and $\{\lambda_1 = -\lambda_2\}$, leading to \ref{IR-3} moves.
The diagram is given in Figure~\ref{fig:X9}, see also \cite[Figure 8i]{FiedlerKurlin}. 

\begin{figure}
  \includegraphics[width=10cm]{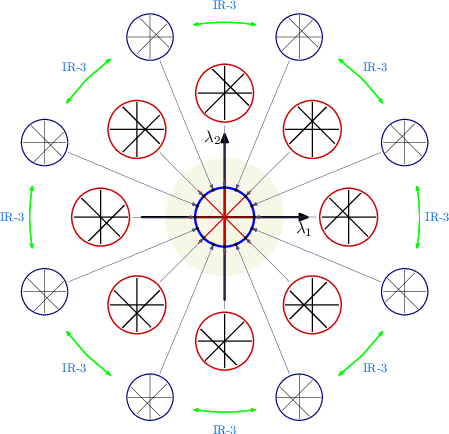}
  \caption{Off-axis quadruple point.}\label{fig:X9}
\end{figure}

\subsection{On-axis and fixed-point cusps}\label{sub:cusp2}

By Lemma~\ref{lem:cusplist}, the on-axis and fixed-point cusps with codimension $2$ are:

\begin{itemize}
  \item on-axis oblique $(2,3)$-cusp;
  \item fixed-point $(2,5)$-cusp;
  \item fixed-point $(3,4)$-cusp.
\end{itemize}
We display these in Figure~\ref{fig:on-the-axis-cusps}.

\begin{figure}
  \begin{tikzpicture}
    \draw (-6,0) node  {\includegraphics[width=2.5cm]{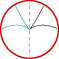}};
    \draw (-3,0) node  {\includegraphics[width=2.5cm]{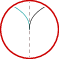}};
    \draw (0,0) node  {\includegraphics[width=2.5cm]{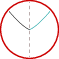}};
  \end{tikzpicture}
  \caption{On-axis and fixed-point cusps. From left-to-right:
  an oblique $(2,3)$-cusp, fixed-point $(2,5)$-cusp, and fixed-point $(3,4)$-cusp.}\label{fig:on-the-axis-cusps}
\end{figure}

\subsubsection{Oblique $(2,3)$-cusp}\label{sub:int_cusp} 
By Lemma~\ref{lem:normal_bound}, an oblique $(2, 3)$-cusp has normal form given by $(x = t^2, y = t^3 + t^2)$ and versal deformation
\[
x(t) = t^2 + \lambda_1,\ y(t) = t^3 + t^2 + \lambda_2t + \lambda_1. 
\]
Note that the deformation of the symmetric branch is given by 
\[
x(s) = -s^2 - \lambda_1,\ y(s) = s^3 + s^2 + \lambda_2 s + \lambda_1.
\]

Along the locus $\{\lambda_2 = 0\}$, the cusp persists as a pair of off-axis cusps, leading to an \ref{IR-1} move. As in Section~\ref{ssub:A4}, to determine when other singularities occur, we start by understanding when multiple points occur in our deformation. Now, however, we must consider both when the cusp branch intersects itself, and when it intersects its symmetric copy. This is due to the fact that our singularity is no longer far away from $\cL$.

To solve for a self-intersection, set $x(t) = x(t')$ and $y(t) = y(t')$ with $t \neq t'$; this results in $t = \pm t'$ from the first equation and $t^3 + \lambda_2 t = 0$ from the second. Whenever $\lambda_2 \leq 0$, we thus have a single self-intersection
\begin{equation}\label{eq:selfintersectoblique}
t = \pm \sqrt{-\lambda_2}\ \text{at } (\lambda_1-\lambda_2, \lambda_1 - \lambda_2).
\end{equation}
If this self-intersection is on $\cL$, then we obtain an on-axis double point; this occurs along the locus $\{\lambda_1 = \lambda_2, \lambda_2 \leq 0\}$ and results in an \ref{M-3} move. Moreover, \eqref{eq:selfintersectoblique} may contribute to a triple point if the $s$-branch passes through the point of self-intersection. This occurs if $s^2=\lambda_2-2\lambda_1$ and $s^2(s+1)+\lambda_2s+\lambda_1=\lambda_1-\lambda_2$, which implies
\[(2\lambda_2-2\lambda_1)(s+1)=0.\] 
There are two ways this can be satisfied. The first is when $s = -1$, in which case substituting back gives $1 = \lambda_2 - 2 \lambda_1$. This does not occur for small values of $\lambda_1, \lambda_2$. The second is when $\lambda_1 = \lambda_2$, but this has already been covered by the previous discussion. 

Now consider when the two symmetric branches intersect each other. Setting $x(t) = x(s)$ and $y(t) = y(s)$, this occurs when
\begin{equation}\label{eq:cusp_eq}t^2+s^2=-2\lambda_1,\ (t-s)(t^2+ts+s^2+t+s+\lambda_2)=0.\end{equation}
One obvious solution to the second equation occurs when $t=s$, in which case we have two points of intersection 
\[
t= s = \pm\sqrt{-\lambda_1}\ \text{at } (0, \pm\sqrt{-\lambda_1}(\lambda_2-\lambda_1))
\]
so long as $\lambda_1 \leq 0$. These occur on $\cL$ and may contribute to the formation of a codimension $1$ singularity in the following manner: if the two points of intersection collide, then we obtain a tangency between the two branches. This occurs if either $\lambda_1 = 0$ or $\lambda_1 = \lambda_2$. The latter case has already been considered. The former induces a line tangency along the locus $\{\lambda_1 = 0\}$, leading to an \ref{R-2} move. 

The other way to solve \eqref{eq:cusp_eq} is for the second factor of the second equation to vanish. Setting $p=t+s$, $q=ts$ turns this into
\begin{equation}\label{eq:cusp_eq1}
  p^2-2q=-2\lambda_1,\ p^2-q+p=-\lambda_2.
\end{equation}
Solving the second equation for $q$ in terms of $p$ and the substituting this into the first gives
\begin{equation}\label{eq:cusp_eq2}
  q=p^2+p+\lambda_2,\ p^2+2p+2(\lambda_2-\lambda_1)=0.
\end{equation}
From the second equation, we see that for small $\lambda_1$, $\lambda_2$, we always have two distinct solutions for $p$. However, each solution in $p$ and $q$ generally corresponds to two solutions in $t$ and $s$, as interchanging the values of $t$ and $s$ does not change $p$ or $q$. Hence a tangency will arise whenever a solution to \eqref{eq:cusp_eq1} in $p$ and $q$ corresponds to a solution with $t = s$, as this reflects a pair of solutions degenerating into a single solution. This will happen when $p^2 = 4q$. Plugging $p^2 = 4q$ back into \eqref{eq:cusp_eq1} gives $q=-\lambda_1$, $p=3\lambda_1-\lambda_2$, and plugging these back into $p^2 = 4q$ gives $(3\lambda_1-\lambda_2)^2=-4\lambda_1$. It is then straightforward to check that we have an on-axis perpendicular tangency along $\{\lambda_2 = 3 \lambda_1 \pm 2 \sqrt{-\lambda_1}, \lambda_1 \leq 0\}$, leading to a \ref{M-2} move.

The diagram is drawn in Figure~\ref{fig:cBD}. As the curve $\lambda_2=3\lambda_1\pm 2\sqrt{-\lambda_1}$ is very close
to the vertical line, for the reader's convenience we zoom in on a portion of the diagram in Figure~\ref{fig:cBD1}. 

\begin{figure}
  \includegraphics[width=10cm]{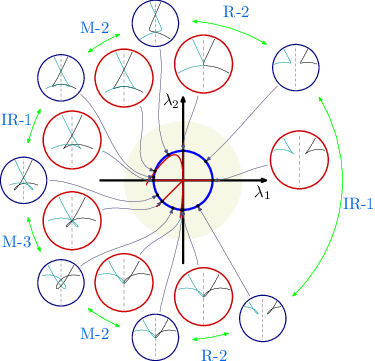}
  \caption{Oblique (2,3)-cusp.}\label{fig:cBD}
\end{figure}
\begin{figure}
  \includegraphics[width=10cm]{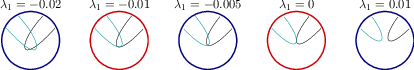}
  \caption{Close-up picture for $\lambda_1 \sim 0$ and $\lambda_2 \sim -0.23$.}\label{fig:cBD1}
\end{figure}

\subsubsection{Fixed-point $(2,5)$-cusp}\label{ssub:25cusp}
By Lemma~\ref{lem:fixedcusp}, a fixed-point $(2, 5)$-cusp has normal form $(x = t^5, y = t^2)$ and versal deformation $(x = t^5 + \lambda_1 t^3 + \lambda_2, y = t^2)$. The analysis is the same as for the off-axis $(2, 5)$-cusp, except that the \ref{IR-1} and \ref{IR-2} moves are replaced with the \ref{R-1} and \ref{R-2} moves, respectively. The diagram is given in Figure~\ref{fig:boundA4}.
\begin{figure}
  \includegraphics[width=10cm]{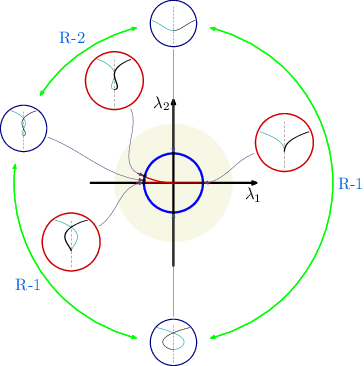}
  \caption{Fixed-point $(2,5)$-cusp.}\label{fig:boundA4}
\end{figure}

\subsubsection{Fixed-point $(3,4)$-cusp}\label{ssub:34cusp}
By Lemma~\ref{lem:dot}, a fixed-point $(3, 4)$-cusp has normal form $(x = t^3, y = t^4)$ and versal deformation $(x = t^3 + \lambda_1 t, y = t^4 + \lambda_2 t^2)$. 

Setting the derivative equal to zero, we see that there are two possibilities for cusps. First, if $\lambda_1 = 0$, then there is a cusp when $t = 0$, at the origin. Hence along the locus $\{\lambda_1 = 0\}$, we have an on-axis $(2, 3)$-cusp, leading to an \ref{R-1} move. Second, along the locus $\{2 \lambda_1 = 3 \lambda_2, \lambda_1 \leq 0\}$, we have a pair of off-axis cusps when $t=\pm\sqrt{-\lambda_1/3}=\pm\sqrt{-\lambda_2/2}$, leading to an \ref{IR-1} move. 

To determine the occurrence of other singularities, we first understand how our branch intersects itself. Setting $p=t+t'$, $q=tt'$ gives the identities
\[
t^2-{t'}^2=(t-t')p,\ t^3-{t'}^3=(t-t')(p^2-q),\ t^4-{t'}^4=(t-t')(p^3-2pq). 
\]
Substituting these into $x(t) = x(t')$ and $y(t) = y(t')$ yields
\begin{equation}\label{eq:34eq}
p^2-q+\lambda_1=0, \ p^3-2pq+\lambda_2p=0.
\end{equation}
Factor the second equation into $p(p^2-2q+\lambda_2)=0$ and substitute $p^2 = q - \lambda_1$ into this to get $p(-q - \lambda_1 + \lambda_2) = 0$. We obtain two classes of solutions corresponding the two factors of this expression vanishing. Combining these with the first equation of \eqref{eq:34eq} gives
\[
p = 0 \quad \text{and} \quad q = \lambda_1
\]
for the first, and
\[
q = -\lambda_1 + \lambda_2 \quad \text{and} \quad p^2 = \lambda_2 - 2 \lambda_1
\]
for the second. In the latter case, we also require $\lambda_2 - 2 \lambda_1 \geq 0$. 

These self-intersections can contribute to the formation of a codimension $1$ singularity in several different ways. First, if a solution to \eqref{eq:34eq} in $p$ and $q$ corresponds to a solution where $t$ or $t'$ is $0$, then we have a fixed double point. For the first class of solution, this occurs when $\lambda_1 = 0$, which is subsumed by our casework for cusps. For the second, this occurs when $-\lambda_1 + \lambda_2 = 0$. Hence we obtain a fixed double point along the locus $\{\lambda_1 = \lambda_2, \lambda_2 \leq 0\}$, leading to an \ref{M-1} move. Second, if a solution to \eqref{eq:34eq} in $p$ and $q$ additionally satisfies $p^2 = 4q$, then it leads to a tangency, as explained in Section~\ref{sub:int_cusp}. Combining $p^2 = 4q$ with the first class of solution to \eqref{eq:34eq} gives $\lambda_1 = 0$, while combining $p^2 = 4q$ with the second gives $2\lambda_1=3\lambda_2$. Both of these are subsumed by our casework for cusps.

However, there is one more possibility for a tangency. This occurs when the two classes of solutions in $p$ and $q$ collide; i.e., when $\lambda_2 - 2\lambda_1 = 0$. Note that in this case we require $\lambda_1 \leq 0$, so that $p = t + t' = 0$ and $q = tt' = \lambda_1$ have solutions in $t$, $t'$. It is then straightforward to check that we have an on-axis perpendicular tangency along the locus $\{2\lambda_1 = \lambda_2, \lambda_1 \leq 0\}$, leading to an \ref{M-2} move. The diagram is drawn in Figure~\ref{fig:boundT34}.

\begin{figure}
  \includegraphics[width=10cm]{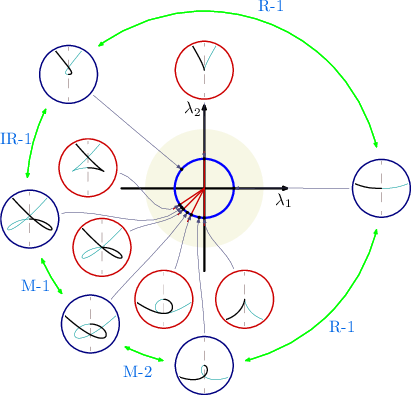}
  \caption{Fixed-point $(3,4)$-cusp.}\label{fig:boundT34}
\end{figure}

\subsection{On-axis and fixed-point strikethroughs}\label{sub:strike2}

By Section~\ref{sub:strike}, the on-axis and fixed-point strikethroughs with codimension $2$ are:

\begin{itemize}
  \item strikethrough of an on-axis $1$-fold line tangency;
  \item strikethrough of an on-axis $2$-fold perpendicular tangency;
  \item strikethrough of an on-axis double point;
  \item strikethrough of a fixed-point $(2, 3)$-cusp;
  \item strikethrough of a fixed double point.
\end{itemize}
These are shown in Figure~\ref{fig:st-on-the-axis}. To provide a deformation for a strikethrough singularity, we move the strikethrough line and separately deform the codimension~1 singularity. For unbalanced homogenous singularities, Theorem~\ref{thm:euh_deform} shows that the resulting deformation is versal. For the two non-simple singularities (i.e., the strikethrough of an on-axis double point and strikethrough of a fixed double point), we use the procedure of freezing the slope of the strikethrough line, as described in Section~\ref{ssub:X9}.

\begin{figure}
  \begin{tikzpicture}
    \draw (-6,0) node  {\includegraphics[width=2.5cm]{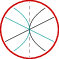}};
    \draw (-3.5,0) node  {\includegraphics[width=2.5cm]{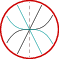}};
    \draw (-1,0) node  {\includegraphics[width=2.5cm]{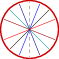}};
    \draw (1.5,0) node  {\includegraphics[width=2.5cm]{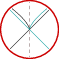}};
    \draw (4.0,0) node  {\includegraphics[width=2.5cm]{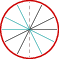}};
  \end{tikzpicture}
  \caption{On-axis and fixed-point strikethroughs. From left-to-right: strikethrough of a line tangency, strikethrough of a perpendicular tangency, strikethrough of an on-axis double point, strikethrough of a fixed-point cusp, and strikethrough of a fixed double point.}\label{fig:st-on-the-axis}
\end{figure}

\subsubsection{Strikethrough of a $1$-fold line tangency}
This has normal form $(x = t^2, y = t)$ and $y = ax$ for some $a \neq 0$. Up to reparameterizing, we can assume $a = 1$. A versal deformation is given by
\[
E_1=\{x=t^2+\lambda_1,y=t\}, \ E_2=\{y=x-\lambda_2\},
\]
with symmetric branches
\[
E_3=\tau(E_1)=\{x=-t^2-\lambda_1, y = t\}, \ E_4=\tau(E_2)=\{y=-x-\lambda_2\}.
\]
For small values of $\lambda_1,\lambda_2$, the only tangency is between $E_1$ and $E_3$. This is a line tangency, which persists along the locus $\{\lambda_1=0\}$. If $\lambda_1>0$, then $E_1$ and $E_3$ are disjoint, leading to a regular diagram. For $\lambda_1 \leq 0$, we see that $E_1$ and $E_3$ intersect $\cL$ together at $(0,\pm\sqrt{-\lambda_1})$. The strikethrough line $E_2$ passes through $(0,\pm\sqrt{-\lambda_1})$ when $\lambda_2=\pm\sqrt{-\lambda_1}$. Thus, along the locus $\{\lambda_1=-\lambda_2^2\}$ we have an on-axis double point. The
diagram is drawn in Figure~\ref{fig:sCL}. 

\begin{figure}
  \includegraphics[width=10cm]{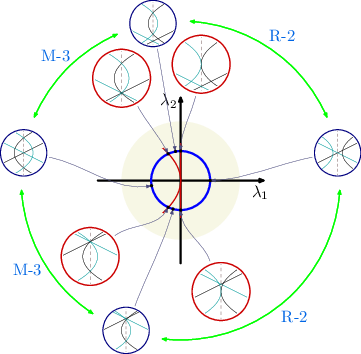}
  \caption{Strikethrough of a $1$-fold line tangency.}\label{fig:sCL}
\end{figure}
\subsubsection{Strikethrough of a $2$-fold perpendicular tangency}
This has normal form $(x = t, y = t^3)$ and $y = ax$ for some $a \neq 0$. Up to reparameterizing, we can assume $a = 1$. A versal deformation is given by
\[
E_1=\{x=t,y=t^3+\lambda_1 t\}, \ E_2=\{y=x+\lambda_2\},
\]
with symmetric branches
\[
E_3=\tau(E_1)=\{x=-t,y=t^3+\lambda_1t\}, \ E_4=\tau(E_2)=\{y=-x+\lambda_2\}.
\]
For small values of $\lambda_1,\lambda_2$, the only tangency is between $E_1$ and $E_3$. This is a perpendicular tangency, which persists along the locus $\{\lambda_1=0\}$.
Next, along the locus $\{\lambda_2=0\}$, all four curves pass through the origin and we have an on-axis double point. This is the only singularity which involves the intersection of $E_2$ with $E_4$. To search for triple points, we thus look for double points involving $E_1$ and $E_3$. These intersect at the origin, as well as $(\pm\sqrt{-\lambda_1},0)$, provided $\lambda_1\le 0$. The line $E_2$ passes through the latter if $\lambda_2=\mp\sqrt{-\lambda_1}$. Along the locus $\{\lambda_1=-\lambda_2^2\}$ we thus obtain an off-axis triple point. The diagram is drawn in Figure~\ref{fig:sPerp}.

\begin{figure}
 \includegraphics[width=10cm]{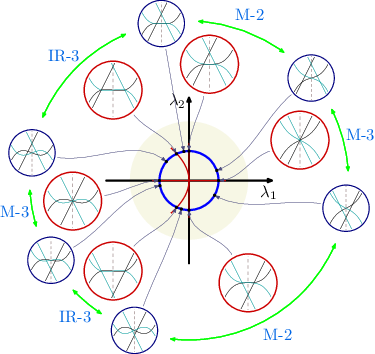}
  \caption{Strikethrough of a $2$-fold perpendicular tangency.}\label{fig:sPerp}
\end{figure}

\subsubsection{Strikethrough of an on-axis double point}\label{sub:sDouble}
We assume the two branches involved in the double point are $y=a_1x$ and $y=a_2x$, and the strikethrough line is $y=a_3x$.
By rescaling $x$ and $y$, we can make $a_1,a_2$ arbitrary, as long as $a_1\neq a_2$, 
but we cannot adjust $a_3$. For explicitness, we suppose $a_1=2.5$, $a_2=1$. The versal deformation is given by
\[
E_1=\{y=2.5x+\lambda_1\}, \ E_2=\{y=x\}, \ E_3=\{y = a_3x+\lambda_2\}, 
\]
with symmetric branches
\[
E_4=\tau(E_1)=\{y=-2.5x+\lambda_1\}, \ E_5=\tau(E_2)=\{y=-x\}, \ E_6=\tau(E_3)=\{y = -a_3x+\lambda_2\}.
\]
The deformation parameters are $\lambda_1,\lambda_2$, and $a_3$. However, as in Section~\ref{ssub:X9}, we can freeze the $a_3$ variable, leading to $a_3$ being constant. For explicitness, we take $a_3=0.4$, so that the slopes of the three lines are approximately $\pm 22.5^\circ$, $\pm 45^\circ$, and $\pm 67.5^\circ$. No matter what the parameters $\lambda_1,\lambda_2$ are, each pair of lines intersects transversely. 

The lines $E_1$, $E_2$, and $E_3$ intersect $\cL$ at the points $(0,\lambda_1)$, $(0,0)$,
$(0,\lambda_2)$. Thus, along any of the loci $\{\lambda_1=0\}$, $\{\lambda_2=0\}$, and $\{\lambda_1=\lambda_2\}$, we have an on-axis double point, where two out
of these three points coincide. This leads to an \ref{M-3} move. Off-axis triple points occur when $E_i\cap E_j\cap E_k$ is nonempty for $i \in \{1, 4\}, j \in \{2, 5\}$, and $k \in \{3, 6\}$. This leads to 8 cases, but by symmetry we may reduce to $4$. The conditions are easy to compute explicitly: $E_1\cap E_2\cap E_3 \neq \emptyset$ along $\{\lambda_2=-\frac25\lambda_1\}$; $E_1\cap E_2\cap E_6\neq\emptyset$ along $\{\lambda_2=-\frac{14}{15}\lambda_1\}$; $E_1\cap E_5\cap E_3\neq\emptyset$ along $\{\lambda_2=\frac25\lambda_1\}$; and $E_1\cap E_5\cap E_6\neq\emptyset$ along $\{\lambda_2=\frac{6}{35}\lambda_1\}$. These give rise to \ref{IR-3} moves. The diagram is drawn in Figure~\ref{fig:sDOA}.

\begin{figure}
  \includegraphics[width=10cm]{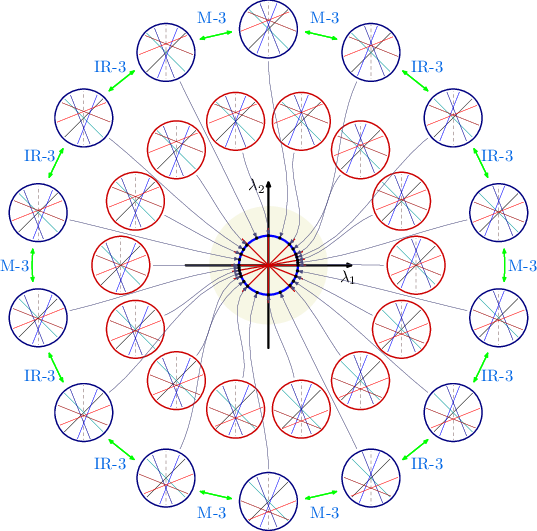}
  \caption{Strikethrough of a central double point.}\label{fig:sDOA}
\end{figure}

\subsubsection{Strikethrough of a fixed-point $(2, 3)$-cusp}\label{ssub:sBcusp}
This has normal form $(x = t^3, y = t^2)$ and $y = ax$ for some $a \neq 0$. Up to reparameterizing, we can assume $a = 1$. A versal deformation is given by
\[
E_1=\tau(E_1)=\{x=t^3-\lambda_2t,\ y=t^2\}, \ E_2=\{y=x+\lambda_1\}, \ E_3=\tau(E_2)=\{y=-x+\lambda_1\}.
\]

Along the locus $\{\lambda_2=0\}$, the fixed-point cusp persists and results in an \ref{R-1} move. The fixed-point branch $E_1$ always intersects $\cL$ at $(0, 0)$; if $E_2$ passes through this, then we obtain a fixed double point. This occurs along the locus $\{\lambda_1 = 0\}$, leading to an \ref{M-1} move. So long as $\lambda_2 \geq 0$, the fixed-point brach also intersects $\cL$ at $(0, \lambda_2)$; if $E_2$ passes through this, we have an on-axis double point. This occurs along the locus $\{\lambda_1 = \lambda_2, \lambda_2 \geq 0\}$, leading to an \ref{M-3} move. 

As for tangencies, $E_2$ is transverse to $E_3$, so we consider tangencies between $E_1$ and $E_2$. For this, we find intersections between $E_1$ and $E_2$ and solve for a double root. Substituting the defining equations for $E_1$ into that of $E_2$, we must find when $t^3-t^2-\lambda_2t+\lambda_1 = 0$ has a double solution. Taking the derivative, this consists of the set of $\lambda_1$, $\lambda_2$ such that $3t^2 - 2t - \lambda_2 = 0$ and $t^3-t^2-\lambda_2t+\lambda_1 = 0$ have a simultaneous solution. We obtain the parameterized curve
\[
C = \{\lambda_2 = 3s^2 - 2s, \lambda_1 = 2s^3 - s^2\}.
\]
Along $C$, there is an off-axis tangency, leading to an \ref{IR-2} move.

The diagram is drawn in Figure~\ref{fig:sBC}. For small values of $\lambda_1$, $\lambda_2$, the curve $C$ is very close to the $\lambda_2$-axis; the full curve $C$ in fact loops back around and intersects the $\lambda_1$ axis again, but we consider only the portion very close to the origin. We thus zoom in on the two singularities corresponding to the intersection of $C$ with $\lambda_1^2 + \lambda_2^2 = \varepsilon$. These are shown in Figures~\ref{fig:sBC1} and \ref{fig:sBC2}.

\begin{figure}
  \includegraphics[width=10cm]{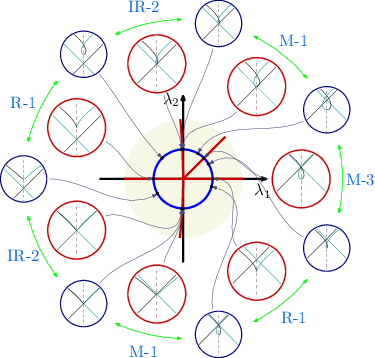}
  \caption{Strikethrough of a fixed-point $(2, 3)$-cusp.}\label{fig:sBC}
\end{figure}

\begin{figure}
  \includegraphics[width=10cm]{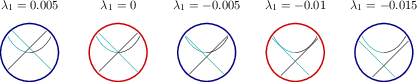}
  \caption{Close-up picture for $\lambda_1 \sim 0$ and $\lambda_2 \sim 0.2$.}\label{fig:sBC1}
\end{figure}
\begin{figure}
  \includegraphics[width=9cm]{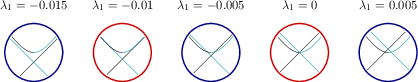}
  \caption{Close-up picture for $\lambda_1 \sim 0$ and $\lambda_2 \sim -0.2$.}\label{fig:sBC2}
\end{figure}

\subsubsection{Strikethrough of a fixed double point}
We assume the branches involved in the fixed double point are $y=0$ and $y=a_1x$, along with its symmetric copy $y = -a_1x$. Let the strikethrough line be $y=a_2x$.
By rescaling $x$ and $y$, we can choose $a_1$, as long as $a_1\neq 0$ and $a_1 \neq a_2$, but we cannot fix $a_2$. For explicitness, we suppose $a_1=1/2$. A versal deformation is given by 
\[
E_1= \tau(E_1) = \{y=0\}, \ E_2=\{y= 0.5x + \lambda_2\}, \ E_3=\{y = a_2x+\lambda_1\}, 
\]
with symmetric branches
\[
E_4=\tau(E_2)=\{y=-0.5x+\lambda_2\}, \ E_5=\tau(E_3)=\{y=-a_2x + \lambda_1\}.
\]
The deformation parameters are $\lambda_1,\lambda_2$, and $a_2$. However, as in Section~\ref{ssub:X9}, we can freeze the $a_2$ variable, leading to $a_2$ being constant. Such a deformation is normal to the stratum defining the strikethrough singularity. For explicitness, we assume $a_2=2$. 

There are no tangencies or cusps, so we consider only multiple points. Along the loci $\{\lambda_1=0\}$ and $\{\lambda_2=0\}$, we have a fixed double point, leading to an \ref{M-1} move. Along the locus $\{\lambda_1 = \lambda_2\}$, we have an on-axis double point, leading to an \ref{M-3} move. There are no other multiple points lying on $\cL$. It is easily checked that $E_1$ must be involved in any multiple point off of $\cL$, so $y = 0$. The intersection of $E_2$ with $E_1$ occurs at $x=-2\lambda_1$, while that of $E_3$ with $E_1$ occurs at $x=-\lambda_2/a_2$. We thus have an off-axis triple point along the locus $\{\lambda_2=2a_2\lambda_1\}$. Replacing $E_3$ by $E_5$ gives a similar triple point locus $\{\lambda_2=-2a_2\lambda_1\}$. The diagram is drawn in Figure~\ref{fig:sBD2p}.

\begin{figure}
  \includegraphics[width=10cm]{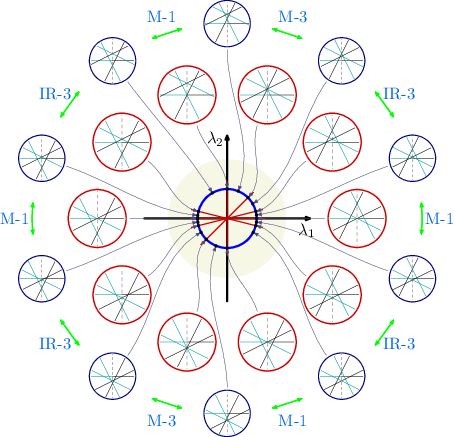}
  \caption{Strikethrough of a fixed double point.}\label{fig:sBD2p}
\end{figure}

\subsection{On-axis and fixed-point tangencies}\label{sub:tang2}
By Lemma~\ref{lem:complicatedtan}, the on-axis and fixed-point tangencies with codimension $2$ are: 
\begin{itemize}
  \item on-axis $2$-fold line tangency;
  \item on-axis $4$-fold perpendicular tangency;
  \item on-axis $1$-fold oblique tangency;
  \item fixed-point $(1,2)$-fold tangency.
\end{itemize}
We also treat the exceptional case from Theorem~\ref{thm:complete} in this section:
\begin{itemize}
\item a fixed-point branch of $\phi$ intersecting an on-axis line tangency of order $1$.
\end{itemize}
We display these in Figure~\ref{fig:st-on-the-axis}.
\begin{figure}
  \begin{tikzpicture}
    \draw (-6,0) node  {\includegraphics[width=2.5cm]{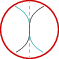}};
    \draw (-3.5,0) node  {\includegraphics[width=2.5cm]{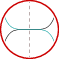}};
    \draw (-1,0) node  {\includegraphics[width=2.5cm]{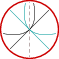}};
    \draw (1.5,0) node  {\includegraphics[width=2.5cm]{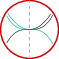}};
    \draw (4,0) node  {\includegraphics[width=2.5cm]{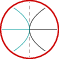}};
  \end{tikzpicture}
  \caption{On-axis and fixed-point tangencies. From left-to-right: $2$-fold line tangency, $4$-fold perpendicular tangency, $1$-fold oblique tangency, fixed-point $(1, 2)$-fold tangency, and a fixed-point branch intersecting a $1$-fold line tangency.}\label{fig:tang-axis}
\end{figure}

\subsubsection{$2$-fold line tangency}\label{ssub:LOTR}
By Lemma~\ref{lem:normal_int_tang}, a $2$-fold line tangency has normal form $(x = t^3, y = t)$ and versal deformation $(x = t^3 + \lambda_1 t + \lambda_2, y = t)$. The analysis is the same as for the off-axis $2$-fold tangency, except that the \ref{IR-2} moves are replaced by \ref{R-2} moves. The diagram is given in Figure~\ref{fig:LOTR}. 

\begin{figure}
  \includegraphics[width=10cm]{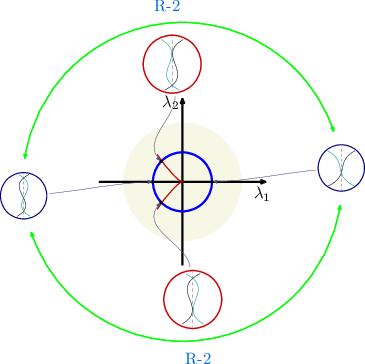}
  \caption{$2$-fold line tangency.}\label{fig:LOTR}
\end{figure}

\subsubsection{$4$-fold perpendicular tangency}\label{ssub:sperp}

By Lemma~\ref{lem:normal_perp}, a $4$-fold perpendicular tangency has normal form $(x = t, y = t^5)$ and versal deformation $(x = t, y = t^5 + \lambda_1t^3+\lambda_2t)$. Note this has symmetric branch $(x = -s, y = s^5 + \lambda_1s^3+\lambda_2s)$. 

Clearly, no cusps and triple points occur. Along $\{\lambda_2 = 0\}$, the tangency persists as a $2$-fold perpendicular tangency, leading to an \ref{M-2} move. To check for other tangencies, we observe that our branch intersects its symmetric copy only at points on the $x$-axis. Hence tangencies occur when $y(t)$ and $y'(t)$ have a common root. This leads to 
\[
t^4 + \lambda_1 t^2 + \lambda_2 = 0 \quad \text{and} \quad 5t^4 + 3\lambda_1 t^2 + \lambda_2 = 0.
\] 
This implies $t^4 + \lambda_1 t^2 = 5t^4 + 3\lambda_1 t^2$ and hence $-2t^2 = \lambda_1$. Substituting this back into the first equation gives $\lambda_1^2 = 4 \lambda_2$. Along the locus $\{\lambda_1^2 = 4\lambda_2, \lambda_1 \leq 0\}$ we thus have an off-axis tangency, leading to an \ref{IR-2} move. The diagram is drawn in Figure~\ref{fig:perp5}.

\begin{figure}
  \includegraphics[width=10cm]{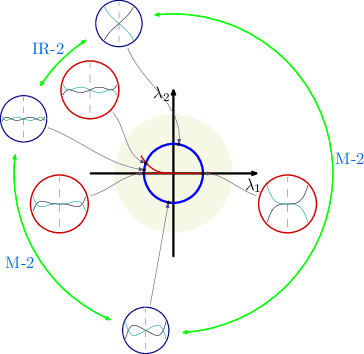}
  \caption{$4$-fold perpendicular tangency.}\label{fig:perp5}
\end{figure}

\subsubsection{$1$-fold oblique tangency}

Lemma~\ref{lem:oblique_form} provides us with a normal form of an oblique tangency as well as with a versal deformation.
Choose coordinates $u=x+y$, $v=x-y$, so that $\cL$ is given by $u = v$ and $\tau$ exchanges $u$ and $v$. A versal deformation is given by 
\[
E_1 = \{v = \lambda_1\}, \ E_2 = \{v = u^2 + \lambda_2\}
\]
with symmetric branches
\[
E_3 = \tau(E_1) = \{u = \lambda_1\}, \ E_4 = \tau(E_2) = \{u = v^2 + \lambda_2\}.
\]

There are clearly no cusps. For $\lambda_1 = \lambda_2$, the tangency between $E_1$ and $E_2$ persists, and for small values of $\lambda_1$, $\lambda_2$ there are no other tangencies. Hence we obtain an off-axis tangency along the locus $\{\lambda_1 = \lambda_2\}$, leading to an \ref{IR-2} move. To understand further singularities, we first determine the intersection $E_1\cap E_2$. This is given by $(\pm\sqrt{\lambda_1-\lambda_2},\lambda_1)$, provided $\lambda_1\ge \lambda_2$. If this lies on $\cL$, we obtain an on-axis double point. Setting $\pm\sqrt{\lambda_1-\lambda_2}=\lambda_1$, this occurs along the locus $\{\lambda_2 = \lambda_1 - \lambda_1^2\}$, leading to an \ref{M-3} move.

We now check for triple points. We see that $E_1\cap E_2\cap E_3$ is non-empty if $\pm\sqrt{\lambda_1-\lambda_2}=\lambda_1$, but this is subsumed by the previous paragraph. Likewise, $E_1\cap E_2\cap E_4$ is non-empty if $\pm\sqrt{\lambda_1-\lambda_2} = \lambda_1^2+\lambda_2$. Squaring both sides and applying various elementary manipulations gives the equivalent condition
\[
(\lambda_2-\lambda_1+\lambda_1^2)(\lambda_1^2+\lambda_1+\lambda_2+1) =0.
\]
The first factor vanishing is again subsumed by the previous paragraph, while the second factor does not vanish for $\lambda_1,\lambda_2$ small. There are no other singularities for small $\lambda_1,\lambda_2$. The diagram is drawn in Figure~\ref{fig:dtan}. 

\begin{figure}
  \includegraphics[width=10cm]{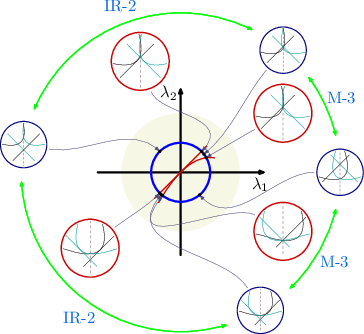}
  \caption{$1$-fold oblique tangency.}\label{fig:dtan}
\end{figure}

\subsubsection{Fixed-point $(1, 2)$-fold tangency}
Lemma~\ref{lem:fixed_point_tan} allows us to find local coordinates in which a $(1, 2)$-fold tangency has form with two symmetric branches $(x = t, y = t^3)$ and $(x = -t, y = t^3)$, and fixed-point branch $(x = t, y = h(t))$, where $h(t)$ is a symmetric function starting with $h(t)=h_2t^2+\dots$, $h_2\neq 0$.  A  deformation transverse to the defining set is given by 
\[
  E_1 = \{y = x^3 + \lambda_1 x + \lambda_2\}, \ E_2 = \tau(E_1) = \{y = - x^3 - \lambda_1 x + \lambda_2\}, \ E_3 = \{y = h(x)\}.
\]
Here $h(x)=h_2x^2+\dots$. On rescaling linearly $x$ and $y$, we may assume that $h_2=1$.

As usual, we will describe the path of Reidemeister move as the deformation parameters go around the circle $\lambda_1^2+\lambda_2^2=\varepsilon$
for $\varepsilon>0$ sufficiently small. We will make the following simplification. Suppose we have proved that
for fixed $\varepsilon>0$, and $h(t)=t^2$, the
path intersects $\cF^1$-stratum transversely. Then, adding higher order terms to $h$, and shrinking $\varepsilon$ if needed, will not change the type of crossings on that path. Therefore, the case of general $h(t)$ will be reduced to the case $h(t)=t^2$. This is what we will
henceforth assume.

Clearly, there are no cusps. Along $\{\lambda_1 = 0\}$, the tangency between $E_1$ and $E_1$ persists; this is a perpendicular tangency and thus leads to an \ref{M-2} move. To understand other singularities, we begin by understanding double points. The branches $E_1$ and $E_2$ intersect at $(0, \lambda_2)$, as well as $(\pm\sqrt{-\lambda_1},\lambda_2)$ if $\lambda_1 \leq 0$. We consider when $E_3$ intersects these. For the first double point this occurs at the origin along $\{\lambda_2 = 0\}$; this turns the singularity into a fixed double point and leads to an \ref{M-1} move. For the latter two double points, this occurs when $\lambda_2 = - \lambda_1$. We thus have an off-axis triple point along the locus $\{\lambda_2 = - \lambda_1, \lambda_1 \leq 0\}$, leading to an \ref{IR-3} move. 

Finally, we check for tangencies between $E_1$ and $E_3$. This occurs when the equation $x^3 + \lambda_1 x + \lambda_2 = x^2$ has a double solution. Following the same procedure as in Section~\ref{ssub:sBcusp} yields the parameterized curve
\[
C = \{\lambda_1=2s-3s^2, \lambda_2=-s^2+2s^3\}.
\]
It is straightforward to check that along $C$, we have a pair of off-axis tangencies, leading to an \ref{IR-2} move. The diagram is drawn in Figure~\ref{fig:cPerp}. For small values of $\lambda_1$, $\lambda_2$, the curve $C$ is very close to the $\lambda_2$-axis. We thus zoom in on the two singularities corresponding to the intersection of $C$ with $\lambda_1^2 + \lambda_2^2 = \varepsilon$. These are shown in Figures~\ref{fig:cPerpA} and \ref{fig:cPerpC}. 
\begin{figure}
  \includegraphics[width=10cm]{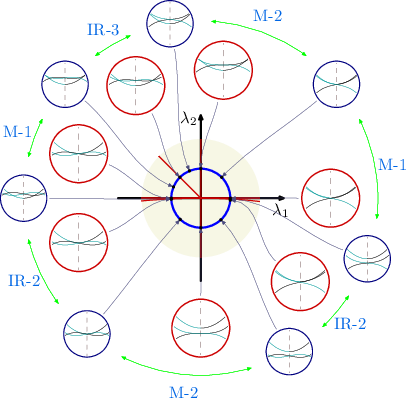}
  \caption{Fixed-point $(1, 2)$-tangency.}\label{fig:cPerp}
\end{figure}
\begin{figure}
  \includegraphics[width=10cm]{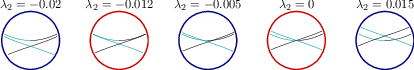}
  \caption{Close-up picture for $\lambda_1\sim 0.2$ and $\lambda_2 \sim 0$.}\label{fig:cPerpA}
\end{figure}
\begin{figure}
  \includegraphics[width=10cm]{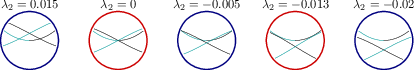}
  \caption{Close-up picture for $\lambda_1\sim -0.2$ and $\lambda_2 \sim 0$.}\label{fig:cPerpC}
\end{figure}

\subsubsection{Intersection of a fixed-point branch with $1$-fold line tangency}\label{ssub:LBP}
We now deal with the exceptional case from Theorem~\ref{thm:complete}, where a fixed-point branch passes through a $1$-fold line tangency. It is not difficult to check that this has normal form with tangency branch $(x = t^2, y = t)$ and fixed-point branch $y = 0$, together with the symmetric branch $(-x = t^2, y = t)$. A versal deformation is given by
\[
E_1 = \{x = y^2 + \lambda_1\}, \ E_2 = \{y = \lambda_2\}, \ E_3 = \tau(E_1) = \{x = - y^2 - \lambda_1\}.
\]

Along $\{\lambda_1 = 0\}$, the line tangency between $E_1$ and $E_3$ persists, leading to an \ref{R-2} move. If $\lambda_1>0$, then $E_1$ and $E_3$ are disjoint, leading to a regular diagram. For $\lambda_1 \leq 0$, we see that $E_1$ and $E_3$ intersect $\cL$ together at $(0,\pm\sqrt{-\lambda_1})$. The fixed-point branch $E_2$ passes through  $(0,\pm\sqrt{-\lambda_1})$ when $\lambda_2 = \pm \sqrt{-\lambda_1}$. Thus, along the locus $\{\lambda_1 = - \lambda_2^2\}$, we have an on-axis double point, leading to an \ref{M-1} move. The diagram is drawn in Figure~\ref{fig:tBP}. 

\begin{figure}
  \includegraphics[width=10cm]{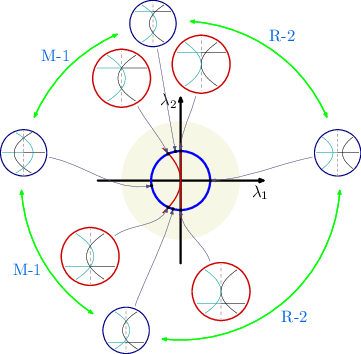}
  \caption{Intersection of a fixed-point branch with $1$-fold line tangency.}\label{fig:tBP}
\end{figure}

\subsection{Loops of Reidemeister moves}\label{sub:local_loops}
We now finally prove Theorem~\ref{thm:intro4} from the introduction. As discussed at the beginning of this section, each of the $18$ codimension $2$ singularities explored above gives rise to a loop of Reidemeister moves. However, there is a slight distinction between \textit{planar} Reidemeister moves and actual Reidemeister moves. The former refers to moves that relate two maps $\phi \colon \cS \rightarrow \R^2$; i.e., maps which do not have any overcrossing or undercrossing data. The latter refers to moves that relate two maps $\phi \colon \cS \rightarrow \R^2$ with given lifts to $\R^3$; i.e., maps which have the overcrossing and undercrossings specified. 

For clarity, we begin with a version of Theorem~\ref{thm:intro4} which is purely for planar Reidemeister moves.

\begin{theorem}\label{thm:two_loops}
Suppose we have two regular paths $\phi_{s, 0}$ and $\phi_{s, 1}$ between two regular maps $\phi_{0, 0} = \phi_{0, 1}$ and $\phi_{1, 0} = \phi_{1, 1}$. By Theorem~\ref{thm:equiv_reid}, these give rise to two sequences of (planar) Reidemeister moves. We claim that one can pass from one sequence of Reidemeister moves to the other via a sequence of the following operations:
  \begin{enumerate}[label=(C-\arabic*)]
    \item Replacing the order of two planar Reidemeister moves appearing in different places. \label{item:repl}
    \item Introducing a birth or death of a pair of mutually inverse planar Reidemeister moves; that is, a planar Reidemeister move followed by its inverse.\label{item:birth}
    \item Replacing a fragment of a local loop of planar Reidemeister moves from Table~\ref{tab:local_loops} by the complement of that fragment. \label{item:loop} 
  \end{enumerate}
\end{theorem}

\begin{table}
  \begin{tabular}{|c|c|c|p{6cm}|}\hline\hline
    \multicolumn{4}{|c|}{\textbf{Off-axis}, Figure~\ref{fig:off-the-axis}}\\ \hline
    L-1 & $(2,5)$-cusp & Fig.~\ref{fig:A4} & (IR-1)$\mapsto$(IR-2)$\mapsto$(IR-1)$\mapsto\cdots$ \\ \hline
    L-2 & $2$-fold tangency & Fig.~\ref{fig:A5} & (IR-2)$\mapsto$(IR-2)$\mapsto\cdots$ \\ \hline
    L-3 & strikethrough of a $(2, 3)$-cusp & Fig.~\ref{fig:sA2} & (IR-3)$\mapsto$(IR-2)$\mapsto$(IR-1)$\mapsto$(IR-2)$\mapsto$(IR-1)$\mapsto\cdots$ \\ \hline

    L-4 & strikethrough of a $1$-fold tangency & Fig.~\ref{fig:sA3} & (IR-2)$\mapsto$(IR-3)$\mapsto$(IR-3)$\mapsto$(IR-2)$\mapsto\cdots$ \\ \hline

    L-5 & four-tuple point & Fig.~\ref{fig:X9} & loop of $8$ (IR-3) moves \\ \hline\hline
    \multicolumn{4}{|c|}{\textbf{On-axis and fixed-point cusps}, Figure~\ref{fig:on-the-axis-cusps}}\\ \hline
    L-6 & oblique $(2, 3)$-cusp & Fig.~\ref{fig:cBD} & (IR-1)$\mapsto$(R-2)$\mapsto$(M-2)$\mapsto$(IR-1)$\mapsto$(M-3)$\mapsto$(M-2)$\mapsto$(R-2)$\mapsto\cdots$\\ \hline
    L-7 & fixed-point $(2,5)$-cusp & Fig.~\ref{fig:boundA4} & (R-1)$\mapsto$(R-2)$\mapsto$(R-1)$\mapsto\cdots$\\ \hline
    L-8 & fixed-point $(3,4)$-cusp & Fig.~\ref{fig:boundT34} & (R-1)$\mapsto$(R-1)$\mapsto$(IR-1)$\mapsto$(M-1)$\mapsto$(M-2)$\mapsto\cdots$\\ \hline\hline
    \multicolumn{4}{|c|}{\textbf{On-axis and fixed-point strikethroughs}, Fig.~\ref{fig:st-on-the-axis}}\\ \hline
    L-9 & $1$-fold line tangency  & Fig.~\ref{fig:sCL} & (R-2)$\mapsto$(M-3)$\mapsto$(M-3)$\mapsto$(R-2)$\mapsto\cdots$\\ \hline
    L-10 & $2$-fold perpendicular tangency  & Fig.~\ref{fig:sPerp} & (M-3)$\mapsto$(M-2)$\mapsto$(IR-3)$\mapsto$(M-3)$\mapsto$(IR-3)$\mapsto$(M-2)$\mapsto\cdots$\\ \hline
    L-11 & on-axis double point  & Fig.~\ref{fig:sDOA} & 
(IR-3)$\mapsto$(IR-3)$\mapsto$(M-3)$\mapsto$(M-3)$\mapsto$(IR-3)$\mapsto$(IR-3)$\mapsto$(M-3)$\mapsto$(IR-3)$\mapsto$(IR-3)$\mapsto$(M-3)$\mapsto$(M-3)$\mapsto$(IR-3)$\mapsto$(IR-3)$\mapsto$(M-3)$\mapsto\cdots$\\\hline
L-12 & fixed-point $(2,3)$-cusp&Fig.~\ref{fig:sBC}&(M-1)$\mapsto$(IR-2)$\mapsto$(R-1)$\mapsto$(IR-2)$\mapsto$(M-1)$\mapsto$(R-1)$\mapsto\cdots$\\\hline
L-13 & fixed double point  & Fig.~\ref{fig:sBD2p} & 
    (IR-3)$\mapsto$(M-3)$\mapsto$(M-1)$\mapsto$(IR-3)$\mapsto$(M-1)$\mapsto$(IR-3)$\mapsto$(M-3)$\mapsto$(M-1)$\mapsto$(IR-3)$\mapsto$(M-1)$\mapsto\cdots$\\\hline\hline
    \multicolumn{4}{|c|}{\textbf{On-axis and fixed-point tangencies}, Figure~\ref{fig:tang-axis}}\\ \hline
    L-14 & $2$-fold line tangency & Fig.~\ref{fig:LOTR} & (R-2)$\mapsto$(R-2)$\mapsto\cdots$\\ \hline
    L-15 & $4$-fold perpendicular tangency  & Fig.~\ref{fig:perp5} & (M-2)$\mapsto$(IR-2)$\mapsto$(M-2)$\mapsto\cdots$\\ \hline
    L-16 & $1$-fold oblique tangency & Fig.~\ref{fig:dtan} & (M-3)$\mapsto$(IR-2)$\mapsto$(IR-2)$\mapsto$(M-3)$\mapsto\cdots$\\ \hline
    L-17 & fixed-point $(1, 2)$-fold tangency & Fig.~\ref{fig:cPerp} & (M-1)$\mapsto$(M-2)$\mapsto$(IR-3)$\mapsto$(M-1)$\mapsto$(IR-2)$\mapsto$(M-2)$\mapsto$(IR-2)$\mapsto\cdots$\\ \hline
    L-18 & fixed-point strikethrough of tangency & Fig.~\ref{fig:tBP} & (R-2)$\mapsto$(M-1)$\mapsto$(M-1)$\mapsto$(R-2)$\mapsto\cdots$\\ \hline
\hline
  \end{tabular}
  \caption{Local loops of equivariant Reidemeister moves. Each loop can be traversed in a positive (clockwise) or negative direction.}\label{tab:local_loops}
\end{table}

\begin{proof}
We first observe that the two paths $\phi_{s, 0}$ and $\phi_{s, 1}$ are obviously path-homotopic via the linear homotopy 
 \[
 \phi_{s, t}=(1-t)\phi_{s, 0} + t\phi_{s, 1}
 \]
defined for $s, t \in [0, 1]$. Note that each $ \phi_{s, t}$ is indeed an equivariant map from $\cS$ into $\R^2$. We aim to perturb $\smash{\phi_{s, t}}$ to be generic, but first we have to define the space $\cF^2$. Let $\cF^2_0$ be the subset of maps in $\cF$ with a coincidence of two codimension~1 singularities. For $i = 1, \dots, 18$, let $\cF^2_i$ correspond to item (L-$i$) in Table~\ref{tab:local_loops}. These are the $18$ codimension $2$ singularities listed previously. We then set
 \[
 \cF^2 = \cF^2_0 \cup \cF^2_1 \cup \cdots \cup \cF^2_{18} \quad \text{and} \quad  \wt{\cF}^3 = \cF \setminus(\cF^0\cup\cF^1\cup\cF^2).
 \]
By Theorem~\ref{thm:complete}, $\wt{\cF}^3$ has codimension $3$. An argument similar to that of Theorem~\ref{thm:complem} then shows that we can perturb $\smash{\phi_{s, t}}$ rel boundary so that it:
  \begin{itemize}
    \item misses $\wt{\cF}^3$;
    \item hits $\cF^2$ transversely at finitely many points;
    \item intersects $\cF^1$ non-transversely at finitely many points. 
  \end{itemize}
We first show that each point of intersection with $\cF^2$ leads to a loop of Reidemeister moves. Let $(s_0, t_0) \in [0, 1]^2$ be such a point and write $\phi =  \phi_{s_0, t_0}$. First suppose $\phi$ is not a coincidence, but rather lies in $\cF^2_j$ for $j > 0$, corresponding to one of our $18$ codimension $2$ singularities.

Let $D$ be a small disk around $(s_0, t_0)$. Then the pair $(\phi_{s, t}, D)$ as $(s, t)$ varies over $D$ is a deformation of $\phi$ with $2$-dimensional base. If the singularity is simple, transversality implies that this deformation is versal. That is, it is induced from the deformation described in previous sections. Let $\lambda_1$, $\lambda_2$ thus be coordinates as in Sections~\ref{sub:off2} through \ref{sub:tang2} and let $D'$ be an even smaller disk centered at $(s_0, t_0)$ whose boundary corresponds to our small loop of parameters. Then we can understand the behavior of the path $\phi_{s, t}$ as $t$ crosses $t_0$ as follows. The path $\phi_{s, t_0 - \epsilon}$ is homotopic to the path displayed in Figure~\ref{fig:circle} which has a segment that goes along the lower arc of $D'$. Likewise, the path $\phi_{s, t_0 + \epsilon}$ is homotopic to the path in Figure~\ref{fig:circle} that goes along the upper arc of $D'$. It is then clear that $\phi_{s, t_0 - \epsilon}$ and $\phi_{s, t_0 + \epsilon}$ are related by replacing a Reidemeister loop fragment with its complement. 
  
  \begin{figure}
    \includegraphics[width=10cm]{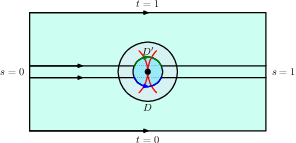}
    \caption{Proof of Theorem~\ref{thm:two_loops}. The two paths of Reidemeister moves differ by replacing a part of $\partial D'$
      with the complement going in the opposite direction. The two arcs passing through the center of the disk represent the discriminant locus.}\label{fig:circle}
  \end{figure}

If the singularity is not simple, then a few more details are needed.
In each non-simple case instead of specifying a versal deformation, we produced a family transverse to the defining set
of singularity (as the singularity is not simple Theorem~\ref{thm:inf_vers} does not apply). Each such family is a two-parameter family.
Whenever a family $\phi_{s,t}$ hits the non-simple singularity at $(s_0,t_0)$, we locally replace $\phi_{t,s}$ by the pre-defined
transverse family. 

We illustrate this rather standard procedure of a strikethrough of an on-axis double point. Suppose $s_0,t_0$ is the parameter
at which $\phi_{s_0,t_0}$ acquires this singularity.

For a strikethrough of an on-axis double point, we have a 3-parameter versal deformation: $\lambda_1,\lambda_2$, and $a$, where in all three cases $a$ is the slope of the strikethrough line and $\lambda_1=\lambda_2=0$ corresponds to $\phi_{s_0,t_0}$. By versality, for parameters $s,t$ close to $s_0,t_0$, the parameters $\lambda_1,\lambda_2$, and $a$ can be regarded as function of $(s,t)$. In fact, as the multijet extension
of $\phi_{s,t}$
is transverse to the defining equation $\lambda_1=\lambda_2=0$, we may assume that $(s,t)\mapsto (\lambda_1,\lambda_2)$
is locally invertible near $(s_0,t_0)$. That is,
the functions $\lambda_1,\lambda_2$ form a local coordinate
  system near $(0,0)$. Replace the deformation in a neighborhood of the origin by freezing $a$ (that is, letting $a$ be constant near the origin). This can be done by an explicit formula that changes $\phi^r_s$ locally.
  Namely, let the points in the domain which are involved in the singularity be $t_1, t_2, t_3 \in \Sfree$. Then $\phi$ near these points is given by
  \[
    (t-t_1,2.5(t-t_1)+\lambda_1),\ (t-t_2,0.4(t-t_2)),\ (t-t_3,a(\lambda_1,\lambda_2)(t-t_3)+\lambda_2),
  \]
  Write now the path
  \[
    (t-t_1,2.5(t-t_1)+\lambda_1),\ (t-t_2,0.4(t-t_2)),\ (t-t_3,a'(\lambda_1,\lambda_2)(t-t_3)+\lambda_2),
  \]
  where 
  \[a'(\lambda_1,\lambda_2)=a(0,0)\eta(\lambda_1,\lambda_2)+a(\lambda_1,\lambda_2)(1-\eta(\lambda_1,\lambda_2)),\]
  and $\eta$ is a cutoff function, equal to $1$ near $(0,0)$ and vanishing away from
  a small neighborhood of $(0,0)$.
  It is routine to check that this change modifies each diagram by an isotopy. Once the parameter $a$ has been frozen, the rest of the proof is as above.  A similar procedure works for other non-simple singularities.

Now suppose $\phi$ is a coincidence of two codimension~1 singularities $S_1$ and $S_2$. Let $D$ and $D'$ be as before. We have a versal deformation with parameters $\lambda_1,\lambda_2$ corresponding to $S_1$ and $S_2$; this has discriminant locus $\{\lambda_1\lambda_2=0\}$. Note that crossing the discriminant locus corresponds to doing the corresponding Reidemeister move. The resulting loop of Reidemeister moves consists of doing $S_1$, then $S_2$, and then undoing $S_1$, and then undoing $S_2$. Thus, $\phi_{s, t_0 - \epsilon}$ and $\phi_{s, t_0 + \epsilon}$ differ by replacing a fragment of this four-move loop with its complement. It is straightforward to check that this is some combination of \ref{item:repl} and \ref{item:birth}. For example, we can replace $S_1$-and-then-$S_2$ with $S_2$-and-then-$S_1$; this is precisely \ref{item:repl}. Or, we can replace $S_1$-and-then-$S_2$-and-then-$S_1$-reverse with $S_2$; this can be regarded as first doing \ref{item:repl} and then canceling a pair of moves via \ref{item:birth}.

Finally, we show that each tangency between $\phi_{s, t}$ and $\cF^1$ leads to a \ref{item:birth} move. As $\cF^1$ is a union of codimension~1 strata, it locally separates $\cF$. Lack of transversality means that as $t$ varies, the number of intersections of the path $\phi_{s, t}$ with $\cF^1$ changes by $2$. This corresponds to a birth or death of a pair of Reidemeister moves, as drawn in Figure~\ref{fig:non_transverse}.
\end{proof}
\begin{figure}
  \includegraphics[width=4cm]{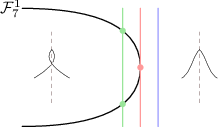}
  \caption{A schematic of a path which is non-transverse to $\cF^1$. The three vertical lines represent three paths in the function
    space. The one to the right does not intersect the $\cF^1_7$-stratum. The middle one is tangent to the stratum. Moving the path to the left
  creates two transverse intersection points with the $\cF^1_7$-stratum. Each of the two points corresponds to an (R-1) move.}\label{fig:non_transverse}
\end{figure}

Finally, we prove Theorem~\ref{thm:intro4} from the introduction.

\introfour*

\begin{proof}
  The proof follows the pattern used in Corollary~\ref{cor:perturb} and Corollary~\ref{cor:Reidemeister}.
  Define $\phi_{s,t}=\pi\circ\wt{\phi}_{s,t}$. Perturb $\phi_{s,t}$ to $\phi_{s,t,n}$, $n=1,\dots$ so that for each $n$
  $\phi_{s,t,n}$ is regular as a two-parameter family, that is,
  it misses codimension~3 singularities and has only finitely many codimension~2 singularities. We assume that $\phi_{s,t,n}$ converge,
  as $n\to\infty$ to $\phi_{s,t}$ in the $C^\infty$ topology. Set $\wt{\phi}_{s,t,n}=\wt{\phi}_{t,s}-\phi_{t,s}+\phi_{t,s,n}$. Then,
  $\wt{\phi}_{s,t,n}$ converges to $\wt{\phi}_{s,t}$ in the $C^\infty$ topology, in particular, in the $C^1$-norm. That is, for sufficiently
  large $n$, $\wt{\phi}_{s,t,n}$ is an embedding for all $s,t$.

  Apply Theorem~\ref{thm:two_loops} to $\phi_{s,t,n}$. Note that the proof of Theorem~\ref{thm:two_loops}
  actually constructs the homotopy $\phi_{s,t,n}$. Here, we insist that the homotopy come from the path $\wt{\phi}_{s,t,n}$. 

  As a consequence of Theorem~\ref{thm:two_loops}, we can connect $\phi_{0,t,n}$ and $\phi_{1,t,n}$ by a path of planar moves \ref{item:repl}--\ref{item:loop}. The data that $\phi_{s,t,n}$ is actually a projection of the map $\wt{\phi}_{s,t,n}$ allows us to lift the planar loops to
  loops of actual Reidemeister moves, because there is a consistent choice of which strand is a bridge and which strand is a tunnel.
\end{proof}
\section{Cobordisms of involutive links}\label{sec:4}

It is a standard fact that a cobordism between knots can be decomposed into the sequence of elementary cobordisms.
The corresponding statement in the equivariant setting is less well-known.  In this section we decompose equivariant cobordisms into ``elementary" cobordisms by studying equivariant Morse functions.

\subsection{Introduction}
Let $M$ be a closed $m$-dimensional smooth manifold with a smooth involution $\tau\colon M\to M$ such that
the fixed point set is an $n$-dimensional manifold 
\[M^\tau=\Fix(\tau).\]

The following result is standard; we include it for completeness.
\begin{lem}\label{lem:local_coor}
  Suppose $u_0\in M^\tau$. There are local coordinates $(x_1,\dots,x_n,y_1,\dots,y_{m-n})$ in a neighborhood of $u_0$ such that
  \begin{itemize}
    \item $u_0=(0,\dots,0)$;
    \item $\tau(x_1,\dots,x_n,y_1,\dots,y_{m-n})=(-x_1,\dots,-x_n,y_1,\dots,y_{m-n})$, in particular, $M^\tau$ is given by $x_1=\dots=x_{m-n}=0$.
  \end{itemize}
\end{lem}
\begin{proof}
  Consider the tangent space $V=T_{u_0}M$. As $u_0$ is a fixed point, the derivative
  of $\tau$ acts as an automorphism on $V$. Call this action $T\colon V\to V$. Then $T^2=\Id$ and so there is a basis
  $\wt{x}_1,\dots,\wt{x}_s,\wt{y}_1,\dots,\wt{y}_{m-s}$ of $V$ such that
  \[T(\wt{x}_1,\dots,\wt{x}_s,\wt{y}_1,\dots,\wt{y}_{m-s})=(-\wt{x}_1,\dots,-\wt{x}_s,\wt{y}_1,\dots,\wt{y}_{m-s}).\]
  We need to show that $s=n$.
  Choose an equivariant Riemannian metric on $M$. The map $\exp\colon V\to M$ is a local diffeomorphism such that
  $\exp(T\wt{x})=\tau\exp(x)$. That is, there are open subsets $U_V\subset V$ and $U_M\subset M$
  containing $0\in V$, respectively $u_0\in M$, such that $\exp\colon U_V\to U_M$ is a diffeomorphism. On replacing $U_V$
  by $U_V\cap T(U_V)$ and $U_M$ by $U_M\cap\tau(U_M)$ we may assume that $U_V$ and $U_M$ are $T$-invariant, respectively $\tau$-invariant.
  The maps $\exp$ being equivariant, takes $U_V\cap\Fix T$ to $U_M\cap \Fix\tau$. In particular, $n=s$.
  
  Let $x_1,\dots,x_n,y_1,\dots,y_{m-n}$ be the coordinates on $U_M$ such that
  \[\exp(\wt{x}_1,\dots,\wt{y}_{m-n})=(x_1,\dots,y_{m-n}).\]
  These coordinates satisfy both items of Lemma~\ref{lem:local_coor}.
\end{proof}
To have a uniform notation, we  sometimes write  the coordinates as $z_1,\dots,z_m$, with $z_i=x_i$ for $i\le n$ and $z_i=y_{i-s}$
for $i>n$. We set then
\begin{equation}\label{eq:def_delta_i}
  \delta_i=\begin{cases} -1 & i\le n \\ 1 & i>n.\end{cases}
\end{equation}
With this notation, we write the coordinates of $\tau$ as $(\tau_1,\dots,\tau_m)$, that is, $\tau(u)=(\tau_1(u),\dots,\tau_m(u))$. We have
\begin{equation}\label{eq:diff_tau}
  \frac{\partial\tau_j}{\partial z_i}=\begin{cases} 0 & i\neq j\\ \delta_i & i=j.\end{cases}
\end{equation}

\begin{defn}\label{def:invariant}
  A function $F\colon M\to\R$ is \emph{$\tau$-invariant} (in short: invariant), if $F(\tau u)=F(u)$ for all $u\in M$.
  A function $F\colon M\to\R$ is \emph{anti-invariant} if $F(\tau u)=-F(u)$ for all $u\in M$.

  If $\delta\in\{\pm1\}$ we also say that $F$ is $\delta$-invariant, if $F(\tau u)=\delta F(u)$.
\end{defn}
Clearly, an anti-invariant function vanishes on $M^\tau$.
\begin{lem}\label{lem:derivative}
  Suppose $F$ is invariant, $u_0\in M^\tau$ and $(x_1,\dots,y_{n-s})$ is a coordinate system in a neighborhood $U\subset M$ of $u_0$ 
  as in Lemma~\ref{lem:local_coor}. Then:
  \begin{itemize}
    \item[(a)] for any $i=1,\dots,n$, the derivative $\frac{\partial F}{\partial x_i}$ is anti-invariant;
    \item[(b)] for any $j=1,\dots,m-n$, the derivative $\frac{\partial F}{\partial y_j}$ is invariant.
  \end{itemize}
\end{lem}
\begin{proof}
  We calculate
  \[
    \left.\frac{\partial F}{\partial z_i}\right|_{u}=\left.\frac{\partial F\circ \tau}{\partial z_i}\right|_{u}=
	\left.\sum_{j=1}^m\frac{\partial F}{\partial z_j}\right|_{\tau u}\frac{\partial\tau_j}{\partial z_i}=
  \left.\delta_i\frac{\partial F}{\partial z_i}\right|_{\tau u}.\]
  Here, the first equality is invariance of $F$. The second is the chain rule and third is \eqref{eq:diff_tau}. Hereafter, we use the notation
  $\frac{\partial F}{\partial z_i}|_{\tau u}$ to emphasize the order of operations. This notation means `first differentiate, then substitute'.
  The notation for the opposite operation is $\frac{\partial (F\circ\tau)}{\partial z_i}|_u$.

  The displayed equation shows that the derivative is invariant if $\delta_i=1$, and anti-invariant if $\delta_i=-1$.
\end{proof}
\begin{corollary}\label{cor:vanish}
  If $F$ is invariant, then for any $u\in M^\tau$ we have $\frac{\partial F}{\partial x_j}(u)=0$.
\end{corollary}
\begin{proof}
  An anti-invariant function vanishes on the fixed point set.
\end{proof}
Our next result shows that the second derivative of $F$ has a block structure. This will have important implications on Morse functions.
\begin{lem}\label{lem:block}
  Let $F$ be an invariant function, $u_0\in M^\tau$ and $(x_1,\dots,y_{m-n})$ coordinates as in Lemma~\ref{lem:local_coor}
  defined in an open set $U\ni u_0$. For any $u\in U\cap M^\tau$, the second derivative $D^2F(u)$ has a block structure
  $D^2_xF(u)\oplus D^2_yF(u)$, where $D^2_xF(u)$ is the matrix of derivatives over the $x$-variables, $D^2_yF(u)$ is
  the matrix of derivatives over the $y$-variables.
\end{lem}
\begin{proof}
  The statement is equivalent to saying that $\frac{\partial^2 F}{\partial x_i\partial y_j}(u)=0$ for $u\in U\cap M^\tau$, $i=1,\dots,n$,
  $j=1,\dots,m-n$.

  Note that by Lemma~\ref{lem:derivative}, the derivative $\frac{\partial F}{\partial y_i}$ is $\tau$-invariant. Therefore,
  the derivative $\frac{\partial}{\partial y_j}\frac{\partial F}{\partial x_i}$ is
  anti-invariant, so by Corollary~\ref{cor:vanish}, it vanishes on $U\cap M^\tau$.
\end{proof}

\subsection{Ambient equivariant Morse theory}
Let $M$ be a closed smooth $m$-dimensional
manifold endowed with a smooth involution $\tau\colon M\to M$, preserving orientation and such that $M^\tau$
is a smooth $n$-dimensional manifold.
\begin{defn}
  A smooth, $\tau$-invariant function $F\colon M\to\R$ is called an \emph{equivariant Morse function}, if the determinant $\det D^2F(u)$ is
  nonzero at any point $x$ such that $DF(u)=0$.
\end{defn}
\begin{lem}\label{lem:restrict}
  Suppose $F$ is an equivariant Morse function. Set $f=F|_{M^\tau}$. Then $f$ is a Morse function on $M^\tau$.
\end{lem}
\begin{proof}
  Take $u\in M^\tau$, choose coordinates $(x_1,\dots,y_{m-n})$ as in Lemma~\ref{lem:local_coor}.
  Suppose $u$ is a critical point of~$f$. This amounts to saying that $\frac{\partial f}{\partial y_i}(u)=0$
  for all $i=1,\dots,m-n$. As $y_1,\dots,y_{m-n}$ are coordinates on $M^\tau$, this means that $\frac{\partial F}{\partial y_i}(u)=0$
  for all $i=1,\dots,m-n$. By Corollary~\ref{cor:vanish}, $\frac{\partial F}{\partial x_j}(u)=0$ for $j=1,\dots,n$. This means 
  that $u$ is a critical point of $F$.

  By Lemma~\ref{lem:block}, the second derivative $D^2F(u)$ has a block structure $D^2F(u)=D^2_xF(u)\oplus D^2_yF(u)$.
  The Morse condition implies that $\det D^2F(u)\neq 0$, but $\det D^2F(u)=\det D^2_xF(u)\cdot\det D^2_yF(u)$.
  Hence, $\det D^2_yF(u)\neq 0$, but $D^2_yF(u)=D^2f(u)$. Hence, $u$ is a Morse critical point.
\end{proof}
In the proof of Lemma~\ref{lem:restrict}, we have proved the following result, which we state explicitly for future reference.
\begin{corollary}\label{lem:critical}
  A critical point $u\in M^\tau$ of $f$ is also a critical point of $F$.
\end{corollary}

We will now continue our study of equivariant Morse functions by recalling the equivariant Morse Lemma, proved
originally by Wassermann.
\begin{thm}[Equivariant Morse Lemma]\label{thm:wass_mors}
  Let $F$ be an equivariant Morse function. If $u_0\in M^\tau$ is a critical point, then there are coordinates $x_1,\dots,x_n,y_1,\dots,y_{m-n}$
  on $M$ near $x_0$, such that $u_0=(0,0,\dots,0)$, $\tau(x_1,\dots,y_{m-n})=(-x_1,\dots,-x_n,y_1,\dots,y_{m-n})$ (in particular,
  $x_1,\dots,x_n$ are coordinates normal to $M^\tau$, $y_1,\dots,y_{m-n}$ are coordinates on $M^\tau$. Furthermore,
  in these coordinates
  \[F(x_1,\dots,y_{n-m})=F(u_0)-x_1^2-\dots-x_k^2+x_{k+1}^2+\dots+x_n^2-y_1^2-\dots-y_\ell^2+y_{\ell+1}^2+\dots+y_{m-n}^2\]
  for some integers $k$ and $\ell$.
\end{thm}
\begin{proof}
  The statement is proved in \cite[Lemma 4.1]{Wassermann}.
\end{proof}
\begin{defn}
  The pair $(k,\ell)$ is called the \emph{multi-index} of a critical point. The number $k$ is the \emph{anti-invariant index},
  while $\ell$ is called the \emph{invariant index}.
\end{defn}

The next result is also proved by Wassermann \cite[Density Lemma 4.8]{Wassermann}:
\begin{thm}\label{thm:density_of_Morse}
  Let $\cF$ be the set of functions $F\colon M\to\R$ that are $\tau$-invariant and of class $C^N$ for $N\ge 3$. Then, the subspace of equivariant Morse
  functions is open-dense.
\end{thm}
\subsection{Embedded equivariant Morse theory}\label{sub:EEM_theory}
Suppose $\Omega$ is a closed, $m+k$-dimensional manifold, $k>0$, 
and $M\subset\Omega$ is a closed submanifold of dimension $m$.
We assume that there is an action $\tau\colon\Omega\to\Omega$ with fixed point set $\Omega^\tau$, which is  an $n$-dimensional manifold.
We assume that $M\cap \Omega^\tau$ is a manifold of dimension $r<n$, but we do not insist that the intersection of $M$ and $\Omega^\tau$ be transverse.
The following notion generalizes the definition of \cite{BP}.

\begin{defn}\label{def:EEM}
  A smooth invariant function $F\colon\Omega\to\R$ is called a \emph{embedded equivariant Morse}, in short EEM, 
  if it has the following properties
  \begin{enumerate}[label=(EEM-\arabic*)]
    \item (Morse on $\Omega)$ $F$ is a Morse function on $\Omega$;\label{item:M_morse}
    \item (Morse on $M$) the restriction $f=F|_M$ is Morse on $M$.\label{item:M_rest}
    \item (genericity of position of critical points) all critical points of $F$ are on $\Omega\setminus M$;\label{item:M_gen}
  \end{enumerate}
\end{defn}
By convention, we will also refer to critical points of $F_M$ occasionally as critical points of $F$, when it will not cause confusion.

Suppose $u_0$ is a critical point of $F$ or of $f=F|_M$. 
To study the local behavior of $F$ near $u_0$, we consider the following cases. We recall that $M^\tau$ and $\Omega^\tau$ denote
\begin{itemize}
  \item if $u_0\notin M$ and $u_0\notin \Omega^\tau$, we use the standard Morse Lemma to find local behavior of $F$. Note, that the point $\tau u_0$
    is also a critical point of $F$;
  \item if $u_0\notin M$, but $u_0\in \Omega^\tau$, we use the Equivariant Morse Lemma (Theorem~\ref{thm:wass_mors} above);
  \item if $u_0\in M$, but $u_0\notin \Omega^\tau$, we use the Embedded Morse Lemma \cite[Lemma 2.17]{BP}. In this case, $\tau u_0$ is another
    critical point;
  \item if $u_0\in M\cap \Omega^\tau$, we need an additional tool, that is, Equivariant Embedded Morse Lemma. The statement and proof follow promptly.
\end{itemize}
For a critical point $u_0$ of $f$ on $M\cap \Omega^\tau$, we have the following variant of Morse lemma.
\begin{thm}[Equivariant Embedded Morse Lemma]\label{thm:eem}
  Suppose $u_0\in M\cap \Omega^\tau$ is a critical point of $f$.
  There exists coordinates $x_1,\dots,x_r,y_1,\dots,y_{m-r},v_1,\dots,v_{n-r},w_1,\dots,w_{s}$ (with $n-r+s=k$)
  on $\Omega$ near $u_0$ (i.e. $u_0=(0,\dots,0)$) such that
  \begin{enumerate}[label=(EML-\arabic*)]
    \item $\tau$ is linear, that is, \label{item:EM_linear}
    \[\tau(x_1,\dots,w_s)=(-x_1,\dots,-x_r,y_1,\dots,y_{m-r},-v_1,\dots,-v_{n-r},w_1,\dots,w_s).\]
  \item $M$ is given by $\{v_1=\dots=v_{n-r}=w_1=\dots=w_s=0\}$;\label{item:EM_M}
  \item $\Omega^\tau$ is given by $\{x_1=\dots=x_r=v_1=\dots=v_{n-r}=0\}$;\label{item:EM_N}
  \item the function $F$ in these coordinates is given by:\label{item:EM_F}
    \[F(x_1,\dots,w_s)=F(u_0)-x_1^2-\dots-x_a^2+x_{a+1}^2+\dots+x_r^2-y_1^2-\dots-y_b^2+y_{b+1}^2+\dots+y_{m-r}^2+w_1.\]
  \end{enumerate}
\end{thm}
\begin{proof}
  We first construct a local coordinate system refining the one of Lemma~\ref{lem:local_coor}. To this end, we
  need to impose conditions on the Riemannian metric used to define the $\exp$ map.

  To this end, choose a neighborhood $U$ of $u_0$ in $\Omega$ such that $U$ is equivariantly diffeomorphic to $U_M\times U_T$,
  where $U_M\subset M$ and $U_T$ is a subset of $\R^k$. Such $U$ exists by the equivariant tubular neighborhood theorem,
  see \cite[Theorem VI.2.2]{Bredon}.
  We can choose coordinates on $U_M$ and on $U_T$ separately using Lemma~\ref{lem:local_coor}. We call the coordinates on
  $M$ $x_1,\dots,x_r,y_1,\dots,y_{m-r}$ and the coordinates on $U_T$ are $v_1,\dots,v_{n-r},w_1,\dots,w_s$. Conditions~\ref{item:EM_linear},
  \ref{item:EM_M}, and \ref{item:EM_N} are automatically satisfied.

  It remains to prove~\ref{item:EM_F}. We will follow the proof of \cite[Lemma 2.17]{BP} with adaptations to the equivariant case. 
  By Theorem~\ref{thm:wass_mors}, we can make 
  an equivariant transformation to variables $x_1,\dots,y_{n-r}$ in such a way that $f=F|_M$ has form
  \[f(x_1,\dots,y_{n-r})=f(x_0)-x_1^2-\dots-x_a^2+x_{a+1}^2+\dots+x_r^2-y_1^2-\dots-y_b^2+y_{b+1}^2+\dots+y_{n-r}^2.\]
  Next, the point $u_0$ is not a critical point of $F$, thus, there exists either $i=1,\dots,m-r$ or $j=1,\dots,s$
  such that $\frac{\partial F}{\partial v_i}(x_0)\neq 0$ or $\frac{\partial F}{\partial w_j}(x_0)\neq 0$. The first option contradicts
  Corollary~\ref{cor:vanish}. Hence, $\frac{\partial F}{\partial v_i}(x_0)\neq 0$ for some $i$. On reindexing
  variables, we may  assume that $\frac{\partial F}{\partial w_1}(x_0)\neq 0$.
  Consider the map
  \[\Phi(x_1,\dots,y_{n-r},v_1,v_2,\dots,w_1,\dots,w_s)=(x_1,\dots,y_{n-r},v_1,v_2,\dots,w_1',\dots,w_s),\]
  where
  \[w_1=F(x_1,\dots,w_s)-(-x_1^2-\dots+y_{n-r}^2)-F(u_0).\]
  The map $\Phi$ is equivariant. Moreover, by the implicit function theorem, $\Phi$ is a local diffeomorphism. Furthermore,
  $\Phi$ preserves $M$. In the
  coordinates $(x_1,\dots,y_{n-r},v_1,v_2,\dots,w_1',w_2,\dots,w_s)$, $F$ has the form~\ref{item:EM_F}.
\end{proof}

We now pass to proving openness and density of the space of Equivariant Embedded Morse functions. The result relies on
the density for ambient equivariant Morse functions, Theorem~\ref{thm:density_of_Morse}, but some effort has to be made.

Let $N\ge 2$ be an integer. We know that if $F,G\in C^N(\Omega;\R)$,
then $FG\in C^N(\Omega;\R)$. There is a constant $\Pi_N$, depending only on $\dim\Omega$ and $N$, and on conventions 
used to define the $C^N$-norms, such that
$||FG||_{C^N}\le \Pi_N||F||_{C^N}||G||_{C^N}$.
 
\begin{thm}\label{thm:density}
  Let $C^N_\tau(\Omega)$ be the space of $C^N$-smooth functions $F\colon \Omega\to\R$ such that $F(\tau u)=F(u)$ for all $u$. Let $\cM$ be the
  subspace of functions satisfying conditions \ref{item:M_morse} -- \ref{item:M_gen}. Then $\cM$ is open-dense in $C^N_\tau(\Omega)$.
\end{thm}
\begin{proof}
  Openness of each condition is straightforward. Fix an invariant metric on $\Omega$. Let $F$ be an $\tau$-invariant function
  on $\Omega$ with bounded $C^N$-norm. Let $\varepsilon>0$. To prove Theorem~\ref{thm:density} we need several lemmas.

  \begin{lem}\label{lem:step1}
    Let $\cM_1\subset C^N_\tau(\Omega)$ be the subspace of functions satisfying~\ref{item:M_morse}. Then, $\cM_1$ is open-dense.
  \end{lem}
  \begin{proof}
    This is an immediate consequence of 
    Theorem~\ref{thm:density_of_Morse}.
  \end{proof}

  \begin{lem}\label{lem:step2}
    Let $\cM_2\subset C^N_\tau(\Omega)$ be the subspace of functions whose restriction to $M$ is Morse (that is, which satisfy \ref{item:M_rest}). Then $\cM_2$ is open-dense.
  \end{lem}
  \begin{proof} 
Openness is clear, we focus on density.
    Suppose $F\in C^N_\tau(\Omega)$.
 Let $f=F|_M$. Choose an invariant tubular neighborhood $U$ of $M$ in $\Omega$ (which exists by \cite[Theorem VI.2.2]{Bredon}). Fix an equivariant projection $\pi\colon \ol{U}\to M$. Let $\kappa_N$ be the norm of the map $P_\pi\colon C^N(M)\to C^N(\ol{U})$
 given by $P_\pi(g)=g\circ\pi$.

 Let $U'$ be an invariant tubular neighborhood of $M$ contained in $U$.
 Let $\phi\colon\Omega\to[0,1]$ be a $C^N$-smooth
 bump function supported on $U$ and equal to $1$ on $U'$.
 On replacing $\phi$ by $\frac12(\phi+\phi\circ\tau)$ we may assume that $\phi$ is $\tau$-invariant. Let $\lambda$
 be the $C^N$-norm of $\phi$.

 As equivariant Morse functions on $M$ are dense, there exists a function $f'\colon M\to\R$ such
 that $||f-f'||_{C^N(M)}<\Pi_N^{-1}\kappa_N^{-1}\lambda^{-1}\varepsilon$.

 Define
 \[F'=F+\phi(\pi^{-1}(f-f')).\]
 Then $F'|_M=f'$ is Morse on $M$. By construction, $F'$ is invariant and $C^N$-smooth. Moreover, 
 \begin{multline*}
   ||F-F'||_{C^N(\Omega)}=||\phi(\pi^{-1}(f-f'))||_{C^N(\ol{U})}\le \Pi_N||\phi||_{C^N(\ol{U})}||\pi^{-1}(f-f')||_{C^N(\ol{U})}\le \\
   \Pi_N\lambda_N ||\pi^{-1}(f-f')||_{C^N(\ol{U})}\le \Pi_N\kappa_N\lambda_N||f-f'||_{C^N(M)}<\varepsilon.
 \end{multline*}
 That is, near each function $F\in C^N_\tau(M)$, we can find a function $F'$, arbitrarily close to $F$ in the $C^N$-norm,
 such that $F'|_M$ is Morse. 
\end{proof}
\begin{lem}\label{lem:step3}
  The set $\cM_3\subset C^N_\tau(\Omega)$ of functions satisfying~\ref{item:M_gen} is open-dense.
\end{lem}
\begin{proof}
  As $M$ is closed, a function without critical points on $M$ has in its $C^N$-neighborhood only
  functions without critical points on $M$. This shows openness.

  For density, suppose $F\in C^N_\tau(\Omega)$. Let $\varepsilon>0$ and choose a function $F_1\in\cM_1\cap\cM_2$ such that $||F-F_1||<\varepsilon/2$. This is possibly by Lemmata~\ref{lem:step1} and~\ref{lem:step2}. We aim
 to push all critical points of $F_1$ off the submanifold $M$. If $F_1$ has no such critical points, we have already found a function
 in $C^N_\tau(\Omega)$ satisfying \ref{item:M_gen}. If not, as critical points of $F_1$ are Morse, they are isolated. In particular,
 there is a finite number $p$ of critical points of $F_1$ on $M$. Choose such a critical point $u$.
 Let $U$ be an equivariant neighborhood of $u$ on which $\tau$ acts linearly. We may assume that there are coordinates
 $x_1,\dots,w_s$ on $U$ satisfying \ref{item:EM_linear}, \ref{item:EM_M}, and~\ref{item:EM_N}. Shrink $U$ to ensure that $F|_M$ has no
 critical points on $U$ other than $u$. 

 We assume that such neighborhoods corresponding to different
 critical points are pairwise disjoint. Let $U'$ be a neighborhood of $u$ such that $\ol{U'}\subset U$. Choose a bump function $\phi\colon\Omega\to[0,1]$ that is supported on $U$ and equal to $1$ on $U'$. Let $\lambda_N$ be the $C^N$-norm of $\phi$.

 Given that $n>r$, for $\eta>0$, we define a function
 \begin{equation}\label{eq:replacement_m3}F'_1=F_1+\phi\varepsilon\eta w_1.\end{equation}
 \begin{lem}\label{lem:on_m3}
   If $\eta$ is sufficiently small, then:
   \begin{itemize}
     \item $||F'_1-F_1||_{C^N}<\varepsilon/2p$;
     \item $F'_1$ has no critical points on $U\cap M$.
   \end{itemize}
 \end{lem}
 \begin{proof}
   The first item is straightforward, we need to show that $||\eta\phi w_1||_{C^N(U)}<1/p$. This can be done by estimating the $C^N$-norm
   of $w_1$ (which is finite) and using the finiteness of the norm of $\phi$.

   For the second part, we note that $F'_1$ has no critical points on $U'\cap M$ regardless of $\eta$ (as long as it is positive). 
   In fact, on $U'$ we have $\phi\equiv 1$, so $F'_1=F_1+\varepsilon\eta w_1$. As $u$ is an isolated critical point of $F_1|_M$,
   at each point of $U'|_M$ except at $u$, the derivative of $F_1|_M$ is not zero. The derivative of $F_1'$ with respect to $w_1$
   is not zero at $u$ (it can be zero at some points of $U'|_M$). That is, $F'_1$ has no critical points on $U'|_M$.

   On $(U\setminus U')|_M$ the function $F_1$ has no critical points, hence its derivative is bounded from below by a constant. Denote it $\alpha>0$. Next, let $\beta$ be the upper bound on the derivative of $\phi\cdot w_1$ on $\ol{U\setminus U'}|_M$.
   If $\varepsilon\eta<\frac{\alpha}{\beta}$, then $F_1+\varepsilon\eta\phi w_1$ has non-vanishing derivative
   on $U\setminus U'$.
 \end{proof}
 Given Lemma~\ref{lem:on_m3}, we continue with the proof of Lemma~\ref{lem:step3}.
 The replacement \eqref{eq:replacement_m3} can be done near all critical points independently.
 The resulting function $\wt{F}\in C^N_\tau(\Omega)$
 satisfies $||\wt{F}-F||_{C^N(\Omega)}<\varepsilon$ and $\wt{F}$ has no critical points on $M$, that is, it satisfies~\ref{item:M_gen}.
 This means, that if $\cM_3$ denotes the subset of functions satisfying~\ref{item:M_gen}, then it is dense (openness is clear).
\end{proof}

 \emph{Finishing the proof of Theorem~\ref{thm:density}.} 
 We have defined three subspaces $\cM_1$, $\cM_2$, and $\cM_3$ of functions of $C^N_\tau(\Omega)$ satisfying respectively \ref{item:M_morse},
 \ref{item:M_rest}, and~\ref{item:M_gen}. We have showed that each of the three spaces is open-dense. Their intersection consists
 of Morse functions.
 \end{proof}

 \subsection{Equivariant cobordisms of involutive links}\label{sub:applications}
In this section we apply equivariant embedded Morse theory to study cobordisms of involutive links.
Let $\Omega=S^3\times[0,1]$ and $\tau\colon\Omega\to\Omega$ be a smooth involution fixing $S^3\times\{t\}$
for all $t\in[0,1]$. That is, we choose the standard involution of $S^3\times[0,1]$, by rotation about a fixed axis in $S^3$.
We will often use the fact that $\Omega^\tau=\Fix\tau$ is a 2-dimensional manifold whose intersection with $S^3\times\{0\}$ and $S^3\times\{1\}$ is a circle.

Consider an equivariant link cobordism $M\subset S^3\times[0,1]$, that is, a smooth oriented two-dimensional surface fixed by $\tau$. 
Throughout Section~\ref{sub:applications}, we make the following assumption.
\begin{con}\label{con:no_iso}
The action of $\tau\colon M\to M$ has no isolated fixed points, and $\tau$ is not the identity on any connected component of $M$.
\end{con}
This assumption is justified by the following result.
\begin{lemma}\label{lem:no_isolated}
  Suppose $z_0$ is an isolated fixed point of the involution $\tau\colon M\to M$. Then the action of $\tau$ on the connected
  component $M_0$ of $M$ containing $z_0$ is orientation-preserving.
\end{lemma}
\begin{proof}
  Let $N=M^\tau$ be the fixed set. By the assumption
  $z_0\in N$.
  As $\tau(z_0)=z_0$, $\tau$ acts on $T_{z_0}M$ linearly. Since $\tau$ acts smoothly, $(T_{z_0}M)^\tau$ is the tangent space $T_{z_0}N$.
  The point $z_0$ is an isolated fixed point, therefore $\tau\colon T_{z_0}M\to T_{z_0}M$ is a linear involution of $\R^2$ with
  only $0$ as the fixed point. Hence, $\tau$ is the symmetry about $0$, so it preserves the orientation of $T_{z_0}M$. Hence, it
  preserves the orientation of the whole connected component of $M$ containing $z_0$.
\end{proof}
\begin{remark}
  The case when $\tau$ has isolated fixed points on $M$ is discussed in Section~\ref{sub:cob_isolated} below.
\end{remark}

 Perturb $F$ to be an embedded equivariant Morse function on $\Omega$ with respect to $M$. As the original function did not have critical points on $\Omega$, its perturbation can be chosen not to have such points either.
In the notation of Subsection~\ref{sub:EEM_theory} we have $m=2$, $k=2$, $n=2$ and $r=1$.
The embedded Equivariant Morse Lemma (Theorem~\ref{thm:eem}) gives the following possibilities for critical points on $M$.
\begin{itemize}
  \item Critical points outside $M\cap \Omega^\tau$. The local form is $\pm x_1^2\pm x_2^2+w_1$. Depending on signs we have:
    \begin{itemize}
      \item A pair of births with the local form $x_1^2+x_2^2+w_1$;
      \item A pair of saddle points with the local form $-x_1^2+x_2^2+w_1$;
      \item A pair of deaths with the local form $-x_1^2-x_2^2+w_1$.
    \end{itemize}
  \item A critical point in $M\cap \Omega^\tau$ has local form $\pm x_1^2\pm y_1^2+w_1$. Depending on signs, we distinguish:
    \begin{itemize}
      \item An invariant birth with local form $x_1^2+y_1^2+w_1$, creating an invariant circle on which $\tau$ acts by reflection;
      \item Invariant saddles with local forms either $-x_1^2+y_1^2+w_1$ or $x_1^2-y_1^2+w_1$. Replacing $F$ by $-F$ and $w_1$ by $-w_1$
	 swaps the two local forms;
      \item An invariant death with local form $-x_1^2-y_1^2+w_1$.
    \end{itemize}
\end{itemize}
It is worth recalling that $x_i,y_i$ are the coordinates on $M$, while $v_j,w_j$ are the coordinates normal to $M$. The $\tau$-action
changes the sign of the $x_i,v_j$ coordinates, and fixes the $y_i,w_j$ coordinates.

The change of the topology of the link $M\cap F^{-1}(t)$ near the critical point can be presented as a change of the link diagram:
a birth should create an unknotted circle or a pair of unknotted circles etc. In order to describe this change rigorously, we need to find
a projection of a link that is compatible with the change of topology, for example, the circle created during the birth should not be mapped to a segment. In Theorem~\ref{thm:526} below, we prove that this is indeed the case. In particular, we show that under some precise genericity conditions, any critical point induces an equivariant Morse move, which we are now going to define.
\begin{figure}
  \begin{tikzpicture}
    \draw(-5,0) node [rectangle, minimum width=2cm,fill=yellow!20] {\includegraphics[height=1.5cm]{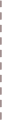}};
    \draw(-2,0) node [rectangle, minimum width=2cm,fill=yellow!20] {\includegraphics[height=1.5cm]{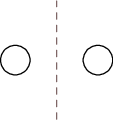}};
    \draw(2,0) node [rectangle, minimum width=2cm,fill=yellow!20] {\includegraphics[height=1.5cm]{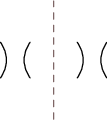}};
    \draw(5,0) node [rectangle, minimum width=2cm,fill=yellow!20] {\includegraphics[height=1.5cm]{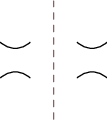}};
    \draw[->, decorate, decoration={snake,amplitude=0.5pt},thin] (-5,0.9) .. controls ++ (45:1) and ++ (135:1) .. node [above, scale=0.8] {birth} (-2,0.9);
    \draw[->, decorate, decoration={snake,amplitude=0.5pt},thin] (-2,-0.9) .. controls ++ (225:1) and ++ (315:1) .. node [above, scale=0.8] {death} (-5,-0.9);
    \draw[<->, decorate, decoration={snake,amplitude=0.5pt},thin] (2,0.9) .. controls ++ (45:1) and ++ (135:1) .. node [above, scale=0.8] {saddle} (5,0.9);
  \end{tikzpicture}
  \caption{Critical points outside the symmetry axis}\label{fig:outside_crits}
\end{figure}
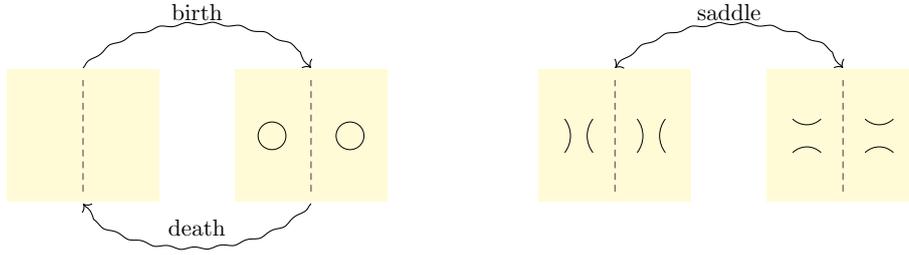

\begin{figure}
  \begin{tikzpicture}
    \draw(-5,0) node [rectangle, minimum width=2cm,fill=yellow!20] {\includegraphics[height=1.5cm]{pics/crits-7.eps}};
    \draw(-2,0) node [rectangle, minimum width=2cm,fill=yellow!20] {\includegraphics[height=1.5cm]{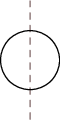}};
    \draw(2,0) node [rectangle, minimum width=2cm,fill=yellow!20] {\includegraphics[height=1.5cm]{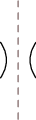}};
    \draw(5,0) node [rectangle, minimum width=2cm,fill=yellow!20] {\includegraphics[height=1.5cm]{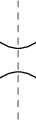}};
    \draw[->, decorate, decoration={snake,amplitude=0.5pt},thin] (-5,0.9) .. controls ++ (45:1) and ++ (135:1) .. node [above, scale=0.8] {birth} (-2,0.9);
    \draw[->, decorate, decoration={snake,amplitude=0.5pt},thin] (-2,-0.9) .. controls ++ (225:1) and ++ (315:1) .. node [above, scale=0.8] {death} (-5,-0.9);
    \draw[<->, decorate, decoration={snake,amplitude=0.5pt},thin] (2,0.9) .. controls ++ (45:1) and ++ (135:1) .. node [above, scale=0.8] {saddle} (5,0.9);
  \end{tikzpicture}
  \caption{Critical points on the symmetry axis}\label{fig:inside_crits}
\end{figure}
\begin{defn}
  A transformation of an equivariant link depicted on Figure~\ref{fig:outside_crits} or Figure~\ref{fig:inside_crits} is called an \emph{equivariant Morse move}. The moves on Figure~\ref{fig:outside_crits} are called \emph{outer} or \emph{off-axis} moves, while the moves
  on Figure~\ref{fig:inside_crits} are called \emph{inner} or \emph{on-axis} moves.
\end{defn}

Our study is now divided into steps. First, we specify local coordinates. This step does not require the function to be Morse. Next,
in Subsection~\ref{sub:local_comp}, we define a local compression of $M$, and show that a Morse critical point. In Subsection~\ref{sub:movies},
we prove Theorem~\ref{thm:intro2} about the movies associated with a cobordism.  Finally, in Subsection~\ref{sub:good_exists}, we show
that good projections (see Definition~\ref{con:generic}) exist.

\subsubsection{Local coordinates and good projections}\label{sub:local_coors}
We now start the discussion of local changes of the diagram after passing through a critical point. This analysis will culminate in
Theorem~\ref{thm:morse_genericity}, which is the key step in proving Theorem~\ref{thm:526}.
We begin with choosing coordinates.
Suppose a critical point $u_0\in M$ 
of $F|_M$ occurs at a level set $t_0$. 
Set $\varepsilon>0$, which we will eventually adjust to be sufficiently small. 
Choose an arc $\omega\in F^{-1}[t_0-\varepsilon,t_0+\varepsilon]\cap \Omega^\tau$, missing $M$ and transverse to the
level sets of $F$. Such arc is constructed by picking first a point in $\Omega^\tau\setminus M$ at the level set $F^{-1}(t_0)$ and
extending it in a normal direction to the level set. We think
of this arc as a choice of the point at infinity for the level set $F$. 

We let $S_t\cong S^3$ denote $F^{-1}(t)$, and $R_t=S_t\setminus\omega$. Then, $R_t$ is diffeomorphic to $\R^3$ for all $t\in[t_0-\varepsilon,t_0+\varepsilon]$. Moreover, these diffeomorphisms can be chosen by a fixed stereographic projection from $\omega\cap S_t$ to $\R^3$. Hence,
the identification of $R_t$ with $\R^3$ depends smoothly on $t$.
Set $X_t=M\cap R_t$, with $X_{t_0}$ being the singular level set. Near $u_0$, we choose coordinates $(x_1,y_1)$ on $M$. 
If $u_0\in M^\tau$, we require that $\tau(x_1,y_1)=(-x_1,y_1)$.
The next result shows a local form of $f$. Note that we \emph{do not} assume that $u_0$ is a Morse critical point of $F|_{M}$.
\begin{lemma}\label{lem:fix_coor}
  The coordinate system on $M$ can be completed to a local coordinate system $(x_1,y_1,v_1,w_1)$ near $u_0$ on $\Omega$
  in such a way that $F$ in these coordinates is equal to $F(u_0)+f(x_1,y_1)+w_1$, where in the local coordinates,
  $f=F|_M$. Moreover, if $u_0\in M^{\tau}$,
  then $\tau(v_1,w_1)=(-v_1,w_1)$.
\end{lemma}
\begin{proof}[Proof of Lemma~\ref{lem:fix_coor}]
  The proof is similar for cases $u_0\in M^{\tau}$ and $u_0\notin M^{\tau}$. We give the proof in the case $u_0\in M^{\tau}$, which
  Extend $(x_1,y_1)$ to equivariant coordinates $(x_1,y_1,v_1,w_1)$ near $u_0$ on $\Omega$ in such a way that $M=\{v_1=w_1=0\}$,
  and $\tau(x_1,y_1,v_1,w_1)=(-x_1,y_1,-v_1,w_1)$.
By Lemma~\ref{lem:derivative} $\frac{\partial F}{\partial v_1}(u_0)=0$, and as $F|_M$ has critical point at $u_0$, we have
$\frac{\partial F}{\partial x_1}(u_0)=\frac{\partial F}{\partial y_1}(u_0)=0$. As $DF(u_0)\neq 0$, 
we must have $\frac{\partial F}{\partial w_1}(u_0)\neq 0$.
Then, the change of variables $(x_1,y_1,u_1,w_1)\mapsto
(x_1,y_1,u_1,F(x_1,y_1,u_1,w_1)-F(x_1,y_1,0,0))$ is a local equivariant diffeomorphism. Applying this, we arrive at the situation where
\begin{equation}\label{eq:F_non_morse}
  F(x_1,y_1,v_1,w_1)=w_1+f(x_1,y_1),
\end{equation}
where $f(x_1,y_1):=F(x_1,y_1,0,0)$ is smooth and $\Z_2$-invariant. 
\end{proof}
\begin{figure}
  \begin{tikzpicture}
    \draw(-1,0) node [rectangle, minimum width=2cm] {\includegraphics[height=4cm]{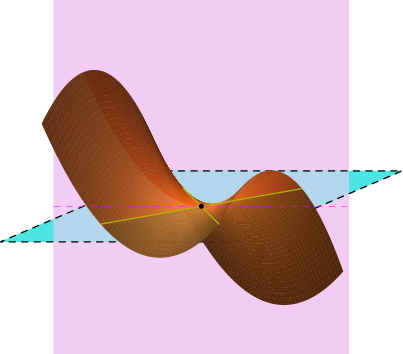}};
    \definecolor{mycolor}{rgb}{0.3,0.9,0.9}
    \fill[mycolor] (3,-1.5) -- (6,-1.5) -- (6,1.5) -- (3,1.5) -- cycle;
    \draw (3.5,-1) -- (5.5,1);
    \draw (3.5,1) -- (5.5,-1);
    \draw (-2.3,1.8) -- ++ (-0.5,0.5) -- node[scale=0.7,above,midway] {symmetry plane $\Omega^\tau$} ++ (-2,0);
    \draw (-2.9,-0.65) -- ++ (-0.5,0.5) -- node[scale=0.7,above,midway] {singular level set $R=R_{t_0}$}  ++ (-2.5,0);
    \fill[yellow] (-1,-0.32) circle (0.03);
    \draw[->] (-1,-0.32) -- node[above,midway,scale=0.5] {$y_1$} ++(0.5,0);
    \draw[->] (-1,-0.32) -- node[above,left,scale=0.5] {$x_1$} ++(-0.4,-0.4);
    \draw[->] (-1,-0.32) -- node[above,left,scale=0.5] {$w_1$} ++(0,0.5);
    \draw (-2.4,1.3) -- ++ (-0.5,0.5) -- node [scale=0.7,above,midway] {$M$} ++ (-1.3,0);
    \draw[densely dotted] (3,0) -- (6,0); \draw (5.6,0) -- ++ (0.5,0.5) -- node[scale=0.7,above,midway] {symmetry axis} ++ (2,0);
    \draw[->] (4.5,0) -- node[above,midway,scale=0.5] {$y_1$} ++ (0.5,0);
    \draw[->] (4.5,0) -- node[midway,left,scale=0.5] {$x_1$} ++ (0,0.5);
    \fill[yellow] (-1,-0.32) circle (0.03);
  \end{tikzpicture}
  \caption{The surface $M$ (left) and its intersection with the singular level set (right). We omit the $v_1$--coordinate on both pictures.}\label{fig:coors}
\end{figure}
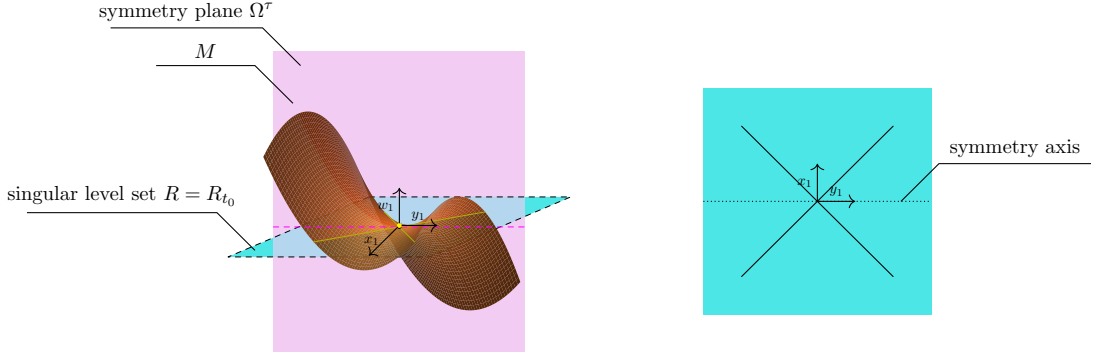
Choose a neighborhood $U$ of $u_0$ such that the coordinates are specified. We assume that the arc $\omega_{t}$ misses $t$.
As $w_1$ is transverse to the level sets of $F$, the coordinates $x_1,y_1,v_1$ provide a coordinate system on $U\cap R_t$
for each $t\in[t_0-\varepsilon,t_0+\varepsilon]$. We extend these coordinates to the whole of $R_t$ arbitrarily, but
if $u_0\in M^\tau$, we require that 
$\tau(x_1,y_1,v_1)=(-x_1,y_1,-v_1)$, see Figure~\ref{fig:coors}.

In these coordinates, we can specify what we mean by a good projection. Here, by a projection, we mean a smooth equivariant fibration $\R^3\to\R^2$ with fibers $\R$. In general, we will consider projections that are $C^1$-small perturbations of linear projections.
\begin{defn}\label{con:generic}\
  A projection $\pi\colon R_{t_0}\to\R^2$ is \emph{good} if it satisfies the following properties.
  \begin{enumerate}[label=(P-\arabic*)]
    \item If $u_0\notin M^{\tau}$, then $\pi(u_0)\notin\cL$;\label{item:nocL}
    \item The vectors $\frac{\partial}{\partial x_1},\frac{\partial}{\partial y_1}$ have linearly independent images under $D\pi(0,0,0)$. \label{item:proj1}
    \item The only points on $X_{t_0}$ mapped to $\pi(U\cap R_{t_0})$ are in $X_{t_0}\cap (U\cap R_{t_0})$. \label{item:proj2}
  \end{enumerate}
\end{defn}
We defer the proof that good projections exist until Subsection~\ref{sub:good_exists}.

\subsubsection{Local compression of $M$ and equivariant Morse moves}\label{sub:local_comp}
Suppose $u_0$ is a critical point of $F|_M$, but not of $F$ itself. Let $\pi$ be a good projection. Define the surface
\[M'\subset\R^2\times[t_0-\varepsilon,t_0+\varepsilon]=\{(\pi(u),F(u))\colon u\in M\cap F^{-1}[t_0-\varepsilon,t_0+\varepsilon]\}.\]
We refer to $M'$ as the \emph{compressed surface} and the map $(\pi,F)$ as the \emph{compression map}.
While $M'$ might be singular, by items~\ref{item:proj1}, \ref{item:proj2} of Definition~\ref{con:generic}, there exists a neighborhood $U'$
of $(0,0,F(u_0))$ such that $M'\cap U'$ is smooth. Indeed, \ref{item:proj1} indicates that a small neighborhood of $u_0$ in $M$ is mapped
diffeomorphically onto its image, and \ref{item:proj2} prohibits existence of multiple points in $U'$. Write $U''$ for $U'\cup\tau U'$.
This allows us to handle the situation, where $u_0\notin M^\tau$, so $\tau u_0$ is another critical point. Notice that if $\pi$ is a good
projection for $u_0$, by symmetry, it is also a good projection for $\tau u_0$.

By construction, $M'\cap F^{-1}(t)$ for $t\in [t_0-\varepsilon,t_0+\varepsilon]$ is the diagram of the link. Perturb the compression map
away from the preimage of $U''$ in such a way that $M'\cap F^{-1}(t_0)$ is a regular link diagram (in the sense of Definition~\ref{def:regular-half-knot}). This is possible using the techinques from the proof of Theorem~\ref{thm:generic_half_knot} applied to the curve
\[\cS:=(M\setminus (\pi,F)^{-1}(U''))\cap R_{t_0}\]
and then extending the perturbation to other level sets. On shrinking $\varepsilon$ if necessary, we may assume that $M'\cap F^{-1}(t)$
is a regular link diagram away from $U''$ for all $t\in[t_0-\varepsilon,t_0+\varepsilon]$. That is, the only change of the diagram
occurs in $U''$. This change depends on the exact properties of the function $f$.

\begin{thm}\label{thm:morse_genericity}
  Suppose $\pi\colon R\to\R^2$ is a good projection. Let $L_t^0=(M\cap U'\cap F^{-1}(t)$, $D^0_t=\pi_t(L_t)$, where $t\in[t_0-\varepsilon,t_0+\varepsilon]$. Suppose $f(x_1,y_1)$ is Morse.
  \begin{itemize}
    \item if $u_0\notin M^\tau$, then $D_{t_0+\varepsilon}^0$ differs from $D_{t_0-\varepsilon}^0$ by a pair of births, saddle moves or births (depending on the index), see Figure~\ref{fig:outside_crits}.
    \item if $u_0\in M^\tau$, then the change is as in Figure~\ref{fig:inside_crits}. More precisely:
  \begin{itemize}
    \item if $u_0$ is a birth, then $D^0_t=\emptyset$ for $t<t_0$. For $t>t_0$, $D^0_t$ is a circle intersecting the symmetry axis
      at two points;
    \item if $u_0$ is a death, then the situation is the same as for birth but with time reversed;
    \item if $u_0$ is a saddle point, for $t\neq t_0$, $D^0_t$ is a union of two arcs.
  \end{itemize}
  \end{itemize}
\end{thm}
\begin{proof}
  The change of the diagram is controlled by the function $f$ in $U''$. If $f$ has local maximum at $u_0$ (and at $\tau u_0$ if $u_0\notin M^\tau$), then the change creates a circle. For $u_0\in M^\tau$, the circle is given locally by $x_1^2+y_1^2=t-t_0$, so it intersects
  the symmetry axis at two points. The same argument holds if $f$ has local maximum.

  If $f$ has a saddle point, the change of the link is by performing a saddle move. Note that in the case $u_0\in M^\tau$, the hyperbolas
  $x_1^2-y_1^2=\pm(t-t_0)$ intersect the axis $\cL=\{x_1=0\}$ at zero or two points depending on the sign of $\pm(t-t_0)$.
\end{proof}

\subsubsection{Movies}\label{sub:movies}

\begin{defn}\label{def:movie}
	A \emph{movie} (more precisely, an \emph{equivariant movie}) is a sequence $D_i$ for $i\in \{0,\ldots, n\}$ of equivariant link diagrams, together with a sequence $\scS_i$ of terms, each of which is either an equivariant Morse move, equivariant Reidemeister move, or equivariant plane isotopy, from $D_i$ to $D_{i+1}$.
\end{defn}
Associated to an equivariant movie, there is an equivariant cobordism as follows.  Each move $\scS_i$ specifies a cobordism from $D_i$ to $D_{i+1}$ in $S^3\times [0,1]$ in an apparent way (e.g. equivariant Reidemeister moves by their trace, etc.).  For convenience, say the cobordisms built from each $\scS_i$ this way are product cobordisms near $S^3\times \{0\}$ and $S^3\times \{1\}$.  The cobordisms from each $\scS_i$ may be stacked together to obtain a composite $D_0\to D_{n}$.
\begin{defn}
  The cobordism associated to a movie as above is called the \emph{trace} of the (equivariant) movie.
\end{defn}
It is evident from the construction that the equivariant isotopy class of the trace of a movie is well-defined, although the trace itself is not well-defined (it depends on choosing a diffeomorphism $[0,n]\to [0,1]$).

We combine the proofs of Theorem~\ref{thm:intro1} and Theorem~\ref{thm:morse_genericity} to obtain the following statement, proving
Theorem~\ref{thm:intro2} from the introduction, up to the case of isolated fixed points of the $\tau$-action, which is explained in
Section~\ref{sub:cob_isolated} below.

\begin{thm}\label{thm:526}
  Suppose there exists an equivariant cobordism $\Sigma$ between two involutive links $L_0,L_1\subset S^3$, where $L_0,L_1$ are in general position with underlying transvergent diagrams $D_0,D_1$. Assume that the $\Z_2$-action
  on $\Sigma$ has no isolated fixed points.  Then there is a movie $\mathcal{M}$ from $D_0$ to $D_1$ whose trace is equivariantly isotopic to $\Sigma$.
\end{thm}
\begin{proof}
  Let $M\subset S^3\times[0,1]$ be an equivariant cobordism between $L_0$ and $L_1$. Let $F\colon S^3\times[0,1]\to[0,1]$ be the projection
  onto the second factor. Assume that $F$ is an equivariant embedded Morse function (see Definition~\ref{def:EEM}). If it is not the case, we perturb $F$ to $F'$. Such perturbation does not create any critical points of $S^3\times[0,1]$. In particular, by stability of Morse functions
  (see \cite[Theorem III.2.2]{GoluGuille}, but the proof applies in the equivariant setting), there is an equivariant diffeomorphism of $S^3\times[0,1]$ taking $F$ to $F'$. Such a diffeomorphism can be thought of as an equivariant perturbation of $M$. From now on, we assume that
  $F|_M$ is equivariant Morse.

  If $F|_M$ has no critical points at all, we conclude that $L_0$ and $L_1$ are equivariantly isotopic, and the statement follows directly from Theorem~\ref{thm:intro1}. If $F|_M$ has more than one critical points, we can split the cobordisms into pieces, when each piece has precisely one critical point of $F$. Therefore, it is enough to consider only the situation when $F|_M$ has precisely one critical point. We assume the critical value is $t_0=\frac12$. The discussion below deals with the case of the critical point on the axis, while the critical point off the axis can be done in a similar manner.

  Choose $\omega_{1/2}\in (S^3\times\{1/2\})^\tau$  disjoint from $M$,
where we recall that the superscript denotes the set of fixed points of $\tau$. 
  Set $R=(S^3\times\{1/2\})\setminus\{\omega_{1/2}\}$, and let $\pi\colon R\to\R^2$ be a good projection. Extend $\omega_t$ to a path of points disjoint from $M$ as above, and $\pi$ to the projection from $(S^3\times\{t\})\setminus \omega_t$, where $t\in[1/2-\varepsilon,1/2+\varepsilon]$.

  Let $L_t$ be the link $M\cap (S^3\times\{t\})$ (it is singular at $t=1/2$). Finally, let $D_t$ be the diagram of $L_t$, $D_t=\pi_t(L_t)$.
  Note that fixing an equivariant parameterization $\Phi\colon\Sigma\to S^3\times[0,1]$ of $M$, where $\Sigma$ is an abstract surface
  endowed with a $\Z_2$-action, gives us an equivariant parameterization $\wt{\phi}_t\colon\cS_t\to S^3\times\{t\}$ of $L_t$ for all $t\neq\frac12$.
  On choosing an appropriate good projection, by Theorem~\ref{thm:morse_genericity}, we may find $\varepsilon>0$ such that
  for $t\in[\frac12-\varepsilon,\frac12+\varepsilon]$, the change of the diagrams is a single equivariant Morse move.

  For $t\in[0,\frac12-\varepsilon]$, $\cS_t=\cS_0$, so $\wt{\phi}_t$ can be regarded as a map from $\cS_0$. That is, this is a path of maps from $\cS_0$ to $S^3$. Note that for $t=0,\frac12-\varepsilon$, the map $\wt{\phi}_t$ is regular. On perturbing the path $\phi_t:=\pi\circ\wt{\phi}_t$
  as in Theorem~\ref{thm:intro1}, we may construct a regular path $\phi'_t$ equal to $\phi_t$ at the end points and as close as we please
  to the original path. That is, $\phi_t$ induces a movie consisting of
  Reidemeister moves and Morse moves. Note that sufficiently perturbation preserves isotopy class. 
  That is, with $\wt{\phi'}_t\colon\cS_0\to S^3$
  being the lift, we note that the trace
  \[\{(\wt{\phi}_t(x),t)\subset S^3\times[0,\frac12-\varepsilon]\colon x\in\cS_0,\ t\in[0,\frac12-\varepsilon]\}\]
  is equivariantly isotopic to $M\cap S^3\times[0,\frac12-\varepsilon]$. 
  This shows two things at once.
  \begin{itemize}
    \item The links $L_0$ and $L_{\frac12-\varepsilon}$ are connected by a movie consisting of equivariant Reidemeister moves and the I-move;
    \item The trace of this movie is equivariantly isotopic to $M\cap S^3\times[0,\frac12-\varepsilon]$.
  \end{itemize}
  An analogous discussion can be done for $t\in[\frac12+\varepsilon,1]$.
  Combining this with the part for $t\in[\frac12-\varepsilon,\frac12+\varepsilon]$, we obtain the statement of the theorem.
\end{proof}

\subsubsection{Existence of good projections}\label{sub:good_exists}
The good projection of Definition~\ref{con:generic} will be obtained by perturbing
the following \emph{canonical projection}:
\[\pi_\phi(x_1,y_1,v_1)=(\cos(\phi)x_1+\sin(\phi)v_1,y_1),\]
where $\phi$ is a parameter, and we require that $\phi$ is not an integer  multiple of $\frac{\pi}{2}$.
We begin with \ref{item:nocL}, which we prove for $u_0\notin M^{\tau}$.
\begin{lemma}\label{lem:nocL_part}
  There are two parameters $\phi\in S^1$ for which $\pi_\phi$ takes $u_0$ to $\cL$. That is, for all but two parameters $\phi$,
  \ref{item:nocL} is satisfied.
\end{lemma}
\begin{proof}
  Suppose $u_0=(u_{x},u_y,u_v)$. Write $(u_x,u_v)=r_u(\cos\phi_u,\sin\phi_u)$. As $u_0\notin M^\tau$, $r_0>0$. The parameters such that $\pi_\phi$
  takes $u_0$ to $(0,u_y)\in\cL$ are $\phi_u\pm\pi$.
\end{proof}

The remaining part of the proof assumes $u_0\in M^{\tau}$. The case $u_0\notin M^\tau$ can be easily adapted.
%
%
%

\begin{lem}\label{lem:linear_is_enough}
  The conditions in Definition~\ref{con:generic} can be satisfied after an arbitrary small perturbation of $\pi_\phi$.
\end{lem}
\begin{proof}
  Condition~\ref{item:proj1} is satisfied for all map $\pi_\phi$ with $\sin(2\phi)\neq 0$ without even a need of perturbation.

  For \ref{item:proj2}, we can use a quick argument involving parameter counting. It might also be instructive for the reader
  to give an alternative
  proof, not invoking transversality theorem.

  Choose an equivariant bump function $\theta\colon R_{t_0}\to[0,1]$,
  with supported on $U\cap R_{t_0}$ and equal to $1$ on a smaller neighborhood of $(0,0,0)\in R_{t_0}$. 
  Consider an equivariant map $\Xi\colon R_{t_0}\times\R^2\to \R^2\times\R^2$ by
  \[\Xi(x_1,y_1,v_1,\sigma,\phi)=(\cos(\phi)x_1+\sin(\phi)v_1,y_1+\sigma(1-\theta)x_1^2,\sigma,\phi).\]
  If for some $\phi,\sigma$, the preimage $\Xi^{-1}(0,0,\sigma,\phi)$ does not contain any point of $X$ outside $U\cap R_{t_0}$ (and $\cos(\phi)\neq 0$), then we replace the coordinate system on $R_{t_0}$ by an automorphism $(x_1,y_1,v_1)\mapsto (x_1,y_1+\sigma(1-\theta)x_1^2,v_1)$ and \ref{item:proj2} is satisfied for this particular value of $\phi,\sigma$. Note that the derivative of this automorphism at $(0,0,0)$ is the identity, hence replacing the coordinate system does not affect \ref{item:proj1}. Therefore, we need to find such $\phi,\sigma$.

  Suppose towards contradiction that there exists a compact set $A\subset \R^2\setminus\{\phi=\frac{k}{2}\pi\}$, $k\in\Z$, of positive 2-dimensional Lebesgue measure such for $(\phi,\sigma)\in A$, the preimage $\Xi^{-1}(0,0,\phi,\sigma)$ contains a point of $X\times\{(\phi,\sigma)\}$ outside $(U\cap R_{t_0})\times\{(\phi,\sigma)\}$. Define $\wt{X}\subset (R_{t_0}\setminus U)\times A$ as the intersection of $(R_{t_0}\setminus U)\times\R^2$ with the preimage of $(0,0)\times A$. 
  As $\Xi\colon \wt{X}\to (0,0)\times A$ is onto, by coarea formula (see e.g. \cite[Theorem 3.2.11]{federer}), the subset $\wt{X}$
  has positive 2-dimensional Hausdorff measure.
  
  Now, consider the projection of $\wt{X}$ onto $X$. We claim the fibers of the projection have at most two points.
  To see this, suppose the points $(x_1,y_1,v_1,\sigma,\phi)$ and $(x_1,y_1,v_1,\sigma',\phi')$ both belong to $\wt{X}$. Then,
  \begin{align*}
    y_1+\sigma(1-\theta)x_1^2&=y_1+\sigma'(1-\theta)x_1^2 =0 ,\\
    \cos(\phi)x_1+\sin(\phi)v_1&=\cos(\phi')x_1+\sin(\phi')v_1 =0.
  \end{align*}
  As the points are outside $U$, $\theta=0$. Hence $\sigma x_1^2=\sigma' x_1^2$. If $x_1=0$ then $y_1=0$
  and $\sin(\phi)v_1=\sin(\phi')v_1=0$. As we excluded $\phi=\frac{k}{2}\pi$ from $A$, $\sin(\phi),\sin(\phi')\neq 0$, so $v_1=0$. But then $(x_1,y_1,w_1)=(0,0,0)\in U$, contradicting that these points are away from $U$. 

  Therefore, $x_1\neq 0$. Then, $\sigma=\sigma'$. Now let $\vec{a}=(\cos\phi,\sin\phi)$, $\vec{b}=(\cos\phi',\sin\phi')$, $\vec{c}=(x_1,v_1)$. The second displayed equation reads $(\vec{a}-\vec{b})\perp \vec{c}$. With $\vec{c}$ and $\vec{a}$ known, we see that
  $\vec{b}$ belongs to an affine line in $\R^2$. But $||\vec{b}||=1$, so $\vec{b}$ belongs to an intersection of a circle and a line, yielding
  at most two possibilities for $\phi'$, one being $\phi$.

  Given that the projection of $\wt{X}$ onto $X$ has at most $2$ points in each fiber, and that the one-dimensional Hausdorff measure of $X$
  is one, we invoke the coarea formula once again to deduce that the one-dimensional Hausdorff measure of $\wt{X}$ is finite. But this is impossible, since the two-dimensional Hausdorff measure of $\wt{X}$ was shown to be positive.
\end{proof}

\subsection{Cobordisms with isolated fixed points}\label{sub:cob_isolated}
We keep the notation of Section~\ref{sub:applications}. That is, $\Omega=S^3\times[0,1]$ is endowed with the standard action of $\Z_2$ (having $S^1$ as a fixed set on each $S^3\times\{t\}$)
and $M\subset\Omega$ is an embedded surface such that $\tau$ preserves $M$ setwise. We think of $M$ as an equivariant cobordism
between the involutive links $M\cap S^3\times\{0\}$ and $M\cap S^3\times\{1\}$. Let $f\colon\Omega\to[0,1]$ be the projection onto the second factor. On perturbing
$f$, if needed, we can assume that $f$ restricts to an equivariant Morse function on $M$.

Assume that $z_0\in M$ is an isolated fixed point of the action of $\tau$. 
By Corollary~\ref{cor:vanish}, $z_0$ is a critical point of $f|_M$.
By the embedded Equivariant Morse Lemma (Theorem~\ref{thm:eem}), we can find local coordinates near $z_0$. These coordinates are $(x_1,x_2,w_1,w_2)$, where $\tau$ fixes $w_1,w_2$, $\tau x_1=-x_1$, $\tau x_2=-x_2$,
and $M$ is given locally as $\{w_1=w_2=0\}$. In these coordinates $f$ is given by
\begin{equation}\label{eq:form_of_f_iso}
f(x_1,x_2,w_1,w_2)=f(z_0)\pm x_1^2\pm x_2^2 + w_1.
\end{equation}
Let $t_0=f(z_0)$. By chosing consistently points at infinity as in Section~\ref{sub:applications} above, for $t$ close to $t_0$ we consider
the level set $f^{-1}(t)$ with the point at infinity removed. We keep using the notation $R_t$ for $f^{-1}(t)$ with the point at infinity
removed. Clearly, $R_t\cong\R^3$. The flow of $\nabla f$
allows us to identify $R_t$ with $R_{t_0}$, so that level sets $f^{-1}(t)\cap M$ can be viewed as belonging to the same space $R_{t_0}\cong\R^3$.

As $t_0$ is a critical point of $f|_M$, the topology of the level sets $f^{-1}(t)\cap M$ changes as $t$ crosses $t_0$.
Depending on the index of the critical point $z_0$ we can have a birth, a death or a saddle. The corresponding level sets are depicted on
the left pictures of Figure~\ref{fig:iso_fix1} for a birth or a death, respectively Figure~\ref{fig:iso_fix2} for a saddle. For simplicity, we assume that the local coordinates $(x_1,x_2,w_2)$ on $R_{t_0}$ have been extended to the whole of $R_{t_0}$. The point $z_0$ has coordinates $(0,0,0)$.

To study the changes of the link diagram associated with $f^{-1}(t)\cap M$, we choose an equivariant projection $\pi\colon R_{t_0}\to\R^2$ and look at the images
$\pi(f^{-1}(t)\cap M)$.
Ideally, the projection should satisfy the axioms of a good projection of Definition~\ref{con:generic}. Unfortunately, this can never be achieved, as the following result shows.
\begin{lemma}\label{lem:noP2}
  Suppose $\pi\colon R_{t_0}\to\R^2$ is equivariant. The images of $\frac{\partial}{\partial x_1}$ and $\frac{\partial}{\partial x_2}$ under $D\pi(0,0,0)$ are linearly dependent.
\end{lemma}
\begin{proof}
  Let $x,y$ be the coordinates in $\R^2$. The map $\pi$ can be then written as $\phi=(\phi_x,\phi_y)$. As $\tau x=-x$ and
  $\tau y=y$, the map $\phi_x$ is anti-invariant and
  $\phi_y$ is $\tau$-invariant, see Definition~\ref{def:invariant}. By Corollary~\ref{cor:vanish}, 
  $\frac{\partial\phi_y}{\partial x_1}$ and $\frac{\partial\phi_y}{\partial x_2}$ vanish at $(0,0,0)$. Hence, the image of
  $\frac{\partial}{\partial x_1},\frac{\partial}{\partial x_2}$ under $D\pi(0,0,0)$ is contained in a one-dimensional space.
\end{proof}
It follows from Lemma~\ref{lem:noP2} that $\pi$ restricted to $\{w_2=0\}$ is
singular at $(0,0,0)$.
We aim to study typical (mildest) singularities that occur.
\begin{lemma}\label{lem:phi_equiv}
  On perturbing $\phi$ equivariantly, we may assume that
  \begin{itemize}
    \item there exists a neighborhood $U$ of $(0,0,0)$ in $R_{t_0}$ such that $f$
      has the form \eqref{eq:form_of_f_iso} in $U$ and neighborhoods $V$ of $(0,0)$ $\R^2$ such that $\pi^{-1}(V)$ does not
      contain any point in $f^{-1}(t)\cap M$ away from $U$, whenever $t$ is close to $t_0$;
    \item the map $\phi$ restricted to $w_2=0$ is locally a fold, i.e. $\phi_x$ has non-vanishing linear term, and $\phi_y$ can be written as
      $w_1+P(x_1,x_2)+\dots$, where $P$ is a quadratic non-degenerate function
			and dots mean higher order terms in $x_1,x_2$ and $w_1$.
  \end{itemize}
\end{lemma}
\begin{proof}
  Recall that $f^{-1}(t)\cap M$ belongs initially to $R_t$, but we have identified $R_t$ with $R_{t_0}$ for $t$ close to $t_0$.
  Choose first $U$, a neighborhood of $(0,0,0)$ in $R_{t_0}$ such that $f$ is written as in \eqref{eq:form_of_f_iso}. Perturb $\phi$ in such a way
  that with $L_{t_0}=f^{-1}(t_0)\cap M$, $\pi|_{L_{t_0}\setminus U}$ misses $(0,0)$. This is possible, because the condition
  on hitting $(0,0)$ is of codimension~1. If $\pi|_{L_{t_0}\setminus U}$ misses $(0,0)$, by openness, there exists $\varepsilon>0$ and $V\subset\R^2$ containing
  $(0,0)$ such that with $L_t=f^{-1}(t)\cap M\subset R_{t_0}$, and $t\in[t_0-\varepsilon,t_0+\varepsilon]$, the map $\pi|_{L_t\setminus U}$
  misses $V$. This concludes the proof of the first property.

  As for the second item, notice that each of the three conditions at $(0,0,0)\in R_{t_0}$:
  \begin{itemize}
    \item $\frac{\partial\phi_x}{\partial x_1}=\frac{\partial\phi_x}{\partial x_2}=0$;
    \item $\frac{\partial\phi_y}{\partial w_1}=0$;
    \item $\det D^2_{x_1,x_2}\phi_y=0$,
  \end{itemize}
  where $D^2_{x_1,x_2}$ is the 2 by 2 matrix of second derivatives of $\phi_y$ with respect to $x_1,x_2$ variables, is of parametric
  codimension~1 in the appropriate equivariant jet space. Therefore, all three situations can be avoided by perturbing $\phi$. We leave the details to the reader.

	Suppose the three conditions are negated. Then, we may assume that $\frac{\partial\phi_x}{\partial x_1}=1$ and $\frac{\partial\phi_x}{\partial x_2}=0$ (by a linear automorphism of $x_1,x_2$), so indeed $\phi_x(x_1,x_2)=x_1+\dots$. Likewise,
	the negation of the next two conditions is that $\phi_y=w_1+P(x_1,x_2)+\dots$. Here, $P$ is a non-degenerate quadratic polynomial and the dots mean higher order terms (degree $2$ or more in $w_1$ or degree $4$ or more in $x_1,x_2$). This concludes the proof.
\end{proof}

The first part of Lemma~\ref{lem:phi_equiv} implies that the change of the diagram on passing the point $t_0$ is local, that is,
does not interact with the remaining part of the diagram.
The second part implies that $\phi$ restricted to $\{w_1=0\}$ is a fold map,
see \cite{Whitney_Euclidean} and \cite[Chapter 1.5]{AVG}. The level sets $f^{-1}(t)\cap M$ are drawn in Figures~\ref{fig:iso_fix1}
(case of a birth or a death) and~\ref{fig:iso_fix2} (case of a saddle). The images of the level sets under $\pi$ acquire self-intersections.
A critical point of index $0$ or $2$ leads to creating, respectively destroying, a circle with a single self-intersection on axis,
while a critical point of index~$1$ leads to a handle with a crossing on axis. These considerations
lead us to the following lemma, which completes the proof of Theorem~\ref{thm:intro2} from the introduction.
\begin{lemma}\label{lem:equi_moves}
  Suppose $z_0\in M$ is an isolated fixed point of the $\tau$-action. Then, the diagrams of the link right below and right after
  the critical point differ by
  \begin{itemize}
    \item a creation or destruction of a singular circle in case of a birth, respectively a death, see Figure~\ref{fig:iso_fix3};
    \item a singular saddle move on-axis, see Figure~\ref{fig:iso_fix4}.
  \end{itemize}
\end{lemma}

\begin{figure}
  \begin{tikzpicture}
    \node at (-5,0) {\includegraphics[height=3cm]{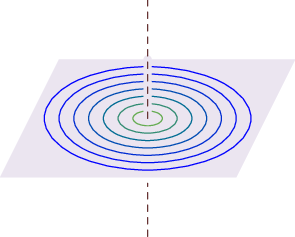}};
    \node at (0,0) {\includegraphics[height=3cm]{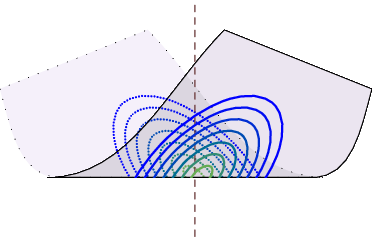}};
    \node at (5,0) {\includegraphics[height=3cm]{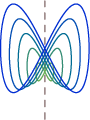}};
  \end{tikzpicture}
  \caption{Birth/death on-axis associated with an isolated fixed point of $\tau$. Left: level sets of $f$. Middle: the map $\pi$ is a fold map.
  Right: the resulting level sets on the diagram.}\label{fig:iso_fix1}
\end{figure}
\begin{figure}
  \begin{tikzpicture}
    \node at (-5,0) {\includegraphics[height=3cm]{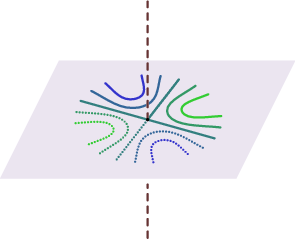}};
    \node at (0,0) {\includegraphics[height=3cm]{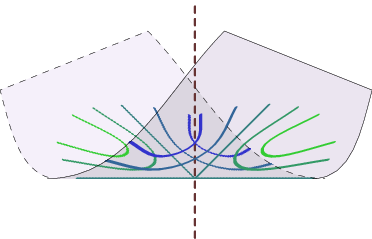}};
    \node at (5,0) {\includegraphics[height=3cm, width=3.5cm]{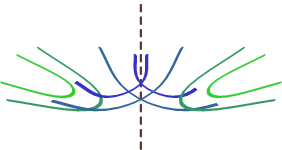}};
  \end{tikzpicture}
  \caption{Saddle on-axis associated with an isolated fixed point of $\tau$. Left: level sets of $f$. Middle: the map $\pi$ is a fold map.
  Right: the resulting level sets on the diagram.}\label{fig:iso_fix2}
\end{figure}
\begin{figure}
  \begin{tikzpicture}
    \node at (-4,0) {\includegraphics[height=2cm]{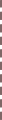}};
    \node at (0,0) {\includegraphics[height=2cm]{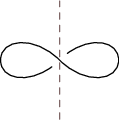}};
  \end{tikzpicture}
  \caption{The Morse move associated to a birth/death on-axis for an isolated fixed point of $\tau$. The crossing on-axis can be
  positive or negative.}\label{fig:iso_fix3}
\end{figure}
\begin{figure}
  \begin{tikzpicture}
    \node at (-4,0) {\includegraphics[height=3cm]{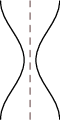}};
    \node at (0,0) {\includegraphics[height=3cm]{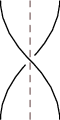}};
  \end{tikzpicture}
  \caption{The Morse move associated to a saddle on-axis for an isolated fixed point of $\tau$. The crossing on-axis can be
  positive or negative.}\label{fig:iso_fix4}
\end{figure}

\section{Movie moves}\label{sec:carter_saito}
We continue the study of cobordisms of involutive links started in Section~\ref{sec:4}. Our goal is to explain possible movie moves
for 1-parameter families. The main result is Theorem~\ref{thm:main_loop} below, echoing Theorem~\ref{thm:intro3} from the introduction.
\begin{remark}
  Throughout Section~\ref{sec:carter_saito}, we assume that cobordisms do not have isolated fixed points.
\end{remark}

\subsection{Genericity conditions for cobordisms}\label{sub:genericity}
To fix the notation, suppose $L_0,L_1$ are involutive links, parameterized as $\wt{\phi}_0\colon\cS_0\to S^3$, $\wt{\phi}_1\colon\cS_1\to S^3$,
where $\cS_0,\cS_1$ are two closed $1$-dimensional manifold with a $\Z_2$ action admitting at most two fixed points on each connected component.
Suppose $\Sigma$ is a compact oriented surface whose boundary is $\cS_0\sqcup \cS_0$ endowed with a $\Z_2$ action $\tau\colon \Sigma\to\Sigma$, such that the fixed point set, $\Fix\tau$, intersects each boundary component
of $\Sigma$ at at most two points. Assume that $\Phi_s$, $s\in[0,1]$ is a family of smooth equivariant maps from $\Sigma$ to $S^3\times[0,1]$
such that $\Phi_s$ extends $\wt{\phi}_0$ and $\wt{\phi_1}$, i.e. it takes  $\cS_0$ to $L_0\subset S^3\times\{0\}$ and $\cS_1$ to $L_1\subset S^3\times\{1\}$. Recall that we assume that the action of $\tau$ on $S^3\times[0,1]$ is standard.

For each $s$, $\Phi_s(\Sigma)$ is an equivariant cobordism between  $L_0$ and $L_1$. As such, under genericity conditions, which we will spell out in a moment, it introduces a sequence of Reidemeister moves, I-moves and Morse moves between strongly invertible diagrams of $L_0$ and $L_1$.
Suppose these genericity conditions hold for $s=0$ and $s=1$. We ask, how differ the sequence of Reidemeister moves, I-moves and Morse moves for $s=0$ and $s=1$.

Genericity conditions are obtained by summarizing the discussion in Subsections~\ref{sub:reg_path}, \ref{sub:ES3} and~\ref{sub:applications}.
Before passing to a formal definition, we recall the main ideas.

In Subsection~\ref{sub:applications}, our main assumption was 
that the composition $H_s=F\circ\Phi_s$, where $F$ is  the projection $F\colon S^3\times[0,1]\to[0,1]$, is Morse. 
For the critical points of $H_s$ there was a list of conditions, we imposed \ref{item:nocL}, \ref{item:proj1}, and \ref{item:proj2},
to hold for a projection $\pi\colon\R^3\to\R^2$ used to define the Morse moves, see Subsection~\ref{sub:local_coors}.
These conditions guaranteed that at critical points of $H_s$, the
local change is via the equivariant Morse moves; compare Subsection~\ref{sub:local_comp}.

For non-critical values of  $H_s$, there are two types of conditions. We imposed conditions for transversality at infinity in Subsection~\ref{sub:ES3}. Finally, for the parts of diagrams in $\R^3$, we required that the suitably defined path $\phi_s$ is regular in the sense of
Definition~\ref{def:regular_path_half}.

To formalize the discussion, 
set $\cS_{ts}=H_s^{-1}(t)$, and let 
$\wt{\phi}_{ts}$ be the embedding $\Phi_s$ restricted to $\cS_{ts}$. If $t$ is a non-critical level set of $H_s$, then $\wt{\phi}_{ts}$ is
a link diagram.

With $\cS_{ts}^0=\{u\in\cS_{ts}\colon \wt{\phi}_{ts}(u)\neq\infty\}$, we set $\phi_{ts}\colon\cS_{ts}^0\to\R^2$ to be the diagram. Note that
if $\wt{\phi}_{ts}$ actually hits infinity, the space $\cS_{ts}^0$ is not compact.
If $(t,s)$ vary in a neighborhood of a point $(t_0,s_0)$ and all $\cS_{ts}^0$ in that neighborhood are equivariantly diffeomorphic,
we simply write $\cS^0$ instead of $\cS_{ts}^0$.

We are now ready to specify regularity conditions. Throughout the list, $s\in[0,1]$ is fixed. Also, the points $t_i^\infty$,
$t^M_j$, $u_i^M$, $w_i^M$ and $t_i^M$ depend on $s$. Even their number is not fixed for $s\in[0,1]$.
\begin{enumerate}[label=(CS-\arabic*)]
  \item For any $s\in[0,1]$, the path $t\mapsto\wt{\phi}_{ts}$ is regular at infinity in the sense of Definition~\ref{defn:reg_at_inft}.
    For any $s$, we specify the points
    $0<t_1^\infty<\dots<t_{r}^\infty<1$, for which $\Phi_s$ hits the point at infinity.\label{item:cs_ifirst}
  \item We require that each point $t_i^\infty$ is a non-critical value of the map $H_s$. \label{item:cs_isecond}
  \item We assume that for any $t_i^\infty$, $\phi_{t_i^\infty s}$ is a regular map in the sense of Definition~\ref{def:generic_R2};  \label{item:cs_imove} 
  \item Let $t_1^M,\dots,t_{r'}^M$ be the set of critical values of $H_s$. Set $K^M_i=\Phi_s(H_s^{-1}(t_i^\infty))$.  We require that the diagram  $D_i^M$ obtained as $\pi(K_i^M)$ is regular  except at the images of the orbit of critical points;\label{item:cs_morse2}
  \item To each critical value of $H_s$ corresponds a unique orbit of critical points $u_i^M\in\Sigma$;\label{item:cs_morse_regular}
  \item If $u_i^M$ is a critical point of $H_s$, and $u_i^M$ is not fixed by the $\tau$-action, then $\pi(u_i)$ is not fixed
    by the $\tau$-action, that is \ref{item:nocL} is satisfied;\label{item:cs_nocL}
  \item For any critical point $u_i^M$ of $H_s$, there are local coordinates $x_1,y_1$ on $\Sigma$, (if $u_i^M\in \Fix\tau$, then $\tau$
    fixes $y_1$ and multiplies $x_1$ by $-1$), such that \ref{item:proj1} is satisfied;\label{item:cs_proj1}
  \item For any critical point $u_i^M$ of $H_s$, write $z_i^M=\Phi_s(u_i^M)\subset \R^3\times\{t_i^M\}$ and $w_i^M(s)=\pi(z_i^M(s))$. Then,
    the preimage $\pi^{-1}(w_i^M)$ in $\R^3\times\{t_i^M\}$ consists only of the orbit of $z_i^M$;
    \label{item:cs_proj2}
  \item If $u_i^M$ is a critical point of $H_s$, and \ref{item:cs_proj1} is satisfied, then the compressed surface \label{item:cs_morse}
    \[\Sigma':=\{(\pi\Phi_s(u),H_s(u))\colon u\in\Sigma\}\subset\R^2\times\R\]
     is smooth near the image $v_i^M$ of $u_i^M$ (the map $u\mapsto (\pi\Phi_s(u),H_s(u))$ takes $\Sigma$ to $\Sigma'$). 
    The projection $\R^2\times\R\to\R$ restricts to $\Sigma'$, and $u_i^M$ is a critical point of it. We require it to be Morse. 
  \item If $I$ is an interval in $[0,1]$ not containing any of $t_i^M$ and $t_i^\infty$, then the path $I\ni t\mapsto\phi_{ts}$ is regular in the sense of Definition~\ref{def:regular_path_half}.\label{item:cs_reid}
\end{enumerate}
We comment on the above conditions. Item~\ref{item:cs_proj2} is an analog of \ref{item:proj2}. We could require that $H_s$ is Morse,
and then deduce \ref{item:cs_morse} from that condition. Instead, we require the projection $\R^2\times\R\to\R$ to restrict
to a Morse function in \ref{item:cs_morse}. This leads to a simpler description of possible singularities: we can invoke a result of \cite{BB}
in that case.

We conclude the subsection with a statement that follows from previous discussions, but the explicit statement will be helpful for
the later statement of Theorem~\ref{thm:main_loop}.
\begin{lemma}\label{lem:o_dense}
  Suppose $L_0,L_1$ are two involutive links in $S^3\times\{0\}$ and $S^3\times\{1\}$, respectively. Let $\Sigma$ be a surface with a $\Z_2$-action whose boundary is divided into two parts (each part being a union of circles).

  Let $\cG(\Sigma)$ be the space of smooth equivariant maps from $\Sigma\to S^3\times[0,1]$ that take part of the boundary of $\Sigma$ to $L_0$
  and the other part to $L_1$. Let $\cG^0(\Sigma)$ be the space of maps satisfying \ref{item:cs_ifirst} -- \ref{item:cs_reid} (where we ignore
  parameter $s$).

  Then $\cG^0(\Sigma)$ is residual in $\cG(\Sigma)$.
\end{lemma}
\begin{proof}
  The proof is done on a case-by-case analysis. We define $\wt{\cG}_i^1$ the space of maps violating condition (CS-$i$).
  We will show that the complement of each $\wt{\cG}_i^1$ is residual. In most cases we will invoke previously proved results. Sometimes,
  the results from previous sections do not literally prove residuality of the complement, but rather the arguments will carry over.

  For \ref{item:cs_ifirst}, that is $\wt{\cG}^1_1$, we use the arguments of Theorem~\ref{thm:reg_at_inf}, where a path fails to be regular if it hits $\cI^2\cup\wt{\cI}^3$, and this union is essentially $\wt{\cG}^1_1$. The only technical difference is that in Subsection~\ref{sub:ES3}, we considered families of maps from $\cS$, that is, maps from $\cS\times[0,1]$ to $\cI$. Here, the source is $\Sigma$, usually not a product. However, the argument is valid in the present situation.

  For \ref{item:cs_isecond}, $\wt{\cG}^1_2$, coincidence of critical values of $H$ and hitting a point at infinity is given by conditions that there are points $x,y\in\Sigma$ such that $\Phi_s(x)=\{\infty\}\times\{t_i^\infty\}$, $\Phi_s(y)\in S^3\times\{t_i^\infty\}$, $DH_s(y)=0$. A straightforward
  argument involving equivariant jet transversality (this time, the source is the surface) leads to showing that $\cG\setminus\wt{\cG}^1_2$ is residual.

  Failure of regularity of $\phi_{t_i^\infty}$ is of codimension~1; we have a family $\phi_{t_i^\infty}$ parameterized by a finite set of
  indices. At each index, we use the arguments of Theorem~\ref{thm:generic_half_knot}.  This takes care of \ref{item:cs_imove}. An
  analogous argument works for \ref{item:cs_morse2}.

  Item~\ref{item:cs_nocL} is handled in Lemma~\ref{lem:nocL_part}.
  Items~\ref{item:cs_proj1} and~\ref{item:cs_proj2} are dealt with as in the proof of Lemma~\ref{lem:linear_is_enough}.
  Items~\ref{item:cs_morse} and~\ref{item:cs_morse_regular} are consequence of density of regular Morse functions (i.e. Morse functions for which distinct critical points have different critical values), and item~\ref{item:cs_reid}
  is proved in Theorem~\ref{thm:complem}.
\end{proof}

\subsection{Codimension 1 singularities of cobordisms}\label{sub:col_deg}

In Lemma~\ref{lem:o_dense}, we essentially proved that any equivariant cobordism can be perturbed to a cobordism satisfying items \ref{item:cs_ifirst} through \ref{item:cs_reid}. However, a generic one-parameter family can violate some conditions. In the following, we list all possible
singularities that may occur in such families. For the reader's convenience we divide them into two groups. The first one corresponds to singularities occurring at distinct places. They correspond to coincidences of moves. The second group are more complex singularities corresponding
either to loops of Reidemeister moves, or to interactions between Morse moves and Reidemeister moves.

The first group of singularities is as follows.
\begin{enumerate}[label=(Coi-\arabic*)]
  \item The map $\phi_{t_i^\infty s}$ of item~\ref{item:cs_imove} is not regular; \label{item:Col1}
  \item The diagram $D_i^M(s)$ of item~\ref{item:cs_morse2} has a single codimension~1 singularity;
  \item There is an equality $t_i^\infty(s)=t_{i'}^M(s)$ for one pair of indices $i,i'$;\label{item:col_R_I}
  \item There are two orbits of critical point of $H_s$ at the same level set,
    violating \ref{item:cs_morse_regular}; 
  \item The family $\alpha\times I\to\R^2$ of item~\ref{item:cs_reid} has a member $s\in I$ for which two separate codimension~1 singularities
    occur.\label{item:Col2}
\end{enumerate}
All these items are of codimension~1 and may occur on a generic one-parameter family. These singularities correspond to:
\begin{itemize}
  \item coincidence of a Reidemeister move and an I-move or an S-move;
  \item coincidence of an equivariant Morse move and a Reidemeister move;
  \item coincidence of an I-move or an S-move and an equivariant Morse move;
  \item coincidence of two equivariant Morse moves;
  \item coincidence of two Reidemeister moves.
\end{itemize}
Note that there cannot be any coincidence of two I-moves, or S-moves because these would imply failure of regularity at infinity. Such a coincidence is in fact a Reidemeister move at infinity, described below.

The second group of singularities is more complex and leads to loops of Reidemeister moves, I-moves, S-moves and equivariant Morse moves.
\begin{enumerate}[label=(ML-\arabic*)]
  \item The map $\wt{\phi}_{ts}$ is not regular at infinity for some $(t,s)$, so \ref{item:cs_ifirst} is violated;\label{item:ml_iloop}
  \item The map $H_s$ has a critical point at infinity, violating \ref{item:cs_isecond};\label{item:ml_morse_inf}
  \item A critical point $u_i^M$ of $H_s$ such that $u_i^M\neq\tau u_i^M$ can be mapped by $\pi$ to the symmetry axis (violating \ref{item:nocL} and
    \ref{item:cs_nocL};\label{item:ml_morse_on_line}
  \item There is a Morse critical point $u_i^M$ for which condition~\ref{item:proj1} (equivalently \ref{item:cs_proj1}) 
    is violated; \label{item:ml_twist_morse}
  \item There is a Morse critical point $u_i^M$ for which condition~\ref{item:proj2} (equivalently \ref{item:cs_proj2}) is violated; \label{item:ml_morse_on_diagram}
  \item There is a critical point $u_i^M$ such that \ref{item:proj1} and \ref{item:proj2} are satisfied, but $v_i^M$ is not a Morse
    critical point;\label{item:ml_morse_loop} 
  \item With $I$ as in~\ref{item:cs_reid}, there is a point $t_0\in I$ such that the family $t\mapsto \phi_{ts}$ has a codimension~2 singularity at $t_0$.\label{item:ml_reid}
\end{enumerate}
We describe these singularities in separate subsections. Each of these singularities will lead to a loop. Several loops of Reidemeister moves
were described in Section~\ref{sec:codim2}: that notion does not require further explanations. If we take into account Morse moves,
it might be worthwile to give a precise definition.
\begin{defn}
  A \emph{loop} of moves (i.e. Reidemeister moves, I-move, S-moves and Morse moves) 
  is a movie (see Definition~\ref{def:movie}) such that $D_0$ and $D_n$ are the same link diagram.
\end{defn}
We stress that the trace of a loop of moves need not be a product cobordism.

Subsections~\ref{sub:iloop} -- Subsection~\ref{sub:reid} will describe 39 loops of moves, which we refer to as \emph{elementary loops}.\footnote{In the count of 39, we do not add collisions, i.e. loops of moves occurring at different places of the diagram.}

\subsection{Case \ref{item:ml_iloop}. Singularities at infinity}\label{sub:iloop}
Suppose $\wt{\phi}_{ts}$ (with $t\in[0,1]$ and $s$ fixed) 
is not regular at infinity at a point $t_0$, but we assume that $\wt{\phi}_{ts}$ is regular at infinity as a 2-parameter family
as in Definition~\ref{def:two_regular}.
We assume that $t_0$ is not a critical value of $\Phi_s$. In fact,
if it is, then either this is a critical point away from infinity, leading to a collision as in item~\ref{item:col_R_I}, or there is a singularity at infinity, in which case we land in Case~\ref{item:ml_morse_inf}.

As $t_0$ is not a critical value of $H_s$, if we set $\cS=H_s^{-1}(t_0)$, then by implicit function theorem,
for $s'$ close to $s$ and $t'$ close to $t_0$
there is a map $\wt{\phi}_{s',t'}\colon\cS\to S^3$ whose image is equal to the image of $\Phi_{s'}$ intersected with $H_{s'}^{-1}(t')$.
These maps depend on two parameters, so we are in the situation of Definition~\ref{defn:reg_at_inft}. That is, the two-parameter
family of maps $\phi^\iota_{s',t'}$ (where $\iota$ indicates that we are composing with inversion of $S^3$ as in Subsection~\ref{sub:ES3})
acquires a codimension $1$ singularity at $0\in\R^2$; compare \ref{item:Sigma_codim} of Definition~\ref{def:two_regular}. These singularities were described in Subsection~\ref{sub:reg_path}.
More precisely, we can have.
\begin{itemize}
  \item Cusp on the axis at $0$;
  \item Central tangency on the axis at $0$;
  \item Boundary double point at $0$;
  \item Central perpendicularity at $0$;
  \item Central double point on the axis at $0$;
\end{itemize}
The family $\phi^{\iota}_{s',t'}$ is a versal deformation of the singularity by a deformation (leading to a Reidemeister move) combined with
shifting the singular point $0$ up and down along the symmetry axis $\cL$. After returning to the diagram $\phi_{s',t'}$ (without the inversion $\iota$), we see the following five loops of moves. They correspond to a versal deformation of the singularity after or before passing through $\infty$. These are:
\begin{enumerate}[label=(RI-\arabic*)]
  \item The R-1 move at infinity, see Figure~\ref{fig:RI1};
  \item The R-2 move at infinity, see Figure~\ref{fig:RI2};
  \item The M-1 move at infinity, see Figure~\ref{fig:MI1};
  \item The M-2 move at infinity, see Figure~\ref{fig:MI2};
  \item The M-3 move at infinity, see Figure~\ref{fig:MI3}.
\end{enumerate}
\begin{figure}
  \begin{tikzpicture}
    \node at (-6,0){\includegraphics[width=2cm]{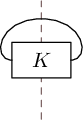}};
    \node at (-3,0){\includegraphics[width=2cm]{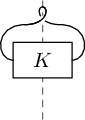}};
    \node at (0,0){\includegraphics[width=2cm]{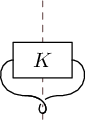}};
    \node at (3,0){\includegraphics[width=2cm]{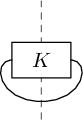}};
    \node at (6,0){\includegraphics[width=2cm]{pics/movie_moves-45.eps}};
  \end{tikzpicture}
  \caption{R-1 move at infinity induces a loop. We have R-1 move, then a sequence of I- and S-moves, the inverse of R-1 move and the I-move again.}\label{fig:RI1}
\end{figure}
\begin{figure}
  \begin{tikzpicture}
    \node at (-6,0){\includegraphics[width=2cm]{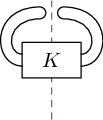}};
    \node at (-3,0){\includegraphics[width=2cm]{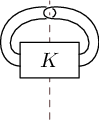}};
    \node at (0,0){\includegraphics[width=2cm]{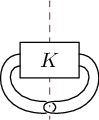}};
    \node at (3,0){\includegraphics[width=2cm]{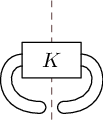}};
    \node at (6,0){\includegraphics[width=2cm]{pics/movie_moves-39.eps}};
  \end{tikzpicture}
  \caption{R-2 move at infinity induces a loop. We have R-2 move, then a sequence of two S-moves, the inverse of R-2 move and then an isotopy.}\label{fig:RI2}
\end{figure}
\begin{figure}
  \begin{tikzpicture}
    \node at (-6,0){\includegraphics[width=2cm]{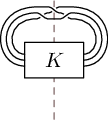}};
    \node at (-3,0){\includegraphics[width=2cm]{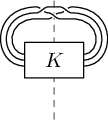}};
    \node at (0,0){\includegraphics[width=2cm]{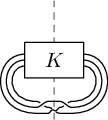}};
    \node at (3,0){\includegraphics[width=2cm]{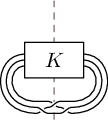}};
    \node at (6,0){\includegraphics[width=2cm]{pics/movie_moves-51.eps}};
  \end{tikzpicture}
  \caption{The loop of M-1 moves around infinity. We have M-1 move, then a sequence of I- and S-moves, the inverse of M-1 move and then a sequence of I- and S- moves.}\label{fig:MI1}
\end{figure}
\begin{figure}
  \begin{tikzpicture}
    \node at (-6,0){\includegraphics[width=2cm]{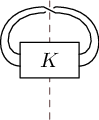}};
    \node at (-3,0){\includegraphics[width=2cm]{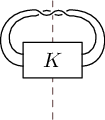}};
    \node at (0,0){\includegraphics[width=2cm]{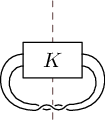}};
    \node at (3,0){\includegraphics[width=2cm]{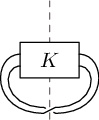}};
    \node at (6,0){\includegraphics[width=2cm]{pics/movie_moves-3.eps}};
  \end{tikzpicture}
  \caption{The loop of M-2 moves around infinity. The M-2 move is followed by S-moves and Reidemeister moves, then M-2 is undone and the S-move is applied.}\label{fig:MI2}
\end{figure}
\begin{figure}
  \begin{tikzpicture}
    \node at (-6,0){\includegraphics[width=2cm]{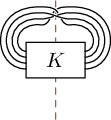}};
    \node at (-3,0){\includegraphics[width=2cm]{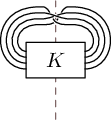}};
    \node at (0,0){\includegraphics[width=2cm]{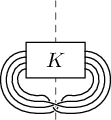}};
    \node at (3,0){\includegraphics[width=2cm]{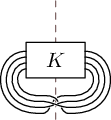}};
    \node at (6,0){\includegraphics[width=2cm]{pics/movie_moves-55.eps}};
  \end{tikzpicture}
  \caption{The loop of M-3 moves around infinity. The M-3 move is done. Then follows a sequence of S-moves and Reidemeister moves, then M-3 is undone. Finally, a sequence of S-moves and Reidemeister moves is done.}\label{fig:MI3}
\end{figure}
We pause for the moment to describe the loops involving S-moves, but not I-moves that is, corresponding to R-2, M-2, and M-3 moves at infinity.
As an S-move can be replaced by a series of equivariant Reidemeister moves (different link diagrams yield possibly different series),
the loops involving an S-move and Reidemeister moves can be replaced by loops of Reidemeister moves. As such, via Theorem~\ref{thm:intro4},
they can be decomposed as loops of Reidemeister moves given by Table~\ref{tab:local_loops}.

In other words, these three loops can be written as other loops. We include them because it might sometimes be easier to verify invariance
of a map associated to a cobordism (e.g. in Khovanov theory)
under a short local loop involving an S-move.

\subsection{Case \ref{item:ml_morse_inf}. The Morse point at infinity.}\label{sub:ml_morse_inf}
This is the situation where there is a point $x\in\Sigma$ that is mapped to $\infty\times\{t^\infty_i(s)\}$ and it is a
critical point of $H_s$ at the same time. A dimension counting argument does not prohibit this situation. 
However, again by a dimension counting argument, we can show
that the non-Morse situation, or the situation such that $\pi^\iota=\pi\circ\iota$ violates either \ref{item:proj1} or \ref{item:proj2} is of higher
codimension and does not occur on a 1-parameter family of maps $H_s$.

The discussion in Subsection~\ref{sub:applications} implies that the diagram of the link obtained via $\pi^\iota$ changes as in Figure~\ref{fig:inside_crits}. That is, we might have an equivariant birth at infinity, equivariant saddle at infinity and equivariant death.

With that type of critical points we associate a loop as in Subsection~\ref{sub:iloop}. The loop consists of moving the critical point
along the axis but off the infinity point, or performing the Morse move.
Eventually, we have the following two loops.
\begin{enumerate}[label=(IM-\arabic*)]
  \item Birth/death at infinity, see Figure~\ref{fig:IMorse1};
  \item Saddle at infinity, see Figure~\ref{fig:IMorse2}.
\end{enumerate}
\begin{figure}
  \begin{tikzpicture}
    \node at (0.3,3){\includegraphics[width=2cm]{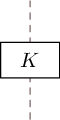}};
    \node at (-5,0){\includegraphics[width=2cm]{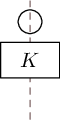}};
    \node at (-0.3,0){\includegraphics[width=2.4cm]{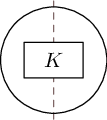}};
    \node at (5.2,0){\includegraphics[width=2cm]{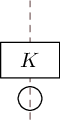}};
    \draw[<->,blue] (-1,3) .. controls ++ (-3,0) and ++(0,1) .. node [near start,above,scale=0.8] {Morse move} (-5,2);
    \draw[<->,blue] (1.6,3) .. controls ++ (3,0) and ++(0,1) .. node [near start,above,scale=0.8] {Morse move} (5.2,2);
    \draw[<->,green!40!black] (-3.8,0) -- node[midway, above, scale=0.8]{I-move} (-1.7,0);
    \draw[<->,green!40!black] (1.5,0) -- node[midway, above, scale=0.8]{I-move} (4,0);
  \end{tikzpicture}
  \caption{A loop of birth/death at infinity.}\label{fig:IMorse1}
\end{figure}
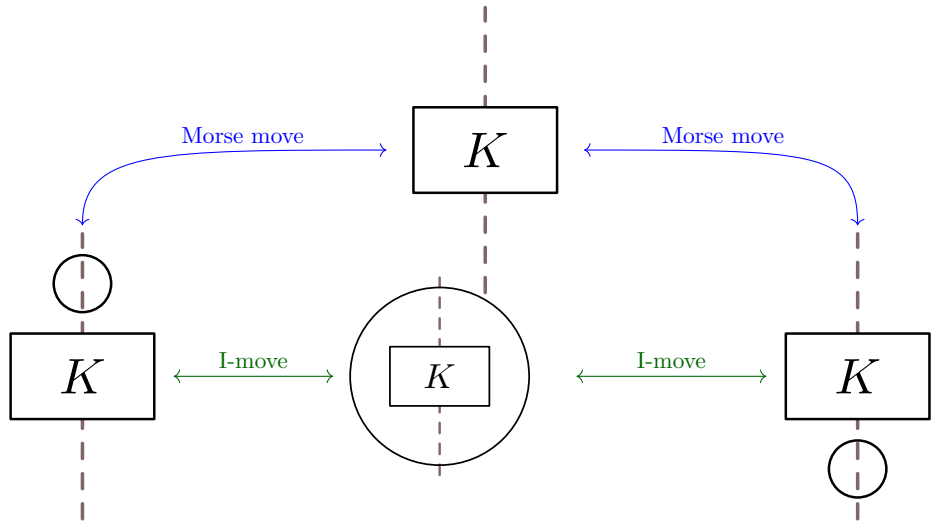
\begin{figure}
  \begin{tikzpicture}
    \node at (-3,3){\includegraphics[width=2cm]{pics/movie_moves-39.eps}};
    \node at (-3,0){\includegraphics[width=2cm]{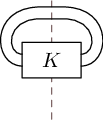}};
    \node at (3,0.5){\includegraphics[width=2cm]{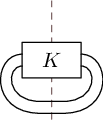}};
    \node at (3,3.5){\includegraphics[width=2cm]{pics/movie_moves-42.eps}};
    \draw[<->,blue] (-4,3.5) .. controls ++(-1.5,0) and ++(-1.5,0) .. (-4,0) node[scale=0.8,midway, right] {Saddle move};
    \draw[<->,blue] (4.1,3.7) .. controls ++(1.5,0) and ++(1.5,0) .. (4.1,0.2) node[scale=0.8,midway, left] {Saddle move};
    \draw[<->,green!40!black] (-1.5,3) -- node[scale=0.8,midway,above] {isotopy} (1.5,3);
    \draw[<->,green!40!black] (-1.5,0) -- node[scale=0.8,midway,above] {2x I-move} (1.5,0);
  \end{tikzpicture}
  \caption{The loop of saddle at infinity.}\label{fig:IMorse2}
\end{figure}

Said differently, for the movie at $s<s_0$, there is a saddle near infinity (to the south)  while for $s>s_0$, part of the diagram is pulled through infinity (to the south), and a saddle happens near north infinity.  

\subsection{Case \ref{item:ml_morse_on_line}. An orbit of handles hitting the axis}\label{sub:morse_online}
Suppose $u_0$ is one of the $u_i^M$: a critical point of $H_s$ such that $u_0\notin\Sigma^\tau$, but $z_0:=\pi\Phi_s(u_0)$ belongs to the axis
$\cL$. Set $u_1=\tau u_0$. It is a critical point of $H_s$ by equivariance, $u_1\neq u_0$, but $\pi\Phi_s(u_0)=\tau z_0=z_0$, so the two critical points occur on the axis.

If $u_0$ is a birth (then $u_1$ is a birth, too), after passing the critical level, two ellipses are created on the diagram (we assume that
\ref{item:proj1} is satisfied, and that \ref{item:proj2} except that $(\pi\Phi_s)^{-1}=\{u_0,u_1\}$). In a generic situation, the axes of the
ellipses are transverse (it is enough that neither axis is parallel to or perpendicular to the symmetry axis $\cL$), the non-generic leads
to a codimension 2 behavior. With these choices, the diagram above the level set looks locally as in the middle picture
in the bottom row of Figure~\ref{fig:birth_on_line}. The two ellipses are mapped to each other by a $\tau$-action.

A deformation moves the point $u_0$ off-axis in both directions, leading to a situation above the critical level set as in the
row in Figure~\ref{fig:birth_on_line}, whereas before the critical level set, that part of the diagram is empty, see Figure~\ref{fig:birth_on_line} for the full loop of moves.

The situation is analogous if $u_0,u_1$ are saddles. The singular level set for each point is a pair of lines passing through $z_0$. In a generic situation, these lines are pairwise transverse, see Figure~\ref{fig:sad_on_line}. Changing the level set leads to a resolution of the intersection
points: each pair of lines is replaced by a pair of hyperbolas. This can be seen in the middle part of Figure~\ref{fig:saddle_on_line}.

Moving $u_0$ and $u_1$ off the axis separates the two crossings of the singular level set. Resolutions lead to the loop as in Figure~\ref{fig:saddle_on_line}.

\begin{figure}
  \includegraphics[width=2cm]{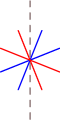}
  \caption{Singular level set for a pair of saddles on-axis. The blue and red lines indicate singular level sets for $u_0$ and $u_1$.}\label{fig:sad_on_line}
\end{figure}

\begin{figure}
  \begin{tikzpicture}
    \node at (0,3){\includegraphics[height=2cm]{pics/reidemeister-20.eps}};
    \node at (-4,0){\includegraphics[height=2cm]{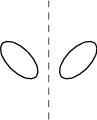}};
    \node at (-2,0){\includegraphics[height=2cm]{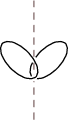}};
    \node at (0,0){\includegraphics[height=2cm]{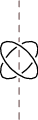}};
    \node at (2,0){\includegraphics[height=2cm]{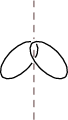}};
    \node at (4,0){\includegraphics[height=2cm]{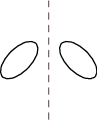}};
    \draw[<->,green!50!black] (-4,1.2) .. controls ++(0,0.5) and ++(0,0.5) .. node [scale=0.8,above,midway]{R-2} (-2,1.2);
    \draw[<->,green!50!black] (4,1.2) .. controls ++(0,0.5) and ++(0,0.5) .. node [scale=0.8,above,midway]{R-2} (2,1.2);
    \draw[<->,green!50!black] (2,-1.2) .. controls ++(0,-0.5) and ++(0,-0.5) .. node [scale=0.8,above,midway]{IR-2} (0.1,-1.2);
    \draw[<->,green!50!black] (-2,-1.2) .. controls ++(0,-0.5) and ++(0,-0.5) .. node [scale=0.8,above,midway]{IR-2} (-0.1,-1.2);
    \draw[<->,blue] (-1,3) .. controls ++ (-4,0) and ++ (-1,0) .. node [scale=0.8,left,midway]{Morse} (-5,0);
    \draw[<->,blue] (1,3) .. controls ++ (4,0) and ++ (1,0) .. node [scale=0.8,right,midway]{Morse} (5,0);
  \end{tikzpicture}
  \caption{A loop for an orbit of handles hitting the symmetry axis. The Morse label indicates birth or death depending
  on the direction. The diagram on the top is an empty diagram (other components of the link might be present, but do not interfere.}\label{fig:birth_on_line}
\end{figure}
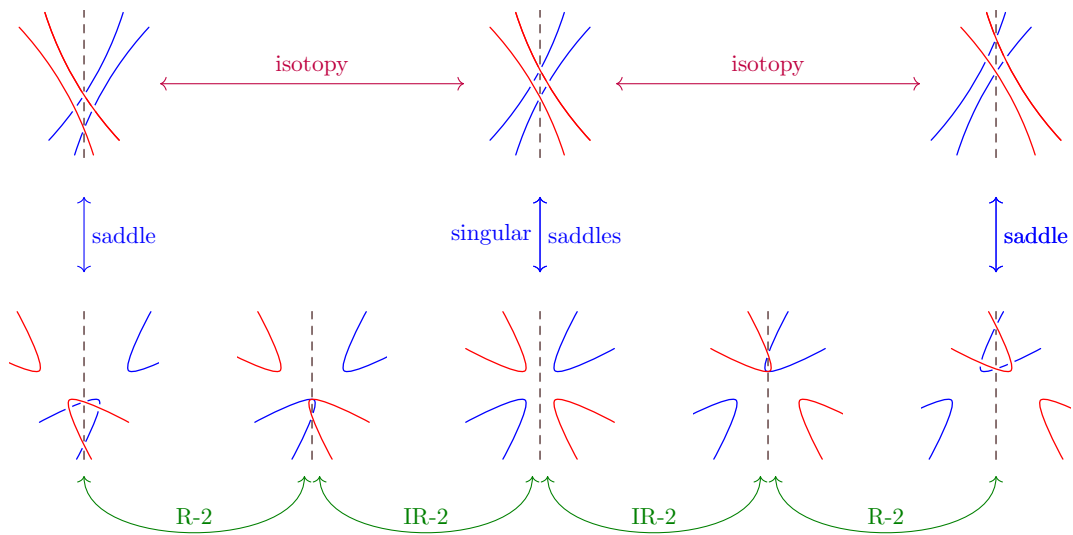
\begin{figure}
  \begin{tikzpicture}
    \node at (-6,0){\includegraphics[height=2cm]{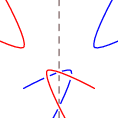}};
    \node at (-3,0){\includegraphics[height=2cm]{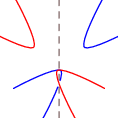}};
    \node at (0,0){\includegraphics[height=2cm]{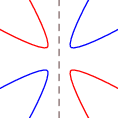}};
    \node at (3,0){\includegraphics[height=2cm]{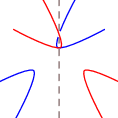}};
    \node at (6,0){\includegraphics[height=2cm]{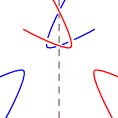}};
    \node at (-6,4){\includegraphics[height=2cm]{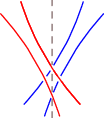}};
    \node at (0,4){\includegraphics[height=2cm]{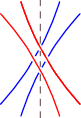}};
    \node at (6,4){\includegraphics[height=2cm]{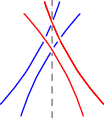}};
    \draw[<->,green!50!black] (-6,-1.2) .. controls ++(0,-1.0) and ++(0,-1.0) .. node [scale=0.8,above,midway]{R-2} (-3.1,-1.2);
    \draw[<->,green!50!black] (6,-1.2) .. controls ++(0,-1.0) and ++(0,-1.0) .. node [scale=0.8,above,midway]{R-2} (3.1,-1.2);
    \draw[<->,green!50!black] (2.9,-1.2) .. controls ++(0,-1.0) and ++(0,-1.0) .. node [scale=0.8,above,midway]{IR-2} (0.1,-1.2);
    \draw[<->,green!50!black] (-2.9,-1.2) .. controls ++(0,-1.0) and ++(0,-1.0) .. node [scale=0.8,above,midway]{IR-2} (-0.1,-1.2);
    \draw[<->,purple] (-5,4) -- node [scale=0.8,above,midway] {isotopy} (-1,4);
    \draw[<->,purple] (5,4) -- node [scale=0.8,above,midway] {isotopy} (1,4);
    \draw[<->,blue] (0,2.5) -- node[scale=0.8,right,midway] {saddles} (0,1.5);
    \draw[<->,blue] (6,2.5) -- node[scale=0.8,right,midway] {saddle} (6,1.5);
    \draw[<->,blue] (-6,2.5) -- node[scale=0.8,right,midway] {saddle} (-6,1.5);
    \draw[<->,blue] (6,2.5) -- node[scale=0.8,right,midway] {saddle} (6,1.5);
    \draw[<->,blue] (0,2.5) -- node[scale=0.8,left,midway] {singular} (0,1.5);
  \end{tikzpicture}
  \caption{A loop for an orbit of handles on the symmetry axis. Saddle.}\label{fig:saddle_on_line}
\end{figure}

\subsection{Case \ref{item:ml_twist_morse}. Singular equivariant handles}\label{sub:twist_morse}
This situation corresponds to a critical point $x_0\in\Sigma$ for which \ref{item:proj1} is violated. 
That is to say, in the local coordinates $(u_1,u_2)$ on $\Sigma$ near the critical point $x_0$, the map $\pi\circ\Phi_s\colon\Sigma\to\R^2$ (where $\pi\colon\R^3\to\R^2$ is the projection) has a critical point. This situation, depicted in \cite[Figure 14]{CRS}, was first studied by Giller \cite{Giller}.    
Recall that the changes of the link diagram are given by images of level sets of $H_s$ under the map $\pi\circ\Phi_s$.

\subsubsection*{The non-equivariant case.}
Dimension counting arguments indicate that $\pi\circ\Phi_s$ must have the mildest possible (codimension~1) 
singularity as a map from $\R^2$ to $\R^2$.
Low-codimension singularities of such maps have been classified, and the only case in codimension~1 is the fold, see \cite{Whitney_Euclidean} and \cite[Chapter 1.5]{AVG}. That is, in
local coordinates, it is given by $(u_1,u_2)\mapsto (u_1^2,u_2)$.

The map $H_s$ in the coordinates $(u_1,u_2)$ cannot, in general, be put in the form $\pm u_1^2\pm u_2^2$, because the $u_1,u_2$ coordinates
have already been fixed by choosing the local form of the fold. Instead, we assume that $H_s(u_1,u_2)=B(u_1,u_2)+\dots$, where $B$ is a quadratic part and the dots denote higher order terms (of order 3 and more). The quadratic part is nondegenerate. 

Suppose $B$ is definite, which corresponds to a birth or a death. The level sets of $B$ are concentric ellipses. These ellipses
are mapped under the fold map to curves with precisely one transverse double point, see Figure~\ref{fig:concentric}. Since
$H_s$ differs from $B$ by terms of order $3$ or more, the level sets of $H_s$  have the same qualitative behavior as $B$: they are approximate
ellipses such that the fold map takes them to curves with precisely one transverse double point. We leave the details to the reader.

\begin{figure}
  \begin{tikzpicture}
    \node at (-5,0) {\includegraphics[width=3.5cm]{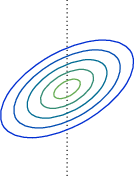}};
    \node at (0,0) {\includegraphics[width=3.5cm]{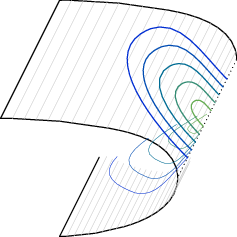}};
    \node at (5,0) {\includegraphics[width=2cm]{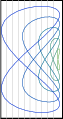}};
  \end{tikzpicture}
  \caption{Case \ref{item:ml_twist_morse}. The dotted vertical line is the fold axis. The images of level sets of $H_s$ under the fold map are circles with one double point.}\label{fig:concentric}
\end{figure}
\begin{figure}
  \begin{tikzpicture}
    \node at (-5,0) {\includegraphics[width=3.5cm]{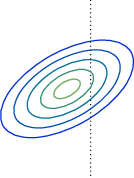}};
    \node at (0,0) {\includegraphics[width=3.5cm]{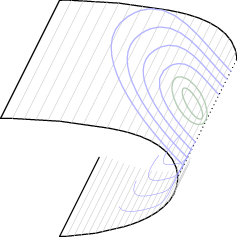}};
    \node at (5,0) {\includegraphics[width=2cm]{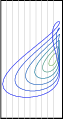}};
  \end{tikzpicture}
  \caption{Case \ref{item:ml_twist_morse}, continued. If the critical point is moved away from the fold line, the level sets for small parameter
  are mapped into simple closed curves. A crossing occurs when the level sets cross the fold line.}\label{fig:concentric2}
\end{figure}
\begin{figure}
  \begin{tikzpicture}
    \node at (-5,0) {\includegraphics[width=3.5cm]{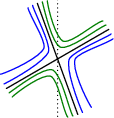}};
    \node at (0,0) {\includegraphics[width=3.5cm]{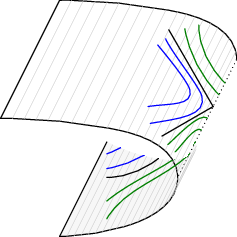}};
    \node at (5,0) {\includegraphics[width=2cm]{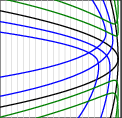}};
  \end{tikzpicture}
  \caption{Case \ref{item:ml_twist_morse}, saddle point. The dotted vertical line is the fold axis. One set of hyperbolas (corresponding
    to the level sets of $H_s$ below or above the critical point) are mapped bijectively by the fold map. The other set (corresponding
  to the opposite level sets) acquire a double point after the fold map.}\label{fig:concentric3}
\end{figure}
The movie is obtained as follows: changing the parameter $s$ moves the critical point away from the fold line (see leftmost figure of Figure~\ref{fig:concentric2}). The level sets that do not intersect the fold line will be mapped to circles with no double point. At some moment, a cusp
is acquired and the double point is created. Moving the critical points from the left of the fold line to the right will lead
to a movie as in Figure~\ref{fig:SMorse1}.

The situation of the saddle point is analogous. The analog of Figure~\ref{fig:concentric} is
sketched in Figure~\ref{fig:concentric3}. It leads to a local loop in Figure~\ref{fig:SMorse2}. This concludes the standard, non-equivariant situation. 
\subsubsection*{The equivariant case off-axis.}  This case is analogous to the non-equivariant case: the same movie is performed on both sides of the axis.
That is, we have the following two off-axis movies from case~\ref{item:ml_twist_morse}.
We refer to them as \emph{singular Morse handles off-axis}:
\begin{enumerate}[label=(SM-\arabic*)]
  \item Singular birth off-axis, see Figure~\ref{fig:SMorse1};
  \item Singular saddle off-axis, see Figure~\ref{fig:SMorse2}.
\end{enumerate}

\begin{figure}
  \begin{tikzpicture}
    \node at (-1.5,3){\includegraphics[height=1cm]{pics/reidemeister-20.eps}};
    \node at (-6,0){\includegraphics[width=1cm]{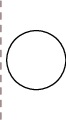}};
    \node at (-3,0){\includegraphics[width=1cm]{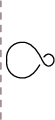}};
    \node at (0,0){\includegraphics[width=1cm]{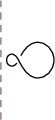}};
    \node at (3,0){\includegraphics[width=1cm]{pics/movie_moves-9.eps}};
    \draw[<->,green!40!black] (-5.2,0) -- node[scale=0.8,above,midway]{IR-1} (-3.8,0);
    \draw[<->,green!40!black] (-2.2,0) -- node[scale=0.8,above,midway]{IR-1} (-0.8,0);
    \draw[<->,green!40!black] (0.8,0) -- node[scale=0.8,above,midway]{IR-1} (2.2,0);
    \draw[<->,blue] (-1.9,3) .. controls ++(-5,0) and ++(-1,0) .. node[scale=0.8,right,midway]{Morse} (-6.7,0);
    \draw[<->,blue] (-1.1,3) .. controls ++(5,0) and ++(1,0) .. node[scale=0.8,right,midway]{Morse} (3.7,0);
  \end{tikzpicture}
\caption{Twisted birth and death off-axis. The diagram is mirrored on the other side of the axis. In the middle diagrams, the smaller loop is done/undone
by an IR-1 move.} \label{fig:SMorse1}
\end{figure}
\begin{figure}
  \begin{tikzpicture}
    \node at (-3,0){\includegraphics[width=1cm]{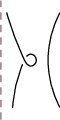}};
    \node at (-0,0){\includegraphics[width=1cm]{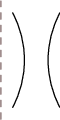}};
    \node at (3,0){\includegraphics[width=1cm]{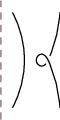}};
    \node at (0,-3){\includegraphics[width=1cm]{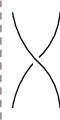}};
    \draw[<->,blue] (3,-1) .. controls ++ (0,-1) and ++(2,0) .. node[midway, scale=0.8,left]{Saddle} (1,-3);
    \draw[<->,blue] (-3,-1) .. controls ++ (0,-1) and ++(-2,0) .. node[midway, scale=0.8,right]{Saddle} (-1,-3);
    \draw[<->,green!40!black] (0.8,0) -- node[scale=0.8,above,midway]{IR-1} (2.2,0);
    \draw[<->,green!40!black] (-2.2,0) -- node[scale=0.8,above,midway]{IR-1} (-0.8,0);
  \end{tikzpicture}
\caption{Twisted saddle off-axis. The diagram is mirrored on the other side of the axis. In the row, the R-1 moves are performed.} \label{fig:SMorse2}
\end{figure}
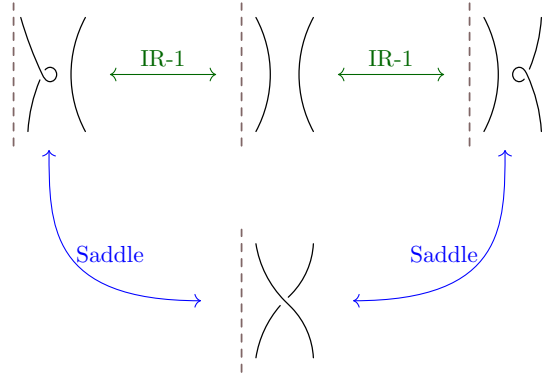
\begin{figure}
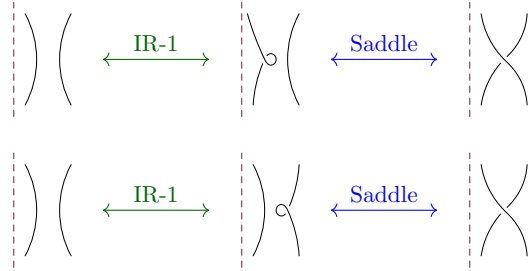

  \begin{tikzpicture}
    \node at (3,-2){\includegraphics[width=0.8cm]{pics/movie_moves-7.eps}};
    \node at (3,0){\includegraphics[width=0.8cm]{pics/movie_moves-7.eps}};
    \node at (-3,-2){\includegraphics[width=0.8cm]{pics/movie_moves-5.eps}};
    \node at (-3,0){\includegraphics[width=0.8cm]{pics/movie_moves-5.eps}};

    \node at (0,0){\includegraphics[width=0.8cm]{pics/movie_moves-6.eps}};
    \node at (0,-2){\includegraphics[width=0.8cm]{pics/movie_moves-8.eps}};
    \draw[<->,blue] (0.8,0) -- node[scale=0.8,above,midway]{Saddle} (2.2,0);
    \draw[<->,green!40!black] (-2.2,0) -- node[scale=0.8,above,midway]{IR-1} (-0.8,0);
    \draw[<->,blue] (0.8,-2) -- node[scale=0.8,above,midway]{Saddle} (2.2,-2);
    \draw[<->,green!40!black] (-2.2,-2) -- node[scale=0.8,above,midway]{IR-1} (-0.8,-2);
  \end{tikzpicture}
  \caption{An actual movie move associated to the loop in Figure~\ref{fig:SMorse2}.} \label{fig:SMorse2a}
\end{figure}

\subsubsection*{Equivariant case on-axis} We have a point $u_0\in\Sigma^\tau$ that is mapped by $\pi\circ\Phi_s$ to a point in $\cL$. We choose
local coordinates $x_1,y_1$ near $u_0$, with $\tau x_1=-x_1$, $\tau y_1=y_1$.
By equivariant Morse Lemma (Theorem~\ref{thm:wass_mors}), 
we can adjust the coordinates equivariantly in such a way that $H_s(x_1,y_1)=H_s(u_0)+\pm x_1^2+\pm y_1^2$.

Choose $x_2,y_2$ the coordinates in the target, with $\tau x_2=-x_2$, $\tau y_2=y_2$. The map $\pi\circ\Phi_s$ can be given as
$x_2=\rho_x(x_1,y_1)$, $y_2=\rho_y(x_1,y_1)$.
The maps $\rho_x,\rho_y$ must satisfy equivariance conditions, so that we can write the local Taylor expansion:
\[\begin{cases}
    \rho_1(x_1,y_1)&=x_1(a_1+a_2 y_1+a_{31}x^2_1+a_{32}y^2_1+\dots)\\
    \rho_2(x_1,y_1)&=b_0+b_1 y_1+b_{21}x^2_1+b_{22}y^2_1+b_{31}y^3_1+b_{32}x^2_1y_1+\dots.
\end{cases}\]
Above, the first number in the subscript denotes the order of the polynomial, while the second number, if present, enumerates
polynomials of the fixed order. We may assume that $b_0=0$.
The condition that $\pi\circ\Phi_s$ has a critical point at $u_0$ implies that $a_1b_1=0$.

\emph{Case 1. $b_1=0$.} We show that this case cannot occur. To see this, choose a nonzero vector $v\in T_{u_0}\Sigma^\tau$. As $\Phi_s$ is an embedding, $v':=D\Phi_s(v)\in T_{u'}(S^3\times[0,1])$ is nonzero, where $u'=\Phi_s(u_0)$. Next, $F_s|_{\Phi_s(\Sigma)}$ has a critical point at $u'$, that is, $T_{u'}(\Phi_s(\Sigma))\subset \ker DF_s$.  Hence, $v'\in\ker DF_s$. Let $t_0=F_s(u_0)$. The fact that $v'\in\ker DF_s$ implies that $v'\in T_{u'}R_{t_0}$. 

The map $\pi\colon R_{t_0}\to\R^2$ is a projection taking the set $R_{t_0}^\tau=\wt{\cL}$ to $\cL=(\R^2)^{\tau}$. In particular, even
if $\pi$ has one-dimensional kernel, $\pi$ restricted to $\wt{\cL}$ is an isomorphism onto its image. As $v'$ is tangent to $\wt{\cL}$,
it follows that $D\pi(v')\neq 0$. That is the original vector $v$ gets mapped to a nonzero vector under $\pi\circ\Phi_s$.

To show that this contradicts $b_1=0$, we note that in the local coordinates we can take $v=(0,1)$. Then, the image of $v$ under
$DH_s$ is $D(\rho_1,\rho_2)$ acted on $(0,1)$, that is $(0,b_1)$.

\smallskip
\emph{Case 2. $a_1=0$.} On counting dimensions, we may assume that $b_1a_2\neq 0$ (the case $b_1a_2=0$ has higher codimension singularity). Up to higher order terms, the map
$(\rho_1,\rho_2)$ is given by $(x_1,y_1)\mapsto(a_2x_1y_1,y_1)$, that is, it contracts the line $y_1=0$ to a single point.

If $D^2H_s$ is definite, the level sets of $H_s$ are approximate ellipses. Each circle intersects the line $\{y_1=0\}$ at two points. The map
$\pi\circ\Phi_s=(\rho_1,\rho_2)$ will glue these points to create a (generically) transverse self-intersection. The diagram is stable, in
the sense that higher order terms of $(\rho_1,\rho_2)$ will not change the picture.
The situation is very similar to the off-axis case: if the critical point is moved above the line $\{y_1=0\}$, small ellipses (level sets of $H_s$)
are disjoint from $\{y_1=0\}$, and so, they are mapped curves with no self-intersections. Once they touch $\{y_1=0\}$, they acquire a cusp and higher values of $H_s$ lead to the circle with self-intersection. 

The situation for saddles is analogous. We are led to the following two loops of moves.
\begin{enumerate}[label=(OM-\arabic*)]
  \item Singular birth on-axis, see Figure~\ref{fig:OMorse1};
  \item Singular saddle on-axis, see Figure~\ref{fig:OMorse2}.
\end{enumerate}
\begin{figure}
  \begin{tikzpicture}
    \node at (-1.5,3){\includegraphics[height=1cm]{pics/reidemeister-20.eps}};
    \node at (-6,0){\includegraphics[width=1cm]{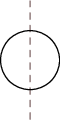}};
    \node at (-3,0){\includegraphics[width=1cm]{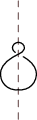}};
    \node at (0,0){\includegraphics[width=1cm]{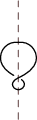}};
    \node at (3,0){\includegraphics[width=1cm]{pics/movie_moves-16.eps}};
    \draw[<->,green!40!black] (0.8,0) -- node[scale=0.8,above,midway]{R-1} (2.2,0);
    \draw[<->,green!40!black] (-0.8,0) -- node[scale=0.8,above,midway]{R-1} (-2.2,0);
    \draw[<->,green!40!black] (-3.8,0) -- node[scale=0.8,above,midway]{R-1} (-5.2,0);
    \draw[<->,blue] (-1.9,3) .. controls ++(-5,0) and ++(-1,0) .. node[scale=0.8,right,midway]{Morse} (-6.7,0);
    \draw[<->,blue] (-1.1,3) .. controls ++(5,0) and ++(1,0) .. node[scale=0.8,right,midway]{Morse} (3.7,0);
  \end{tikzpicture}
\caption{Twisted birth/death on-axis.} \label{fig:OMorse1}
\end{figure}
\begin{figure}
  \begin{tikzpicture}
    \node at (-3,0){\includegraphics[width=1cm]{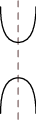}};
    \node at (-6,0){\includegraphics[width=1cm]{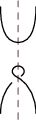}};
    \node at (-3,-3){\includegraphics[width=0.8cm]{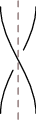}};
    \node at (0,0){\includegraphics[width=1cm]{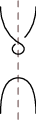}};
    \draw[<->,blue] (-6,-2) .. controls ++(0,-0.5) and ++(-0.5,0) .. (-3.5,-3);
    \draw[<->,blue] (0,-2) .. controls ++(0,-0.5) and ++(0.5,0) .. (-2.5,-3);
    \draw[<->,green!40!black] (-3.8,0) -- node[scale=0.8,above,midway]{R-1} (-5.2,0);
    \draw[<->,green!40!black] (-2.2,0) -- node[scale=0.8,above,midway]{R-1} (-0.8,0);
  \end{tikzpicture}
  \caption{Twisted saddle on-axis. The arrows indicate the Morse moves (saddles). The row consists of Reidemeister moves.} \label{fig:OMorse2}
\end{figure}
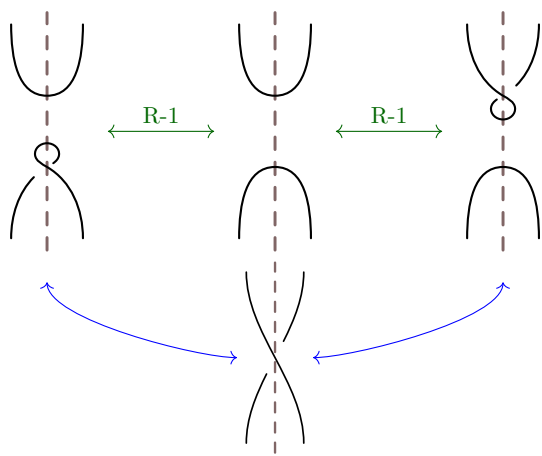

\subsection{Case \ref{item:ml_morse_on_diagram}. Handles crossing the diagram}\label{sub:morse_on_diagram}
We discuss now the situation, where \ref{item:proj2} holds, but \ref{item:proj1} is violated. In the codimension~1 situation, we might have
the following cases.
\begin{enumerate}[label=(DM-\arabic*)]
  \item A birth off-axis over a smooth point of the diagram, see Figure~\ref{fig:DMorse1};
  \item A saddle off-axis over a smooth point of the diagram, see Figure~\ref{fig:DMorse2};
  \item A birth on-axis over a double point, see Figure~\ref{fig:DMorse3};
  \item A saddle on-axis over a double point, see Figure~\ref{fig:DMorse4}.
\end{enumerate}
Notice that neither the birth nor the saddle can possibly occur at the image of $\cS^{\Z_2}$. Indeed, this would mean that two fixed point of the $\tau$ action (the critical point and the point of $\cS^{\Z_2}$) are mapped to the same point on $\cL$. This would imply that they are the same point, violating the assumption that $\phi_s$ is an embedding.

The deformation in each of the four cases consists of moving the image of the critical point off the diagram. In the equivariant setting,
the image of critical point is moved off the diagram along the axis.

\begin{figure}
  \begin{tikzpicture}
    \node at (0,3){\includegraphics[width=1cm]{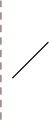}};
    \node at (-3,0){\includegraphics[width=1cm]{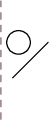}};
    \node at (0,0){\includegraphics[width=1cm]{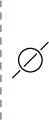}};
    \node at (3,0){\includegraphics[width=1cm]{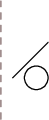}};
    \draw[<->,blue] (-1,3) .. controls ++(-0.5,0) and ++(0,0.5) .. (-3,1);
    \draw[<->,blue] (1,3) .. controls ++(0.5,0) and ++(0,0.5) .. (3,1);
    \draw[<->,green!40!black] (-2.2,0) -- node[scale=0.8,above,midway]{IR-2} (-0.8,0);
    \draw[<->,green!40!black] (2.2,0) -- node[scale=0.8,above,midway]{IR-2} (0.8,0);
  \end{tikzpicture}
\caption{Birth/death over a point off-axis.} \label{fig:DMorse1}
\end{figure}
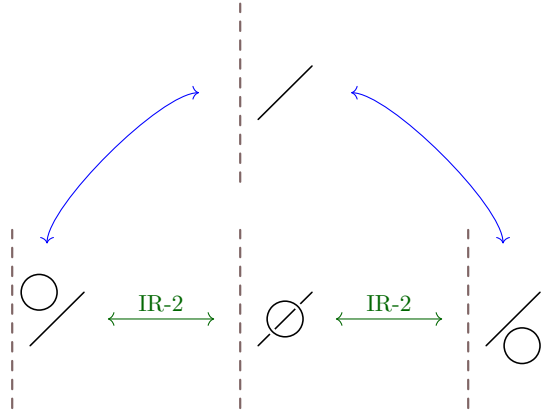
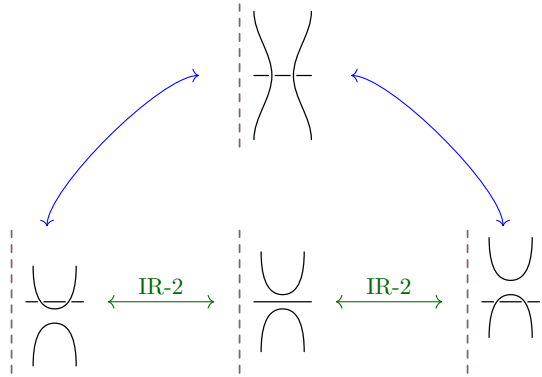
\begin{figure}
  \begin{tikzpicture}
    \node at (0,0){\includegraphics[width=1cm]{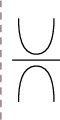}};
    \node at (-3,0){\includegraphics[width=1cm]{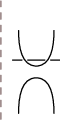}};
    \node at (0,3){\includegraphics[width=1cm]{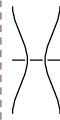}};
    \node at (3,0){\includegraphics[width=1cm]{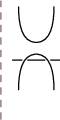}};
    \draw[<->,blue] (-1,3) .. controls ++(-0.5,0) and ++(0,0.5) .. (-3,1);
    \draw[<->,blue] (1,3) .. controls ++(0.5,0) and ++(0,0.5) .. (3,1);
    \draw[<->,green!40!black] (-2.2,0) -- node[scale=0.8,above,midway]{IR-2} (-0.8,0);
    \draw[<->,green!40!black] (2.2,0) -- node[scale=0.8,above,midway]{IR-2} (0.8,0);
  \end{tikzpicture}
\caption{Saddle over a point off-axis. The arrows indicate the Morse move.} \label{fig:DMorse2}
\end{figure}
\begin{figure}
  \begin{tikzpicture}
    \node at (0,3){\includegraphics[width=1cm]{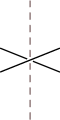}};
    \node at (-3,0){\includegraphics[width=1cm]{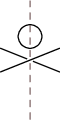}};
    \node at (0,0){\includegraphics[width=1cm]{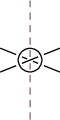}};
    \node at (3,0){\includegraphics[width=1cm]{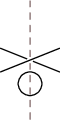}};
    \draw[<->,blue] (-1,3) .. controls ++(-0.5,0) and ++(0,0.5) .. (-3,1);
    \draw[<->,blue] (1,3) .. controls ++(0.5,0) and ++(0,0.5) .. (3,1);
    \draw[<->,green!40!black] (-2.2,0) -- node[scale=0.8,above,midway]{M-1} (-0.8,0);
    \draw[<->,green!40!black] (2.2,0) -- node[scale=0.8,above,midway]{M-1} (0.8,0);
  \end{tikzpicture}
\caption{Birth/death over a point on-axis} \label{fig:DMorse3}
\end{figure}
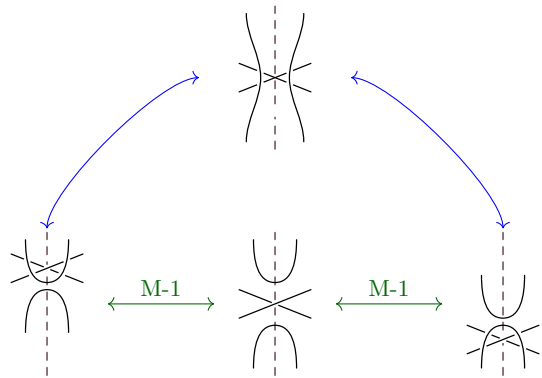
\begin{figure}
  \begin{tikzpicture}
    \node at (0,0){\includegraphics[width=1cm]{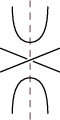}};
    \node at (-3,0){\includegraphics[width=1cm]{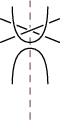}};
    \node at (0,3){\includegraphics[width=1cm]{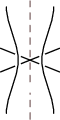}};
    \node at (3,0){\includegraphics[width=1cm]{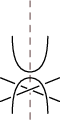}};
    \draw[<->,blue] (-1,3) .. controls ++(-0.5,0) and ++(0,0.5) .. (-3,1);
    \draw[<->,blue] (1,3) .. controls ++(0.5,0) and ++(0,0.5) .. (3,1);
    \draw[<->,green!40!black] (-2.2,0) -- node[scale=0.8,above,midway]{M-1} (-0.8,0);
    \draw[<->,green!40!black] (2.2,0) -- node[scale=0.8,above,midway]{M-1} (0.8,0);
  \end{tikzpicture}
\caption{Saddle over a point on-axis} \label{fig:DMorse4}
\end{figure}
\subsection{Case \ref{item:ml_morse_loop}. Non Morse critical point}\label{sub:morse_loop}
This is the case when \ref{item:cs_morse} is violated. We recall that near $u_i^M(s)\in\Sigma$, the map 
$\Sigma\ni u\mapsto (\pi\Phi_s(u),H_s(u))\in\R^2\times\R$ 
locally parameterizes a surface in $\R^3$. We denote it $\Sigma'$. The projection $\R^2\times\R$ to $\R$ restricted to $\Sigma'$
has a non-Morse critical point $v_i^M$, which is the image of $u_i^M(s)$. For near parameters $s'$, the non-Morse critical points
splits into Morse critical points. That is, on $v_i^M$ we have a codimension~1 singularity of Morse functions.

Codimension~1 non-Morse singularities of $\Z_2$-equivariant functions were studied 
in detail
in \cite{BB}; compare also \cite{Akaho,Akaho2}.
Each such singularity induces a loop of equivariant Morse moves.
Degeneracies are the following:
\begin{enumerate}[label=(NM-\arabic*)]
  \item birth+saddle off-axis, see Figure~\ref{fig:NM1};
  \item birth+saddle on-axis, see Figure~\ref{fig:NM2};
  \item collision of two saddles and a birth, see Figure~\ref{fig:NM3};
  \item collision of two births and a saddle, see Figure~\ref{fig:NM4}.
\end{enumerate}
\begin{figure}
  \begin{tikzpicture}
    \node at (-6,0){\includegraphics[width=1cm]{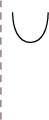}};
    \node at (-3,0){\includegraphics[width=1cm]{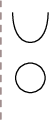}};
    \node at (0,0){\includegraphics[width=1cm]{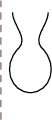}};
  \end{tikzpicture}
  \caption{Birth and saddle off-axis can cancel.} \label{fig:NM1}
\end{figure}
\begin{figure}
  \begin{tikzpicture}
    \node at (-6,0){\includegraphics[width=1cm]{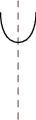}};
    \node at (-3,0){\includegraphics[width=1cm]{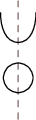}};
    \node at (0,0){\includegraphics[width=1cm]{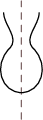}};
  \end{tikzpicture}
  \caption{Birth and saddle on-axis can cancel.} \label{fig:NM2}
\end{figure}
\begin{figure}
  \begin{tikzpicture}
    \node at (-6,0){\includegraphics[width=1cm]{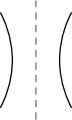}};
    \node at (-3,0){\includegraphics[width=1cm]{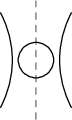}};
    \node at (0,0){\includegraphics[width=1cm]{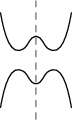}};
  \end{tikzpicture}
  \caption{Collision of critical points. A birth on the axis and a pair of saddles off-axis are equivalent to a saddle on-axis.} \label{fig:NM3}
\end{figure}
\begin{figure}
  \begin{tikzpicture}
    \node at (-6,0){\includegraphics[height=1cm]{pics/reidemeister-20.eps}};
    \node at (-3,0){\includegraphics[width=1cm]{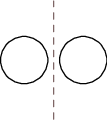}};
    \node at (0,0){\includegraphics[width=1cm]{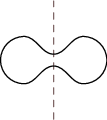}};
    \node at (3,0){\includegraphics[height=1cm]{pics/reidemeister-20.eps}};
  \end{tikzpicture}
  \caption{Collision of critical points. A birth off the axis and a saddle on the axis can be traded for a birth on the axis.} \label{fig:NM4}
\end{figure}

\subsection{Case \ref{item:ml_reid}. Codimension two singularity of the diagram}\label{sub:reid}
Suppose $s_0,t_0$ are such that $t_0$ is a non-critical point of $F\circ\Phi_{s_0}$. Set as usual, $\cS=\Phi_{s_0}(\Sigma)\cap H_{s_0}^{-1}(t_0)$.
Then, by implicit function theorem, there exists a family $\wt{\phi}_{s,t}\colon\cS\to\R^3$ of equivariant maps such that
the image of $\cS$ parameterizes $\Phi_s(\Sigma)\cap H_{s}^{-1}(t)$, and $(s,t)$ are close to $(s_0,t_0)$. We let $\phi_{s,t}$ be the diagrams
obtained by the projection.

It might happen, for finitely many pairs $(s_0,t_0)$ that $\phi_{s_0,t_0}\in\cF^2$, in which case the whole family is an unfolding of a codimension~2 Reidemeister singularity.
This leads to one of the 18 loop (L-1)--(L-18) displayed in Table~\ref{tab:local_loops}.

\subsection{Summary. Loops and movie moves}\label{sub:list}
The discussion of Section~\ref{sec:carter_saito} can be summarized in the following 
list of loops.
\begin{enumerate}[label=(MM-\arabic*)]
  \item Commuting two moves occurring at different places, e.g. equivariant Reidemeister move and a Morse move, that is:\label{item:MM1}
    \begin{itemize}
      \item commuting two Reidemeister moves;
      \item commuting an I-move or an S-move and a Reidemeister move;
      \item commuting an I-move or an S-move with an equivariant Morse move;
      \item commuting two equivariant Morse moves;
      \item commuting a Reidemeister move and an equivariant Morse move;
      \item inserting and removing an I-move or an S-move.
    \end{itemize}
  \item A loop of Reidemeister moves of Theorem~\ref{thm:intro4};\label{item:MM3}
  \item A loop of Reidemeister moves around the point at infinity, leading to:\label{item:MM2}
    \begin{itemize}
      \item The R-1 loop at infinity, see Figure~\ref{fig:RI1};
      \item  \emph{The R-2 loop at infinity, see Figure~\ref{fig:RI2}};\footnote{Italicized items can be deduced from loops of Reidemeister
	moves, but in a non-local way.}
      \item The M-1 loop at infinity, see Figure~\ref{fig:MI1};
      \item \emph{The M-2 loop at infinity, see Figure~\ref{fig:MI2};}
      \item \emph{The M-3 loop at infinity, see Figure~\ref{fig:MI3}.}
    \end{itemize}
  \item A loop associated to a Morse move for a critical point hitting the axis; \label{item:MM25}
    \begin{itemize}
      \item A loop of a birth or a death on the symmetry line, see Figure~\ref{fig:birth_on_line};
      \item A loop of a saddle on the symmetry line, see Figure~\ref{fig:saddle_on_line}.
    \end{itemize}
  \item A loop of equivariant Morse moves around the point at infinity:\label{item:MM4}
    \begin{itemize}
      \item Birth/death at infinity, see Figure~\ref{fig:IMorse1};
      \item Saddle at infinity, see Figure~\ref{fig:IMorse2}.
    \end{itemize}
  \item A loop related to a singular equivariant Morse move:\label{item:MM5}
    \begin{itemize}
      \item Singular birth off-axis, see Figure~\ref{fig:SMorse1};
      \item Singular saddle off-axis, see Figure~\ref{fig:SMorse2}.
      \item Singular birth on-axis, see Figure~\ref{fig:OMorse1};
      \item Singular saddle on-axis, see Figure~\ref{fig:OMorse2}.
    \end{itemize}
  \item A loop related to an equivariant Morse move over another point in the diagram:\label{item:MM6}
    \begin{itemize}
      \item A birth off-axis over a smooth point of the diagram, see Figure~\ref{fig:DMorse1};
      \item A saddle off-axis over a smooth point of the diagram, see Figure~\ref{fig:DMorse2};
      \item A birth on-axis over a double point, see Figure~\ref{fig:DMorse3};
      \item A saddle on-axis over a double point, see Figure~\ref{fig:DMorse4}.
    \end{itemize}
  \item A loop of equivariant Morse moves related to a failure of Morse condition:\label{item:MM7}
    \begin{itemize}
      \item birth+saddle off-axis, see Figure~\ref{fig:NM1};
      \item birth+saddle on-axis, see Figure~\ref{fig:NM2};
      \item collision of two saddles and a birth, see Figure~\ref{fig:NM3};
      \item collision of two births and a saddle, see Figure~\ref{fig:NM4}.
    \end{itemize}
\end{enumerate}
\begin{defn}
  A loop of the above list is called an \emph{elementary loop}.
\end{defn}

The way an elementary loop produces an elementary movie move is the following.
Suppose a loop consists of moves  $\scS_1,\dots,\scS_n$. We choose $i,j=1,\dots,n$ as parameters with $i\le j$. We consider
movies $\scM_1$ consisting of  $\scS_{i+1},\scS_{i+2},\dots,\scS_j$ and $\scM_2$ consisting of $\scS_i,\scS_{i-1},\scS_{i-2}\dots,\scS_1,\scS_n,\dots,\scS_{j+1}$.
Informally, the loop is broken into two sets of movies. We say that $i,j$ are \emph{admissible} if the traces of $\scM_1$ and $\scM_2$
are equivariantly isotopic.
\begin{defn}
  An \emph{elementary movie move} consists of exchanging $\scM_1$ with $\scM_2$ or $-\scM_1$ with $-\scM_2$, where
  $\scM_1,\scM_2$ are two sets of movies coming from admissible $i,j$ and one of the loop in the list \ref{item:MM1}--\ref{item:MM7} above.
  Here the sign `$-$' on the movie denotes the reverse movie.
\end{defn}

\begin{example}
  If $\scS_1,\dots,\scS_n$ is a loop of Reidemeister moves from Table~\ref{tab:local_loops}, then any $i,j$ are admissible.
\end{example}
\begin{example}
  Consider the loop associated with a collision as in Figure~\ref{fig:NM4}. Here $\scS_1$ is a birth off-the axis, $\scS_2$ is a saddle on-axis, $\scS_3$ is the death on-axis. Only the choice $i=0$, $j=2$ is possible: the trace of $\scS_1,\scS_2$ is the disk with two minima, while
  the trace of $-\scS_3$ is a disk with the minimum. All other choices lead to non-isotopic traces.
\end{example}
\begin{example}
  The loop L-1 from Table~\ref{tab:local_loops} induces an elementary movie move replacing (IR-1) by (IR-1) and (IR-2) moves. It also induces
  an elementary movie move replacing one (IR-2) move by two (IR-1) moves.
\end{example}
\begin{example}
  Consider a local loop described in Figure~\ref{fig:SMorse2}. It induces an elementary movie move replacing (IR-1)+Saddle by another 
  type (IR-1)+Saddle;
  see Figure~\ref{fig:SMorse2a}. However,
  it does not induce an elementary movie move replacing two (IR-1) moves by two saddles, because the genera of traces are not the same.
\end{example}

We obtain now a theorem, expanding on Theorem~\ref{thm:intro3}
from the introduction.
\begin{theorem}\label{thm:main_loop}
  Suppose $L_0,L_1$ are two equivariant links in $S^3\times\{0\}$ and $S^3\times\{1\}$, respectively. Suppose $\Phi_s\colon \Sigma\to S^3\times[0,1]$, $s\in[0,1]$ is a family of equivariant cobordisms connecting $L_0$ and $L_1$. Assume that $\Phi_s(\Sigma)$ 
  does not have any isolated fixed points of the action of $\tau$.
  Suppose for $s=0,1$ the maps $\Phi_0,\Phi_1$ satisfy conditions \ref{item:cs_ifirst} -- \ref{item:cs_reid}.
  The diagrams $D_0$ and $D_1$ of $L_0$ and $L_1$
  can be connected by the movies associated with $\Phi_0$ and $\Phi_1$. The two movies can be transformed one into other
  by the elementary movie moves.
\end{theorem}
\begin{remark}
  The assumptions that \ref{item:cs_ifirst} -- \ref{item:cs_reid} are satisfied for $\Phi_0$ and $\Phi_1$ is not
  restrictive: we can guarantee it by a small equivariant perturbation of maps, see Lemma~\ref{lem:o_dense}.
\end{remark}
\begin{proof}[Proof of Theorem~\ref{thm:main_loop}]
  Suppose $\Phi_s$ is a path of equivariant embeddings of $\Sigma$ to $S^3\times[0,1]$. For some values of $s$, conditions \ref{item:cs_ifirst} -- \ref{item:cs_reid} can
  be violated. These singularities have been described in Subsection~\ref{sub:col_deg}.
  In particular, the moves~\ref{item:MM1} are described as consequences of \ref{item:Col1} -- \ref{item:Col2}.

  In Subsection~\ref{sub:col_deg}, we described also other types of violations of conditions \ref{item:cs_ifirst} -- \ref{item:cs_reid},
  namely the list \ref{item:ml_iloop} -- \ref{item:ml_reid}. Item \ref{item:ml_iloop} of that list is handled in Subsection~\ref{sub:iloop}
  and leads to \ref{item:MM2}. Indeed, say the map $\wt{\phi}_{ts}$, at $t=t_0,s=s_0$, has a diagram that is not regular at infinity.
  By the discussion in Subsection~\ref{sub:iloop}, for the loop $\wt{\phi}_{t's'}$, with $t'=t_0+r\cos\gamma$, $s'=s_0+r\sin\gamma$,
  with $r$ sufficiently small, then $\gamma\in[0,2\pi]$ is one of the elementary loops corresponding to \ref{item:MM2}. It follows
  from the same arguments as in the proof of Theorem~\ref{thm:two_loops} indicates that the movies associated to paths
  $t\mapsto \wt{\phi}_{ts_1}$ and $t\mapsto\wt{\phi}_{ts_2}$ for $s_1=s_0-r$, $s_2=s_0+r$ are related by an elementary movie move.

 Analogously, the failure of regularity described in \ref{item:ml_morse_inf} is discussed in Subsection~\ref{sub:ml_morse_inf}. It leads
 to \ref{item:MM4}. Next, case \ref{item:ml_twist_morse} (Subsection \ref{sub:twist_morse}) creates loops in \ref{item:MM5}.
 Item \ref{item:ml_morse_on_diagram} is expanded on in Subsection \ref{sub:morse_on_diagram}, leading to \ref{item:MM6}.  
Item \ref{item:ml_morse_loop} is shown in Subsection \ref{sub:morse_loop} to create loops \ref{item:MM7}. Finally,
the case of Reidemeister loops is discussed in Subsection \ref{sub:reid}, and in the whole Section~\ref{sec:codim2}.
The list of loops is given in item \ref{item:MM3}.
\end{proof}

\appendix
\section{Equivariant Malgrange theorem}\label{sec:app}
The versality Theorem~\ref{thm:versality_theorem} requires  Malgrange preparation
theorem in its proof. As alluded to in Subsection~\ref{sub:versal}, the equivariant version of versality theorem relies
on an equivariant version of Malgrange preparation theorem. That result is known to the experts but not commonly stated in the literature.
For the reader's convenience, we include a short discussion. We begin with the classical, non-equivariant statement.

\begin{thm}[Malgrange preparation theorem, see \expandafter{\cite[Theorem IV.3.6]{GoluGuille}}]\label{thm:mpt}
  Let $X,Y$ be two smooth, finite dimensional manifolds. Let $p\in X$, $q\in Y$ and $f\in C^\infty(X,Y)$ take $p$ to $q$.
  Denote by $\cO_p,\cO_q$ the rings of germs at $p$ (respectively: at $q$) of $C^\infty$ smooth functions from $X$ (respectively: from $Y$)
  to $\R$. Let $\mm_p,\mm_q$ be the maximal ideals.

  A finitely generated $\cO_p$-module $A$ is finitely generated as an $\cO_q$-module if and only if $A/\mm_q A$ is a finite
  dimensional $\R$-vector space.
\end{thm}
Here, $\cO_p$ is regarded as an $\cO_q$ module via the map $f$. That is, $A$ is regarded as a $\cO_q$-module, as well.

In the $\Z_2$-equivariant setting, the statement is analogous.

\begin{thm}[Equivariant Malgrange preparation theorem]\label{thm:empt-y}
  Let $X,Y$ be two smooth, finite dimensional manifolds acted on smoothly by $\Z_2$. Let $p\in X^{\Z_2}$, $q\in Y^{\Z_2}$ 
  and $f\in C^\infty(X,Y)$ be an equivariant function taking $p$ to $q$.
  Denote by $\cE_p,\cE_q$ the rings of germs at $p$ (respectively: at $q$) of $C^\infty$ smooth $\Z_2$-invariant 
  functions from $X$ (respectively: from $Y$)
  to $\R$. Let $\me_p,\me_q$ be the maximal ideals.

  A finitely generated $\cE_p$-module $E$ is finitely generated as an $\cE_q$-module if and only if $E/\mm_q E$ is a finite
  dimensional $\R$-vector space.
\end{thm}
\begin{proof}
  It is enough to suppose that $X,Y$ are vector spaces acted linearly by $\Z_2$.
  Then, $\cE_p$ is the ring of germs of equivariant
  functions.

  Let $\pi_x\colon X\to\R^k$, $\pi_y\colon Y\to\R^\ell$ be the quotient maps by the group action. Let $p'=\pi_x(p)$, $q'=\pi_y(q)$. The equivariant map $f$
  induces a (germ of) smooth map $f'\colon\R^k\to\R^n$

  By a result of Whitney \cite{Whitney}, there is a canonical isomorphism between $\cE_p$ and $\cO_{p'}$. Canonical means that the diagram
  \[
    \begin{tikzcd}
      \cE_q\ar[r,"f^*"]\ar[d] & \cE_p\ar[d]\\
      \cO_{q'}\ar[r,"{f'}^*"] & \cO_{p'}
    \end{tikzcd}
  \]
  commutes. Then $E$ is a finitely generated $\cO_{p'}$-module and it is finitely generated $\cO_{q'}$-module
  if and only if $E/\mm_{q'}E$ is finitely dimensional.
\end{proof}

\bibliographystyle{amsalpha}
\def\MR#1{}
\bibliography{bib}

\typeout{Number of still unresolved questions: \arabic{nparcount}.}
\typeout{Number of resolved questions: \arabic{sparcount}.}

\end{document}